\documentclass[11pt, oneside]{article}   	
\usepackage{geometry}                		
\usepackage[dvipdfmx]{graphicx}
\usepackage[usenames]{xcolor}		
\usepackage{amsmath,amsthm,amssymb}
\usepackage{amscd}
\usepackage[all]{xy}

\usepackage[colorinlistoftodos, backgroundcolor=yellow!10]{todonotes}

\theoremstyle{definition}
\newtheorem{dfn}{Definition}[section]
\newtheorem{rem}[dfn]{Remark}
\newtheorem*{prob}{Problem}
\newtheorem{exple}[dfn]{Example}

\theoremstyle{plain}
\newtheorem{thm}[dfn]{Theorem}
\newtheorem{prop}[dfn]{Proposition}
\newtheorem{lem}[dfn]{Lemma}
\newtheorem{cor}[dfn]{Corollary}

\newcommand{\K}{\mathbb{K}}
\newcommand{\Z}{\mathbb{Z}}

\newcommand{\tDer}{\mathrm{tDer}}

\newcommand{\taut}{\mathrm{tAut}}

\newcommand{\Lie}{\operatorname{Lie}}

\newcommand{\gn}{{(g,n+1)}}
\newcommand{\KV}{\operatorname{KV}}

\newcommand{\pa}{\partial}

\newcommand{\bch}{\mathrm{bch}}

\newcommand{\gr}{\mathrm{gr}}
\newcommand{\ad}{\mathrm{ad}}

\newcommand{\alt}{\mathrm{Alt}}
\newcommand{\rot}{\mathrm{rot}}

\title{The Goldman-Turaev Lie bialgebra and the Kashiwara-Vergne problem in higher genera}
\author{Anton Alekseev\thanks{Department of Mathematics, University of Geneva, 
7-9 rue du Conseil G\'en\'eral, 1211 Geneva 4, Switzerland \texttt{e-mail:anton.alekseev@unige.ch}}, 
Nariya Kawazumi\thanks{Department of Mathematical Sciences, University of Tokyo, 3-8-1 Komaba, Meguro-ku, Tokyo 153-8914, Japan \texttt{e-mail:kawazumi@ms.u-tokyo.ac.jp}},
Yusuke Kuno\thanks{Department of Mathematics, Tsuda University, 2-1-1 Tsuda-machi, Kodaira-shi, Tokyo 187-8577, Japan \texttt{e-mail:kunotti@tsuda.ac.jp}} 
and Florian Naef\thanks{School of Mathematics, Trinity College Dublin, 17 Westland Row, Dublin, Ireland \texttt{e-mail:naeff@tcd.ie}}}
\date{}							

\begin{document}

\maketitle

\begin{abstract}
For a compact oriented surface $\Sigma$ of genus $g$ with $n+1$ boundary components, the space $\mathfrak{g}(\Sigma)$ spanned by free homotopy classes of loops in $\Sigma$ carries the structure of a Lie bialgebra.
The Lie bracket was defined by Goldman \cite{Go86}, and it depends only on the orientation of the surface.
The Lie cobracket was originally defined by Turaev \cite{Tu91}, and the version used in this paper depends on the choice of framing on $\Sigma$.
The Lie bialgebra $\mathfrak{g}(\Sigma)$ has a natural decreasing filtration such that the Goldman bracket and Turaev cobracket have degree $(-2)$.

In this paper, we address the following Goldman-Turaev (GT) formality problem: construct a Lie bialgebra homomorphism $\theta$ from $\mathfrak{g}(\Sigma)$ to its associated graded ${\rm gr}\, \mathfrak{g}(\Sigma)$ such that ${\rm gr} \, \theta = {\rm id}$.
In order to solve it, we define a family of higher genus Kashiwara-Vergne (KV) problems for an element  $F\in {\rm Aut}(L)$, where $L$ is a free Lie algebra.
In the case of $g=0$ and $n=2$, it is the classical KV problem from Lie theory \cite{KV78, AT12}.
For $g>0$, these KV problems are new.

Our main results are as follows.
An element $F \in {\rm Aut}(L)$ induces a GT formality map if and only if it is a solution of the KV problem. 
Solving the KV problem reduces to two important cases: $g=0, n=2$ which admits solutions by \cite{AT12}, and $g=1, n=0$ for which we construct solutions in terms of certain elliptic associators following Enriquez \cite{Enriquez}.
By combining these two results, we obtain a proof of the GT formality for every $g$ and $n$, with the exception of some framings for $g=1$ in which case the GT formality actually does not hold.

Furthermore, we introduce pro-unipotent groups ${\rm KV}$ and ${\rm KRV}$ which act on the space of solutions of the KV problem freely and transitively.
The groups ${\rm KV}$ also acts by automorphisms of (the completion of) the Lie bialgebra $\mathfrak{g}(\Sigma)$. There are injective group homomorphisms ${\rm GT}_1 \to {\rm KV}, {\rm GRT}_1 \to {\rm KRV}$ from Grothendieck-Teichm\"uller groups which depend on the choice of a pair of pants decomposition of $\Sigma$. 

As part of our study, we prove a uniqueness theorem for non-commutative divergence cocycles on the group algebra of a free group which is of independent value. 
As a topological application of our results, we show that the Johnson obstruction given by the Turaev cobracket coincides with the one given by the Enomoto-Satoh trace.
\end{abstract}

\section{Introduction}
\label{sec:intro}

\subsection{Formality and homomorphic expansions}

One often encounters algebraic structures (associative algebras, Hopf algebras, Lie algebras, categories of some class, etc.) equipped with a natural filtration compatible with structure morphisms.
Given such a structure $V$, one can construct a (completed) associated graded ${\rm gr}\, V$ of $V$.
The grading of ${\rm gr}\, V$ gives rise to a filtration of ${\rm gr}\, V$, and the {\em formality problem} arises: 
\[
\text{
Is (the completion of) $V$ isomorphic to ${\rm gr}\, V$?
}
\]
Even if the formality of $V$ holds, in general there is no canonical choice for an isomorphism between $V$ and ${\rm gr}\, V$.
Following the terminology of Bar-Natan and Dancso \cite{BND13}, each choice
\[
\theta \colon V \overset{\cong}{\to} {\rm gr}\, V
\]
with the property ${\rm gr}\, \theta = {\rm id}$ is called a {\em homomorphic expansion}. Furthermore, the automorphism groups ${\rm Aut}(V)$ and ${\rm Aut}({\rm gr}\, V)$ act freely and transitively on the set of homomorphic expansions, 
\[
{\rm Aut}({\rm gr}\, V)
\curvearrowright
\{ \text{homomorphic expansions} \} 
\curvearrowleft
{\rm Aut}(V),
\]
and each choice of $\theta$ induces a group isomorphism ${\rm Aut}(V) \cong {\rm Aut}({\rm gr}\, V)$.

By nature, the associated graded ${\rm gr}\, V$ only captures the lowest order term of the original algebraic structure.
Hence, the formality is expected to lead to better understanding of the original (more complicated) structure $V$.
Not only that, the significance of the concepts of formality and homomorphic expansions stems from the fact that many important mathematical objects fit into this framework: expansions for free groups \cite{Bou71, Lin93, Ka05} which generalize the classical Magnus expansion \cite{Mag35, MKS},
Drinfeld associators and their connection with parenthesized tangles \cite{Dr1, Dr2, BN97, LM96}, and universal finite type invariants for various knotted objects \cite{BND13, BND16, BND17}.

This paper is about formality phenomena in 2-dimensional topology.
The algebraic structure that we will focus on is a Lie bialgebra associated with a (framed) oriented surface, called the {\em Goldman-Turaev Lie bialgebra}. 
Generalizing our previous work \cite{genus0} on surfaces of genus zero, we will establish formality of the Goldman-Turaev Lie bialgebra for surfaces of higher genus with boundary.
A key ingredient in \cite{genus0} was a link between the formality of the Goldman-Turaev Lie bialgebra for surfaces of genus zero and the Kashiwara-Vergne (KV) problem from Lie theory \cite{KV78, AT12}.
Based on this, we introduce
{\em generalized KV problems} for each framed oriented surface with boundary. We show that in most cases (with some notable exceptions in genus one) these new KV problems admit solutions. Last but not least, we explain how solutions of KV problems imply the formality of the Goldman-Turaev Lie bialgebras.

In the rest of this section, we give background on the Goldman-Turaev formality and state our main results.

\subsection{Goldman-Turaev Lie bialgebras and formality problem}
Let $\mathbb{K}$ be a field of characteristic zero 
and let $\Sigma$ be a compact connected oriented 2-manifold of genus $g$ with $n+1$ boundary components, where $g$ and $n$ are non-negative integers.
We label the boundary components by the set $\{ 0, 1, \dots, n\}$ and choose the basepoint $*$ on the boundary component with label $0$. 
For $\pi = \pi_1(\Sigma)$ the fundamental group of $\Sigma$,
the space
$$
\mathfrak{g}(\Sigma) = |\K \pi| := \mathbb{K} \pi/[\mathbb{K}\pi, \mathbb{K}\pi],
$$
where $[\K \pi, \K \pi]$ is the subspace spanned by commutators in $\K\pi$,
is isomorphic to the $\mathbb{K}$-span of homotopy classes of free loops in $\Sigma$.
We denote by $a \mapsto |a|$ the natural projection $\mathbb{K} \pi \to \mathfrak{g}(\Sigma)$.
In \cite{Go86}, Goldman showed that the space $\mathfrak{g}(\Sigma)$ carries the natural structure of a Lie algebra defined in terms of intersections of curves representing elements of $\pi$.

Under the Goldman bracket $[\cdot, \cdot]_{\rm Goldman} \colon \mathfrak{g}(\Sigma) \otimes \mathfrak{g}(\Sigma) \to \mathfrak{g}(\Sigma)$, the homotopy class of a trivial loop ${\bf 1} \in \mathfrak{g}(\Sigma)$ is a central element. Hence, the quotient space 
$$
\tilde{\mathfrak{g}}(\Sigma)=\mathfrak{g}(\Sigma)/\mathbb{K} {\bf 1}
$$
is also a Lie algebra.
In \cite{Tu91}, Turaev showed that $\tilde{\mathfrak{g}}(\Sigma)$ carries the canonical structure of a Lie bialgebra where the Lie bracket is (induced by) the Goldman bracket $[\cdot, \cdot]_{\rm Goldman}$ and the Lie cobracket $\delta_{\rm Turaev} \colon \tilde{\mathfrak{g}}(\Sigma) \to \tilde{\mathfrak{g}}(\Sigma) \otimes \tilde{\mathfrak{g}}(\Sigma)$ is defined in terms of self-intersections of curves on $\Sigma$, see Fig. \ref{fig:GTstructure} for sample calculations of the Goldman bracket and Turaev cobracket.

\begin{figure}
\begin{center}
\input{fig_GTstructure.tex}
\end{center}
\caption{The Goldman bracket and Turaev cobracket.}
\label{fig:GTstructure}
\end{figure}

Let $f\colon T\Sigma \overset{\cong}{\to} \Sigma \times \mathbb{R}^2$ be a framing on $\Sigma$ (a trivialization of the tangent bundle of $\Sigma$). To each immersed free loop $\alpha$ in $\Sigma$, the framing assigns a rotation number,
$$
\alpha \mapsto {\rm rot}^f(\alpha) \in \mathbb{Z}.
$$
A choice of framing defines a Lie bialgebra structure on $\mathfrak{g}(\Sigma)$ with Lie bracket the Goldman bracket and the Lie cobracket $\delta^f_{\rm Turaev} \colon \mathfrak{g}(\Sigma) \to \mathfrak{g}(\Sigma) \otimes \mathfrak{g}(\Sigma)$ which depends on $f$.
This lift of Turaev's Lie cobracket comes from Furuta's observation on a $1$-cocycle of the mapping class group \cite[\S 4]{Mo97}.
For every $f$, this Lie bialgebra structure descends to the canonical Lie bialgebra structure on $\tilde{\mathfrak{g}}(\Sigma)$. 

The Goldman-Turaev (GT) Lie bialgebra $(\mathfrak{g}(\Sigma), [\cdot, \cdot]_{\rm Goldman}, \delta^f_{\rm Turaev})$ carries a natural filtration. In the case of  $n=0$, it can be described as follows: the group algebra $\mathbb{K}\pi$ carries a natural decreasing filtration by powers of the augmentation ideal which descends to a filtration on $\mathfrak{g}(\Sigma)$  (for the general case, see Section~\ref{sec:filt}). Under this filtration, the Goldman bracket and Turaev cobracket  have filtration degree $(-2)$. 

Denote by $\widehat{\mathfrak{g}}(\Sigma)$ the completion of $\mathfrak{g}(\Sigma)$ with respect to the filtration described above. For every framing $f$, the Lie bialgebra structure  $(\mathfrak{g}(\Sigma), [\cdot, \cdot]_{\rm Goldman}, \delta^f_{\rm Turaev})$  extends to  $\widehat{\mathfrak{g}}(\Sigma)$. The associated graded vector space ${\rm gr} \, \widehat{\mathfrak{g}}(\Sigma)$ carries the induced Lie bracket $[ \cdot, \cdot]_{\rm gr}$ and the induced Lie cobracket $\delta^f_{\rm gr}$. Both these operations are of degree $(-2)$.

The vector space ${\rm gr} \, \widehat{\mathfrak{g}}(\Sigma)$ admits a nice explicit description as follows: let
$H:= H_1(\Sigma, \mathbb{K})$
be the first homology group of $\Sigma$ and $A=\widehat{T}(H)$ the completed free associative algebra spanned by $H$. Then,
$$
{\rm gr} \, \widehat{\mathfrak{g}}(\Sigma) \cong |A| := A/[A,A],
$$
where $[A, A] \subset A$ is the (closed) subspace spanned by commutators. 
To be more precise, when both $g$ and $n$ are positive, one needs to replace $H$ with its associated graded space with respect to a certain filtration; see Section~\ref{sec:filt}.
One can also identify ${\rm gr} \, \widehat{\mathfrak{g}}(\Sigma)$ with the completed graded vector space spanned by cyclic words in $H$.
Furthermore, the Lie bialgebra structure on ${\rm gr} \, \widehat{\mathfrak{g}}(\Sigma)$ is identical to the necklace Lie bialgebra structure \cite{BLeB, Ginzburg, Schedler} associated to a certain quiver determined by $g$ and $n$.

The following question naturally arises in the framework that we have described above:

\begin{prob}[Formality problem for Goldman-Turaev Lie bialgebras]
Find a filtration preserving Lie bialgebra isomorphism
$\theta\colon \widehat{\mathfrak{g}}(\Sigma) \to {\rm gr} \, \widehat{\mathfrak{g}}(\Sigma)$
such that its associated graded map is the identity: ${\rm gr} \, \theta= {\rm id}$.
\end{prob}

This question can be addressed for any framing $f$ on $\Sigma$. A positive solution of the GT formality problem for some $f$ implies a positive solution for the canonical GT Lie bialgebra structure on the completion of $\tilde{\mathfrak{g}}(\Sigma)$. 

\subsection{Known results} \label{subsec:known_results}

We briefly describe some of the known results concerning the GT formality. 

First, we need some notation. Among the maps $\theta\colon  \widehat{\mathfrak{g}}(\Sigma) \to {\rm gr}\, \widehat{\mathfrak{g}}(\Sigma)$ there is a class which admits an easy description.
The key is to notice that $\pi$ is a free group of finite rank.
It is well-known that $A = \widehat{T}(H)$ can be identified 
with the associated graded of $\K\pi$ as a (completed) Hopf algebra: $A = {\rm gr}\, \K \pi$.
Let $\theta \colon \widehat{\mathbb{K}\pi} \to A$ be a homomorphism of filtered completed Hopf algebras such that 
$$
\theta(\gamma ) = 1 + [\gamma] + \cdots
$$
for any $\gamma \in \pi$, where $[\gamma] \in H$ is the homology class of $\gamma$ and $\cdots$ stand for higher degree terms. Then, $\theta$ induces an isomorphism of filtered vector spaces $\widehat{\mathfrak{g}}(\Sigma) \to {\rm gr}\, \widehat{\mathfrak{g}}(\Sigma)$ (which we denote by the same letter).

As the rank of $\pi$ is $2g+n$,
one can choose a free generating set $\alpha_i, \beta_i, i=1, \dots, g$ and $\gamma_j, j=1, \dots, n$ of $\pi$ represented by simple curves which only intersect at the basepoint $*$ such that $\gamma_j$ is a loop around the boundary component with label $j$ and the element 
$$
\gamma_0 = 
\prod_{i=1}^g \alpha_i \beta_i {\alpha_i}^{-1} {\beta_i}^{-1} \prod_{j=1}^n \gamma_j
$$
corresponds to the boundary component with label $0$ (with suitable orientation).
Denote by $[\alpha_i]=x_i, [\beta_i]=y_i, i=1, \dots, g$ and $[\gamma_j]=z_j, j=1, \dots, n$ the corresponding generators of $H$.
Then, there is a unique Hopf algebra homomorphism $\theta_{\rm exp} \colon \widehat{\K \pi} \to A$ such that $\theta_{\exp}(\alpha_i) = e^{x_i}, 
\theta_{\rm exp}(\beta_i) = e^{y_i}, \theta_{\rm exp}(\gamma_j)=e^{z_j}$. All other maps $\theta$ can be represented as 
$$
\theta = \theta_F := F \circ \theta_{\rm exp},
$$
where $F \in {\rm Aut}^+(L)$ is an automorphism of the free Lie algebra $L$ with generators $x_i, y_i, z_j$ which induces a Hopf algebra automorphism of $A$ and whose associated graded is the identity.
Our aim will be to encode the properties of the map $\theta_F$ in terms of the properties of the automorphism $F$.

\subsubsection*{Formality for Goldman bracket}

Before studying the GT formality, one can address the following easier question:

\begin{prob}[Formality problem for Goldman Lie algebras]
Find a filtration preserving Lie algebra isomorphism
$\theta\colon \widehat{\mathfrak{g}}(\Sigma) \to {\rm gr} \, \widehat{\mathfrak{g}}(\Sigma)$
such that its associated graded map is the identity: ${\rm gr} \, \theta= {\rm id}$.
\end{prob}

This problem admits a very nice solution which plays an important role in our study:

\begin{thm}   \label{thm:intro_Goldman}
Let $F \in {\rm Aut}^+(L)$ be such that for each $j \in \{1,\ldots,n \}$ there exists a group-like element $f_j \in A$ satisfying $F(z_j) = {f_j}^{-1}z_j f_j$, and 
$$
\theta_F(\gamma_0) = \exp \big(\sum_{i=1}^g [x_i, y_i] + \sum_{j=1}^n z_j \big).
$$
Then, the map $\theta_F$ solves the formality problem for the Goldman Lie algebra of the surface of genus $g$ with $n+1$ boundary components.
\end{thm}

This theorem has several proofs available in the literature: Kawazumi and Kuno \cite{KK14} establish it by viewing the map $\theta_F$ in terms of relative cohomology of Hopf algebras and by interpreting the Goldman bracket using cup products.  The cup product interpretation was also used by Hain \cite{Hain17} to give a proof based on the theory of mixed Hodge structures. In \cite{MT13}, Massuyeau and Turaev give a proof using the theory of non-degenerate Fox pairings. In \cite{Naef}, Naef gives a proof using moment maps in non-commutative Poisson geometry.

\subsubsection*{Goldman-Turaev formality in genus zero}

There are several proofs of the GT formality for genus zero surfaces in the literature. In \cite{Mas15}, Massuyeau establishes it using the Kontsevich integral. In \cite{AN17}, Alekseev and Naef gave a proof over $\mathbb{C}$ using the Knizhnik-Zamolodchikov connection. 

Recall the idea of the proof that we gave in \cite{genus0}. 
We denote by  $a \mapsto |a|$  the natural projection $A \to |A| = A/[A, A]$, and recall that the group ${\rm Aut}^+(L)$ carries a 1-cocycle
$$
{\sf j} \colon {\rm Aut}^+(L) \to |A|.
$$
This cocycle is a non-commutative analogue of the log-Jacobian 1-cocycle in differential geometry. 
The main result of \cite{genus0} is the following theorem:

\begin{thm}
Suppose $g=0$.
Let $F \in {\rm Aut}^+(L)$ such that for each $j \in \{1,\ldots,n \}$ there exists a group-like element $f_j \in A$ satisfying $F(z_j) = {f_j}^{-1}z_j f_j$, 
and 
$$
\theta_F(\gamma_0) =  \exp( z_1 + \cdots + z_n),
\quad
{\sf j}(F) =\sum_{j=1}^n |h(z_j)| - |h(z_1 + \dots + z_n)|,
$$
where $h(s) \in s\mathbb{K}[[s]]$ is a formal power series in one variable (the Duflo function).
Then, the map $\theta_F$ solves the formality problem for Goldman-Turaev Lie bialgebras for a surface of genus zero with $n+1$ boundary components.
\end{thm}

For $n=2$, the conditions on $F$ coincide with the Kashiwara-Vergne (KV) problem in Lie theory \cite{KV78, AT12}. By \cite{AT12}, the KV problem admits solutions parametrized by Drinfeld associators. Furthermore, solutions of the problem for $n>2$ are easily constructed from solutions of the $n=2$ problem.

\subsubsection*{Higher genus Goldman-Turaev formality}

The first solution of the higher genus GT formality problem was announced in \cite{short} and appeared in the previous verison of this paper (arXiv:1804.09566v2).
It uses the higher genus version of the KV problem, and it represents a nontrivial extension to higher genera of the methods of \cite{genus0}. Another proof of the GT formality was given by Hain \cite{Hain_GT} using the theory of mixed Hodge structures.

\subsection{Structure of the paper and main results}

We describe the structure of the paper and give a brief account of the main results.

\vspace{1em}
In Section~\ref{sec:GTL}, we define the Goldman bracket and Turaev cobracket, and we recall their main properties.
In fact, it is more convenient to work with their refinements, the double bracket (in the sense of van den Bergh)
$$
\kappa \colon \mathbb{K}\pi \otimes \mathbb{K}\pi \to \mathbb{K}\pi \otimes \mathbb{K}\pi
$$
and the coaction maps (which depend on the choice of framing on the surface)
$$
\mu^f_r \colon \mathbb{K}\pi \to |\mathbb{K}\pi| \otimes \mathbb{K}\pi, \hskip 0.3cm
\mu^f_l \colon \mathbb{K}\pi \to \mathbb{K}\pi \otimes |\mathbb{K}\pi|.
$$
These operations are defined in a similar manner to $[\cdot, \cdot]_{\rm Goldman}$ and $\delta^f_{\rm Turaev}$ (that is, in terms of intersections and self-intersections of curves) but for based loops.
The advantage of using these maps is that they obey interesting product formulas, and hence they are uniquely determined by their values on generators of $\mathbb{K}\pi$.
An important property of the map $\kappa$ is that it naturally induces a map
\[
|\K \pi| \to {\rm Der}(\K \pi),
\quad
|a| \mapsto \{|a|, \cdot \}_{\kappa}
\]
which can be regarded as a based loop extension of the adjoint action of the Goldman bracket.
Section \ref{sec:GTL} does not contain new results, and its purpose is to fix notation and to serve as a review for convenience of the reader.

\vspace{1em}
In Section~\ref{sec:filt}, we describe the weight filtration on $\mathbb{K}\pi$, its completion $\widehat{\mathbb{K}\pi}$ and its associated graded $A = {\rm gr} \, \mathbb{K} \pi$.
When $n =0$, this filtration coincides with the powers of the augmentation ideal of $\K\pi$.
The Goldman bracket, the Turaev cobracket and the operations $\kappa, \mu^f_r, \mu^f_l$ give rise to the corresponding graded operations on $A$ and $|A|$. The material of Section~\ref{sec:filt} is rather standard. It is needed to fix notation and to establish a setup for the Goldman-Turaev formality problem.

\vspace{1em}
In Section~\ref{sec:div}, we develop a theory of non-commutative divergence and log-Jacobian cocycles.
It can be read as a separate text without relation to the topology of surfaces.  

Denote by $\Gamma$ a finitely generated free group.
The non-commutative divergence on the group algebra $\mathbb{K} \Gamma$ is a certain Lie algebra 1-cocycle
$$
{\sf Div}\colon {\rm Der}(\mathbb{K} \Gamma) \to |\mathbb{K} \Gamma|^{\otimes 2}
$$
defined on the Lie algebra of derivations of $\mathbb{K}\Gamma$.
A priori, the cocycle ${\sf Div}$ depends on the choice of generating set $\gamma=\{ \gamma_1, \dots, \gamma_n\}$ (this is similar to dependence of the ordinary divergence on the coordinate system).
Hence, we are using the notation ${\sf Div}_\gamma$.
It naturally extends to the $1$-cocycle
\[
{\sf Div}_{\gamma} \colon {\rm Der}(\widehat{\K\Gamma}) \to |\widehat{\K\Gamma}|^{\otimes 2}
\]
defined on the Lie algebra of derivations of the completed group algebra $\widehat{\K\Gamma}$.
The main result of Section~\ref{sec:div} establishes an (almost) uniqueness of the non-commutative divergence (for more details, see Theorem~\ref{thm:div_independent}):
\begin{thm}      \label{thm:intro1}
For a finitely generated free group $\Gamma$, the non-commutative divergence cocycle ${\sf Div}_{\gamma}$ is independent of the generating set $\gamma$ modulo the lattice $\Lambda \cong H_1(\Gamma, \mathbb{Z})$.
\end{thm}
This result can be viewed as a non-commutative counterpart of the uniqueness of left or right-invariant Haar measures on locally compact topological groups.
The group algebra $\K\Gamma$ has a Hopf algebra structure with the coproduct $\Delta(x) = x\otimes x$ and antipode $\iota(x) = x^{-1}$ for $x \in \Gamma$, which naturally extends to the completion $\widehat{\K\Gamma}$.
This defines a Lie subalgebra ${\rm Der}(\widehat{\K\Gamma}, \Delta) \subset {\rm Der}(\widehat{\K\Gamma})$ of Hopf derivations of $\widehat{\K\Gamma}$.
When restricted to ${\rm Der}(\widehat{\K\Gamma}, \Delta)$, the cocycle ${\sf Div}_{\gamma}$ takes the form
\[
{\sf Div}_{\gamma} = \tilde{\Delta} \circ {\sf div}_{\gamma},
\]
where ${\sf div}_{\gamma} \colon {\rm Der}(\widehat{\K\Gamma}, \Delta) \to |\widehat{\K\Gamma}|$ is a certain Lie algebra $1$-cocycle and $\tilde{\Delta}\colon |\widehat{\K\Gamma}| \to |\widehat{\K\Gamma}|^{\otimes 2}$ is the map induced from the twisted coproduct $({\rm id}\otimes \iota)\circ \Delta$. 
The cocycle ${\sf div}_{\gamma}$ integrates to the group $1$-cocycle ${\sf j}_{\gamma}\colon {\rm Aut}^+(\widehat{\K\Gamma}, \Delta) \to |\widehat{\K\Gamma}|$ on the group of Hopf algebra automorphisms of $\widehat{\K\Gamma}$ which serves as a non-commutative counterpart of the log-Jacobian cocycle.
It gives rise to the corresponding graded $1$-cocycle
\[
{\sf j}_{\rm gr}\colon {\rm Aut}^+(L) \to |A|.
\]

\vspace{1em}
In Section~\ref{sec:Turaev}, we establish a link between the topologically defined operations on $\mathbb{K}\pi$, where $\pi$ is a fundamental group of an oriented surface, and the non-commutative divergence map.
To simplify the presentation, in the rest of the introduction we assume that the surface $\Sigma$ has only one boundary component, namely $n=0$.
The first main result of this section is as follows (see Theorem~\ref{thm:Div^f}):
\begin{thm}       \label{thm:intro2}
    Let $\pi=\pi_1(\Sigma)$.
    Then, for each framing $f$ on $\Sigma$ there is a canonical Lie algebra 1-cocycle ${\sf Div}^f \colon {\rm Der}(\widehat{\mathbb{K}\pi}) \to |\widehat{\mathbb{K} \pi}|^{\otimes 2}$ which is cohomologous to ${\sf Div}_{\gamma}$.
\end{thm}
This result eliminates an ambiguity in Theorem \ref{thm:intro1} by adding topological information on the framing of the surface.
When restricted to ${\rm Der}(\widehat{\K\pi},\Delta)$, the cocycle ${\sf Div}^f$ takes the form ${\sf Div}^f = \tilde{\Delta} \circ {\sf div}^f$ where ${\sf div}^f \colon {\rm Der}(\widehat{\K\pi}, \Delta) \to |\widehat{\K\pi}|$ is a certain Lie algebra $1$-cocycle.
The $1$-cocycle ${\sf div}^f$ integrates to a canonical group $1$-cocycle
\[
{\sf j}^f\colon {\rm Aut}^+(\widehat{\K\pi}, \Delta) \to |\widehat{\K\pi}|
\]

The canonical cocycle ${\sf Div}^f$ is used in the following theorem (for a more precise statement, see Theorem~\ref{prop:delta=c_sigma}):
\begin{thm}      \label{thm:intro3}
Let $f$ be a framing on $\Sigma$.
Then,
\begin{equation} \label{eq:intro_key}
\delta^f_{\rm Turaev}(|a|) = {\sf Div}^f( \{ |a|, \cdot \}_{\kappa}).
\end{equation}
\end{thm}

The equation \eqref{eq:intro_key} is of crucial importance for constructions in the paper.
While $\kappa$ and $\delta^f_{\rm Turaev}$ have topological definitions in terms of intersections and self-intersections of curves,
the canonical cocycle ${\sf Div}^f$ does not have a direct topological definition.

Theorem~\ref{thm:intro3} follows from a more detailed statement (Theorem~\ref{thm:mu=ch}) which expresses the coaction maps
$\mu^f_r, \mu^f_l$ in terms of the 1-cocycle ${\sf Div}^f$ and the double bracket $ \kappa$. 
To prove this statement, we develop some machinery to construct an algebraic counterpart to the coaction maps $\mu^f_r, \mu^f_l$. 
The treatment in full generality (that is, the case of arbitrary $g, n$ and framing) is involved, and technical parts in Sections~\ref{sec:div} and \ref{sec:Turaev} are devoted to this purpose.
We make use of the multiplicative property of the coaction maps and their algebraic counterpart, and the argument is reduced to comparing these maps on generators of $\pi$. 

\vspace{1em}
In Section~\ref{sec:expansions+KV}, we establish a link between the Goldman-Turaev formality and a certain algebraic problem (the Kashiwara-Vergne problem). 
As in Section~\ref{subsec:known_results}, we let $\alpha_1, \dots, \alpha_g, \beta_1, \dots, \beta_g$ be a generating set of $\pi=\pi_1(\Sigma)$ such that $\gamma_0 = \prod_{i=1}^g \alpha_i \beta_i {\alpha_i}^{-1} {\beta_i}^{-1}$ corresponds to the loop around the boundary of $\Sigma$.
The homology classes $x_i = [\alpha_i], y_i = [\beta_i] \in H$ constitute a generating set for the free assosiative algebra $A = \widehat{T}(H)$ and the corresponding free Lie algebra $L$.
Denote by $\theta_{\rm exp}\colon \widehat{\mathbb{K}\pi} \to A$ the map defined on generators by exponential formulas $\theta_{\exp}(\alpha_i) = e^{x_i}, \theta_{\exp}(\beta_i) = e^{y_i}$.
We define elements ${\bf r}, {\bf p}^f \in |A|$ by 
$$
{\bf r} = \sum_{i=1}^g |r(x_i) + r(y_i)|,
\qquad 
{\bf p}^f = \sum_{i=1}^g |{\rm rot}^f(\alpha_i) y_i - {\rm rot}^f(\beta_i) x_i|.
$$
Here, $r(s) = \log((e^s-1)/s)$.
The element ${\bf p}^f$ encodes information on the framing.
Also, define elements $\xi, \omega \in L$ by 
\[
\xi = \log \big( \prod_{i=1}^g [e^{x_i}, e^{y_i}] \big),
\qquad \omega = \sum_{i=1}^g [x_i, y_i].
\]
(Notice that $\xi = \log \theta_{\exp}(\gamma_0)$.)

\begin{dfn}     \label{dfn:KV_intro}
    An element $F \in {\rm Aut}^+(L)$ is a solution of the Kashiwara-Vergne (KV) problem associated to the surface $\Sigma$ with framing $f$ if
\[
\begin{array}{ll}
{\rm KV}\, {\rm I} :& F(\xi) = \omega, \\
{\rm KV}\, {\rm II}:&  \text{${\sf j}_{\rm gr}(F) + F({\bf r} - {\bf p}^f) = |h (\omega)|$ for some $h \in s^2\mathbb{K}[[s]]$}.
\end{array}
\]
\end{dfn}

For a more general form of this definition which applies to surfaces with many boundary components, see Definition \ref{dfn:KV_problem}. In particular, in the case of a sphere with three boundary components and the framing ${\rm rot}^f(\gamma_1)={\rm rot}^f(\gamma_2)=-1$, one obtains the classical KV problem from Lie theory in the formulation of \cite{AT12}. This observation explains the terminology.

As Theorem~\ref{thm:intro_Goldman} shows, the first equation ${\rm KV}\, {\rm I}$ for $F$ implies that $\theta_F = F \circ \theta_{\exp}$ solves the formality problem for the Goldman bracket.
It turns out that assuming ${\rm KV}\, {\rm I}$ the second equation ${\rm KV}\, {\rm II}$ corresponds to the formality of the Turaev cobracket (Proposition~\ref{prop:homomorphic=center}).
Here, the formula~\eqref{eq:intro_key} plays a significant role.
The importance of the two equations ${\rm KV}\, {\rm I}$ and ${\rm KV}\, {\rm II}$ combined stems from the following result:
\begin{thm}     \label{thm:intro4}
Let $\theta_F =F \circ \theta_{\rm exp}\colon \widehat{\mathbb{K} \pi} \to A$ with   
$F \in {\rm Aut}^+(L)$. Then, $\theta_F$ induces an isomorphism of Lie bialgebras
$$
\theta_F \colon (\mathfrak{g}(\Sigma), [\cdot, \cdot]_{\rm Goldman}, \delta^f_{\rm Turaev}) \cong
({\rm gr}\,  \mathfrak{g}(\Sigma), [\cdot, \cdot]_{\rm gr}, \delta^f_{\rm gr}) 
$$
if and only if the element $F$ is a solution of the Kashiwara-Vergne problem for the surface $\Sigma$ with framing $f$ (composed with conjugation by a group-like element in $A$).
\end{thm}

For a more precise statement of Theorem \ref{thm:intro4}, see Theorem \ref{thm:GThomomorphic} in the body of the paper.

\begin{rem}
If $\Sigma$ has more than one boundary component, we actually need to work with automorphisms of $L$ {\em tangential} to $z_j$, namely $F \in {\rm Aut}^+(L)$ such that $F(z_j)$ is conjugate to $z_j$ by a group-like element in $A$, as in Theorem~\ref{thm:intro_Goldman}.
Accordingly, we will need to introduce various tangential objects (automorphisms, derivations and $1$-cocycles on them), and we use these to formulate the KV problem for any framed surface with nonempty boundary.
\end{rem}

\vspace{1em}
In Section~\ref{sec:solveKV}, we discuss solutions of the KV problem. The main result of this section (and of the paper) is as follows:
\begin{thm} \label{thm:intro5}
Let $\Sigma$ be a surface of genus $g \geq 1$ with $n+1$ boundary components and framing $f$. The corresponding Kashiwara-Vergne problem admits solutions if and only if

\begin{enumerate}
\item[$(i)$] $g\geq 2$ and $f$ is an arbitrary framing on $\Sigma$;

\item[$(ii)$] $g=1$ and ${\rm rot}^f(\gamma_j) \neq -1$ for at least one $j \in \{1, \ldots, n\}$;

\item[$(iii)$] $g=1, {\rm rot}^f(\gamma_j) = -1$ for all $j \in \{1, \dots, n\}$ and ${\rm rot}^f(\alpha_1) ={\rm rot}^f(\beta_1)=0$.
\end{enumerate}
\end{thm}

For a more precise statement of this result, see Theorem~\ref{thm:KVsolve}. Note that the genus zero KV problem has been settled in \cite{AT12, AET}. 
The strategy of the proof of Theorem~\ref{thm:intro5} is as follows: it turns out that one can reduce the higher genus KV problem for arbitrary $g$ and $n$ to two special cases. The first one is the case of $g=0$ and $n=2$ (this is the classical KV problem mentioned above), and the second one is the case of $g=1$ and $n=1$. This latter case admits  solutions in terms of special elliptic associators defined by Enriquez \cite{Enriquez}.

\vspace{1em}
In Section~\ref{sec:uniqueness}, we consider symmetries of the KV problem. 
We will denote the set of solutions of the KV problem with framing $f$ by ${\rm SolKV}^f$.
The following definition is a somewhat simplified version of Definition \ref{dfn:KV_KRV}:
\begin{dfn}
For a surface $\Sigma$ of genus $g$ with one boundary component equipped with a framing $f$ we define the groups
$$
\begin{array}{lll}
{\rm KV}^f & = & \{ G \in {\rm Aut}^+(L);
\text{$G(\xi) = \xi$, ${\sf j}^f(G) = |h(\xi)|$ for some $h\in s^2\mathbb{K}[[s]]$} \}, \\
{\rm KRV}^f & = & \{ G \in {\rm Aut}^+(L);
\text{$G(\omega) = \omega$, ${\sf j}^f_{\rm gr}(G) =  |h(\omega)|$ for some $h\in s^2\mathbb{K}[[s]]$} \}.
\end{array}
$$
Here, we regard ${\sf j}^f$ as a $1$-cocycle ${\rm Aut}^+(L) \to |A|$ through the isomorphisms $|\widehat{\K \pi}| \cong |A|$ and ${\rm Aut}^+(\widehat{\K \pi}, \Delta) \cong {\rm Aut}^+(L)$ induced from the map $\theta_{\exp}$.
The map ${\sf j}^f_{\rm gr}$ is the associated graded map of ${\sf j}^f$. For $n=0$, it is independent of framing and coincides with ${\sf j}_{\rm gr}$.
\end{dfn}
The following theorem explains the relation between solutions of the KV problem and the Kashiwara-Vergne groups ${\rm KV}^f, {\rm KRV}^f$:
\begin{thm}      \label{thm:intro6}
   Assume that the set of solutions ${\rm SolKV}^f$ of the Kashiwara-Vergne problem for a surface $\Sigma$ with framing $f$ is nonempty. Then, it carries a free and transitive action of the group ${\rm KV}^f$ by right translations and of the group ${\rm KRV}^f$ by left translations:
   \[
{\rm KRV}^f
\curvearrowright
{\rm SolKV}^f
\curvearrowleft
{\rm KV}^f.
\]
Every element $F \in {\rm SolKV}^f$ defines a group isomorphism ${\rm KRV}^f \to {\rm KV}^f$ given by formula $G \mapsto F^{-1}GF$.
\end{thm}

For a more precise statement of Theorem \ref{thm:intro6}, see Theorem \ref{thm:KV_KRV}.
The Kashiwara-Vergne groups serve as automorphism groups of topologically defined objects:
\begin{thm}      \label{thm:intro7}
   The groups ${\rm KV}^f$ and ${\rm KRV}^f$ are isomorphic to automorphism groups of the following natural structures:
   $$
{\rm KV}^f \cong {\rm Aut}^+(\widehat{\mathbb{K}\pi}, \Delta, \kappa, \delta^f_{\rm Turaev}), \hskip 0.3cm
{\rm KRV}^f \cong {\rm Aut}^+(L, \kappa_{\rm gr}, \delta_{\rm gr}).
$$
Here ${\rm Aut}^+(\widehat{\mathbb{K}\pi}, \Delta, \kappa, \delta^f_{\rm Turaev})$ is the group of positive Hopf automorphisms of the completed group algebra $\widehat{\mathbb{K} \pi}$ preserving the double bracket $\kappa$ on $\widehat{\mathbb{K} \pi}$ and the cobracket $\delta^f_{\rm Turaev}$ on $|\widehat{\mathbb{K} \pi}|$, and ${\rm Aut}^+(L, \kappa_{\rm gr}, \delta_{\rm gr})$ corresponds to the graded versions of all the structures.
\end{thm}
For a more precise statement of Theorem \ref{thm:intro7}, see  Theorem \ref{thm:KRV_acts} in the main body of the paper. In general, not much is known about the structure of the groups ${\rm KV}^f$ and ${\rm KRV}^f$. However, the following is true:
\begin{thm}      \label{thm:intro8}
 There is a family of injective group homomorphisms from Grothendieck-Teichm\"uller groups to Kashiwara-Vergne groups parametrized by trivalent oriented rooted trees $T$:
 $$
\nu^T\colon {\rm GRT}_1 \to {\rm KRV}^f, \hskip 0.3cm
\tilde{\nu}^T\colon {\rm GT}_1 \to {\rm KV}^f.
 $$
\end{thm}
For a more detailed statement of Theorem \ref{thm:intro8} including the conditions on trivalent trees $T$, see Theorem \ref{thm:treeGTKV}.

\vspace{1em}
In Section~\ref{sec:Johnson_obstruction}, we give an application of the study of the Goldman-Turaev formality to surface automorphisms.
Let $\Sigma$ be a surface of genus $g \geq 1$ with one boundary component, i.e., $n=0$.
Denote by $\mathcal{I}$ the Torelli group of $\Sigma$ (that is, the subgroup of the mapping class group acting trivially on $H=H_1(\Sigma)$).
Fix a framing $f$ on $\Sigma$ and let $\mathcal{I}^f \subset \mathcal{I}$ be the framed Torelli group (that is, the subgroup of the Torelli group preserving $f$).

The group $\mathcal{I}$ carries a decreasing filtration called the Johnson filtration $\mathcal{I}(k)$, and the framed Torelli group $\mathcal{I}^f$ (for any framing $f$) contains its second term  $\mathcal{I}(2) \subset \mathcal{I}^f$. The Torelli group $\mathcal{I}$ naturally acts on the fundamental group $\pi = \pi_1(\Sigma)$ which carries the weight filtration (in the case of $n=0$, this filtration coincides with the lower central series of $\pi$).
Taking the associated graded of the action of $\mathcal{I}$ on $\pi$ under these two filtrations gives rise to a graded Lie algebra homomorphism called the Johnson homomorphism:
$$
\tau = (\tau_k)_{k \ge 1}: {\rm Lie}_{\rm gr}(\mathcal{I}) \to {\rm Der}^+_\omega(L).
$$
Here, $L$ is the free Lie algebra spanned by $H$, $\omega=\sum_{i=1}^g [x_i, y_i]$ is the canonical element corresponding to the boundary loop, and ${\rm Der}^+_{\omega}(L)$ is the Lie algebra of positive derivations annihilating $\omega$.

Following the pioneering work of Morita \cite{MoAQ}, Enomoto and Satoh \cite{ES} introduced the trace map
$$
{\sf ES}: {\rm Der}^+_\omega(L) \to |A|
$$
with the property that ${\sf ES} \circ \tau=0$ on the degree $\ge 2$ part of ${\rm Lie}_{\rm gr}(\mathcal{I})$. Hence, the image of the Johnson homomorphism is contained in the kernel of the Enomoto-Satoh trace
\begin{equation}   \label{eq:intro_tau_ES}
{\rm im}(\tau_{k}) \subset {\rm ker}({\sf ES}_k), \quad k \ge 2.
\end{equation}
In this way, the nonvanishing Enomoto-Satoh trace is an obstruction for a derivation to be in the image of the Johnson homomorphism.

Since $\mathcal{I}^f$ is a subgroup of the mapping class group, its action preserves intersections and self-intersections of curves on $\Sigma$. Then, it is (relatively) easy to establish the following facts: for each $\varphi \in \mathcal{I}^f$
we have $\varphi(\xi) = \xi$ and 
\begin{equation} \label{eq:intro_j^f}
{\sf j}^f(\varphi)=0.
\end{equation}
In particular, this implies that $\mathcal{I}^f \subset {\rm KV}^f$.
Furthermore, it is easy to see that for $\varphi \in \mathcal{I}^f(k)$
(that is, $\varphi$ is of degree $k$ under the Johnson filtration) vanishing of the lowest degree part of ${\sf j}^f(\varphi)$ is equivalent to vanishing of the Enomoto-Satoh trace on the graded derivation associated to $\varphi$. This observation yields a new proof of equation~\eqref{eq:intro_tau_ES}. And one can conclude that potentially equation~\eqref{eq:intro_j^f} is a strengthening of equation~\eqref{eq:intro_tau_ES}.

As an application of our results, we show that unfortunately this is not the case:
\begin{thm}       \label{thm:intro9}
    Equation \eqref{eq:intro_j^f} is equivalent to the Enomoto-Satoh obstruction \eqref{eq:intro_tau_ES}.
\end{thm}
This result clarifies the topological meaning of the Enomoto-Satoh obstruction. Namely, it precisely captures the fact that any diffeomorphism of the framed surface $\Sigma$ preserves the loop operations $\kappa, \mu^f_r$ and $\mu^f_l$.
For a more precise statement of Theorem~\ref{thm:intro9}, see Theorem~\ref{thm:JohnsonES} and Theorem~\ref{thm:grKer=Kergr}.

\vspace{1em}
In Appendix, we give technical details on the facts that we use in the main body of the text.
Among other things, in Appendix~\ref{subsec:bra_kappa_gr} we give a purely algebraic proof of the result describing the center of the graded Lie algebra $(\widehat{\mathfrak{g}}(\Sigma), [\cdot, \cdot]_{\rm gr})$. A similar result (with a different proof) was first established in \cite{CBEG07}. By formality of Goldman Lie algebras, this result implies the description of the center of the completed Goldman Lie algebra.

\vskip 1em
\noindent \textbf{Acknowledgements.}
We are grateful to R. Hain and V. Turaev for interesting remarks and suggestions.
Research of A.A. was supported in part by the grant MODFLAT of the European Reseacrh Council (ERC), by the grants number 200400 and 208235 of the Swiss National Science Foundation
(SNSF) and by the NCCR SwissMAP of the SNSF.
F.N. was supported by the Early Postdoc.Mobility grant of the SNSF.
N.K. was supported in part by the grant JSPS KAKENHI 15H03617, 20H00115 and 22H01120.
Y.K. was supported in part by the grant JSPS KAKENHI 26800044, 18K03308 and 23K03121.

\tableofcontents

\subsection*{Convention.}

Here we collect our convention used in the rest of the paper.

\begin{itemize}
\item 
A {\em filtration} on a $\K$-vector space $V$ is meant a decreasing filtration $\cdots V(-1) \supset V(0) \supset V(1) \supset \cdots$ of $V$ indexed by integers which is bounded below by $V$, i.e., $V = V(l)$ for some $l \in \Z$.
The filtration $\{ V(m) \}_m$ is called {\em complete} if the natural map $V \to \widehat{V}:=\varprojlim_m V/V(m)$ is an isomorphism.
The space $\widehat{V}$ is called the {\em completion} of $V$ and is equipped with a complete filtration defined by the kernel of the projections $\widehat{V} \to V/V(m)$, $m \in \Z$.
The {\em associated graded} of a filtered vector space $V$ is defined to be the complete vector space
\[
{\rm gr}\, V := \prod_k V(k)/V(k+1)
\]
which is filtered by the subspaces $\prod_{m \ge k} V(k)/V(k+1)$, $m \in \Z$.

\item
If $V$ and $W$ are complete $\K$-vector spaces, the tensor product $V \otimes W$ is naturally filtered: the $m$th term of the filtration is the sum of the supspaces $V(k) \otimes W(l)$, where $k$ and $l$ run through integers such that $k + l = m$.
The corresponding completion of $V \otimes W$, usually denoted by $V \widehat{\otimes} W$, is called the {\em complete tensor product} of $V$ and $W$.
To simplify the notation, we always use the symbol $\otimes$ instead of $\widehat{\otimes}$ for complete tensor product.
    \item Let $V$ and $W$ be (complete) $\K$-vector spaces and $v\in V, w\in W$.
We denote
\[
(v\otimes w)^{\circ} := w\otimes v \in W\otimes V.
\]

\item 
Let $V$ be a (finite dimensional) $\K$-vector space.
The complete tensor algebra generated by $V$ is defined to be $\widehat{T}(V) = \prod_{k=0}^{\infty} H^{\otimes k}$. If $V$ is graded, then $\widehat{T}(V)$ is naturally graded.

\item 
We denote by $A(z_1,\ldots,z_n)$ the completed free associative algebra generated by indeterminates $z_1, \ldots, z_n$.
It is identified with the ring $\K\langle \langle z_1, \ldots, z_n \rangle \rangle$ of formal power series in $z_1, \ldots, z_n$ and also with the complete tensor algebra $\widehat{T} ( \bigoplus_{i=1}^n \K z_i )$.
We denote by $L(z_1,\ldots,z_n)$ the completed free Lie algebra generated by $z_1, \ldots, z_n$.
\end{itemize}

\begin{itemize}
\item
Let $\mathfrak{A}$ be a (topological) associative $\K$-algebra with unit.
We denote by $[\mathfrak{A},\mathfrak{A}]$ the closure of the $\K$-span of the elements of the form $ab-ba$ with $a,b\in \mathfrak{A}$.
The {\em trace space} of $\mathfrak{A}$ is defined to be the $\K$-linear space
\[
|\mathfrak{A}|:=\mathfrak{A}/[\mathfrak{A},\mathfrak{A}].
\]
We use the symbols $|\cdot |$ or ${\rm Tr}$ for the natural projection $\mathfrak{A}\to |\mathfrak{A}|$.
We use the following shorthand notation for the class of the unit:
\[
{\bf 1} := |1| = {\rm Tr}(1) \in |\mathfrak{A}|.
\]
Adopting this notation, the underlying space of the Goldman-Turaev Lie bialgebra is given by $\mathfrak{g}(\Sigma) = | \K \pi|$.

The trace space of a (complete) tensor product is computed as follows: $|\mathfrak{A} \otimes \mathfrak{B}| = |\mathfrak{A}| \otimes |\mathfrak{B}|$.
To see this, notice that $[a\otimes b, a' \otimes b'] = [a,a'] \otimes bb' + aa' \otimes [b,b']$ for any $a,a'\in \mathfrak{A}$ and $b,b' \in \mathfrak{B}$.
Hence the projection $a\otimes b \mapsto |a|\otimes |b|$ induces a map $|\mathfrak{A} \otimes \mathfrak{B}| \to |\mathfrak{A}| \otimes |\mathfrak{B}|$.
Since $[a,a'] \otimes b = [a\otimes b, a' \otimes 1]$ and $a \otimes [b,b'] = [a\otimes b, 1\otimes b']$, this map has a well-defined inverse $|\mathfrak{A}| \otimes |\mathfrak{B}| \to |\mathfrak{A}\otimes \mathfrak{B}|, |a|\otimes |b| \mapsto |a\otimes b|$.
\end{itemize}

In our study, the theory of double brackets in the sense of van den Bergh \cite{vdB08} plays a key role.
Here we briefly outline basic constructions to fix our convention.

\begin{dfn}
\label{dfn:db}
Let $\mathfrak{A}$ be an associative $\K$-algebra with unit. 
A $\K$-linear map $\Pi\colon \mathfrak{A}\otimes \mathfrak{A}\to \mathfrak{A}\otimes \mathfrak{A}$ is called a \emph{double bracket} on $\mathfrak{A}$ if for any $a,b,c\in \mathfrak{A}$,
\begin{align}
\Pi(a,bc) & = \Pi(a,b) (1\otimes c) + (b\otimes 1) \Pi(a,c), \label{eq:a,bc} \\
\Pi(ab,c) & = \Pi(a,c) (b\otimes 1) + (1\otimes a) \Pi(b,c). \label{eq:ab,c}
\end{align}
\end{dfn}

The following notation of Sweedler type is sometimes convenient:
\[
\Pi(a,b)=\Pi(a,b)'\otimes \Pi(a,b)''.
\]
Given a double bracket $\Pi$, we can consider the following auxiliary operations.

\begin{enumerate}
\item[(i)]
The map
\[
(a,b) \mapsto \Pi(a,b)'\Pi(a,b)''=:\{|a|,b\}_{\Pi} \in \mathfrak{A}
\]
descends to a map from $|\mathfrak{A}|\otimes \mathfrak{A}$, and is a derivation in the second variable.
Namely,
\[
\{|a|,b_1b_2\}_{\Pi}=\{|a|,b_1\}_\Pi \, b_2+b_1 \, \{|a|,b_2\}_\Pi,
\quad
b_1,b_2\in \mathfrak{A}.
\]
This can be summarized in the following map:
$$
\sigma^\Pi\colon |\mathfrak{A}| \to {\rm Der}(\mathfrak{A}),
\quad |a| \mapsto
\sigma^{\Pi}(|a|) = \{ |a|, \cdot \}_{\Pi}.
$$

\item[(ii)]
The map
\[
(a,b) \mapsto \Pi(a,b)''\Pi(a,b)'=:\{a,|b|\}_{\Pi} \in \mathfrak{A}
\]
descends to a map from $\mathfrak{A}\otimes |\mathfrak{A}|$, and is a derivation in the first variable.
Namely,
\[
\{a_1a_2,|b|\}_{\Pi} =\{a_1,|b|\}_\Pi \, a_2+ a_1\, \{ a_2,|b|\}_\Pi,
\quad
a_1,a_2\in \mathfrak{A}.
\]

\item[(iii)]
The maps (i) and (ii) descend to the same map $|\mathfrak{A}|\otimes |\mathfrak{A}|\to |\mathfrak{A}|$, given by the formula
\[
|a|\otimes |b| \mapsto \{|a|,|b|\}_{\Pi}:=
|\{|a|,b\}_{\Pi}|=|\{a,|b|\}_{\Pi}| \in |\mathfrak{A}|.
\]
\end{enumerate}

All the three constructions above are called the \emph{bracket} associated with $\Pi$.

\section{Goldman-Turaev Lie bialgebra}
\label{sec:GTL}

In this section, we define the Goldman-Turaev Lie bialgebra and its upgrades.
In the presentation, we mainly follow \cite{Go86}, \cite{Tu91}, \cite{KK15}, \cite{MT14}. When we need to modify the proofs, we explain the necessary changes.

Let $\Sigma$ be a compact oriented surface (that is, a smooth connected oriented 2-manifold) with nonempty boundary $\partial \Sigma$.
Fix a basepoint $*\in \pa \Sigma$.
We consider the group algebra $\K\pi$ of the fundamental group
$\pi:=\pi_1(\Sigma,*)$, as well as its quotient $\K$-vector space
\[
|\K\pi| =\K\pi/[\K\pi,\K\pi].
\]
The space $|\K \pi|$ is denoted by $\mathfrak{g}(\Sigma)$ in Introduction.
Another description of $|\K \pi|$ is as follows.
Let $\hat{\pi}$ be the set of homotopy classes of free loops in $\Sigma$:
\[
\hat{\pi} := \{ \text{free loops in $\Sigma$} \}/ \text{homotopy}.
\]
Forgetting the basepoint of a  loop, we obtain a map $\pi_1(\Sigma,p) \to \hat{\pi}, \gamma \mapsto |\gamma|$ for any $p\in \Sigma$.
Then the $\K$-linear extension of the map $\pi=\pi_1(\Sigma,*)\to \hat{\pi}$ induces a $\K$-linear isomorphism
\[
|\K \pi |\cong \K\hat{\pi}.
\]
In what follows, we often use this identification without mentioning explicitly.
In particular, ${\bf 1} \in |\mathbb{K}\pi|$ corresponds to the homotopy class of a constant loop. 

The Goldman bracket and Turaev cobracket are topologically defined operations on the space $|\K\pi|$. More precisely, 
the Goldman bracket is a map
$$
[\cdot, \cdot]_{\rm Goldman}\colon |\K \pi| \otimes |\K\pi| \to |\K \pi|.
$$
The Turaev cobracket was originally defined on the quotient $|\K \pi|/\K {\bf 1}$:
$$
\delta_{\rm Turaev}\colon |\K\pi|/\K{\bf 1} \to (|\K \pi|/\K {\bf 1}) \otimes (|\K\pi|/\K{\bf 1}).
$$
It turns out that a choice of framing $f\colon T\Sigma \to \Sigma \times \mathbb{R}^2$ allows us to lift the Turaev cobracket to a map
 $$
 \delta^f = \delta^f_{\rm Turaev}\colon |\K \pi| \to |\K \pi| \otimes |\K\pi|
 $$
 which depends on $f$. For every framing, $(|\K\pi|, [\cdot, \cdot]_{\rm Goldman}, \delta^f_{\rm Turaev})$ is an involutive Lie bialgebra, as was established by Chas \cite{Cha04} for the original Turaev cobracket.
 
 It is actually more convenient to work with operations defined directly on the group algebra $\K \pi$ rather than on its quotient $|\K \pi|$. The natural lifts of the Goldman bracket and Turaev cobracket are the double bracket (in the sense of van den Bergh)
$$
\kappa\colon \K \pi \otimes \K \pi \to \K \pi \otimes \K \pi
$$ 
and the coaction maps
$$
\mu^f_r \colon \K \pi \to |\K\pi| \otimes \K \pi,
\hspace{2em}
\mu^f_l \colon \K\pi \to \K \pi \otimes |\K\pi|
$$ 
which depend on the framing $f$. We refer to these maps as {\em upgrades} of the Goldman-Turaev operations.
The Goldman bracket and Turaev cobracket can be easily recovered from the maps $\kappa$, $\mu^f_r$ and $\mu^f_l$. The great advantage of working with $\kappa$, $\mu^f_r$ and $\mu^f_l$ is the fact that they are completely defined by their values on generators of $\pi$. This is in contrast with the situation of the Goldman bracket and Turaev cobracket. Indeed, we are not aware of a convenient to handle generating system for the Goldman Lie algebra.

All the maps listed above are defined using the notions of intersections and self-intersections of curves.
We say that two immersed curves in $\Sigma$ are \emph{generic} or \emph{in general position} if their intersections consist of finitely many transverse double points.
Likewise, an immersed curve is said to be \emph{generic} if its self-intersections consist of a finite number of transverse double points.

We fix the following notation: for an ordered pair of linearly independent tangent vectors $(v_1, v_2)$ at some $p\in \Sigma$, let $\varepsilon(v_1,v_2):=+1$
if $(v_1,v_2)$ is a positive basis for $T_p\Sigma$,
and $\varepsilon(v_1,v_2):=-1$ otherwise.

\subsection{Goldman bracket and  double bracket $\kappa$}
\label{subsec:bra}

In this subsection, we recall the definition and properties of the Goldman bracket $[ \cdot, \cdot]_{\rm Goldman}$ and of its upgrade, the double bracket $\kappa$.

Let $\alpha$ and $\beta$ be free loops on $\Sigma$ in general position.
For each intersection $p\in \alpha\cap \beta$,
let $\dot{\alpha}_p, \dot{\beta}_p\in T_p\Sigma$ be the tangent vectors of $\alpha$ and $\beta$ at $p$, respectively.
Let $\alpha_p\in \pi_1(\Sigma,p)$ be the loop $\alpha$ based at $p$ and define $\beta_p$ in the same way.
Then the concatenation $\alpha_p \beta_p\in \pi_1(\Sigma,p)$ and its projection $|\alpha_p\beta_p|\in \hat{\pi}$ are well defined.
The \emph{Goldman bracket} of $\alpha$ and $\beta$ is defined by the formula
\begin{equation}
\label{eq:Gbra}
[\alpha,\beta]_{\rm Goldman}:=\sum_{p\in \alpha\cap \beta}
\varepsilon(\dot{\alpha}_p,\dot{\beta}_p)\, |\alpha_p\beta_p|
\in \K \hat{\pi} \cong |\K \pi|.
\end{equation}
Figure \ref{fig:Goldman} illustrates a computation of the Goldman bracket, where $\alpha$ and $\beta$ have only one transverse intersection.
As shown in \cite[Theorem 5.3]{Go86},
the $\K$-bilinear extension of this formula defines a Lie bracket on $|\K \pi|$.
That is, $[\cdot,\cdot]=[\cdot,\cdot]_{\rm Goldman}$ depends only on homotopy classes of loops and satisfies
\begin{itemize}
\item
the skew-symmetry condition:
$
[y,x]=-[x,y]
$
for any $x,y\in |\K\pi|$,
\item
the Jacobi identity:
$
[x,[y,z]]+[y,[z,x]]+[z,[x,y]]=0
$
for any $x,y,z\in |\K\pi|$.
\end{itemize}

\begin{figure}
\begin{center}
\input{fig_Goldman.tex}
\end{center}
\caption{The Goldman bracket.}
\label{fig:Goldman}
\end{figure}

In order to introduce an upgrade of the Goldman bracket, we choose an orientation preserving embedding $\nu \colon [0,1] \to \pa \Sigma$ with $\nu(1)=*$, and let $\bullet:=\nu(0)$.
The path $\nu$ defines an isomorphism of groups
\begin{equation*}
\pi=\pi_1(\Sigma,*) \overset{\cong}{\to} \pi_1(\Sigma,\bullet),
\quad
\gamma \mapsto \nu \gamma \overline{\nu},
\end{equation*}
through which we identify $\pi$ with $\pi_1(\Sigma,\bullet)$.
Here, $\overline{\nu}$ is the inverse path of $\nu$.

Let $\alpha$ and $\beta$ be generic loops based at $\bullet$ and $*$, respectively.
For each $p\in \alpha\cap \beta$, let $\alpha_{\bullet p}$ be the path from $\bullet$ to $p$ along $\alpha$, and define $\alpha_{p\bullet}$, $\beta_{*p}$ and $\beta_{p*}$ in a similar way.
Then the concatenations $\beta_{*p}\alpha_{p\bullet}\nu$ and $\overline{\nu}\alpha_{\bullet p}\beta_{p*}\in \pi$ are defined.
Set
\begin{equation}
\label{eq:kaab}
\kappa(\alpha,\beta):=
\sum_{p\in \alpha \cap \beta}
\varepsilon(\dot{\alpha}_p,\dot{\beta}_p)\, 
\beta_{*p}\alpha_{p\bullet}\nu \otimes \overline{\nu}\alpha_{\bullet p}\beta_{p*}
\in \K \pi \otimes \K \pi.
\end{equation}
Figure \ref{fig:kappa} shows a computation of $\kappa$, where $\alpha$ and $\beta$ intersect in one transverse double point.
Extending $\K$-bilinearly, we obtain a map
\[
\kappa\colon \K\pi \otimes \K\pi \to \K\pi \otimes \K\pi.
\]

\begin{figure}
\begin{center}
\input{fig_kappa.tex}
\end{center}
\caption{The operation $\kappa$}
\label{fig:kappa}
\end{figure}

\begin{prop}
\label{prop:kdb}
The map $\kappa$ is a double bracket on $\K\pi$ in the sense of Definition \ref{dfn:db}.
Furthermore, for any $a,b\in \K\pi$, we have
\begin{equation}
\label{eq:yxxy}
\kappa(b,a)=-\kappa(a,b)^{\circ}
+a\otimes b+b\otimes a-ab\otimes 1-1\otimes ba.
\end{equation}
\end{prop}

\begin{proof}
For the first statement, see \cite[Lemma 4.3.1]{KK15} and \cite[\S 7]{MT14}.
For convenience of the reader, we verify the property \eqref{eq:ab,c} for $\Pi = \kappa$. Let $\alpha, \beta$ be loops based at $\bullet$ and $\gamma$ a loop based at $*$. We assume that all intersections are in general position. Then,
\begin{align*}
\kappa(\alpha \beta, \gamma) = &
\textstyle\sum_{p \in \alpha \cap \gamma} \varepsilon(\dot{\alpha}_p, \dot{\gamma}_p) \,
(\gamma_{*p} \alpha_{p \bullet} \nu)(\overline{\nu} \beta \nu) \otimes \overline{\nu} \alpha_{\bullet p} \gamma_{p *} \\
& + \textstyle\sum_{q \in \beta \cap \gamma} \varepsilon(\dot{\beta}_q, \dot{\gamma}_q) \,
\gamma_{*q}\beta_{q\bullet} \nu \otimes (\overline{\nu} \alpha \nu)(\overline{\nu} \beta_{\bullet q} \gamma_{q *}) \\
= & \ \kappa(\alpha, \gamma) (\beta \otimes 1) + (1 \otimes \alpha) \kappa(\beta, \gamma),
\end{align*}
as required. Here in the first and the second line we inserted a trivial loop $\nu \overline{\nu}$ based at $\bullet$ to make the calculation more transparent. The proof of the property \eqref{eq:a,bc} is similar to the above.

For the second statement, see the proof of Lemma~7.2 in \cite{MT14}.
\end{proof}

\begin{cor}
\label{cor:kuv}
Let $a_1,\ldots,a_l,b_1,\ldots,b_m\in \K\pi$.
Then,
\[
\kappa(a_1\cdots a_l, b_1\cdots b_m)=
\sum_{i,j} 
(b_1\cdots b_{j-1} \otimes a_1\cdots a_{i-1})
\kappa(a_i,b_j)
(a_{i+1}\cdots a_l \otimes b_{j+1}\cdots b_m).
\]
\end{cor}

\begin{proof}
This is a direct consequence of \eqref{eq:a,bc} and \eqref{eq:ab,c}.
\end{proof}

Applying the construction explained after Definition \ref{dfn:db} to the double bracket $\kappa$, we obtain the bracket $\{\cdot,\cdot \}_{\kappa}\colon |\K\pi|\otimes |\K\pi|\to |\K\pi|$ which equals the Goldman bracket.
Another bracket operation $\{\cdot, \cdot \}_{\kappa}\colon |\K\pi|\otimes \K\pi \to \K\pi$ coincides with the action $\sigma$ in \cite[\S 3.2]{KK14}.
More explicitly, 
$$
\sigma = \sigma^{\kappa} \colon |\mathbb{K} \pi| \to {\rm Der}(\mathbb{K}\pi)
$$
maps a free loop $\alpha$ to the derivation $\sigma(\alpha) = \{ \alpha,\cdot\}_{\kappa}$. If $\beta\in \pi$ is represented by an immersed based loop in general position with $\alpha$, then
\begin{equation} \label{eq:actionsigma}
\sigma(\alpha)(\beta) = \{ \alpha,\beta\}_{\kappa} =
\sum_{p\in \alpha\cap \beta} \varepsilon(\dot{\alpha}_p,\dot{\beta}_p) \,
\beta_{*p}\alpha_p \beta_{p*}.
\end{equation}
Yet another bracket operation $\{\cdot, \cdot \}_{\kappa} \colon \K \pi \otimes |\K \pi| \to |\K \pi|$ is related to the previous one by the 
following skew-symmetric property: for any $a, b \in \K \pi$,
\begin{equation} \label{eq:kappauvvu}
\{ a, |b| \}_{\kappa} = - \{ |b|, a \}_{\kappa}.
\end{equation}
This follows by applying the multiplication of $\K \pi$ to equation \eqref{eq:yxxy}.

\begin{rem}
\begin{enumerate}
    \item[(i)]
Papakyriakopoulos \cite{Papa} and Turaev \cite{Tu78} independently introduced essentially the same operations as the map $\kappa$.
    \item[(ii)]
The sign convention in formula \eqref{eq:kaab} is different from \cite{genus0} and \cite{KK15}.
\end{enumerate}
\end{rem}

We end this section with the following formula which we will use later.

\begin{prop}[Lemma 7.4 in \cite{MT13}, see also equation (5.4) in \cite{KK16}] \label{prop:ksabc}
For any $a, b, c \in \K \pi$, we have
\begin{equation*} 
\kappa(\sigma(|a|)(b), c) + \kappa(b, \sigma(|a|)(c)) = \sigma(|a|) (\kappa(b,c)).
\end{equation*}
\end{prop}

\subsection{Turaev cobracket and its upgrade $\delta^f$}
\label{subsec:cob}

The original version \cite{Tu91} of the Turaev cobracket $\delta_{\rm Turaev}$ is a Lie cobracket on the quotient space $|\K\pi|/\K {\bf 1}$, where ${\bf 1}\in \hat{\pi}$ is the class of a constant loop. In this section, we introduce a natural lift of $\delta_{\rm Turaev}$ to the space $|\K\pi|$.
Our definition of the lift $\delta^f = \delta^f_{\rm Turaev}$ is motivated by Furuta's observation \cite[\S 4]{Mo97}, and it depends on the choice of framing.
Here, a \emph{framing} on $\Sigma$ is (the homotopy class of) a trivialization of the tangent bundle of $\Sigma$, $f \colon T\Sigma \overset{\cong}{\to} \Sigma\times \mathbb{R}^2$.
Since by assumptions $\partial \Sigma$ is nonempty, the tangent bundle $T\Sigma$ is trivial and framings on $\Sigma$ always exist.

Let $\hat{\pi}^+$ be the set of regular homotopy classes of immersed free loops on $\Sigma$. 
Recall that two immersed curves $\alpha_0$ and $\alpha_1$ are {\em regularly homotopic} if there exists a homotopy $\{ \alpha_s\}_{s \in [0,1]}$ between them such that for each $s \in [0, 1]$ the curve $\alpha_s$ is an immersed curve.
Given a framing $f$, one can define the rotation number function on $\hat{\pi}^+$:
\[
\rot^f \colon \hat{\pi}^+ \to \Z.
\]

\begin{rem}
\label{rem:framing}
The set of homotopy classes of framings on $\Sigma$ is a torsor under the action of $H^1(\Sigma,\Z)$.
This action can be seen through the injective assignment $f\mapsto \rot^f$. 
Namely, if $f$ and $f'$ are framings, then there exists a unique cohomology class $\chi\in H^1(\Sigma,\Z)\cong {\rm Hom}(\pi,\Z)$ such that for any $\gamma\in \hat{\pi}^+$,
\begin{equation}
\label{eq:ff'chi}
\rot^{f'}(\gamma)=\rot^f(\gamma) + \chi(\gamma).
\end{equation}
\end{rem}

\begin{figure}
\begin{center}
{\unitlength 0.1in%
\begin{picture}(37.0000,8.7000)(2.0000,-12.7000)%
%
\special{pn 8}%
\special{pa 2300 1200}%
\special{pa 2300 400}%
\special{fp}%
%
\special{pn 8}%
\special{pa 3800 1200}%
\special{pa 3800 900}%
\special{fp}%
%
\special{pn 8}%
\special{pa 3800 400}%
\special{pa 3800 700}%
\special{fp}%
%
\special{pn 8}%
\special{ar 3700 900 100 200 4.7123890 6.2831853}%
%
\special{pn 8}%
\special{ar 3700 800 100 100 3.1415927 4.7123890}%
%
\special{pn 8}%
\special{ar 3700 800 100 100 1.5707963 3.1415927}%
%
\special{pn 8}%
\special{ar 3700 700 100 200 6.2831853 1.5707963}%
%
\special{pn 8}%
\special{pa 800 1200}%
\special{pa 800 900}%
\special{fp}%
%
\special{pn 8}%
\special{pa 800 400}%
\special{pa 800 700}%
\special{fp}%
%
\special{pn 8}%
\special{ar 900 900 100 200 3.1415927 4.7123890}%
%
\special{pn 8}%
\special{ar 900 800 100 100 4.7123890 6.2831853}%
%
\special{pn 8}%
\special{ar 900 800 100 100 6.2831853 1.5707963}%
%
\special{pn 8}%
\special{ar 900 700 100 200 1.5707963 3.1415927}%
%
\special{pn 8}%
\special{pa 800 1000}%
\special{pa 825 1100}%
\special{fp}%
\special{pa 800 1000}%
\special{pa 775 1100}%
\special{fp}%
%
\special{pn 8}%
\special{pa 2300 1000}%
\special{pa 2325 1100}%
\special{fp}%
\special{pa 2300 1000}%
\special{pa 2275 1100}%
\special{fp}%
%
\special{pn 8}%
\special{pa 3800 1000}%
\special{pa 3825 1100}%
\special{fp}%
\special{pa 3800 1000}%
\special{pa 3775 1100}%
\special{fp}%
\put(24.0000,-6.0000){\makebox(0,0)[lb]{$\alpha$}}%
\put(39.0000,-6.0000){\makebox(0,0)[lb]{$\alpha^+$}}%
\put(9.0000,-6.0000){\makebox(0,0)[lb]{$\alpha^-$}}%
\put(32.0000,-14.0000){\makebox(0,0)[lb]{{\small ${\rm rot}^f(\alpha^+) = {\rm rot}^f(\alpha) +1$}}}%
\put(2.0000,-14.0000){\makebox(0,0)[lb]{{\small ${\rm rot}^f(\alpha^-) = {\rm rot}^f(\alpha) -1$}}}%
%
\special{pn 8}%
\special{pn 8}%
\special{pa 2000 800}%
\special{pa 1937 800}%
\special{fp}%
\special{pa 1906 800}%
\special{pa 1898 800}%
\special{fp}%
\special{pa 1868 800}%
\special{pa 1805 800}%
\special{fp}%
\special{pa 1774 800}%
\special{pa 1766 800}%
\special{fp}%
\special{pa 1736 800}%
\special{pa 1673 800}%
\special{fp}%
\special{pa 1642 800}%
\special{pa 1634 800}%
\special{fp}%
\special{pa 1604 800}%
\special{pa 1541 800}%
\special{fp}%
\special{pa 1510 800}%
\special{pa 1502 800}%
\special{fp}%
\special{pa 1472 800}%
\special{pa 1409 800}%
\special{fp}%
\special{pa 1379 800}%
\special{pa 1371 800}%
\special{fp}%
\special{pa 1340 800}%
\special{pa 1300 800}%
\special{fp}%
\special{sh 1}%
\special{pa 1300 800}%
\special{pa 1367 820}%
\special{pa 1353 800}%
\special{pa 1367 780}%
\special{pa 1300 800}%
\special{fp}%
%
\special{pn 8}%
\special{pn 8}%
\special{pa 2600 800}%
\special{pa 2663 800}%
\special{fp}%
\special{pa 2694 800}%
\special{pa 2702 800}%
\special{fp}%
\special{pa 2732 800}%
\special{pa 2795 800}%
\special{fp}%
\special{pa 2826 800}%
\special{pa 2834 800}%
\special{fp}%
\special{pa 2864 800}%
\special{pa 2927 800}%
\special{fp}%
\special{pa 2958 800}%
\special{pa 2966 800}%
\special{fp}%
\special{pa 2996 800}%
\special{pa 3059 800}%
\special{fp}%
\special{pa 3090 800}%
\special{pa 3098 800}%
\special{fp}%
\special{pa 3128 800}%
\special{pa 3191 800}%
\special{fp}%
\special{pa 3221 800}%
\special{pa 3229 800}%
\special{fp}%
\special{pa 3260 800}%
\special{pa 3300 800}%
\special{fp}%
\special{sh 1}%
\special{pa 3300 800}%
\special{pa 3233 780}%
\special{pa 3247 800}%
\special{pa 3233 820}%
\special{pa 3300 800}%
\special{fp}%
\end{picture}}%
\end{center}
\caption{Inserting monogons}
\label{fig:insertingmonogon}
\end{figure}

For a given framing $f$,  we now define a framed version of the Turaev cobracket.
Let $\alpha\in \hat{\pi}$ and represent it by a generic immersion.
By inserting monogons into $\alpha$, we can change the rotation number of $\alpha$ with respect to $f$ without changing its homotopy class
(see Figure \ref{fig:insertingmonogon}, where $\alpha^+$ and $\alpha^-$ are curves obtained by inserting a positive or negative monogon into $\alpha$).

We require that after this operation $\alpha$ is represented by a generic immersed loop with $\rot^f(\alpha)=0$.
Each self-intersection $p$ of $\alpha$ is traversed by $\alpha$ twice, and there are two linearly independent tangent vectors of $\alpha$ at $p$.
Temporarily, we label them by $v_1,v_2\in T_p\Sigma$ so that the ordered pair $(v_1,v_2)$ is a positive basis for $T_p\Sigma$.
Define $\alpha^1_p$ as the loop starting at $p$ in the direction of $v_1$ and going along $\alpha$ until the first return to $p$.
Define $\alpha^2_p$ similarly by starting in the direction of $v_2$ (see Figure \ref{fig:splittingloop}).
Then, the \emph{framed version of the Turaev cobracket} is defined by formula
\begin{equation}
\label{eq:cob}
\delta^f_{\rm Turaev}(\alpha):=\sum_p \alpha^1_p \wedge \alpha^2_p =
\sum_p \left( \alpha^1_p \otimes \alpha^2_p - \alpha^2_p \otimes \alpha^1_p \right) \in |\K\pi| \otimes |\K\pi|,
\end{equation}
where the sum is taken over all self-intersections of $\alpha$.

\begin{figure}
\begin{center}
\input{fig_splittingloop.tex}
\end{center}
\caption{Splitting $\alpha$ at a self-intersection}
\label{fig:splittingloop}
\end{figure}

\begin{prop}
\label{prop:dfcob}
The map $\delta^f = \delta^f_{\rm Turaev}$ is a Lie cobracket on $|\K\pi|$.
That is, $\delta^f$ satisfies
\begin{itemize}
\item
the coskew-symmetry condition:
$\delta^f(a)^{\circ}=-\delta^f(a)$ for any $a\in |\K\pi|$,
\item
the coJacobi identity:
$
N(\delta^f \otimes {\rm id}) \delta^f=0 \colon |\K\pi| \to |\K\pi|^{\otimes 3}.
$
\end{itemize}
Here, $N(a\otimes b\otimes c)=a\otimes b \otimes c+b\otimes c \otimes a+c\otimes a\otimes b$
for $a,b,c\in |\K\pi|$.
\end{prop}

\begin{proof}
The coskew-symmetry condition follows from  equation \eqref{eq:cob}.

To prove the coJacobi identity, let $\alpha$ be a generic immersed loop with $\rot^f(\alpha)=0$.
Then
\[
(\delta^f\otimes {\rm id})\delta^f (\alpha)=
\sum_p \delta^f(\alpha_p^1)\otimes \alpha_p^2-\delta^f(\alpha_p^2)\otimes \alpha_p^1.
\]
There are two contributions on the right hand side. The first one comes from the sum over self-intersections $q$ of $\alpha_p^1$ in the first term and of $\alpha_p^2$ in the second term. 
It is easy to check (see the original argument of Turaev in the proof of \cite[Theorem 8.3]{Tu91})
that the sum of these contributions is in the kernel of the map $N$.

The second contribution comes from the fact that in order to compute $\delta^f(\alpha_p^1)$ and $\delta^f(\alpha_p^2)$ one needs to insert a certain number of monogons in $\alpha_p^1$ and $\alpha_p^2$ to insure vanishing of their rotation numbers.
Notice that for each self-intersection $p$,
\[
\rot^f(\alpha_p^1)+\rot^f(\alpha_p^2)=\rot^f(\alpha)=0.
\]
Therefore, if we insert $m_p$ positive monogons into $\alpha_p^1$ to achieve $\rot^f(\alpha_p^1)=0$, then we need $m_p$ negative monogons for $\alpha_p^2$.
Now, each positive (resp. negative) monogon of $\gamma$ contributes to $\delta^f(\gamma)$ as $\gamma \wedge {\bf 1}$ (resp. $-\gamma \wedge {\bf 1}$).
Therefore, the second contribution is computed as
\[
\sum_p m_p ((\alpha_p^1\wedge {\bf 1}) \otimes \alpha_p^2+
(\alpha_p^2\wedge {\bf 1}) \otimes \alpha_p^1),
\]
and this is in the kernel of $N$.
This completes the proof.
\end{proof}

\begin{rem}
The maps $\delta^f$ for all $f$ descend to a unique map
$\delta_{\rm Turaev}\colon |\K\pi|/\K {\bf 1} \to (|\K\pi|/\K {\bf 1})^{\otimes 2}$. Furthermore, the element ${\bf 1} \in |\K\pi|$ is central under the Goldman bracket. Therefore, $\K {\bf 1}$ is a Lie ideal and the Goldman bracket descends to the space $|\K\pi|/\K {\bf 1}$ which carries a canonical (framing independent) structure of a Lie bialgebra.
\end{rem}

\subsection{Operations $\mu^f_r$ and $\mu^f_l$}
\label{subsec:mu}

In this subsection, we introduce a lift of the framed Turaev cobracket to the group algebra $\K \pi$ and describe its properties.

Choose a framing $f$ in such a way that it is trivial on a neighborhood of $\nu$.
In more detail, the velocity vector field of $\nu$ corresponds to $(1,0)\in \mathbb{R}^2$ through $f$. Denote by $v_*\in T_*\Sigma$ and $v_{\bullet}\in T_{\bullet}\Sigma$ the inward tangent vectors corresponding to $(0,1)\in \mathbb{R}^2$,
see Figure \ref{fig:nutrivial}.

\begin{figure}
\begin{center}
{\unitlength 0.1in%
\begin{picture}(28.6000,10.0000)(1.4000,-12.4800)%
%
\special{pn 13}%
\special{pa 1400 1200}%
\special{pa 3000 1200}%
\special{fp}%
%
\special{pn 4}%
\special{sh 1}%
\special{ar 1800 1200 16 16 0 6.2831853}%
%
\special{pn 4}%
\special{sh 1}%
\special{ar 2600 1200 16 16 0 6.2831853}%
\put(17.5000,-14.0000){\makebox(0,0)[lb]{$\bullet$}}%
\put(25.5000,-14.0000){\makebox(0,0)[lb]{$*$}}%
%
\special{pn 8}%
\special{pa 1800 1200}%
\special{pa 1800 800}%
\special{fp}%
\special{sh 1}%
\special{pa 1800 800}%
\special{pa 1780 867}%
\special{pa 1800 853}%
\special{pa 1820 867}%
\special{pa 1800 800}%
\special{fp}%
%
\special{pn 8}%
\special{pa 2600 1200}%
\special{pa 2600 800}%
\special{fp}%
\special{sh 1}%
\special{pa 2600 800}%
\special{pa 2580 867}%
\special{pa 2600 853}%
\special{pa 2620 867}%
\special{pa 2600 800}%
\special{fp}%
\put(17.5000,-7.0000){\makebox(0,0)[lb]{$v_{\bullet}$}}%
\put(25.5000,-7.0000){\makebox(0,0)[lb]{$v_*$}}%
\put(6.0000,-11.9500){\makebox(0,0)[lb]{$(1,0)$}}%
\put(1.4000,-7.9500){\makebox(0,0)[lb]{$(0,1)$}}%
%
\special{pn 4}%
\special{pa 1860 1134}%
\special{pa 1960 1134}%
\special{fp}%
\special{sh 1}%
\special{pa 1960 1134}%
\special{pa 1893 1114}%
\special{pa 1907 1134}%
\special{pa 1893 1154}%
\special{pa 1960 1134}%
\special{fp}%
\special{pa 1860 1134}%
\special{pa 1860 1034}%
\special{fp}%
\special{sh 1}%
\special{pa 1860 1034}%
\special{pa 1840 1101}%
\special{pa 1860 1087}%
\special{pa 1880 1101}%
\special{pa 1860 1034}%
\special{fp}%
%
\special{pn 8}%
\special{pa 2200 1200}%
\special{pa 2080 1240}%
\special{fp}%
\special{pa 2200 1200}%
\special{pa 2080 1160}%
\special{fp}%
\put(21.6000,-13.4500){\makebox(0,0)[lb]{$\nu$}}%
%
\special{pn 8}%
\special{pa 300 1100}%
\special{pa 300 900}%
\special{fp}%
\special{sh 1}%
\special{pa 300 900}%
\special{pa 280 967}%
\special{pa 300 953}%
\special{pa 320 967}%
\special{pa 300 900}%
\special{fp}%
%
\special{pn 8}%
\special{pa 300 1100}%
\special{pa 500 1100}%
\special{fp}%
\special{sh 1}%
\special{pa 500 1100}%
\special{pa 433 1080}%
\special{pa 447 1100}%
\special{pa 433 1120}%
\special{pa 500 1100}%
\special{fp}%
%
\special{pn 4}%
\special{sh 1}%
\special{ar 300 1100 8 8 0 6.2831853}%
%
\special{pn 4}%
\special{sh 1}%
\special{ar 1860 1133 8 8 0 6.2831853}%
%
\special{pn 4}%
\special{pa 2020 1134}%
\special{pa 2120 1134}%
\special{fp}%
\special{sh 1}%
\special{pa 2120 1134}%
\special{pa 2053 1114}%
\special{pa 2067 1134}%
\special{pa 2053 1154}%
\special{pa 2120 1134}%
\special{fp}%
\special{pa 2020 1134}%
\special{pa 2020 1034}%
\special{fp}%
\special{sh 1}%
\special{pa 2020 1034}%
\special{pa 2000 1101}%
\special{pa 2020 1087}%
\special{pa 2040 1101}%
\special{pa 2020 1034}%
\special{fp}%
%
\special{pn 4}%
\special{sh 1}%
\special{ar 2020 1133 8 8 0 6.2831853}%
%
\special{pn 4}%
\special{pa 2180 1134}%
\special{pa 2280 1134}%
\special{fp}%
\special{sh 1}%
\special{pa 2280 1134}%
\special{pa 2213 1114}%
\special{pa 2227 1134}%
\special{pa 2213 1154}%
\special{pa 2280 1134}%
\special{fp}%
\special{pa 2180 1134}%
\special{pa 2180 1034}%
\special{fp}%
\special{sh 1}%
\special{pa 2180 1034}%
\special{pa 2160 1101}%
\special{pa 2180 1087}%
\special{pa 2200 1101}%
\special{pa 2180 1034}%
\special{fp}%
%
\special{pn 4}%
\special{sh 1}%
\special{ar 2180 1133 8 8 0 6.2831853}%
%
\special{pn 4}%
\special{pa 2340 1134}%
\special{pa 2440 1134}%
\special{fp}%
\special{sh 1}%
\special{pa 2440 1134}%
\special{pa 2373 1114}%
\special{pa 2387 1134}%
\special{pa 2373 1154}%
\special{pa 2440 1134}%
\special{fp}%
\special{pa 2340 1134}%
\special{pa 2340 1034}%
\special{fp}%
\special{sh 1}%
\special{pa 2340 1034}%
\special{pa 2320 1101}%
\special{pa 2340 1087}%
\special{pa 2360 1101}%
\special{pa 2340 1034}%
\special{fp}%
%
\special{pn 4}%
\special{sh 1}%
\special{ar 2340 1133 8 8 0 6.2831853}%
%
\special{pn 4}%
\special{pa 2500 1134}%
\special{pa 2600 1134}%
\special{fp}%
\special{sh 1}%
\special{pa 2600 1134}%
\special{pa 2533 1114}%
\special{pa 2547 1134}%
\special{pa 2533 1154}%
\special{pa 2600 1134}%
\special{fp}%
\special{pa 2500 1134}%
\special{pa 2500 1034}%
\special{fp}%
\special{sh 1}%
\special{pa 2500 1034}%
\special{pa 2480 1101}%
\special{pa 2500 1087}%
\special{pa 2520 1101}%
\special{pa 2500 1034}%
\special{fp}%
%
\special{pn 4}%
\special{sh 1}%
\special{ar 2500 1133 8 8 0 6.2831853}%
%
\special{pn 4}%
\special{pa 1660 1134}%
\special{pa 1760 1134}%
\special{fp}%
\special{sh 1}%
\special{pa 1760 1134}%
\special{pa 1693 1114}%
\special{pa 1707 1134}%
\special{pa 1693 1154}%
\special{pa 1760 1134}%
\special{fp}%
\special{pa 1660 1134}%
\special{pa 1660 1034}%
\special{fp}%
\special{sh 1}%
\special{pa 1660 1034}%
\special{pa 1640 1101}%
\special{pa 1660 1087}%
\special{pa 1680 1101}%
\special{pa 1660 1034}%
\special{fp}%
%
\special{pn 4}%
\special{sh 1}%
\special{ar 1660 1133 8 8 0 6.2831853}%
%
\special{pn 4}%
\special{pa 1500 1134}%
\special{pa 1600 1134}%
\special{fp}%
\special{sh 1}%
\special{pa 1600 1134}%
\special{pa 1533 1114}%
\special{pa 1547 1134}%
\special{pa 1533 1154}%
\special{pa 1600 1134}%
\special{fp}%
\special{pa 1500 1134}%
\special{pa 1500 1034}%
\special{fp}%
\special{sh 1}%
\special{pa 1500 1034}%
\special{pa 1480 1101}%
\special{pa 1500 1087}%
\special{pa 1520 1101}%
\special{pa 1500 1034}%
\special{fp}%
%
\special{pn 4}%
\special{sh 1}%
\special{ar 1500 1133 8 8 0 6.2831853}%
%
\special{pn 4}%
\special{pa 2660 1134}%
\special{pa 2760 1134}%
\special{fp}%
\special{sh 1}%
\special{pa 2760 1134}%
\special{pa 2693 1114}%
\special{pa 2707 1134}%
\special{pa 2693 1154}%
\special{pa 2760 1134}%
\special{fp}%
\special{pa 2660 1134}%
\special{pa 2660 1034}%
\special{fp}%
\special{sh 1}%
\special{pa 2660 1034}%
\special{pa 2640 1101}%
\special{pa 2660 1087}%
\special{pa 2680 1101}%
\special{pa 2660 1034}%
\special{fp}%
%
\special{pn 4}%
\special{sh 1}%
\special{ar 2660 1133 3 3 0 6.2831853}%
%
\special{pn 4}%
\special{pa 2820 1134}%
\special{pa 2920 1134}%
\special{fp}%
\special{sh 1}%
\special{pa 2920 1134}%
\special{pa 2853 1114}%
\special{pa 2867 1134}%
\special{pa 2853 1154}%
\special{pa 2920 1134}%
\special{fp}%
\special{pa 2820 1134}%
\special{pa 2820 1034}%
\special{fp}%
\special{sh 1}%
\special{pa 2820 1034}%
\special{pa 2800 1101}%
\special{pa 2820 1087}%
\special{pa 2840 1101}%
\special{pa 2820 1034}%
\special{fp}%
%
\special{pn 4}%
\special{sh 1}%
\special{ar 2820 1133 3 3 0 6.2831853}%
\put(2.0000,-4.0000){\makebox(0,0)[lb]{$\mathbb{R}^2$}}%
\put(22.0000,-4.0000){\makebox(0,0)[lb]{$\Sigma$}}%
\end{picture}}%
\end{center}
\caption{Trivializing $f$ near $\nu$}
\label{fig:nutrivial}
\end{figure}

Let $\Pi\Sigma(\bullet,*)$ be the set of homotopy classes of paths from $\bullet$ to $*$, and define $\Pi\Sigma(*,\bullet)$ similarly.
The path $\nu$ induces natural identifications
\[
\pi \cong \Pi\Sigma(\bullet,*), \quad
\gamma \mapsto \nu \gamma,
\quad \text{and} \quad
\pi \cong \Pi\Sigma(*,\bullet), \quad
\gamma \mapsto \gamma \overline{\nu}.
\]

Let $\gamma \in \pi$.
Using the bijection $\pi\cong \Pi\Sigma(\bullet,*)$, we choose a representative of $\gamma$ to be a generic immersed path
$\gamma\colon ([0,1],0,1) \to (\Sigma, \bullet, *)$ such that $\dot{\gamma}(0) = v_{\bullet}$ and $\dot{\gamma}(1) = -v_*$.
We further assume that $\gamma$ has the rotation number $-1/2$ with respect to $f$.

For each self-intersection $p$ of $\gamma$, we denote $\gamma^{-1}(p)=\{ t_1^p,t_2^p \}$ with $0<t_1^p <t_2^p<1$.
The velocity vectors $\dot{\gamma}(t_1^p), \dot{\gamma}(t_2^p)$ are linearly independent in $T_p\Sigma$.
Let $\gamma_{t_1^p t_2^p}$ be the restriction of $\gamma$ to the interval $[t_1^p,t_2^p]$.
This becomes a loop since $\gamma(t_1^p)=\gamma(t_2^p)=p$.
Define the paths $\gamma_{0t_1^p}$ and $\gamma_{t_2^p 1}$ similarly.
Then the concatenation $\overline{\nu}\gamma_{0 t_1^p}\gamma_{t_2^p 1} \in \pi$ is defined.
Set
\begin{equation}
\label{eq:muf}
\mu^f_r(\gamma):=\sum_p \varepsilon(\dot{\gamma}(t_1^p),\dot{\gamma}(t_2^p))\,
|\gamma_{t_1^p t_2^p}| \otimes \overline{\nu} \gamma_{0 t_1^p} \gamma_{t_2^p 1}
\in |\K\pi| \otimes \K\pi.
\end{equation}
Figure \ref{fig:mu} shows a computation of $\mu^f_r$, where there is only one transverse self-intersection of $\gamma$.
Extending $\K$-linearly, we obtain a map
\[
\mu^f_r \colon \K \pi \to |\K\pi|\otimes \K\pi.
\]

\begin{figure}
\begin{center}
{\unitlength 0.1in%
\begin{picture}(53.0000,8.4500)(10.0000,-16.2000)%
%
\special{pn 13}%
\special{pa 1000 1600}%
\special{pa 2200 1600}%
\special{fp}%
%
\special{pn 13}%
\special{pa 3400 1600}%
\special{pa 4600 1600}%
\special{fp}%
%
\special{pn 13}%
\special{pa 5100 1600}%
\special{pa 6300 1600}%
\special{fp}%
%
\special{pn 4}%
\special{sh 1}%
\special{ar 1200 1600 16 16 0 6.2831853}%
%
\special{pn 4}%
\special{sh 1}%
\special{ar 2000 1600 16 16 0 6.2831853}%
%
\special{pn 4}%
\special{sh 1}%
\special{ar 3600 1600 16 16 0 6.2831853}%
%
\special{pn 4}%
\special{sh 1}%
\special{ar 4400 1600 16 16 0 6.2831853}%
%
\special{pn 4}%
\special{sh 1}%
\special{ar 5300 1600 16 16 0 6.2831853}%
%
\special{pn 4}%
\special{sh 1}%
\special{ar 6100 1600 16 16 0 6.2831853}%
%
\special{pn 13}%
\special{ar 1600 1300 100 100 0.0000000 6.2831853}%
%
\special{pn 8}%
\special{pa 1200 1600}%
\special{pa 1200 1200}%
\special{fp}%
%
\special{pn 8}%
\special{ar 1300 1200 100 100 3.1415927 4.7123890}%
%
\special{pn 8}%
\special{pa 1300 1100}%
\special{pa 1500 1100}%
\special{fp}%
%
\special{pn 8}%
\special{ar 1500 1300 300 200 4.7123890 6.2831853}%
%
\special{pn 8}%
\special{ar 1600 1290 200 200 6.2831853 3.1415927}%
%
\special{pn 8}%
\special{ar 1700 1295 300 200 3.1415927 4.7123890}%
%
\special{pn 8}%
\special{pa 1700 1095}%
\special{pa 1900 1095}%
\special{fp}%
%
\special{pn 8}%
\special{ar 1900 1195 100 100 4.7123890 6.2831853}%
%
\special{pn 8}%
\special{pa 2000 1195}%
\special{pa 2000 1595}%
\special{fp}%
%
\special{pn 8}%
\special{pa 1200 1240}%
\special{pa 1216 1320}%
\special{fp}%
\special{pa 1200 1240}%
\special{pa 1184 1320}%
\special{fp}%
%
\special{pn 4}%
\special{sh 1}%
\special{ar 1592 1108 16 16 0 6.2831853}%
\put(11.6000,-17.2000){\makebox(0,0)[lb]{$\bullet$}}%
\put(19.6000,-17.2000){\makebox(0,0)[lb]{$*$}}%
\put(15.6000,-9.2000){\makebox(0,0)[lb]{$\gamma$}}%
%
\special{pn 13}%
\special{ar 4000 1300 100 100 0.0000000 6.2831853}%
%
\special{pn 8}%
\special{ar 4000 1290 200 200 6.2831853 3.1415927}%
%
\special{pn 13}%
\special{ar 5700 1300 100 100 0.0000000 6.2831853}%
%
\special{pn 8}%
\special{pa 6100 1195}%
\special{pa 6100 1595}%
\special{fp}%
%
\special{pn 8}%
\special{pa 5300 1240}%
\special{pa 5316 1320}%
\special{fp}%
\special{pa 5300 1240}%
\special{pa 5284 1320}%
\special{fp}%
%
\special{pn 8}%
\special{ar 4000 1303 200 180 3.1415927 6.2831853}%
%
\special{pn 8}%
\special{pa 4000 1490}%
\special{pa 4060 1502}%
\special{fp}%
\special{pa 4000 1490}%
\special{pa 4061 1462}%
\special{fp}%
%
\special{pn 8}%
\special{ar 5400 1200 100 100 3.1415927 4.7123890}%
%
\special{pn 8}%
\special{pa 5400 1100}%
\special{pa 6000 1100}%
\special{fp}%
%
\special{pn 8}%
\special{ar 6000 1200 100 100 4.7123890 6.2831853}%
\put(48.0000,-14.0000){\makebox(0,0)[lb]{$\otimes$}}%
\put(27.0000,-14.0000){\makebox(0,0)[lb]{$\mu^f_r(\gamma) =$}}%
\put(35.6000,-17.2000){\makebox(0,0)[lb]{$\bullet$}}%
\put(43.6000,-17.2000){\makebox(0,0)[lb]{$*$}}%
\put(52.6000,-17.2000){\makebox(0,0)[lb]{$\bullet$}}%
\put(60.6000,-17.2000){\makebox(0,0)[lb]{$*$}}%
%
\special{pn 8}%
\special{pa 6100 1600}%
\special{pa 5300 1520}%
\special{fp}%
%
\special{pn 8}%
\special{pa 5300 1520}%
\special{pa 5300 1200}%
\special{fp}%
%
\special{pn 8}%
\special{pa 1640 1600}%
\special{pa 1560 1580}%
\special{fp}%
\special{pa 1640 1600}%
\special{pa 1560 1620}%
\special{fp}%
\put(16.0000,-17.2000){\makebox(0,0)[lb]{$\nu$}}%
\end{picture}}%
\end{center}
\caption{The operation $\mu^f_r$}
\label{fig:mu}
\end{figure}

\begin{rem}
Our sign convention in formula \eqref{eq:muf} is different from \cite{genus0} and \cite{KK15}.
\end{rem}

By exchanging the roles of $*$ and $\bullet$, we obtain another map
\[
\mu^f_l \colon \K\pi \to \K\pi \otimes |\K\pi|.
\]
In more detail, we now use the bijection $\pi \cong \Pi\Sigma(*,\bullet)$.
Let $\gamma\in \pi$ and represent it as a generic immersed path from $*$ to $\bullet$ with $\dot{\gamma}(0)=v_*$ and $\dot{\gamma}(1)=-v_{\bullet}$.
We further assume that $\rot^f(\gamma) = 1/2$.
Then, we have
\begin{equation}
\label{eq:mufbar}
\mu^f_l(\gamma)=-\sum_p \varepsilon(\dot{\gamma}(t_1^p),\dot{\gamma}(t_2^p))\,
\gamma_{0 t_1^p}\gamma_{t_2^p 1} \nu \otimes |\gamma_{t_1^p t_2^p}|
\in \K\pi \otimes |\K\pi|.
\end{equation}
The maps $\mu^f_r$ and $\mu^f_l$ are related by the following equality:
\begin{equation}  \label{eq:mu&mu}
\mu_l^f(\gamma)=
-\mu_r^f(\gamma)^{\circ}+\gamma \otimes {\bf 1}-1\otimes |\gamma|.
\end{equation}

The next proposition describes product formulas for the maps $\mu^f_r$ and $\mu^f_l$  and their relation to the cobracket $\delta^f$.

\begin{prop}
\label{prop:prodmu}
\begin{enumerate}
\item[$(i)$]
For any $a,b\in \K\pi$,
\begin{align}
\mu^f_r(ab) & = \mu^f_r(a)(1\otimes b)+(1\otimes a)\mu^f_r(b)
+(|\cdot |\otimes {\rm id})\kappa(a,b), \label{eq:muuv} \\
\mu^f_l(ab) & = \mu^f_l(a)(b\otimes 1)+(a\otimes 1)\mu^f_l(b)+({\rm id} \otimes |\cdot |)\kappa(b,a). \label{eq:muuv2}
\end{align}
Here, the term $\mu^f_r(a)(1\otimes b)$ in the right hand side of \eqref{eq:muuv} takes multiplication of $\K \pi$ in the following way: $(c \otimes d)(1\otimes b) = c \otimes db$ for $c\in |\K \pi|$, $d\in \K\pi$.
The three other similar terms are interpreted in the same manner.
\item[$(ii)$]
For any $a \in \K\pi$,
\begin{align}
\delta^f(|a|)
& = \alt ({\rm id} \otimes |\cdot |)\mu^f_r(a) + |a|\wedge {\bf 1} \label{eq:altmu} \\
& = ({\rm id} \otimes |\cdot |)\mu^f_r(a)+(|\cdot |\otimes {\rm id})\mu^f_l(a). \label{eq:bul**bul}
\end{align}
Here, $\alt(x\otimes y)=x\otimes y-y\otimes x$ for $x,y\in |\K\pi|$.
\end{enumerate}
\end{prop}

\begin{proof}
For (i), the same argument as in the proof of \cite[Lemma 4.3.3]{KK15} works. 
For convenience of the reader, we give a proof of  equation \eqref{eq:muuv}. For $\alpha, \beta \in \pi$ we choose generic immersed paths  in $\Pi\Sigma(\bullet, *)$ (also denoted by $\alpha$ and $\beta$) such that $\dot{\alpha}(0)=\dot{\beta}(0)=v_\bullet$, $\dot{\alpha}(1)=\dot{\beta}(1)=-v_*$ and $\rot^f(\alpha)=\rot^f(\beta) = -1/2$.
Then, we can modify the path $\alpha \overline{\nu} \beta$ in a neighborhood of $\nu$ so that it also satisfies these conditions.
The resulting path can be used to compute $\mu^f_r(\alpha \beta)$, 
see Figure \ref{fig:modification}.
We will denote by $s$ the parameter on the path $\alpha$ and by $t$ the parameter on $\beta$. Furthermore, we will denote by $p$ self-intersections of $\alpha$, by $q$ self-intersections of $\beta$, by $r$ intersections of $\alpha$ and $\beta$, and by $\varepsilon_p, \varepsilon_q$ and $\varepsilon_r$ the corresponding signs. Then, we have
\begin{align*}
\mu^f_r(\alpha \beta) = & 
\textstyle\sum_p \varepsilon_p |\alpha_{s^p_1 s^p_2}| \otimes (\overline{\nu} \alpha_{0 s^p_1}\alpha_{s^p_2 1})(\overline{\nu} \beta) 
 + \textstyle\sum_q \varepsilon_q |\beta_{t^q_1 t^q_2}| \otimes (\overline{\nu} \alpha)(\overline{\nu} \beta_{0 t^q_1} \beta_{t^q_2 1}) \\
 & + \textstyle\sum_r \varepsilon_r |\alpha_{s^r 1} \overline{\nu} \beta_{0 t^r}| \otimes \overline{\nu} \alpha_{0 s^r} \beta_{t^r 1} \\
= & \ \mu^f_r(\alpha)(1 \otimes \beta) + (1 \otimes \alpha) \mu^f_r(\beta) + (|\cdot| \otimes {\rm id})\kappa(\alpha,\beta).
\end{align*}
Here in the last line we again view $\alpha$ and $\beta$ as elements of $\pi$. Formula \eqref{eq:muuv} now follows by linearity with respect to $\alpha$ and $\beta$.
The proof of formula \eqref{eq:muuv2} is similar, so we omit it.

\begin{figure}
\begin{center}
\input{fig_modification.tex}
\end{center}
\caption{Modifying $\alpha \overline{\nu} \beta$ in a neighborhood of $\nu$}
\label{fig:modification}
\end{figure}

We next prove (ii).
It is sufficient to consider the case where $a = \gamma \in \pi$.
Using the identification $\pi\cong \Pi\Sigma(\bullet,*)$, we choose its representative by a generic immersion $\gamma \colon ([0,1],0,1) \to (\Sigma,\bullet,*)$ such that $\dot{\gamma}(0)=v_{\bullet}$, $\dot{\gamma}(1)=-v_*$, and $\rot^f(\gamma)=-1/2$.
Then inserting a positive monogon into $\gamma \overline{\nu}$, we obtain a generic immersed loop with vanishing rotation number with respect to $f$.
This loop has  a new self-intersection  whose contribution to $\delta^f(|\gamma|)$ is $|\gamma|\wedge {\bf 1}$.
We see that $\alt ({\rm id} \otimes |\cdot |)\mu^f_r(\gamma)$ is equal to the sum of contributions to $\delta^f(|\gamma|)$ from the self-intersections of $\gamma$.
This proves \eqref{eq:altmu}. 
To prove \eqref{eq:bul**bul}, we use \eqref{eq:mu&mu} and \eqref{eq:altmu}.
\end{proof}

The repeated use of the product formula \eqref{eq:muuv} yields the following more general product formula for the operation $\mu^f_r$.
(We have a similar formula for $\mu^f_l$, but we omit it.)

\begin{cor}
\label{cor:muuuu}
Let $a_1,\ldots,a_m\in \K\pi$.
Then,
\begin{align*}
\mu^f_r(a_1\cdots a_m) = &
\textstyle\sum_i (1\otimes a_1\cdots a_{i-1}) \mu^f_r(a_i)
(1\otimes a_{i+1}\cdots a_m) \\
& + \textstyle\sum_i (|\cdot |\otimes {\rm id})\kappa(a_1\cdots a_{i-1},a_i)(1\otimes a_{i+1}\cdots a_m).
\end{align*}
\end{cor}

The next proposition describes how the operations $\delta^f$, $\mu^f_r$ and $\mu^f_l$ depend on the choice of  framing.

\begin{prop}
\label{prop:changef}
Let $f$ and $f'$ be two framings on $\Sigma$ and let $\chi\in H^1(\Sigma, \Z)$ be the cohomology class describing their difference as in Remark \ref{rem:framing}.
Then, for any $\gamma \in \pi$,
\begin{align*}
\mu^{f'}_r(\gamma) & =  \mu^f_r(\gamma)+\chi(\gamma)\, ({\bf 1}\otimes \gamma), \\
\mu^{f'}_l(\gamma) & = \mu^f_l(\gamma)-\chi(\gamma)\, (\gamma \otimes {\bf 1}), \\
\delta^{f'}(|\gamma|) & = \delta^f(|\gamma|)+\chi(\gamma)\, ({\bf 1}\wedge |\gamma|).
\end{align*}
\end{prop}

\begin{proof}
Let $\gamma \colon ([0,1],0,1)\to (\Sigma,\bullet,*)$ be a generic immersion such that $\dot{\gamma}(0)=v_{\bullet}$, $\dot{\gamma}(1)=-v_*$, and $\rot^f(\gamma)=-1/2$.
Then, by \eqref{eq:ff'chi}, we see that by inserting $\chi(\gamma)$ negative monogons to $\gamma$, we obtain a generic immersion $\gamma'$ with $\rot^{f'}(\gamma')=-1/2$.
Since each negative monogon in $\gamma'$ contributes a term ${\bf 1}\otimes \gamma$ to $\mu^{f'}_r(\gamma)$, the formula for $\mu^{f'}_r(\gamma)$ follows.

The other formulas can be proved in a similar way.
\end{proof}

\begin{rem}  \label{rem:mu_red}
As an immediate consequence of Proposition \ref{prop:changef}, the maps $\mu_r^f$ for all framings descend to the canonical map $\mu_r\colon \K \pi \to (|\K\pi|/\K {\bf 1}) \otimes \K\pi$.
Similarly, one obtains the map $\mu_l\colon \K\pi \to \K \pi \otimes (|\K \pi|/\K {\bf 1})$.
These reduced versions of the coaction maps were considered in \cite{KK15}.
\end{rem}

\subsection{Compatibility conditions}
\label{subsec:compa}

The Goldman bracket and framed Turaev cobracket define a Lie bialgebra structure on the space $|\K\pi|$.

\begin{prop}
\label{prop:GTLie}
\begin{enumerate}
\item[$(i)$]
The triple $(|\K\pi|,[\cdot,\cdot]_{\rm Goldman},\delta^f_{\rm Turaev})$ is a Lie bialgebra.
That is, $[\cdot,\cdot]=[\cdot,\cdot]_{\rm Goldman}$ is a Lie bracket on $|\K\pi|$, $\delta^f = \delta^f_{\rm Turaev}$ is a Lie cobracket on $|\K\pi|$, and they satisfy the compatibility condition: for any $x,y\in |\K\pi|$,
\[
\delta^f ([x,y]) =
x(\delta^f(y)) - y(\delta^f(x)).
\]
Here, on the right hand side, $x(\delta^f(y))$ denotes the adjoint action of the Lie algebra $(|\K\pi|,[\cdot, \cdot]_{\rm Goldman})$ on $|\K\pi|\otimes |\K\pi|$.
\item[$(ii)$]
Moreover, this Lie bialgebra structure is involutive in the sense that
\[
[\cdot,\cdot] \circ \delta^f=0 \colon |\K\pi|\to |\K\pi|\otimes |\K\pi|\to |\K\pi|.
\]
\end{enumerate}
\end{prop}

\begin{proof}
To prove the compatibility condition, let $\alpha$ and $\beta$ be generic immersed loops such that $\rot^f(\alpha) = \rot^f(\beta) = 0$.
Observe that for each intersection $p\in \alpha\cap \beta$,
\[
\rot^f(|\alpha_p\beta_p|)=\rot^f(\alpha)+\rot^f(\beta)=0.
\]
Therefore, to compute $\delta^f(|\alpha_p\beta_p|)$, we do not need insertion of monogons.
Then, the same argument as in the proof of \cite[Theorem 8.3]{Tu91} yields
$\delta^f ([\alpha,\beta]) = \alpha(\delta^f(\beta))-\beta(\delta^f(\alpha))$.

The involutivity condition can be proved in the same way as \cite[Proposition B.1]{Cha04}.
\end{proof}

The previous proposition has a refinement of the following form.

\begin{prop} \label{prop:bimodule}
\begin{enumerate}
\item[$(i)$]
The operations $\{ \cdot, \cdot\} = \{ \cdot,\cdot\}_{\kappa} \colon |\K\pi| \otimes \K\pi \to \K\pi$ and $\mu^f_r$ satisfy the following compatibility condition: for all $x\in |\K\pi|$ and $a\in \K\pi$,
\begin{equation*} 
\mu^f_r(\{ x, a\}) =
x(\mu^f_r(a)) + ({\rm id} \otimes \{ \cdot,\cdot\})(\delta^f(x) \otimes a).
\end{equation*}
Here, in the first term of the right hand side the action of $x$ on $|\K\pi|\otimes \K\pi$  is given by $x(y\otimes b) =[x,y]\otimes b + y\otimes \{ x, b\}$ for $y\in |\K\pi|$ and $b \in \K\pi$.
\item[$(ii)$]
The operations $\{\cdot, \cdot \} \colon |\K\pi|\otimes \K\pi \to \K\pi$ and $\mu^f_r$ satisfy the involutivity condition
\[
\{\cdot, \cdot \} \circ \mu^f_r=0 \colon \K\pi \to |\K\pi|\otimes \K\pi \to \K\pi.
\]
\end{enumerate}
\end{prop}

\begin{proof}
An unframed version of this statement was proved in \cite[Proposition 3.2.7]{KK15}, and the same proof works here as well. 
For convenience of the reader, we give an outline of the proof and we point out changes accounting for framing.

(i)
Let $\alpha$ be a free loop with $\rot^f(\alpha) =0$ and let $\beta \colon ([0,1],0,1) \to (\Sigma,\bullet,*)$ be an immersed path with $\dot{\beta}(0) = v_{\bullet}$, $\dot{\beta}(1) = -v_*$ and $\rot^f(\beta) = -1/2$.
We further assume that $\alpha$ and $\beta$ are in general position.
Let us compute
\[
\mu^f_r(\{ \alpha,\beta \}) = \sum_{p\in \alpha \cap \beta}
\varepsilon(\dot{\alpha}_p, \dot{\beta}_p) \,
\mu^f_r(\beta_{\bullet p}\alpha_p\beta_{p*}).
\]
Notice that for every intersection point $p\in \alpha \cap \beta$, one has
$
\rot^f( \beta_{\bullet p} \alpha_p \beta_{p*}) = -1/2
$.
Therefore, to compute $\mu^f_r(\beta_{\bullet p}\alpha_p\beta_{p*})$, we do not need to insert extra monogons.

Self-intersections of $\beta_{\bullet p} \alpha_p \beta_{*p}$ are classified into the following three types:
\begin{enumerate}
    \item[(a)] those coming from self-intersections of $\alpha$;
    \item[(b)] those coming from self-intersections of $\beta$;
    \item[(c)] those coming from intersections of $\alpha$ and $\beta$ that are different from $p$.
\end{enumerate}

We claim that contributions of type  (a) sum up to the expression $({\rm id} \otimes \{ \cdot, \cdot\})(\delta^f(\alpha) \otimes \beta)$.
To see this, write $\delta^f(\alpha) = \alpha' \otimes \alpha''$ by using the Sweedler notation.
Then,
$$
({\rm id} \otimes \{ \cdot, \cdot\})(\delta^f(\alpha) \otimes \beta) = \alpha' \otimes \{ \alpha'', \beta\}
$$ 
is a sum over pairs $(p,q)$, where $p\in \alpha \cap \beta$ and $q$ is a self-intersection of $\alpha$. This is exactly the same set as 
we consider for  contributions of type (a). Checking signs proves the claim.

Contributions of type (b) yield the expression $\alpha(\mu^f(\beta))$.
To see this, write $\mu^f_r(\beta) = \beta' \otimes \beta''$.
Then,
$$
\alpha(\mu^f_r(\beta)) = [\alpha, \beta'] \otimes \beta'' + \beta' \otimes \{ \alpha, \beta''\}.
$$
Let $q$ be a self-intersection of $\beta$.
It contributes to $\mu^f_r(\beta)$ as $\pm |\beta_{t_1^qt_2^q}| \otimes \overline{\nu} \beta_{0t_1^q}\beta_{t_2^q1}$.
If $p \in \alpha \cap \beta$, there are two possibilities:
\begin{enumerate}
    \item[(b-1)] $p$ belongs to the loop $|\beta_{t_1^q t_2^q}|$;
    \item[(b-2)] $p$ belongs to the path  $\beta_{0t_1^q}\beta_{t_1^q1}$.
\end{enumerate}
Checking signs shows that contributions of type (b-1) sum up to 
$[\alpha,\beta'] \otimes \beta''$, and contributions of type (b-2) sum up to 
$\beta' \otimes \{ \alpha,\beta''\}$.

Finally, contributions of type (c) cancel each other.
This is illustrated in Figure \ref{fig:cancellation}.
In this example, there are no self-intersections of type (a) and (b).
One has $\{ \alpha, \beta \} = \varepsilon_p\, \overline{\nu} \ell + \varepsilon_q\, \overline{\nu} m$, where $\varepsilon_p = \varepsilon(\dot{\alpha}_p, \dot{\beta}_p)$ and $\varepsilon_q = \varepsilon(\dot{\alpha}_q, \dot{\beta}_q)$.
On the one hand, we have $\mu^f_r(\ell) = \varepsilon_q \, \gamma \otimes \overline{\nu} \delta$ since $\dot{\ell}(t_1^q) = \dot{\alpha}_q$. 
On the other hand, we have $\mu^f_r(m) = -\varepsilon_p \, \gamma \otimes \overline{\nu} \delta$ since $\dot{m}(t_1^p) = \dot{\beta}_p$.
Therefore,
\[
\mu^f_r(\{ \alpha, \beta\}) =
\varepsilon_p \varepsilon_q \, \gamma \otimes \overline{\nu} \delta
-\varepsilon_q \varepsilon_p \, \gamma \otimes \overline{\nu} \delta
= 0,
\]
as desired.
This concludes the proof of (i).

\begin{figure}
\begin{center}
\input{fig_cancellation.tex}
\end{center}
\caption{Cancellation of the contributions from (c).}
\label{fig:cancellation}
\end{figure}

(ii)
Let $\gamma \colon ([0,1],0,1) \to (\Sigma, \bullet, *)$ be a generic immersed path such that $\dot{\gamma}(0) = v_{\bullet}$, $\dot{\gamma}(1) = -v_*$ and $\rot^f(\gamma) = -1/2$.
The expression $\{ \cdot,\cdot\} (\mu^f_r(\gamma))$ is a sum over ordered pairs $(p,q)$ of distinct self-intersections of $\gamma$ such that $t^p_1 < t_1^q < t_2^p < t_2^q$ or $t^q_1 < t_1^p < t_2^q < t_2^p$.
Moreover, the contribution of the pair $(p,q)$ cancels the contribution of the pair $(q,p)$ in a way similar to contributions of type  (c) in the proof of (i).
Thus we obtain $\{ \cdot, \cdot \} (\mu^f_r(\gamma)) = 0$.
\end{proof}

\subsection{Standard generators and adapted framing}
\label{subsec:stdada}

Propositions~\ref{prop:kdb} and \ref{prop:prodmu} show that the double bracket $\kappa$ and the coaction maps $\mu^f_r, \mu^f_l$ are completely determined by their values on generators of $\pi$. In this subsection, we compute these values  on  standard systems of generators of $\pi$.

By the classification theorem of surfaces, there are unique integers $g,n\ge 0$ such that $\Sigma$ is diffeomorphic to $\Sigma_{g,n+1}$, a compact oriented surface of genus $g$ with $n+1$ boundary components.
Fix a labeling of the boundary components, $\pa \Sigma=\bigsqcup_{j=0}^n \pa_j \Sigma$, such that $*\in \pa_0 \Sigma$.
Then, there is a generating system for $\pi=\pi_1(\Sigma,*)$, $\alpha_i,\beta_i,\gamma_j$, $i=1,\ldots,g$, $j=1,\ldots,n$, such that
\begin{equation} \label{eq:gamma0}
\gamma_0 = \prod_{i=1}^g \alpha_i \beta_i {\alpha_i}^{-1} {\beta_i}^{-1} \prod_{j=1}^n \gamma_j,
\end{equation}
where $\gamma_0$ is the loop around the $0$th boundary $\pa_0 \Sigma$ with negative orientation and $\gamma_j$ is freely homotopic to the $j$th boundary $\pa_j \Sigma$ with positive orientation,
see Figure \ref{fig:abc}.

\begin{figure}
\begin{center}
{\unitlength 0.1in%
\begin{picture}(16.0200,12.0000)(3.8000,-15.2000)%
%
\special{pn 13}%
\special{pa 1980 720}%
\special{pa 1950 720}%
\special{pa 1920 709}%
\special{pa 1864 679}%
\special{pa 1837 661}%
\special{pa 1811 643}%
\special{pa 1785 624}%
\special{pa 1760 604}%
\special{pa 1736 583}%
\special{pa 1712 561}%
\special{pa 1690 538}%
\special{pa 1670 513}%
\special{pa 1659 484}%
\special{pa 1660 480}%
\special{fp}%
%
\special{pn 13}%
\special{pa 1660 488}%
\special{pa 1690 488}%
\special{pa 1720 499}%
\special{pa 1776 529}%
\special{pa 1803 547}%
\special{pa 1829 565}%
\special{pa 1855 584}%
\special{pa 1880 604}%
\special{pa 1904 625}%
\special{pa 1928 647}%
\special{pa 1950 670}%
\special{pa 1970 695}%
\special{pa 1981 724}%
\special{pa 1980 728}%
\special{fp}%
%
\special{pn 13}%
\special{ar 780 720 400 400 2.4980915 5.6396842}%
%
\special{pn 13}%
\special{pa 1218 640}%
\special{pa 1098 480}%
\special{fp}%
%
\special{pn 13}%
\special{pa 1538 640}%
\special{pa 1658 480}%
\special{fp}%
%
\special{pn 13}%
\special{ar 778 720 160 160 0.0000000 6.2831853}%
%
\special{pn 13}%
\special{ar 1378 520 200 200 0.6435011 2.4980915}%
%
\special{pn 13}%
\special{pa 1858 880}%
\special{pa 1978 720}%
\special{fp}%
%
\special{pn 13}%
\special{ar 2018 1000 200 200 3.1415927 3.7850938}%
%
\special{pn 13}%
\special{pa 460 960}%
\special{pa 820 1440}%
\special{fp}%
%
\special{pn 13}%
\special{pa 1820 1000}%
\special{pa 1820 1440}%
\special{fp}%
%
\special{pn 13}%
\special{ar 1320 1440 500 80 0.0000000 6.2831853}%
%
\special{pn 4}%
\special{sh 1}%
\special{ar 1380 1360 16 16 0 6.2831853}%
%
\special{pn 4}%
\special{sh 1}%
\special{ar 1260 1360 16 16 0 6.2831853}%
%
\special{pn 8}%
\special{pn 8}%
\special{pa 1540 648}%
\special{pa 1570 648}%
\special{pa 1597 658}%
\special{fp}%
\special{pa 1651 686}%
\special{pa 1656 689}%
\special{pa 1683 707}%
\special{pa 1700 719}%
\special{fp}%
\special{pa 1748 755}%
\special{pa 1760 764}%
\special{pa 1784 785}%
\special{pa 1793 793}%
\special{fp}%
\special{pa 1835 836}%
\special{pa 1850 855}%
\special{pa 1861 884}%
\special{pa 1860 888}%
\special{fp}%
%
\special{pn 8}%
\special{pn 8}%
\special{pa 940 720}%
\special{pa 965 698}%
\special{pa 987 681}%
\special{fp}%
\special{pa 1039 648}%
\special{pa 1045 645}%
\special{pa 1075 636}%
\special{pa 1098 633}%
\special{fp}%
\special{pa 1160 633}%
\special{pa 1170 634}%
\special{pa 1203 638}%
\special{pa 1220 640}%
\special{fp}%
%
\special{pn 8}%
\special{pa 1380 1360}%
\special{pa 1660 880}%
\special{fp}%
%
\special{pn 8}%
\special{pa 1380 1360}%
\special{pa 1600 840}%
\special{fp}%
%
\special{pn 8}%
\special{ar 1700 880 40 20 3.1415927 4.7123890}%
%
\special{pn 8}%
\special{pa 1700 860}%
\special{pa 1860 880}%
\special{fp}%
%
\special{pn 8}%
\special{pa 1600 840}%
\special{pa 1611 811}%
\special{pa 1600 780}%
\special{pa 1586 749}%
\special{pa 1573 719}%
\special{pa 1549 661}%
\special{pa 1540 640}%
\special{fp}%
%
\special{pn 8}%
\special{pa 1380 1360}%
\special{pa 960 1080}%
\special{fp}%
%
\special{pn 8}%
\special{pa 780 1000}%
\special{pa 812 1001}%
\special{pa 842 1008}%
\special{pa 871 1020}%
\special{pa 899 1036}%
\special{pa 927 1055}%
\special{pa 954 1076}%
\special{pa 960 1080}%
\special{fp}%
%
\special{pn 8}%
\special{pa 1380 1360}%
\special{pa 1120 880}%
\special{fp}%
%
\special{pn 8}%
\special{pa 1380 1360}%
\special{pa 1000 920}%
\special{fp}%
%
\special{pn 8}%
\special{pa 940 720}%
\special{pa 943 752}%
\special{pa 947 784}%
\special{pa 953 816}%
\special{pa 962 846}%
\special{pa 975 875}%
\special{pa 990 903}%
\special{pa 1000 920}%
\special{fp}%
%
\special{pn 8}%
\special{pa 1220 640}%
\special{pa 1235 668}%
\special{pa 1263 726}%
\special{pa 1274 756}%
\special{pa 1284 786}%
\special{pa 1293 817}%
\special{pa 1300 840}%
\special{fp}%
%
\special{pn 8}%
\special{pa 1300 840}%
\special{pa 1380 1360}%
\special{fp}%
%
\special{pn 8}%
\special{ar 780 720 280 280 1.5707963 4.7123890}%
%
\special{pn 8}%
\special{ar 780 720 320 280 4.7123890 6.2831853}%
%
\special{pn 8}%
\special{pa 1100 720}%
\special{pa 1100 752}%
\special{pa 1102 784}%
\special{pa 1106 815}%
\special{pa 1113 847}%
\special{pa 1120 878}%
\special{pa 1120 880}%
\special{fp}%
\put(13.4000,-14.6000){\makebox(0,0)[lb]{$*$}}%
\put(12.2000,-14.6000){\makebox(0,0)[lb]{$\bullet$}}%
%
\special{pn 8}%
\special{pa 1564 918}%
\special{pa 1532 949}%
\special{fp}%
\special{pa 1564 918}%
\special{pa 1570 963}%
\special{fp}%
%
\special{pn 8}%
\special{pa 991 1101}%
\special{pa 1013 1141}%
\special{fp}%
\special{pa 991 1101}%
\special{pa 1036 1107}%
\special{fp}%
%
\special{pn 8}%
\special{pa 1312 916}%
\special{pa 1299 959}%
\special{fp}%
\special{pa 1312 916}%
\special{pa 1338 952}%
\special{fp}%
\put(9.1400,-12.3800){\makebox(0,0)[lb]{$\alpha_1$}}%
\put(13.2400,-8.8800){\makebox(0,0)[lb]{$\beta_1$}}%
\put(16.1400,-10.7800){\makebox(0,0)[lb]{$\gamma_1$}}%
\end{picture}}%
\end{center}
\caption{generators $\alpha_i,\beta_i,\gamma_j$ ($g=1,n=1$)}
\label{fig:abc}
\end{figure}

\begin{rem}
Given a basepoint $*\in \pa \Sigma$ and a labeling $\pa \Sigma=\bigsqcup_{j=0}^n \pa_j \Sigma$ such that $* \in \pa_0 \Sigma$, the choice of a generating system $\alpha_i,\beta_i,\gamma_j$ as above is unique up to an action of the mapping class group of $\Sigma$ relative to the boundary $\pa \Sigma$.
\end{rem}

We compute the operation $\kappa$ on the standard generators $\alpha_i,\beta_i,\gamma_j$.

\begin{prop}
\label{prop:kapvalue}
For $i=1,\ldots,g$,
\begin{align*}
\kappa(\alpha_i,\alpha_i) & =
\alpha_i \otimes \alpha_i - 1 \otimes {\alpha_i}^2, \\
\kappa(\alpha_i,\beta_i) &= \beta_i\otimes \alpha_i, \\
\kappa(\beta_i,\alpha_i) & =
\beta_i \otimes \alpha_i -\alpha_i\beta_i \otimes 1 - 1 \otimes \beta_i \alpha_i, \\
\kappa(\beta_i,\beta_i) & =
\beta_i \otimes \beta_i - {\beta_i}^2\otimes 1.
\end{align*}
For $j=1,\ldots,n$,
\[
\kappa(\gamma_j,\gamma_j)= \gamma_j \otimes \gamma_j -1\otimes {\gamma_j}^2.
\]
If $(x, y) \in \{\alpha_i,\beta_i\} \times \{\alpha_k,\beta_k\}$ with $i<k$,  $(x, y) = (\alpha_i, \gamma_j)$, $(x,y) = (\beta_i, \gamma_j)$ or $(x,y) = (\gamma_j, \gamma_k)$ with $j < k$, then 
\[
\kappa(x,y) = 0
\quad \text{and} \quad
\kappa(y,x) = x\otimes y + y\otimes x- xy \otimes 1-1 \otimes yx.
\]
\end{prop}

\begin{proof}
Figure \ref{fig:pfkap} computes $\kappa(\alpha_1,\beta_1)$ and $\kappa(\beta_1,\alpha_1)$.
The other cases are similar.
\begin{figure}
\begin{center}
\input{fig_computingkappa.tex}
\end{center}
\caption{Computing $\kappa(\alpha_1,\beta_1)$ and $\kappa(\beta_1,\alpha_1)$}
\label{fig:pfkap}
\end{figure}
\end{proof}

\begin{rem}
Since $\kappa$ is a double bracket,
Proposition \ref{prop:kapvalue} gives its characterization.
\end{rem}

There are simple closed curves freely homotopic to $|\alpha_i|,|\beta_i|,|\gamma_j|$.
Abusing notation, we denote them simply by $\alpha_i, \beta_i, \gamma_j$.
A framing on $\Sigma$ is completely determined by the values of its rotation number function on these simple closed curves.

\begin{dfn}
Fix a standard generating system $\alpha_i,\beta_i,\gamma_j$ as above.
The \emph{adapted} framing is the framing $f^{\rm adp}$ with
its rotation number function $\rot^{\rm adp}=\rot^{f^{\rm adp}}$ 
given by
\[
\rot^{\rm adp}(\alpha_i)=\rot^{\rm adp}(\beta_i)=0, \quad
\rot^{\rm adp}(\gamma_j)=-1,
\]
for $i=1,\ldots,g$ and $j=1,\ldots,n$.
\end{dfn}

Note that the Poincar\'e-Hopf theorem implies $\rot^{\rm adp}(\gamma_0) = 2g-1$.

Let $\mu^{\rm adp}_r, \mu^{\rm adp}_l$ be the operations $\mu^f_r, \mu^f_l$ for the adapted framing $f = f^{\rm adp}$.

\begin{prop}
\label{prop:muvalue}
For $i=1,\ldots,g$ and $j = 1, \ldots, n$, we have
\begin{equation*}
\begin{array}{ll}
\mu^{\rm adp}_r(\alpha_i)={\bf 1}\otimes \alpha_i,
& \mu^{\rm adp}_l(\alpha_i) = - 1 \otimes |\alpha_i|, \\
\mu^{\rm adp}_r(\beta_i)=-|\beta_i|\otimes 1, &
\mu_l^{\rm adp}(\beta_i) = \beta_i \otimes {\bf 1}, \\
\mu_r^{\rm adp}(\gamma_j)=0, & \mu_l^{\rm adp}(\gamma_j)=\gamma_j \otimes {\bf 1} - 1 \otimes |\gamma_j|.
\end{array}
\end{equation*}
\end{prop}

\begin{proof}
The values $\mu_r^{\rm adp}(\alpha_i)$ and $\mu_r^{\rm adp}(\beta_i)$ can be seen from Figure \ref{fig:pfmu}, which shows representatives of $\nu \alpha_i$ and $\nu \beta_i$ with rotation number $-1/2$ with respect to $f^{\rm adp}$.
For $\nu \gamma_j$, there is a representative with no self-intersections.
Hence $\mu_r^{\rm adp}(\gamma_j)=0$. The values of $\mu_l^{\rm adp}$ on generators are obtained using equation \eqref{eq:mu&mu}.
\begin{figure}
\begin{center}
\input{fig_computingmu.tex}
\end{center}
\caption{Computing $\mu^{\rm adp}$}
\label{fig:pfmu}
\end{figure}
\end{proof}

Together, Propositions~\ref{prop:muvalue} and \ref{prop:changef} uniquely determine values of the maps $\mu^f_r$ and $\mu^f_l$ on the generating system $\alpha_i, \beta_i, \gamma_j$ for arbitrary framing $f$. This is summarized in the following theorem:
\begin{thm}   \label{thm:values_mu^f}
The maps $\mu^f_r$ and $\mu^f_l$ take the following values on generators $\alpha_i, \beta_i, \gamma_j$.
For $i = 1, \ldots, g$ and $j = 1,\ldots, n$, we have
\begin{equation*}
    \begin{array}{ll}
    \mu^f_r(\alpha_i)=({\rm rot}^f(\alpha_i) +1)({\bf 1}\otimes \alpha_i),
& \mu^f_l(\alpha_i) =-{\rm rot}^f(\alpha_i)(\alpha_i \otimes {\bf 1}) - 1 \otimes |\alpha_i|, \\
\mu^f_r(\beta_i)={\rm rot}^f(\beta_i)({\bf 1} \otimes \beta_i)-|\beta_i|\otimes 1, &
\mu_l^f(\beta_i) = (1 - {\rm rot}^f(\beta_i))(\beta_i \otimes {\bf 1}), \\
\mu_r^f(\gamma_j)=({\rot}^f(\gamma_j) + 1)({\bf 1} \otimes \gamma_j), & 
\mu_l^f(\gamma_j)=- {\rot}^f(\gamma_j)(\gamma_j \otimes {\bf 1}) - 1 \otimes |\gamma_j|.
    \end{array}
\end{equation*}
\end{thm}

\begin{rem}  \label{rem:mu_red_generators}
Note that the canonical coaction maps $\mu_r$ and $\mu_l$ (see Remark \ref{rem:mu_red}) take the following form on generators:
for $i = 1, \ldots, g$ and $j = 1, \ldots, n$,
$$
\begin{array}{ll}
\mu_r(\alpha_i)=0,
& \mu_l(\alpha_i) = - 1 \otimes |\alpha_i|, \\
\mu_r(\beta_i)=-|\beta_i|\otimes 1, &
\mu_l(\beta_i) = 0, \\
\mu_r(\gamma_j)=0, & 
\mu_l(\gamma_j)= - 1 \otimes |\gamma_j|.
\end{array}
$$
\end{rem}

\subsection{Functoriality and automorphisms}
In this subsection, we recall how the group algebra $\K\pi$ and the operations $\kappa$ and $\mu^f_r, \mu^f_l$ behave under embedding of surfaces. We also discuss automorphisms of $\K\pi$ equipped with these operations.

Let $\Sigma$ and $\overline{\Sigma}$ be two compact connected oriented surfaces with nonempty boundary.
Choose basepoints and labeling of the boundary components of $\Sigma$ and $\overline{\Sigma}$ such that $* \in \partial_0 \Sigma$ and $\bar{*} \in \partial_0 \overline{\Sigma}$.
Let $i \colon \Sigma \to \overline{\Sigma}$ be an embedding of oriented surfaces mapping $\partial_0 \Sigma$ to $\partial_0 \overline{\Sigma}$ and $*$ to $\bar{*}$. We denote the fundamental groups by $\pi = \pi_1(\Sigma, *)$ and $\bar{\pi}=\pi_1(\overline{\Sigma}, \bar{*})$. The map $i$ induces a group homomorphism $i_* \colon \pi \to \bar{\pi}$ and a Hopf algebra homomorphism $i_* \colon \K\pi \to \K\bar{\pi}$. The latter induces a map (denoted by the same symbol) $i_* \colon |\K\pi| \to |\K\bar{\pi}|$.

Denote the double brackets on $\K\pi$ and $\K\bar{\pi}$ by $\kappa$ and $\bar{\kappa}$, respectively. Let $\bar{f}$ be a framing on $\overline{\Sigma}$ and denote by $f$ its restriction to $\Sigma$. Denote by $\mu^f_r, \mu^f_l$ and $\mu^{\bar{f}}_r, \mu^{\bar{f}}_l$ the corresponding coaction maps.

\begin{prop} \label{prop:functorial}
The map $i_*$ defines a homomorphism of double bracket and coaction structures
$$
i_* \colon (\K\pi, \kappa, \mu^f_r, \mu^f_l) \to (\K\bar{\pi}, \bar{\kappa}, \mu^{\bar{f}}_r, \mu^{\bar{f}}_l).
$$
Namely, we have $\bar{\kappa} \circ (i_* \otimes i_*) = (i_* \otimes i_*) \circ \kappa$, $\mu^{\bar{f}}_r \circ i_* = (i_* \otimes i_*)\circ \mu^f_r$ and $\mu^{\bar{f}}_l \circ i_* = (i_* \otimes i_*)\circ \mu^f_l$.
\end{prop}

\begin{proof}  
All the operations are defined in terms of intersections and self-intersections of curves, and those are preserved under embeddings of oriented surfaces. Furthermore, rotation numbers of curves in $\Sigma$ are preserved under the embedding $i$ since the framing $\bar{f}$ restricts to $f$. This completes the proof.
\end{proof}

\begin{rem}  \label{rem:capping_trick}
Note that if $i_*$ is injective, the operations $\kappa$ and $\mu^f_r, \mu^f_l$ are completely determined by $\bar{\kappa}$ and $\mu^{\bar{f}}_r, \mu^{\bar{f}}_l$.

\begin{figure}
\begin{center}
\input{fig_cappingtrick.tex}
\end{center}
\caption{Capping trick ($g=1, n=2$)}
\label{fig:cappingtrick}
\end{figure}

An example of such a situation is the following {\em capping trick}.
Let $\Sigma$ be a surface of genus $g$ with $n+1$ boundary components. 
Define the surface $\overline{\Sigma}$ by attaching $n$ copies of $\Sigma_{1,1}$ to the boundary components $\pa_j \Sigma$, $j=1,\ldots,n$:
\[
\overline{\Sigma}=\Sigma \cup_{\pa}
\big( \bigcup_{j=1}^n \Sigma_{1,1} \big).
\]
Here, $\cup_{\pa}$ denotes the gluing along $\bigcup_{j=1}^n \pa_j \Sigma$.
See Figure \ref{fig:cappingtrick}.
The group $\bar{\pi}$ is a free group of rank $2(g+n)$.
Given a standard generating set $\{ \alpha_i, \beta_i \}_{i=1}^g \cup \{ \gamma_j \}_{j=1}^n$ of $\pi$, one can choose a standard generating set of $\bar{\pi}$ in a way that the first $g$ pairs of generators are $i_* \alpha_i, i_* \beta_i$ and the last $n$ pairs $\alpha_{g+j}, \beta_{g+j}$, $j=1, \ldots, n$, satisfy
$i_* \gamma_j = \alpha_{g+j} \beta_{g+j} \alpha_{g+j}^{-1} \beta_{g+j}^{-1}$.
Note that for any framing $\bar{f}$ on $\overline{\Sigma}$ we have ${\rm rot}^f(\gamma_j)=1$ for the induced framing $f$ on $\Sigma$.
\end{rem}

A special case of surface embeddings are self-diffeomorphisms of $\Sigma$. Recall that the {\em mapping class group} of $\Sigma$ is defined as the group of connected components of the diffeomorphism group of $\Sigma$ fixing the boundary pointwise:
$$
\mathcal{M} = \mathcal{M}(\Sigma) := {\rm Diff}(\Sigma, \partial \Sigma)/{\rm Diff}_0(\Sigma, \partial \Sigma).
$$
There is a natural action of $\mathcal{M}$ on the fundamental group $\pi$, which is called the {\em Dehn-Nielsen action}.
It extends naturally to the action of $\mathcal{M}$ on the group algebra $\K \pi$.

For a given framing $f$ on $\Sigma$, denote by $\mathcal{M}^f = \mathcal{M}^f(\Sigma)$ the subgroup of the mapping class group consisting of elements that preserve $f$.
By restriction, the group $\mathcal{M}^f$ acts on the group algebra $\K \pi$.
Our main interest in the {\em framed mapping class group} $\mathcal{M}^f$ comes from the following fact:
\begin{prop} \label{prop:MCGfaction}
The action of $\mathcal{M}^f$ on $\K \pi$ preserves the double bracket $\kappa$ and the coaction maps $\mu^f_r$, $\mu^f_l$.
\end{prop}

\begin{proof}
This follows from Proposition~\ref{prop:functorial}.
\end{proof}

In general, any element $\varphi\in \mathcal{M}$ maps a framing $f$ to another framing $\varphi(f)$ and induces an isomorphism
$$
(\K\pi, \kappa, \mu^{\varphi(f)}_r, \mu^{\varphi(f)}_l) \cong (\K\pi, \kappa, \mu^f_r, \mu^f_l).
$$

\begin{rem} \label{rem:Torelli_n=0}
When $n=0$, the action of $\mathcal{M}$ on $\K \pi$ is faithful.
Thus we can regard the framed mapping class group as a subgroup of the automorphism group of $\K \pi$ together with the double bracket $\kappa$ and the coaction maps $\mu^f_r, \mu^f_l$:
$$
\mathcal{M}^f \subset {\rm Aut}(\K\pi, \kappa, \mu^f_r, \mu^f_l).
$$
In Section~\ref{sec:Johnson}, we use this fact to study the Torelli group of $\Sigma$.

When $n > 0$, i.e., $\Sigma$ has more than one boundary component, the action of $\mathcal{M}$ on $\K \pi$ is not faithful.
In fact, the Dehn twists along the $j$th boundary component ($j = 1,\ldots, n$) acts trivially on $\pi$, but it defines a nontrivial element of $\mathcal{M}$. 
To get a faithful action of the mapping class group, we would need  to consider the action on the \emph{fundamental groupoid} of $\Sigma$, where we choose basepoints from each boundary component of $\Sigma$.
For simplicity our exposition in Section~\ref{sec:Johnson} will focus on the case of $n = 0$.
It is possible to extend it to the case of $n > 0$ following  \cite{KK15, KK16}.
\end{rem}

\section{Filtrations and associated graded} \label{sec:filt}

In this section, we consider the natural {\em weight filtration} on the group algebra $\K \pi$ and the induced filtration on the space $|\K \pi|$. These filtrations give rise to completions $\widehat{\K\pi}$ and $\widehat{|\K\pi|}$ and to their associated graded objects ${\rm gr}^{\rm wt}\,  \K\pi$ and ${\rm gr}^{\rm wt} \, |\K\pi| $. 

All the operations of the previous section extend to completions and descend to graded operations on associated graded objects. In particular, $|\widehat{\K\pi}|$ is a filtered involutive Lie bialgebra and ${\rm gr}^{\rm wt}\, |\K\pi|$ is a graded involutive Lie bialgebra. In both cases, the Lie cobracket depends on the choice of framing. 

We will focus on the double bracket $\kappa$ and on the coaction maps $\mu^f_r, \mu^f_l$ since they uniquely determine the Goldman bracket and Turaev cobracket.

\subsection{Weight filtrations on free groups and free associative algebras}

In this subsection, we recall some standard facts about filtrations on group algebras of free groups and on free associative algebras which represent the corresponding associated graded objects.

Let $\Gamma$ be a free group of finite rank. 
The group algebra $\K \Gamma$ has a structure of a Hopf algebra whose coproduct, augmentation map and antipode are given by formulas
\[
\Delta(\gamma)=\gamma\otimes \gamma,
\quad \varepsilon(\gamma)=1,
\quad
\iota(\gamma)=\gamma^{-1}
\]
for any $\gamma \in \Gamma$. Denote by $I\Gamma=\ker(\varepsilon)$ the augmentation ideal and define the completion of the group algebra
$$
\widehat{\K \Gamma}= \varprojlim_m \, \K \Gamma/(I\Gamma)^m.
$$
The Hopf algebra structure on $\K\Gamma$ naturally extends to a complete Hopf algebra structure on $\widehat{\K\Gamma}$.
Since $\Gamma$ is a free group of finite rank, 
the canonical map $\K\Gamma \to \widehat{\K \Gamma}$ is injective
(see e.g. \cite[Ch.\ 2, Exer.\ \S 5, 1(b)]{Bou71})
and one can identify $\mathbb{K}\Gamma$ with its image in $\widehat{\K \Gamma}$.

The completed group algebra $\widehat{\K \Gamma}$ carries a decreasing filtration by kernels of quotient maps $\widehat{\K \Gamma} \to \K \Gamma/(I\Gamma)^m$, $m\ge 0$. Besides this canonical filtration, there are many other interesting decreasing filtrations on 
$\widehat{\K \Gamma}$, and we describe some of them below.

Let $\{ \gamma_j\}_{j=1}^n$ be a (finite) free generating set of $\Gamma$, where $n$ is the rank of $\Gamma$. The group algebra 
$\K \Gamma$ is generated by elements of the form $(\gamma_j^{\pm 1} -1)$. The completed group algebra $\widehat{\K \Gamma}$ is freely generated by elements 
$$
Z_j=\gamma_j -1
$$
since
$
\gamma^{-1} -1 = \sum_{k=1}^\infty (-1)^k (\gamma-1)^k
$
for any $\gamma \in \Gamma$.
In terms of generators $Z_j$, the Hopf algebra operations on $\widehat{\K \Gamma}$ acquire the form
\begin{equation}    \label{eq:Zs}
\Delta(Z_j)=
 Z_j \otimes 1 + 1 \otimes Z_j + Z_j \otimes Z_j, \hskip 0.3cm 
\varepsilon(Z_j)=0, \hskip 0.3cm
\iota(Z_j) = \sum_{k=1}^\infty (-1)^k {Z_j}^k.
\end{equation}
It is convenient to introduce a notation
$$
Z_{J}=Z_{j_1} \cdots Z_{j_l},
$$
for elements of a basis of the filtered vector space $\widehat{\K \Gamma}$. Here $l\in \mathbb{Z}_{\geq 0}$,  $J=(j_1, \dots, j_l)$  and the basis element corresponding to $l=0$ is the unit of $\widehat{\K \Gamma}$.

\begin{rem} \label{rem:Magnus}
The basis $\{ Z_J \}_J$ defines a $\K$-algebra isomorphism
$$
\widehat{\mathbb{K} \Gamma} \cong
\K \langle \langle Z_1, \ldots, Z_n \rangle \rangle
= \widehat{T} \left( {\textstyle \bigoplus}_{j=1}^n \mathbb{K} Z_j \right).
$$
This isomorphism is given on generators by formula $\gamma_j=1+Z_j$ and is called \emph{the Magnus expansion} \cite{Mag35} (for a discussion of more general expansions, see Section~\ref{sec:expansions+KV}).
\end{rem}

Let us assign positive integer weights $w_j = {\rm wt}(Z_j) \in \mathbb{Z}_{>0}$ to generators $Z_j = \gamma_j - 1$ and weights ${\rm wt}(Z_J) = \sum_{k=1}^l w_{j_k}$ to  elements of the basis $\{ Z_J \}_J$.
This assignment defines a decreasing filtration on $\widehat{\K \Gamma}$ by two-sided ideals $\widehat{\K \Gamma}(m)$ defined by  
$$
\widehat{\K\Gamma}(m)=\{ \sum_J c_J Z_J ; c_J \in \K, {\rm wt}(Z_J)\geq m \},
\quad m \ge 0.
$$
It follows from the definition that it is multiplicative in the sense that $\widehat{\K\Gamma}(m)\cdot \widehat{\K\Gamma}(n) \subset \widehat{\K\Gamma}(m+n)$.
This filtration restricts to a decreasing filtration $\{ \K\Gamma(m)\}_m$ on the group algebra $\K\Gamma \subset \widehat{\K \Gamma}$ by multiplicative two-sided ideals which is called the \emph{weight filtration} on $\K\Gamma$.
More explicitly,
\[
\K \Gamma(m) = \widehat{\K \Gamma}(m) \cap \K \Gamma.
\]

Let us collect some elementary properties of the weight filtration.
First, equations~\eqref{eq:Zs} show that the weight filtrations $\{ \widehat{\K\Gamma}(m) \}_m$ and $\{ \K\Gamma(m)\}_m$ are preserved by Hopf algebra operations (for any choice of weights).
Second, the natural map $\K\Gamma(m)/\K\Gamma(m+1) \to
\widehat{\K\Gamma}(m)/\widehat{\K\Gamma}(m+1)$
is a $\K$-linear isomorphism for any $m\ge 0$.
This follows from the fact that the set $\{ Z_J \ {\rm mod} \ \widehat{\K\Gamma}(m+1) ; {\rm wt}(Z_J) =m \}$ is a $\K$-basis of $\widehat{\K\Gamma}(m)/\widehat{\K\Gamma}(m+1)$.
Consequently, the completion of $\K\Gamma$ with respect to the filtration $\{\K\Gamma(m)\}_m$ is canonically isomorphic to $\widehat{\K\Gamma}$.
Third, an alternative description of $\K\Gamma(m)$ is given as follows:

\begin{prop} \label{prop:KGamma(m)}
For any $m\ge 0$, the set $\mathcal{Z}^{(m)}:=\{ Z_J ; {\rm wt}(Z_J) \ge m\}$ generates $\K\Gamma(m)$ as a two-sided ideal of $\K\Gamma$.
In particular, we have $\K\Gamma(1) = I\Gamma$.
\end{prop}

\begin{proof}
See Appendix~\ref{subsec:alt_wt_filt}.
\end{proof}

We will need the following elementary lemma.
In what follows, we denote by $\equiv_m$ the equivalence modulo elements of $\K\Gamma(m)$.

\begin{lem} \label{lem:commutator_degree}
\begin{enumerate}
    \item[(i)]
    Let $\alpha \in \Gamma$ such that $(\alpha - 1) \in \K\Gamma(m)$.
    Then, $(\alpha^{-1} - 1) \in \K\Gamma(m)$ and $(\alpha^{-1} - 1) \equiv_{m+1} -(\alpha - 1)$.
    \item[(ii)]
    Let $\alpha, \beta \in \Gamma$ such that $(\alpha - 1), (\beta-1) \in \K\Gamma(m)$.
    Then, $(\alpha\beta - 1) \in \K\Gamma(m)$ and
    \[
    (\alpha\beta - 1) \equiv_{m+1} (\alpha - 1) + (\beta - 1).
    \]
    \item[(iii)]
    Let $\alpha, \beta \in \Gamma$ such that $(\alpha -1) \in \K\Gamma(l)$ and $(\beta - 1) \in \K\Gamma(m)$. Then, their group commutator $\gamma=\alpha \beta \alpha^{-1} \beta^{-1}$ satisfies $(\gamma -1) \in \K \Gamma(l+m)$.
    Furthermore, we have
    \[
    (\gamma - 1) \equiv_{l+m+1} [(\alpha - 1), (\beta - 1)].
    \]
\end{enumerate}
\end{lem}

\begin{proof}
(i) Since $\K\Gamma(1) = I\Gamma$ and the weight filtration is multiplicative, we compute $(\alpha^{-1} -1) = -\alpha^{-1}(\alpha -1) \equiv_{m+1} -(\alpha -1) \in \K\Gamma(m)$.

(ii) This follows from
$(\alpha\beta - 1) = (\alpha - 1) + \alpha (\beta - 1) \equiv_{m+1} (\alpha - 1) + (\beta - 1)$.

(iii)
Set $p= l + m + 1$.
We compute
\begin{align*}
 & \hspace{1.5em} \gamma-1 \\
 & = 
(\alpha -1 +1)(\beta-1 +1)(\alpha^{-1} -1 +1)(\beta^{-1}-1+1) -1 \\
& \equiv_{p} 
(\alpha - 1)(\beta - 1) + (\alpha - 1)(\beta^{-1} - 1)
+(\beta - 1)(\alpha^{-1} - 1) + (\alpha^{-1} - 1)(\beta^{-1} - 1) \\
& \equiv_{p} 
(\alpha - 1)(\beta - 1) - (\alpha - 1)(\beta - 1)
-(\beta - 1)(\alpha - 1) + (\alpha - 1)(\beta - 1) \\
& =
[(\alpha -1), (\beta-1)],
\end{align*}
where we have used (i) in the fourth line.
The right hand side is manifestly in $\K\Gamma(l + m)$.
\end{proof}

\begin{rem}
The filtration $\{ \K\Gamma(m)\}_m$ depends on the choice of generating set $\{ \gamma_j\}_{j=1}^n$ and weights $w_j = {\rm wt}(Z_j)$.
If $w_j=1$ for all $j$, we obtain again the canonical filtration by powers of the augmentation ideal $(I \Gamma)^m$. Hence, in this case the filtration is independent of the choice of the generating set.
\end{rem}

Consider the first homology of the group $\Gamma$:
$$
H := H_1(\Gamma, \K) = (\Gamma/ [\Gamma, \Gamma]) \otimes_{\Z} \K.
$$ 
We will denote by $z_j=[\gamma_j] \in H$ the homology classes of generators of $\Gamma$. The choice of weights $w_j$ induces a filtration on $H$:
$$
H^{(m)} = \K \{ z_j ; w_j \geq m \}.
$$
Its associated graded is of the form
$$
{\rm gr}^{\rm wt} \, H = \bigoplus_{m \ge 1} H^{(m)}/H^{(m+1)} \cong \bigoplus_{m \ge 1} \K\{ z_j ; w_j = m \}.
$$
The tensor algebra $T({\rm gr}^{\rm wt} \, H)$ generated by ${\rm gr}^{\rm wt}\, H$ has a structure of a graded Hopf algebra with operations
\begin{equation}  \label{eq:zs}
\Delta(z_j) = z_j \otimes 1 + 1\otimes z_j, \hskip 0.3cm
\varepsilon(z_j)=0, \hskip 0.3cm
\iota(z_j)=-z_j.
\end{equation}
Using grading, one can define a completion
\[
A:= \widehat{T}({\rm gr}^{\rm wt} \, H) = \prod_{m \ge 0} ({\rm gr}^{\rm wt}\, H)^{\otimes m} \cong \mathbb{K}\langle\langle z_1, \dots, z_n \rangle\rangle.
\]
The Hopf algebra operations of $T({\rm gr}^{\rm wt} \, H)$ extend to the completion and give rise to a structure of a (complete) graded Hopf algebra on $A$.

The associated graded of $\K \Gamma$ with respect to the weight filtration is defined as a complete graded vector space
$$
{\rm gr}^{\rm wt}\, \K \Gamma := \prod_{m\geq 0} \K \Gamma(m)/\K \Gamma(m+1).
$$
For an element $a \in \K \Gamma(m)$ we denote by 
$$
{\rm gr}(a) \in \K \Gamma(m)/\K \Gamma(m+1)
$$ 
its projection to the associated graded.
The following proposition describes the Hopf algebra structure of the associated graded ${\rm gr}^{\rm wt} \, \K \Gamma$:

\begin{prop}
\label{prop:grh}
There is a unique isomorphism of complete graded Hopf algebras
\begin{equation*} 
\gr^{\rm wt} \, \K \Gamma \cong
\widehat{T}(\gr^{\rm wt} \, H)
= A
\end{equation*}
which maps ${\rm gr}(Z_j)$ to $z_j$.
\end{prop}

\begin{proof}
The assignment ${\rm gr}(Z_j) \mapsto z_j$ uniquely determines the map on ${\rm gr}(Z_J)$:
$$
{\rm gr}(Z_J)={\rm gr}(Z_{j_1}) \cdots {\rm gr}(Z_{j_l}) \mapsto z_{j_1} \cdots z_{j_l}.
$$
By applying the projection ${\rm gr}$ to equations~\eqref{eq:Zs} we obtain equations~\eqref{eq:zs}.
Hence this map is a Hopf algebra isomomorphism.
\end{proof}

The weight filtration on $\K\Gamma$ and $\widehat{\K\Gamma}$ induce the weight filtration on the spaces $|\K\Gamma|$ and $|\widehat{\K\Gamma}|$.
As a consequence of Proposition \ref{prop:grh}, there is a canonical isomorphism of grade $\K$-vector spaces
\begin{equation} \label{eq:isowithgraded_abs}
{\rm gr}^{\rm wt}\, |\K\Gamma| \cong |\widehat{T}(\gr^{\rm wt} \, H)| = |A|.
\end{equation}

\subsection{The weight filtration on the group algebra $\K\pi$} \label{subsec:filtonKpi}
In this subsection, we apply the formalism of the previous subsection to the fundamental group $\pi = \pi_1(\Sigma)$ of a compact oriented surface $\Sigma$ with nonempty boundary.

Since $\partial \Sigma \neq \emptyset$, $\pi$ is a free group.
We assume that $\Sigma$ is a surface of genus $g$ with $n+1$ boundary components, where $g,n\ge 0$, and we choose a standard generating system of $\pi$ given by $\alpha_i, \beta_i$ for $i=1, \dots, g$ and $\gamma_j$ for $j=1, \dots, n$, as in Section~\ref{subsec:stdada}. We denote the corresponding homology classes in $H = H_1(\pi, \K)$ by
$$
[\alpha_i] = x_i, \hskip 0.3cm 
[\beta_i] = y_i, \hskip 0.3cm
[\gamma_j] = z_j.
$$

\begin{dfn}
The {\em weight filtration} on $\K\pi$ is defined by the assignment
${\rm wt}(\alpha_i - 1) = {\rm wt}(\beta_i - 1) = 1$ for all $i = 1, \ldots, g$ and
${\rm wt}(\gamma_j - 1) = 2$ for all $j = 1, \ldots, n$.
\end{dfn}

The weight filtration on $\K \pi$ induces a weight filtration on $H$ with 
$$
{\rm deg}(x_i) = {\rm deg}(y_i) =1, \hskip 0.3cm {\rm deg}(z_j)=2.
$$
A more invariant description of this filtration is as follows. Consider the intersection pairing $\langle \cdot, \cdot\rangle$ on $H \cong H_1(\Sigma, \K)$.
Then, we have
\[
H^{(1)}=H
\hspace{1em}
\text{and}
\hspace{1em}
H^{(2)}= \{ a \in H ; \text{$\langle a, b \rangle = 0 $ for any $b \in H$} \}.
\]
By abuse of notation, we denote the generators of ${\rm gr}^{\rm wt} \, H$ by the same letters $x_i, y_i, z_j$. 
In particular, $H/H^{(2)}=\bigoplus_{i=1}^g (\K x_i \oplus \K y_i)$.
In view of the isomorphism ${\rm gr}^{\rm wt} \, \K\pi \cong \widehat{T}({\rm gr}^{\rm wt}\, H) = A$ in Proposition \ref{prop:grh}, elements $x_i,y_i,z_j \in {\rm gr}^{\rm wt} \, H$ are given by
\[
x_i = {\rm gr}(\alpha_i - 1), \quad
y_i = {\rm gr}(\beta_i - 1), \quad
z_j = {\rm gr}(\gamma_j - 1).
\]

The intersection pairing descends to a (non-degenerate) skew-symmetric bilinear form
\[
\langle \cdot, \cdot \rangle \colon
(H/H^{(2)}) \times (H/H^{(2)}) \to \K.
\]
On generators, this gives
$$
\langle x_i, x_k \rangle = \langle y_i, y_k\rangle =0, \hskip 0.3cm 
\langle x_i, y_k\rangle = - \langle y_i, x_k \rangle=  \delta_{ik}.
$$
For future use, it is convenient to introduce one more operation on $H^{(2)}=\bigoplus_{j=1}^n \K z_j$:
\[
\mathfrak{z} \colon H^{(2)} \times H^{(2)} \to H^{(2)},
\quad \mathfrak{z}(z_j,z_k):=\delta_{jk} z_k.
\]
Both the skew-symmetric form $\langle \cdot, \cdot\rangle$ and the symmetric bilinear operation $\mathfrak{z}$ naturally extend to ${\rm gr}^{\rm wt} \, H = (H/H^{(2)}) \oplus H^{(2)}$, where $\langle \cdot, \cdot \rangle$ extends by zero on $H^{(2)}$ and $\mathfrak{z}$ extends by zero on $H/H^{(2)}$.

Recall the element $\gamma_0 \in \pi$ which corresponds to the loop around the $0$th boundary component $\pa_0 \Sigma$.
(See equation \eqref{eq:gamma0}.)

\begin{lem} \label{lem:omega}
We have $(\gamma_0 -1) \in \K\pi(2)$.
The element $\omega := {\rm gr}(\gamma_0 -1)$ can be written as  
\begin{equation*}  
\omega = \sum_{i=1}^g [x_i, y_i] + \sum_{j=1}^n z_j.
\end{equation*}
\end{lem}

\begin{proof}
Applying Lemma \ref{lem:commutator_degree}~(ii)(iii) to equation \eqref{eq:gamma0}, we obtain
\begin{equation*}  
\gamma_0 -1 \equiv_3 \sum_{i=1}^g [\alpha_i -1, \beta_i -1] + \sum_{j=1}^n (\gamma_j-1)
\in \K\pi(2).
\end{equation*}
Projecting this equation to $\mathbb{K}\pi(2)/\mathbb{K}\pi(3)$, we obtain the required expression for $\omega$.
\end{proof}

\begin{rem}
Note that under the filtration with ${\rm wt}(\alpha_i - 1)={\rm wt}(\beta_i - 1) = {\rm wt}(\gamma_j - 1)=1$, we would have had $(\gamma_0-1) \in \K\pi(1)$ and ${\rm gr}(\gamma_0-1)=\sum_{j=1}^n z_j$ which corresponds to the relation between homology classes of  boundary components of $\Sigma$.
\end{rem}

\begin{prop}  \label{prop:filtration_independent}
The weight filtration on $\K\pi$ is independent of the choice of standard generating system of $\pi$.
\end{prop}

\begin{proof}
We make use of  the capping trick of Remark \ref{rem:capping_trick}. Recall that $\overline{\Sigma}$ is a surface of genus $g+n$ with one boundary component and $\bar{\pi}$ is its fundamental group. The weight filtration on $\widehat{\mathbb{K} \bar{\pi}}$ and on $\mathbb{K} \bar{\pi}$ is defined by powers of the augmentation ideal. Hence, it is independent of the choice of the generating set. We will show that it restricts to the weight filtrations on $\widehat{\mathbb{K} \pi}$ and $\mathbb{K} \pi$. 

First, we show that the natural injective maps $i_*\colon \mathbb{K} \pi \to \mathbb{K} \bar{\pi}$ and $i_*\colon \widehat{\mathbb{K} \pi} \to \widehat{\mathbb{K} \bar{\pi}}$ preserve the weight filtration.
Let $\alpha_i, \beta_i, \gamma_j$ for $i=1, \dots, g, j=1, \dots, n$ be a standard generating system of $\pi$.
Choose a generating set $\bar{\alpha}_k, \bar{\beta}_k$ of $\bar{\pi}$ such that
\begin{equation}   \label{eq:capping}
\begin{array}{ll}
 i_* \alpha_i = \bar{\alpha}_i, \quad i_* \beta_i = \bar{\beta}_i &
{\rm for} \,\, i=1, \dots, g; \\
i_* \gamma_j= \bar{\alpha}_{g+j} \bar{\beta}_{g+j} \bar{\alpha}^{-1}_{g+j}
\bar{\beta}^{-1}_{g+j} & {\rm for} \,\, j=1, \dots, n.
\end{array}
\end{equation}
Then, with respect to the weight filtration, we have ${\rm wt}(\alpha_i - 1) = {\rm wt}(\beta_i - 1) = 1$ in $\K\pi$, and ${\rm wt}(\bar{\alpha}_i - 1) = {\rm wt}(\bar{\beta}_i - 1) = 1$ in $\K\bar{\pi}$ as well.
Also, ${\rm wt}(\gamma_j - 1) = 2$ by definition, and $i_*(\gamma_j - 1) \in \K\bar{\pi}(2)$ by Lemma~\ref{lem:commutator_degree}~(iii).
This proves the claim.

Next, we consider the associated graded of the injection $i_* \colon \mathbb{K} \pi \to \mathbb{K} \bar{\pi}$.
On generators, the associated graded of $i_*$ takes the form
\begin{equation}  \label{eq:graded_capping}
(i_*)_{\rm gr} \colon x_i \mapsto \bar{x}_i, \quad
y_i \mapsto \bar{y}_i, \quad
z_j \mapsto [\bar{x}_{g+j}, \bar{y}_{g+j}].
\end{equation}
This map is injective, and this implies that $i_* \colon \K\pi/\K\pi(m) \to \K \bar{\pi} / \K \bar{\pi}(m)$ is injective for all $m\ge 0$ as well.
Hence, the restriction of the weight filtration on $\mathbb{K} \bar{\pi}$ to $\mathbb{K} \pi$ coincides with the weight filtration on $\mathbb{K} \pi$.
Therefore, the latter is independent of the choice of standard generating system, as required.
\end{proof}

\begin{rem}
The weight filtration for the group algebra of the fundamental group  of bordered surfaces is compatible with injective maps $\K\pi \to \K\pi'$ induced by embeddings of surfaces $\Sigma \to \Sigma'$ which map $\pa_0\Sigma$ to $\pa_0 \Sigma'$ and $*$ to $*'$.
An example of such a statement is the use of the capping trick in the proof of Proposition~\ref{prop:filtration_independent}.
\end{rem}

\subsection{Associated graded of the double bracket} \label{subsec:gr_kappa}
In this subsection, we define and study the associated graded of the double bracket $\kappa$ under the weight filtration on $\K \pi$.

\begin{dfn}
Let $V$ and $W$ be filtered $\K$-vector spaces which carry decreasing filtrations $\{ V(m)\}_m$ and $\{ W(m) \}_m$. We say that the map $\lambda \colon V \to W$ is of filtration degree $d \in \mathbb{Z}$ if $\lambda(V(m)) \subset W(m+d)$ for any $m$, and the associated graded map
$$
\lambda_{\rm gr} \colon {\rm gr}\, V = \prod_m V(m)/V(m+1) \to \prod_m W(m+d)/W(m+d+1) = {\rm gr}\, W
$$
is non-vanishing.
\end{dfn}

\begin{rem}
The map $\lambda$ as above induces a map $\widehat{V} \to \widehat{W}$ between completions, which has the same filtration degree as $\lambda$.
For simplicity, we denote it by the same letter.
\end{rem}

We can now determine the filtration degree of the double bracket $\kappa$.
To exclude the case where $\pi$ is trivial, we assume that $(g,n) \neq (0,0)$.
The first result is as follows:

\begin{prop}  \label{prop:degree_kappa}
Under the weight filtration on $\K \pi$, the filtration degree
of the double bracket $\kappa$ is greater than or equal to $(-2)$.
\end{prop}

\begin{proof}
It is convenient to use the capping trick of Remark~\ref{rem:capping_trick} and to compute the filtration degree of the double bracket on $\K\overline{\pi}$.
Since both the double bracket and the weight filtration are functorial, the result applies to the double bracket on $\K\pi$.

The weight filtration on $\K\overline{\pi}$ is given by $\K\overline{\pi}(m) = (I\overline{\pi})^m$.
Letting all $u_i,v_j \in I\overline{\pi}$ in Corollary~\ref{cor:kuv}, we see that $\kappa(\K\overline{\pi}^{\otimes 2}(m)) \subset \K\overline{\pi}^{\otimes 2}(m-2)$, as required.
\end{proof}

For all $m \in \mathbb{Z}$, we have the induced map
$$
\kappa_\gr \colon \K\pi^{\otimes 2}(m)/\K\pi^{\otimes 2}(m+1) \to \K\pi^{\otimes 2}(m-2)/\K\pi^{\otimes 2}(m-1)
$$
(the maps for $m \le 1$ are trivial).
Under the canonical isomorphism ${\rm gr}^{\rm wt}\, \K \pi \cong \widehat{T}({\rm gr}^{\rm wt}\, H) = A$, these maps comprise a graded double bracket on $A$:
\[
\kappa_\gr \colon A \otimes A
\to A \otimes A.
\]
The next lemma shows (among other things) that this map is non-vanishing.

\begin{lem}
\label{lem:kgruv}
For any $a,b\in \gr^{\rm wt}\, H$,
\begin{equation*}  
\kappa_\gr(a,b)=\langle a,b \rangle(1\otimes 1)
+ \mathfrak{z}(a,b)\otimes 1 - 1\otimes \mathfrak{z}(a,b).
\end{equation*}
On generators, non-vanishing graded double brackets are as follows:
\begin{equation*}
\kappa_{\rm gr}(x_i, y_i) = - \kappa_{\rm gr}(y_i, x_i) =1\otimes 1, \hspace {1em}
\kappa_{\rm gr}(z_j, z_j) = z_j \otimes 1 - 1 \otimes z_j.
\end{equation*}
\end{lem}

\begin{proof}
Let $a, b \in {\rm gr}^{\rm wt}\, H$ be elements of the basis $\{ x_i, y_i, z_j\}$.
Then, we can check the equation by using Proposition
\ref{prop:kapvalue}. For example, we compute
\[
\kappa(\alpha_i-1,\beta_i-1)
= \beta_i \otimes \alpha_i \equiv_1 1\otimes 1.
\]
This proves $\kappa_\gr(x_i,y_i)=1\otimes 1$.
Let us give one more example:
\begin{align*}
\kappa(\gamma_j-1,\gamma_j-1)
& = \gamma_j \otimes \gamma_j - 1\otimes {\gamma_j}^2 \\
& = (\gamma_j-1)\otimes \gamma_j - 1\otimes (\gamma_j-1)-1\otimes (\gamma_j-1)^2 \\
& \equiv_3 (\gamma_j-1)\otimes 1 - 1\otimes (\gamma_j-1).
\end{align*}
This proves $\kappa_\gr(z_j,z_j)=z_j\otimes 1 - 1\otimes z_j$.
The other cases can be checked similarly.
\end{proof}

Since $\kappa_\gr$ is non-vanishing, we conclude that the filtration degree of $\kappa$ is $(-2)$ indeed and that $\kappa_\gr$ is the associated graded map of $\kappa$.
The following proposition gives an explicit formula for $\kappa_\gr$ on any pair of monomials:

\begin{prop}
\label{prop:kgr}
Let $a=a_1\cdots a_l$, $b=b_1\cdots b_m\in A$ with $a_i,b_j \in \gr^{\rm wt}\, H$.
Then,
\begin{align*}
\kappa_\gr(a,b)
 = & 
\textstyle\sum_{i,j} \langle a_i,b_j \rangle\,
b_1\cdots b_{j-1}a_{i+1}\cdots a_l \otimes a_1\cdots a_{i-1}b_{j+1}\cdots b_m \\
 & + \textstyle\sum_{i,j} b_1\cdots b_{j-1} \, \mathfrak{z}(b_j,a_i)\, a_{i+1}\cdots a_l \otimes a_1\cdots a_{i-1}b_{j+1}\cdots b_m \\
 & -\textstyle\sum_{i,j} b_1\cdots b_{j-1}a_{i+1}\cdots a_l \otimes a_1\cdots a_{i-1} \, \mathfrak{z}(a_i,b_j)\, b_{j+1}\cdots b_m.
\end{align*}
\end{prop}

\begin{proof}
This follows from the previous lemma and the graded version of Corollary \ref{cor:kuv}.
\end{proof}

The graded double bracket $\kappa_\gr$ on $A$ induces the bracket operation $\{ \cdot, \cdot \}_{\kappa_{\rm gr}}$ and the map
$$
\sigma^{\kappa_{\rm gr}} \colon |A| \to {\rm Der}(A), \hskip 0.3cm
\sigma^{\kappa_{\rm gr}}(|a|): b \mapsto
\{ |a|, b\}_{\kappa_{\rm gr}} =
\kappa_{\rm gr}(a,b)' \kappa_{\rm gr}(a,b)''.
$$
Note that $|A| \cong {\rm gr}^{\rm wt}\, |\K \pi|$ by \eqref{eq:isowithgraded_abs}.
To simplify notation, we will write $\{ \cdot, \cdot \}_{\kappa_{\rm gr}} = \{ \cdot, \cdot \}_{\rm gr}$ and $\sigma^{\kappa_{\rm gr}} = \sigma_{\rm gr}$.

Proposition \ref{prop:degree_kappa} implies that the filtration degree of the operation $\sigma$ and the Goldman bracket are at least $(-2)$.
It is $(-2)$ indeed, as the following proposition shows.

\begin{prop}
If $g \ge 1$ or $n\ge 2$, the filtration degree of $\sigma$ and the Goldman bracket are $(-2)$.
Furthermore, their associated graded maps coincide with $\sigma_{\rm gr}$ and $\{ \cdot, \cdot\}_{\rm gr}$, respectively.
\end{prop}

\begin{proof}
We have $\{ |x_i|, |y_i|\}_{\rm gr} = {\bf 1} \neq 0$ and $\{ |{z_1}^2 z_2|,  |{z_1}^2{z_2}^2| \}_{\rm gr} = |{z_1}^2z_2z_1{z_2}^2| - |{z_1}^2{z_2}^2z_1z_2| \neq 0$.
This shows that $\{ \cdot, \cdot \}_{\rm gr}$ is non-vanishing on $|A|$, so is $\sigma_{\rm gr}$.
\end{proof}

We will use the notation $[\cdot, \cdot]_{\rm gr}$ for the associated graded map of the Goldman bracket.

In later sections, we will need some algebraic results on the bracket operations $\sigma_{\rm gr}$ and $[\cdot, \cdot]_{\rm gr}$.
We collect them here.
First, the image and the kernel of the map $\sigma_{\rm gr}\colon |A| \to {\rm Der}(A)$ is characterized as follows:
\begin{prop}  \label{prop:ker_sigma_gr}
The image of the map $\sigma_{\rm gr}$ is given by
\[
{\rm im}(\sigma_{\rm gr}) = \{ u \in {\rm Der}(A) ; \text{$u(\omega) = 0$ and $u(z_j) = [z_j, u_j]$ for some $u_j \in A$} \},
\]
and the kernel is given by
$$
\ker(\sigma_{\rm gr}) = \mathbb{K} {\bf 1} \oplus \bigoplus_{j=1}^n \, |\mathbb{K}[[z_j]]_{\geq 1}|.
$$
\end{prop}

\begin{proof}
See Appendix \ref{subsec:bra_kappa_gr}.
\end{proof}

The next result is about the center of the Lie algebra $(|A|, [\cdot, \cdot]_{\rm gr})$.
Recall that the center of a Lie algebra $(\mathfrak{g}, [\cdot, \cdot])$, denoted by $Z(\mathfrak{g}, [\cdot, \cdot])$, is defined to be the set of elements $u \in \mathfrak{g}$ such that $[u,v]=0$ for any $v\in \mathfrak{g}$.
For an element $a \in A$, we denote by $|\K[[a]]_{\geq m}| \subset |A|$ the set of formal power series of the form $\sum_{k=m}^\infty c_k |a^k|$.

\begin{thm}[\cite{CBEG07}]
\label{thm:center}
The center of the Lie algebra $(|A|, [\cdot, \cdot]_{\rm gr})$ is given as follows:
\[
Z(|A|, [\cdot,\cdot]_\gr)
= \K {\bf 1} \oplus |\K[[\omega]]_{\geq 2} | \oplus \bigoplus_{j=1}^n | \K[[z_j]]_{\ge 1}|.
\]
\end{thm}
A similar result to Theorem~\ref{thm:center} has been proved in \cite{CBEG07} by using Poisson geometry of quiver varieties. In Appendix~\ref{subsec:bra_kappa_gr}, we give a purely algebraic proof of Theorem~\ref{thm:center}.

\begin{rem}
Note that by Theorem \ref{thm:center},
$
|\omega| = \sum_{i=1}^g |[x_i, y_i]| + \sum_{j=1}^n |z_j| =
\sum_{j=1}^n |z_j|
$
is also a central element of the Lie algebra $(|A|,[\cdot,\cdot]_{\rm gr})$.
\end{rem}

The description of the center $Z(|A|, [\cdot, \cdot]_{\rm gr})$ yields the following result.

\begin{prop}  \label{prop:tiDelta_center}
Let $|a| \in Z(|A|, [\cdot, \cdot]_{\rm gr})$. Then,
$$
\tilde{\Delta}(|a|) \in Z(|A|, [\cdot, \cdot]_{\rm gr}) \otimes Z(|A|, [\cdot, \cdot]_{\rm gr}).
$$
Here, $\tilde{\Delta} \colon |A| \to |A|\otimes |A|$ is the map induced from the twisted coproduct $\tilde{\Delta} = ({\rm id} \otimes \iota)\circ \Delta$ of $A$ (see also \eqref{eq:tDelta}).
\end{prop}

\begin{proof}
By Theorem \ref{thm:center}, all elements of $Z(|A|, [\cdot, \cdot]_{\rm gr})$ are linear combinations of elements of the form $|a^k|$, where $a$ is a primitive element of $A$. That is,
$
\tilde{\Delta}(a) = a \otimes 1 - 1 \otimes a
$.
This implies that
$$
\tilde{\Delta}(|a^k|) = \sum_{l=0}^k (-1)^l \frac{k!}{l!(l-k)!} \, |a^{k-l}| \otimes |a^l| \in Z(|A|, [\cdot, \cdot]_{\rm gr}) \otimes Z(|A|, [\cdot, \cdot]_{\rm gr}),
$$
as required.
\end{proof}

Finally, we have the following result on the graded Lie algebra $(|A|, [\cdot, \cdot]_{\rm gr})$ and its center.

\begin{prop}   \label{prop:inner_derivation}
Let $|a| \in |A|$ such that $[|a|, |b|]_{\rm gr}=0$ for all $b \in A$ of sufficiently high degree (that is, ${\rm deg}(b) \geq N$ for some $N$). Then, $|a| \in Z(|A|, [\cdot, \cdot]_{\rm gr})$.
\end{prop}

\begin{proof}
By assumptions, the derivation $\sigma_{\rm gr}(|a|) \colon b \mapsto \{ |a|, b\}_{\rm gr}$ has the property that 
$$
|\sigma_{\rm gr}(|a|)(b)|= [|a|, |b|]_{\rm gr} =0
$$
for all $b \in A$ of sufficiently high degree. By Theorem \ref{thm:idt} in Appendix,
this implies that $\sigma_{\rm gr}(|a|)$ is an inner derivation.
Namely, there exists some $c \in A$ such that 
$
\sigma_{\rm gr}(|a|)(b) = [c, b]
$
for any $b \in A$.
Then, for every $b \in A$ we have
$$
[|a|, |b|]_{\rm gr} = |\sigma_{\rm gr}(|a|)(b)| = |[c, b]|=0.
$$
We conclude that $|a| \in Z(|A|, [\cdot, \cdot]_{\rm gr})$, as required.
\end{proof}

\subsection{Associated graded of operations $\mu^f_r$ and $\mu^f_l$}

In this subsection, we compute the associated graded of the coaction maps $\mu^f_r$ and $\mu^f_l$.

Recall that for a given framing $f$ the rotation number ${\rm rot}^f \colon \hat{\pi}^+ \to \mathbb{Z}$ is defined on regular homotopy classes of immersed free loops on $\Sigma$. 
In general, it does not descend to the first homology $H=H_1(\Sigma, \mathbb{Z})$.
However, for homology classes $z_j$ with $j= 1, \dots, n$ there are distinguished simple embedded curves which represent them: these are the boundary components $\pa_j \Sigma$ with positive orientation.
Furthermore, note that for any standard generating system $\alpha_i, \beta_i, \gamma_j$ for $\pi$, the regular homotopy class of $|\gamma_j|$ coincides with the one of $\pa_j \Sigma$.

We define a $\K$-linear map $q^f \colon {\rm gr}^{\rm wt} \, H \to \K$ by setting
$$
q^f(x_i)=q^f(y_i)=0, \hskip 0.3cm q^f(z_j)={\rm rot}^f(\gamma_j) + 1.
$$

We can now compute the associated graded of the topological operations $\mu^f_r$ and $\mu^f_l$:

\begin{prop}
\label{prop:grmuf}
Ler $f$ be a framing on $\Sigma$. Then, the operations $\mu^f_r$ and $\mu^f_l$ are of filtration degree $(-2)$ and their associated graded $\mu^f_{r, \gr}$ and $\mu^f_{l, \gr}$ satisfy for all $a \in \gr^{\rm wt} \, H$
$$
\mu^f_{r, \gr}(a)=q^f(a)({\bf 1}\otimes 1), \hskip 0.3cm
\mu^f_{l, \gr}(a)=-q^f(a)(1 \otimes {\bf 1}).
$$
Furthermore, for any $a=a_1\cdots a_m \in A$
with $a_i\in \gr^{\rm wt}\, H$,
\begin{align*}
\mu^f_{r, \gr}(a)= & \textstyle\sum_{i=1}^m q^f(a_i)\, {\bf 1} \otimes a_1\cdots a_{i-1}a_{i+1}\cdots a_m \\
& +\textstyle\sum_{j<k} \langle a_j,a_k \rangle \,
|a_{j+1}\cdots a_{k-1}| \otimes a_1\cdots a_{j-1}a_{k+1}\cdots a_m \\
& + \textstyle\sum_{j<k}  |\mathfrak{z}(a_k,a_j) a_{j+1}\cdots a_{k-1}| \otimes a_1\cdots a_{j-1} a_{k+1}\cdots a_m \\
& -\textstyle\sum_{j<k} |a_{j+1}\cdots a_{k-1}|\otimes
a_1\cdots a_{j-1} \mathfrak{z}(a_j, a_k) a_{k+1}\cdots a_m.
\end{align*}
Finally, the map $\mu^f_{l,{\rm gr}}$ satisfies $\mu^f_{l,\gr}(a) = - {\mu^f_{r,\gr}(a)}^{\circ}$ for all $a\in A$.
\end{prop}

\begin{proof}
We first prove that the filtration degree of $\mu^f_{r}$ and $\mu^f_{l}$ are at least $(-2)$.
Using the capping trick as in the proof of Proposition~\ref{prop:degree_kappa}, the problem is reduced to the case of the coaction maps on $\K\overline{\pi}$.
Then, by Corollary~\ref{cor:muuuu} and the fact that the filtration degree of $\kappa$ is $(-2)$, we conclude that the filtration degree of $\mu^f_r$ and $\mu^f_l$ are at least $(-2)$.

We now compute its associated graded component $\mu^f_{r, \gr}$ of degree $(-2)$ in order to check whether it is non-vanishing.
Since the generators $x_i$ and $y_i$ are of degree $1$, $\mu^f_{r, \gr}(x_i)=\mu^f_{r, \gr}(y_i)=0$.
For generators $z_j$, using Theorem~\ref{thm:values_mu^f} we compute
$$
\mu^f_r(\gamma_j-1)=\mu^f_r(\gamma_j) =({\rm rot}^f(\gamma_j) +1)({\bf 1}\otimes \gamma_j)
\equiv_1 q^f(z_j)({\bf 1}\otimes 1)
$$
which implies $\mu^f_{r, {\rm gr}}(z_j)=q^f(z_j)({\bf 1}\otimes 1)$.
Corollary~\ref{cor:muuuu} descends to a product formula for $\mu^f_{r, \gr}$ of the same form, and we have
\begin{align*}
\mu^f_{r, \gr}(a_1\cdots a_m)=&
\textstyle\sum_{i=1}^m ({\bf 1}\otimes a_1\cdots a_{i-1}) \mu^f_{r, \gr}(a_i)
(1\otimes a_{i+1}\cdots a_m) \\
& +\textstyle\sum_{k=1}^m (|\cdot |\otimes {\rm id})\kappa_\gr(a_1\cdots a_{k-1},a_k)(1\otimes a_{k+1}\cdots a_m).
\end{align*}
The first term becomes $\sum_{i=1}^m q^f(a_i)\, {\bf 1} \otimes a_1\cdots a_{i-1}a_{i+1}\cdots a_m$.
By using Proposition~\ref{prop:kgr} we can compute the second term, and we obtain the formula as required.

Note that $\mu^f_{r, \gr}(z_j z_j) = 2q^f(z_j)\, {\bf 1} \otimes z_j + (|\cdot |\otimes {\rm id})\kappa_{\gr}(z_j,z_j) = 2q^f(z_j)\, {\bf 1}\otimes z_j + |z_j|\otimes 1 -{\bf 1} \otimes z_j \neq 0$ and $\mu^f_{r, \gr}(x_i y_i) = (|\cdot| \otimes {\rm id}) \kappa_\gr(x_i, y_i) = {\bf 1} \otimes 1 \neq 0$.
As $(g,n) \neq (0,0)$, this implies that
the filtration degree of $\mu^f_r$ is exactly $(-2)$.

Finally, the associated graded of equation \eqref{eq:mu&mu} implies that the operation $\mu^f_l$ is of filtration degree $(-2)$ as well and $\mu^f_{l,\gr}(a) = - {\mu^f_{r,\gr}(a)}^{\circ}$ for all $a\in A$.
\end{proof}

The operations $\mu_{r, \gr}^f$ and $\mu^f_{l, \gr}$ define the graded Lie cobracket $\delta_\gr^f$ on $|A|$. In more detail, the associated graded of equations \eqref{eq:altmu} and \eqref{eq:bul**bul} yield
\begin{equation*} 
\delta^f_{\gr}(|a|)
= \alt ({\rm id} \otimes |\cdot |)\mu^f_{r, \gr}(a)
=({\rm id} \otimes |\cdot|)\mu^f_{r, \gr}(a) + (|\cdot|\otimes {\rm id})\mu^f_{l, \gr}(a).
\end{equation*}
Note that the second term $|\gamma| \wedge {\bf 1}$ in \eqref{eq:altmu} does not contribute to the associated graded since it is of filtration degree $0$.

\begin{prop}
If $g\ge 1$ or $n\ge 2$, the map $\delta^f_{\rm gr}$ is non-vanishing.
In other words, the filtration degree of the framed Turaev cobracket $\delta^f$ is $(-2)$.
\end{prop}

\begin{proof}
Direct computation shows that $\mu^f_{r, \gr}({x_1}^2{y_1}^2x_1y_1) = {\bf 1} \otimes ((x_1y_1)^2 - {x_1}^2{y_1}^2)$, hence $\delta^f_{\rm gr}(|{x_1}^2{y_1}^2x_1y_1|) = {\bf 1} \wedge (|(x_1y_1)^2 - {x_1}^2{y_1}^2|) \neq 0$.
We also compute $\mu^f_{r, \gr}({z_1}^2 z_2) = 2q^f(z_1)\, {\bf 1} \otimes z_1z_2 + q^f(z_2)\, {\bf 1} \otimes {z_1}^2 + |z_1| \otimes z_2 - {\bf 1} \otimes z_1z_2$, and we see that $\delta^f_{\gr}(|{z_1}^2z_2|) \neq 0$.
\end{proof}

Together with the graded Lie bracket $[\cdot,\cdot]_{\rm gr}$, the map $\delta_\gr^f$ defines a structure of an involutive Lie bialgebra $(|A|, [\cdot, \cdot]_\gr, \delta_\gr^f)$.

\begin{rem}
Note that the only framing dependence of $\mu_{r, \gr}^f(a)$ is in the component ${\bf 1} \otimes A \subset |A| \otimes A$ of the image. In a similar fashion, the framing dependence $\mu^f_{l, \gr}(a)$ is in $A \otimes {\bf 1}$, and of $\delta^f_\gr(|a|)$
in ${\bf 1} \wedge |A|$.
\end{rem}

\begin{rem}
The Lie bialgebra $(|A|, [\cdot, \cdot]_\gr, \delta_\gr^f)$ coincides with the {\em necklace Lie bialgebra} \cite{BLeB, Ginzburg, Schedler} associated to the quiver with $g$ circles and $n$ edges emanating from a distinguished vertex.
\end{rem}

\subsection{Associated graded of surface embeddings} \label{subsec:grsurfemb}
In this subsection, we describe the associated graded of surface embeddings, and in particular of the capping trick.

Recall that by Proposition \ref{prop:functorial} an embedding of oriented surfaces
$i \colon \Sigma \to \overline{\Sigma}$ induces a homomorphism of group algebras together with double bracket and coaction structures. This homomorphism extends to completed group algebras:
$$
i_* \colon (\widehat{\K\pi}, \kappa, \mu^f_r, \mu^f_l) \to (\widehat{\K\bar{\pi}}, \bar{\kappa}, \mu_r^{\bar{f}}, \mu_l^{\bar{f}}).
$$
Furthermore, by taking associated graded we obtain a homomorphism of graded double brackets and coaction maps:
$$
(i_*)_\gr \colon ({\rm gr}^{\rm wt} \, \K\pi, \kappa_\gr , \mu_{r, \gr}^f, \mu^f_{l, \gr}) \to ({\rm gr}^{\rm wt} \, \K\bar{\pi}, \bar{\kappa}_\gr , \mu_{r, \gr}^{\bar{f}}, \mu^{\bar{f}}_{l, \gr}).
$$

\begin{exple} \label{ex:capping_grded}
As an illustration, we consider the capping trick map in Remark~\ref{rem:capping_trick}.
With the generating sets of $\pi$ and $\bar{\pi}$ as in \eqref{eq:capping}, the associated graded map $(i_*)_\gr \colon {\rm gr}^{\rm wt} \, \K\pi = A \to \bar{A} = {\rm gr}^{\rm wt} \, \K\bar{\pi}$ is given by equation~\eqref{eq:graded_capping}.
One can directly check that $(i_*)_\gr$ is a homomorphism of double brackets and coaction maps. For convenience of the reader, we give a couple of examples of such calculations. To simplify notation, denote $z_j=z$ and 
$\bar{x}_{g+j}=\bar{x}, \bar{y}_{g+j}=\bar{y}$. Then, we have
\begin{align*}
\kappa_{\rm gr}( (i_*)_\gr(z), (i_*)_\gr(z) ) & =  
\kappa_{\rm gr}( [\bar{x}, \bar{y}], [\bar{x}, \bar{y}] ) \\
& = [\bar{x}, \bar{y}] \otimes 1 - 1 \otimes [\bar{x}, \bar{y}] \\
& = (i_*)_\gr(z) \otimes 1 - 1 \otimes (i_*)_\gr(z) \\
& = (i_*)_\gr(\kappa_{\rm gr}( z, z)).
\end{align*}
Here we have used Lemma~\ref{lem:formulas2} (i) in the second line.
Furthermore,
$$
\mu^{\bar{f}}_{r, \gr}((i_*)_\gr(z)) = \mu^{\bar{f}}_{r, \gr}([\bar{x}, \bar{y}]) =
(|\cdot| \otimes {\rm id})(\kappa_{\rm gr}( \bar{x}, \bar{y}) - \kappa_{\rm gr}( \bar{y}, \bar{x}))=2({\bf 1} \otimes 1)=(i_*)_\gr(\mu^{f}_{r, \gr}(z)).
$$
Here we have used the facts that $\mu^{\bar{f}}_{r, \gr}(\bar{x}) = \mu^{\bar{f}}_{r, \gr}(\bar{y}) =0$, that ${\rm rot}^f(\gamma_j) =1$ (see Remark~\ref{rem:capping_trick}) and that $q^f(z_j)={\rm rot}^f(\gamma_j) +1 =2$.
Therefore, $\mu^{f}_{r, \gr}(z) = q_j ({\bf 1} \otimes 1) = 2({\bf 1} \otimes 1)$.
\end{exple}

\subsection{Mapping class groups and automorphism groups}
\label{subsec:mcgauto}

The action of the mapping class group $\mathcal{M} = \mathcal{M}(\Sigma)$ on $\K \pi$ extends to an action on the completed group algebra $\widehat{\K \pi}$ preserving the weight filtration.
Let $\mathcal{I} = \mathcal{I}(\Sigma)$ be the subgroup of $\mathcal{M}$ consisting of elements acting trivially on 
\[
\widehat{\K \pi}(1)/\widehat{\K \pi}(2) \cong \K \pi(1)/ \K \pi(2) \cong H/H^{(2)}.
\]
When $n=0$, $\mathcal{I}$ is nothing but the classical {\em Torelli group} of $\Sigma = \Sigma_{g,1}$ (that is, the kernel of the action of $\mathcal{M}$ on the first homology $H$).
When $n >0$, the group $\mathcal{I}$ is the largest Torelli group in the sense of Putman \cite{Putman08}.

The weight filtration on $\widehat{\K \pi}$ induces a decreasing filtration $\{ \mathcal{I}(k) \}_{k \ge 1}$ on the group $\mathcal{I}$.
The $k$th term of this filtration is the kernel of the action on
\[
\widehat{\K \pi}(1)/ \widehat{\K \pi}(k+1)
\cong
\K \pi(1)/ \K \pi(k+1).
\]

\begin{rem}
This filtration (in particular for the case $n=0$) is called the {\em Johnson filtration} \cite{Joh80, Joh83}.
When $n=0$, $\mathcal{I}(k)$ coincides with the kernel of the action on the $k$th nilpotent quotient of $\pi$.
\end{rem}

The filtration $\{\mathcal{I}(k) \}_{k \ge 1}$ of the group $\mathcal{I}$ gives rise to the graded Lie algebra
\[
{\rm Lie}_{\rm gr}(\mathcal{I}):= \bigoplus_{k=1}^{\infty} \mathcal{I}(k)/\mathcal{I}(k+1),
\]
where the Lie bracket is induced from the group commutator in $\mathcal{I}$.

For a given framing $f$ on $\Sigma$, we will be interested in the intersection of the Torelli group with the framed mapping class group:
\[
\mathcal{I}^f := 
\mathcal{I} \cap \mathcal{M}^f.
\]
Note that we have
\[
\mathcal{I}(2) \subset 
\mathcal{I}^f \subset
\mathcal{I}(1) = \mathcal{I}.
\]
(See e.g. \cite[\S 5]{Joh80}.)
By Proposition~\ref{prop:MCGfaction}, we have a natural group homomorphism
\begin{equation} \label{eq:Torellimap}
\mathcal{I}^f \to {\rm Aut}(\widehat{\K\pi}, H/H^{(2)}, \kappa, \mu^f_r, \mu^f_l)
\end{equation}
which is injective when $n=0$.
The filtration $\{\mathcal{I}(k) \}_{k \ge 1}$ restricts to a decreasing filtration on $\mathcal{I}^f$, and it gives rise to the graded Lie algebra
\[
{\rm Lie}_{\rm gr}(\mathcal{I}^f)=\bigoplus_{k=1}^\infty \, \mathcal{I}^f(k)/\mathcal{I}^f(k+1).
\]
Note that ${\rm Lie}_{\rm gr}(\mathcal{I})$ and ${\rm Lie}_{\rm gr}(\mathcal{I}^f)$ differ only in the degree one part.

\section{Divergence and log-Jacobian cocycles and their properties}
\label{sec:div}

In this section, we describe the non-commutative divergence and log-Jacobian cocycles and their properties. 

\subsection{Preliminaries on 1-cocycles} \label{sec:cocycles}
In this subsection, we recall definitions and general properties of 1-cocycles on Lie algebras and on groups.

We begin with the definition of Lie algebra $1$-cocycles:

\begin{dfn}
Let $\mathfrak{g}$ be a Lie algebra and $M$ a (left) $\mathfrak{g}$-module. A linear map ${\sf b}\colon \mathfrak{g} \to M$ is a {\em Lie algebra 1-cocycle} with values in $M$ if it satisfies the following property: for all $u, v \in \mathfrak{g}$,
$$
{\sf b}([u, v]) = u({\sf b}(v)) - v({\sf b}(u)).
$$
\end{dfn}

For every $m \in M$, one can define its {\em coboundary} to be the $1$-cocycle given by
\[
dm \colon \mathfrak{g} \to M,
\quad
u \mapsto u(m).
\]
The first Lie algebra cohomology $H^1(\mathfrak{g}, M)$ is defined as the quotient of the space of 1-cocycles by the space of $1$-coboundaries:
$$
H^1(\mathfrak{g}, M) =
\frac{ \{ \text{$1$-cocycles of $\mathfrak{g}$ with values in $M$} \}}{ \{ \text{$1$-coboundaries} \}}.
$$

Next we recall the definition of additive group 1-cocycles:
\begin{dfn}
Let $\mathcal{G}$ be a group and $M$ a (left) $\mathcal{G}$-module.
A map ${\sf c}\colon \mathcal{G} \to M$ is an (additive) {\em group 1-cocycle} with values in $M$ if it satisfies the following property: for all $F,G \in \mathcal{G}$,
$$
{\sf c}(FG) = {\sf c}(F) + F({\sf c}(G)).
$$
\end{dfn}
Among other things, the 1-cocycle condition implies ${\sf c}(e_{\mathcal{G}})=0$, where $e_{\mathcal{G}}$ is the unit element in $\mathcal{G}$, and 
${\sf c}(F^{-1})=-F^{-1}({\sf c}(F))$ for all $F \in \mathcal{G}$.
For every $m \in M$ one can define the {\em coboundary} of $m$ to be the $1$-cocycle defined by
$$
dm\colon \mathcal{G} \to M, \quad
F \mapsto F(m) - m.
$$
As in the case of Lie algebras, the first group cohomology of $\mathcal{G}$ with values in $M$ is defined as
$$
H^1(\mathcal{G}, M) = \frac{ \{ \text{$1$-cocycles of $\mathcal{G}$ with values in $M$} \} }{ \{ \text{$1$-coboundaries} \}}.
$$

Assume that $\mathfrak{g}$ is a (pro-)nilpotent Lie algebra. Then, the {\em Baker-Campbell-Hausdorff (BCH) series} defines a structure of a group $\mathcal{G}$ on the set $\mathfrak{g}$ with the product
$$
u * v = {\rm bch}(u, v) = 
\log (e^u e^v) = u + v + \frac{1}{2} [u, v] + \cdots.
$$
We call $\mathcal{G}$ the {\em exponentiation} of $\mathfrak{g}$ and denote $\mathcal{G} = \exp(\mathfrak{g})$.
The identity map $\mathfrak{g} \ni u \mapsto u = \exp(u) \in \mathcal{G}$ plays the role of the exponential map.
In our application, $\mathfrak{g}$ is always equipped with a complete filtration $\mathfrak{g} = \mathfrak{g}_1 \supset \mathfrak{g}_2 \supset \cdots$ indexed by positive integers and the Lie bracket on $\mathfrak{g}$ is filtration-preserving.

Let $M$ be a $\mathfrak{g}$-module equipped with a complete filtration
and assume that the Lie algebra action of $\mathfrak{g}$ on $M$ is filtration-preserving (that is, the map $\mathfrak{g} \otimes M \to M, u\otimes m \mapsto u(m)$ is filtration-preserving).
This action integrates to a group action of $\mathcal{G}$ which is given by the following formula: for $F = \exp(u)\in \mathcal{G}$ and $m\in M$,
\[
F(m) = \exp(u)(m) = \sum_{k=0}^{\infty} \frac{1}{k!} u^k(m).
\]
In what follows, all $1$-cocycles of $\mathfrak{g}$ and $\mathcal{G}$ with values in $M$ are assumed to be continuous with respect to the filtrations on $\mathfrak{g}$, $\mathcal{G}$ and $M$. 
We say that the Lie algebra 1-cocycle ${\sf b}\colon \mathfrak{g} \to M$ {\em integrates} to an additive group 1-cocycle ${\sf c}\colon \mathcal{G} \to M$ if for all $u \in \mathfrak{g}$

\begin{equation} \label{eq:cintb}
\left. \frac{d}{dt}\, {\sf c}(\exp(tu))\right|_{t=0} = {\sf b}(u).
\end{equation}
Every Lie algebra 1-cocycle ${\sf b}$ admits a unique integration.
More explicitly, this integration is given by 
\[
{\sf c}(\exp (u)) = 
\frac{e^u - 1}{u} ({\sf b}(u)) = 
\sum_{k=0}^{\infty} \frac{1}{(k+1)!}
u^k ({\sf b}(u)).
\]
Conversely, for any group $1$-cocycle ${\sf c}\colon \mathcal{G} \to M$ the map ${\sf b}\colon \mathfrak{g} \to M$ defined by \eqref{eq:cintb} is a Lie algebra $1$-cocycle, which is called the {\em differentiaion} of ${\sf c}$.
By integration and differentiation, Lie algebra $1$-cocycles and additive group $1$-cocycles correspond to each other bijectively:
\[
\left\{
\hspace{-0.5em}
\begin{array}{ll}
\text{ continuous $1$-cocycles } \\
\text{ of $\mathfrak{g}$ with values in $M$ }
\end{array}
\hspace{-0.5em}
\right\}
\cong 
\left\{
\hspace{-0.5em}
\begin{array}{ll}
\text{ continuous $1$-cocycles } \\
\text{ of $\mathcal{G}$ with values in $M$ }
\end{array}
\hspace{-0.5em}
\right\}.
\]
For more details, see Section~\ref{subsec:App_integration}.

\begin{exple} \label{eq:trivial}
Every $m \in M$ determines the Lie algebra coboundary and group coboundary both of which are denoted by $dm$.
These two Lie algebra and group $1$-cocycles correspond to each other through the above bijection.
For, the integration of the coboundary $dm \colon \mathfrak{g} \to M$ is given by $\exp(u) \mapsto \textstyle\frac{e^u - 1}{u}(dm(u)) = \textstyle\frac{e^u - 1}{u}(u(m)) = e^u (m) - u = dm(\exp(u))$.
\end{exple}

Let $\mathcal{G}$ be a group and $M$ a $\mathcal{G}$-module.
The group $\mathcal{G}$ acts on the space of additive group 1-cocycles of $\mathcal{G}$:
\begin{equation}  \label{eq:action_on_cocycles}
{\sf c}^F(G) = F({\sf c}(F^{-1} GF))={\sf c}(G) + G({\sf c}(F)) - {\sf c}(F).
\end{equation}
The first expression makes it obvious that the map 
$F\colon {\sf c} \mapsto {\sf c}^F$ defines an $\mathcal{G}$-action. The second expression shows that ${\sf c}^F$ is a 1-cocycle cohomologous to ${\sf c}$. One uses the $1$-cocycle property of ${\sf c}$ to pass from one expression to the other.
If $\mathcal{G} = \exp(\mathfrak{g})$ for a pro-nilpotent Lie algebra $\mathfrak{g}$, there is a group action of $\mathcal{G}$ on Lie algebra $1$-cocycles of $\mathfrak{g}$ obtained by differentiating \eqref{eq:action_on_cocycles}.
Actually, we will need to work with a more general setting, which is described below.

Let $\mathfrak{g}$ be a Lie algebra and $M$ a $\mathfrak{g}$-module both of which are equipped with complete filtrations
and the action of $\mathfrak{g}$ on $M$ is filtration-preserving.
Let $\mathfrak{g}_+$ be the Lie subalgebra of $\mathfrak{g}$ that consists of elements of positive filtration degree.
The Lie algebra $\mathfrak{g}_+$ is pro-nilpotent, and the adjoint action of $\mathfrak{g}_+$ on $\mathfrak{g}$ and the Lie algebra action of $\mathfrak{g}_+$ on $M$ integrate to group actions of $\mathcal{G}_+= \exp(\mathfrak{g}_+)$.
We denote by $Z^1(\mathfrak{g}, M)$ the space of continuous $1$-cocycles of $\mathfrak{g}$ with values in $M$.
The space $Z^1(\mathfrak{g}, M)$ admits an action of $\mathcal{G}_+$ given by the following formula: for any $F\in \mathcal{G}_+$, ${\sf b} \in Z^1(\mathfrak{g},M)$ and $u \in \mathfrak{g}$,
\[
{\sf b}^F(u) := F({\sf b} ({\rm Ad}_{F^{-1}} u)).
\]
\begin{prop} \label{prop:bFu}
 Let $F \in \mathcal{G}_+$ and ${\sf b}\in Z^1(\mathfrak{g},M)$.
We denote by ${\sf c}\colon \mathcal{G}_+ \to M$ the integration of the restriction of ${\sf b}$ to $\mathfrak{g}_+$.
Then, we have
\begin{equation*}
{\sf b}^F = {\sf b} + d({\sf c}(F)).
\end{equation*}
\end{prop}
\begin{proof}
See Appendix~\ref{subsec:App_integration}.
\end{proof}

We keep the assumptions on $\mathfrak{g}$ and $M$ as above.
Let ${\sf b}_1, {\sf b}_2 \colon \mathfrak{g} \to M$ be cohomologous $1$-cocycles and let $m \in M$ such that ${\sf b}_2 - {\sf b}_1 = dm$.
We denote by ${\sf c}_1, {\sf c}_2 \colon \mathcal{G}_+ \to M$ the integration of ${{\sf b}_1}|_{\mathfrak{g}_+}$ and ${{\sf b}_2}|_{\mathfrak{g}_+}$, respectively.

\begin{dfn} \label{dfn:R}
For $F\in \mathcal{G}_+$, we define the element $R_{{\sf c}_1,{\sf c}_2}(F) = R(F) \in M$ by 
\begin{equation}  \label{eq:D^f_first}
R(F):= {\sf c}_2(F) - F(m) = {\sf c}_1(F) - m.
\end{equation}
(To be more rigorous, $R(F)$ depends on $m$ as well.
For simplicity, we drop $m$ from the notation.)
\end{dfn}

Since ${\sf c}_2 - {\sf c}_1 = dm$, the equivalence of the two expressions above follows: we compute
\[
{\sf c}_2(F) - F(m) = {\sf c}_1(F) + F(m) - m - F(m) = {\sf c}_1(F) - m.
\]

The expression $R(F)$ will play an important role in our discussion of the KV problem.

\begin{prop} \label{prop:transform_R}
Keep the notations as above.
\begin{enumerate}
\item[(i)]
We have ${\sf b}_1^F - {\sf b}_2 = d(R(F))$.

\item[(ii)]
The map $R\colon \mathcal{G}_+ \to M$ has the following property:
for all $F, G \in \mathcal{G}_+$,
\begin{equation*} 
R(FG) = {\sf c}_2(F) + F(R(G)) = R(F) + F({\sf c}_1(G)).
\end{equation*}
\end{enumerate}
\end{prop}

\begin{proof}
(i) Using Proposition~\ref{prop:bFu} and ${\sf b}_2 - {\sf b}_1 = dm$, we compute
\[
{\sf b}_1^F - {\sf b}_2 = {\sf b}_1 + d( {\sf c}_1(F)) - {\sf b}_2 = d( {\sf c}_1(F) - m) = d(R(F)).
\]

(ii) For the first equality, we compute
$$
R(FG) = {\sf c}_2(FG) - F(G(m)) = {\sf c}_2(F) + F({\sf c}_2(G) - G(m)) =
{\sf c}_2(F) + F(R(G)).
$$
Similarly, for the second equality we have
$$
R(FG) = {\sf c}_1(FG) - m = {\sf c}_1(F) + F({\sf c}_1(G)) - m = R(F) + F({\sf c}_1(G)),
$$
as required.
\end{proof}

\subsection{Gradings and filtrations} \label{subsec:gradfilt}
In preparation to non-commutative calculus, we recall some facts about gradings and filtrations on free associative algebras.

Let $A= \mathbb{K}\langle\langle z_1, \dots, z_n\rangle\rangle$ be the ring of non-commutative formal power series in variables $z_1, \ldots, z_n$. Recall that it carries a Hopf algebra structure with coproduct, augmentation map and antipode given by
$$
\Delta(z_j) =z_j \otimes 1 + 1 \otimes z_j,
\quad
\varepsilon(z_j) = 0,
\quad
\iota(z_j) = - z_j.
$$
For future use, it is convenient to introduce the  twisted coproduct
\begin{equation} \label{eq:tDelta}
\tilde{\Delta} := ({\rm id} \otimes \iota) \circ \Delta \colon A \to A \otimes A.
\end{equation}
In particular, $\tilde{\Delta}(z_j) = z_j \otimes 1 - 1 \otimes z_j$.

For positive integral weights on the generators, $w_j={\rm wt}(z_j)$, $j=1, \ldots, n$, we introduce a grading and a decreasing filtration on $A$: for any $k \ge 0$, we define
\[
A_k = \{ a\in A ; {\rm wt}(a) = k \},
\quad
A_{\geq k} = \{ a \in A; {\rm wt}(a) \geq k \},
\quad
A_{> k} = \{ a \in A; {\rm wt}(a) > k \}.
\]
We have ${\rm gr}\, A = A$ as graded $\K$-algebras.
Note that this identification depends on the choice of generators $z_j$.

Let ${\rm End}(A)$ be the space of continuous endomorphisms of $A$.
We will be interested in the following two subsets of ${\rm End}(A)$: the Lie algebra ${\rm Der}(A)$ of continuous derivations of $A$ and the group ${\rm Aut}(A)$ of continuous algebra automorphisms of $A$.
If $E\in {\rm End}(A)$ satisfies $E(A_{\ge k}) \subset A_{\ge k}$ for all $k$, then one can define the associated graded map ${\rm gr}\, E \in {\rm End}({\rm gr}\, A) = {\rm End}(A)$ which satisfies ${\rm gr}\, E (A_k) \subset A_k$ for all $k$.
With this in mind, we introduce the following Lie subalgebras of ${\rm Der}(A)$:
\begin{align*}
& {\rm Der}^{\ge 0}(A) := \{ u \in {\rm Der}(A) ;
\text{$u(A_{\ge k}) \subset A_{\ge k}$ for all $k$} \}, \\
& {\rm Der}^+(A) := \{ u \in {\rm Der}^{\ge 0}(A) ; {\rm gr}\, u = 0 \}, \\
& {\rm Der}^0(A) : = \{ u \in {\rm Der}^{\ge 0}(A) ; {\rm gr}\, u = u \}.
\end{align*}
The space ${\rm Der}(A)$ becomes a graded Lie algebra with respect to the weight filtration.
For $u \in {\rm Der}^{\ge 0}(A)$, the condition ${\rm gr}\, u = 0$ (resp. ${\rm gr}\, u = u$) is equivalent to that $u$ is a derivation of positive degree (resp. of degree zero).
The Lie algebra ${\rm Der}^+(A)$ is pro-nilpotent, and we have the following semi-direct decomposition:
\[
{\rm Der}^{\ge 0}(A) \cong 
{\rm Der}^0(A) \ltimes {\rm Der}^+(A),
\quad
u \mapsto ( {\rm gr}\, u, u - {\rm gr}\, u) =: (u_0, u_+).
\]
In a similar fashion, we define the following subgroups of ${\rm Aut}(A)$:
\begin{align*}
& {\rm Aut}^{\ge 0}(A) := \{ F \in {\rm Aut}(A) ;
\text{$F(A_{\ge k}) = A_{\ge k}$ for all $k$} \}, \\
& {\rm Aut}^+(A) := \{ F \in {\rm Aut}^{\ge 0}(A) ; {\rm gr}\, F = {\rm id}_A \}, \\ 
& {\rm Aut}^0(A) := \{ F \in {\rm Aut}^{\ge 0}(A) ; {\rm gr}\, F = F \}.
\end{align*}
Note that for $F \in {\rm Aut}(A)$, 
\[
F \in {\rm Aut}^+(A)
\quad
\iff
\quad
\text{$F(z_j) \in z_j + A_{> w_j}$ for all $j = 1, \ldots, n$}.
\]
There is a semi-direct decomposition of the group ${\rm Aut}^{\ge 0}(A)$:
\begin{equation} \label{eq:Aut(w)dec}
{\rm Aut}^{\ge 0}(A) \cong
{\rm Aut}^0(A) \ltimes {\rm Aut}^+(A),
\quad
F \mapsto ({\rm gr}\, F, ({\rm gr}\, F)^{-1}F) =:(F_0, F_+).
\end{equation}

\begin{rem}
If all the weights $w_j$ are equal to each other, we have ${\rm Aut}^{\ge 0}(A) = {\rm Aut}(A)$.
\end{rem}

The positive derivations and automorphisms of $A$ correspond to each other by the exponential and logarithm maps
\begin{align*}
& \exp \colon {\rm Der}^+(A) \to {\rm Aut}^+(A), \quad
u \mapsto \exp (u) = \sum_{k=0}^{\infty} \frac{1}{k!} u^k, \\
& \log \colon {\rm Aut}^+(A) \to {\rm Der}^+(A), \quad
F \mapsto \log (F) = \sum_{k=1}^{\infty} \frac{(-1)^{k-1}}{k} (F - {\rm id}_A)^k,
\end{align*}
through which the group ${\rm Aut}^+(A)$ is identified with the exponentiation of ${\rm Der}^+(A)$.

Next, we define a graded $n$-dimensional vector space $H=\bigoplus_{j=1}^n \mathbb{K}_{w_j}$, where $\mathbb{K}_{w_j}$ stands for one dimensional linear space of weight $w_j$, and consider the following graded unital ring
$$
{\rm Mat} := A^{\rm op} \otimes A \otimes {\rm End}(H),
$$
where $A^{\rm op}$ is the opposite algebra of $A$.
Note that $A^{\rm op}$ is also a (completed) free associative algebra.
From time to time we will make use of the identity algebra anti-homomorphism
\[
A \to A^{\rm op}, \quad a \mapsto a.
\]
It intertwines the action of ${\rm Der}(A)$ and of ${\rm Aut}(A)$ on $A$ and $A^{\rm op}$, maps commutators to commutators, and induces an isomorphism $|A^{\rm op}| \cong |A|$.
Elements of ${\rm Mat}$ are regarded as $n \times n$ matrices with coefficients in $A^{\rm op} \otimes A$.
There are natural actions of the Lie algebra ${\rm Der}(A)$ and the group ${\rm Aut}(A)$ on ${\rm Mat}$.
We denote by ${\rm Mat}^{\ge 0}$, ${\rm Mat}^+$ and ${\rm Mat}^0$ the subspace of elements of non-negative degree, positive degree elements and degree zero elements, respectively. 
In terms of matrix coefficients, we have the following description:
\begin{align*}
X & \in {\rm Mat}^{\ge 0} \hspace{-6em} & \iff & \quad
\text{$X_{ij} \in (A^{\rm op}\otimes A)_{\ge w_j - w_i}$ for all $i,j$}, \\
X & \in {\rm Mat}^+ \hspace{-6em} & \iff & \quad
\text{$X_{ij} \in (A^{\rm op}\otimes A)_{> w_j - w_i}$ for all $i,j$}, \\
X & \in {\rm Mat}^0 \hspace{-6em} & \iff & \quad
\text{$X_{ij} \in (A^{\rm op}\otimes A)_{w_j - w_i}$ for all $i,j$}.
\end{align*}
In particular, any element $X \in {\rm Mat}^{\ge 0}$ satisfies $X_{ij} = 0$ if $w_j - w_i < 0$.
For every $X \in {\rm Mat}^{\ge 0}$, one can define its associated graded ${\rm gr}\, X \in {\rm Mat}^0$ by setting the $(i,j)$ entry to be the degree $w_j - w_i$ part of $X_{ij}$.

Let ${\rm Mat}^*$ be the group of invertible elements in ${\rm Mat}$.
Furthermore, we introduce the following three subgroups:
\begin{align*}
{\rm Mat}^{*, \ge 0} &:= \{ X \in {\rm Mat}^* ; X, X^{-1} \in {\rm Mat}^{\ge 0} \}, \\
{\rm Mat}^{*,+} &:= \{ X \in {\rm Mat}^{*, \ge 0} ; {\rm gr}\, X = I_n \}, \\
{\rm Mat}^{*,0} &:= \{ X \in {\rm Mat}^{*, \ge 0} ; {\rm gr}\, X = X \},
\end{align*}
where $I_n$ is the $n\times n$ identity matrix.
Note that we can also write
$$
{\rm Mat}^{*,+} = \{ I_n + Y; Y \in {\rm Mat}^+ \}
$$
and the logarithm map sends it bijectively onto ${\rm Mat}^+$:
\begin{equation} \label{eq:logbij}
\log \colon {\rm Mat}^{*,+} \to {\rm Mat}^+,
\quad
X \mapsto \sum_{k=1}^{\infty} \frac{(-1)^{k-1}}{k} (X - I_n)^k.
\end{equation}

\subsection{Non-commutative calculus}

In non-commutative differential calculus, partial derivatives are derivations with values in $A \otimes A$ (instead of $A$):
$$
\partial_j = \partial_{z_j} \colon A \to A \otimes A, \hskip 0.3cm
\partial_j z_k=\delta_{jk}(1\otimes 1).
$$
(See also \cite[Section 2.4]{CBEG07}.)
We will use the Sweedler notation: $\partial_j a=\partial'_j a \otimes \partial''_j a$.
As $\partial_j$ is a derivation, we have
\begin{equation} \label{eq:pa_jab}
\partial_j(ab) = a (\partial'_j b) \otimes \partial''_j b + \partial'_j a \otimes (\partial''_j a) b
\end{equation}
for any $a,b \in A$.
For example, $\partial_2(z_1z_2z_3)=z_1 \otimes z_3$.  For a derivation $u \in {\rm Der}(A)$ and an element  $a \in A$ we have
\begin{equation} \label{eq:u(a)}
u(a) = \sum_{j=1}^n (\partial'_j a) u(z_j) (\partial''_j a).
\end{equation}
The partial derivatives commute in the following sense:
\begin{equation} \label{eq:partial_commute}
(\partial_i \otimes 1)\partial_k = (1\otimes \partial_k) \partial_i,
\end{equation}
where the both sides are mappings of $A$ to $A^{\otimes 3}$.
In Appendix~\ref{subsec:tnc}, we collect several other formulas in non-commutative calculus that we will use.

\begin{dfn} \label{dfn:operatorD}
We define a map $D = D_z \colon {\rm End}(A) \to {\rm Mat}$ by formula
$$
D(E)_{ij} = \partial_i E(z_j).
$$
Here matrix elements $\partial_i E(z_j)$ take values in $A \otimes A$, and we understand them as elements of $A^{\rm op} \otimes A$ by using the identity anti-homomorphism
$A \to A^{\rm op}$.
\end{dfn}

Basic properties of the operator $D$ are as follows:

\begin{prop} \label{prop:DuE}
\begin{enumerate}
    \item[$(i)$] For all $u \in {\rm Der}(A)$ and $E \in {\rm End}(A)$, we have
    \[
    D(uE) = D(u)D(E) + u(D(E)).
    \]
    \item[$(ii)$] For all $F \in {\rm Aut}(A)$ and $E \in {\rm End}(A)$, we have
    $
    D(FE)=D(F) F(D(E))
    $.
    \item[$(iii)$] For all $F \in {\rm Aut}(A)$, the matrix $D(F)$ is invertible and
    $
    D(F)^{-1} = F(D(F^{-1}))
    $.
    \item[$(iv)$] For all $u, v \in {\rm Der}(A)$, we have 
\[
D([u,v])=u(D(v))-v(D(u))+[D(u), D(v)].
\]
\end{enumerate}
\end{prop}

\begin{proof}
We denote $E(z_j)=E_j(z_1, \dots, z_n)$ for $E \in {\rm End}(A)$.

(i) Using equations~\eqref{eq:pa_jab} and \eqref{eq:u(a)}, we compute
\begin{align*}
D(uE)_{ij}
& =   \partial_i u(E_j) \\ 
& = \partial_i \textstyle\sum_k (\partial'_k E_j) u(z_k) (\partial''_k E_j)  \\
& = \textstyle\sum_{k} (\partial'_i (\partial'_k E_j)) \otimes (\partial''_i (\partial'_k E_j)) u(z_k) (\partial''_k E_j) + 
\textstyle\sum_k (\partial'_k E_j \otimes 1) \partial_i u(z_k) (1 \otimes \partial''_k E_j) \\
& \hspace{1em} + \textstyle\sum_{k}  (\partial'_k E_j) u(z_k) (\partial'_i (\partial''_k E_j)) \otimes (\partial''_i (\partial''_k E_j)). 
\end{align*}
The second term of the last expression gives $\sum_k D(u)_{ik} D(E)_{kj}$.
(Note that we take products in $A^{\rm op} \otimes A$.)
Using the commutativity of the partial derivatives \eqref{eq:partial_commute}, we see that the other two terms combined give $u(D(E)_{ij}) = u ( \partial'_i(E_j) \otimes \partial''_i(E_j) )$.

(ii) For the second equation, we compute
\begin{align*}
D(FE)_{ij} 
& = \partial_i (F(E_j(z_1, \dots, z_n))) \\
& = \textstyle\sum_k F(\partial'_k E_j)(\partial_i F(z_k)) F(\partial''_k E_j) \\
& = \textstyle\sum_k D(F)_{ik} F(D(E))_{kj}.
\end{align*}
Here, we have used Example~\ref{ex:chainrule} (a non-commutative version of chain rule formula) in the second line.

(iii) Using (ii), we compute
$I_n = D({\rm id}_A)=D(FF^{-1}) = D(F)F(D(F^{-1}))$
and similarly $D(F^{-1})F^{-1}(D(F)) = I_n$.
This implies the desired result.

(iv) Using (i) twice, we obtain
\begin{align*}
D([u,v]) & = D(uv - vu) \\
& = D(u)D(v)+u(D(v)) - D(v)D(u)-v(D(u)) \\
& = u(D(v))-v(D(u))+[D(u), D(v)],
\end{align*}
as required.
This completes the proof.
\end{proof}

\begin{rem}
Proposition~\ref{prop:DuE}~(iv) admits the following interpretation. Consider the semi-direct sum of ${\rm Der}(A)$ and ${\rm Mat}$, where ${\rm Mat}$ is viewed as a Lie algebra under commutator so as we obtain a short exact sequence of Lie algebras
$$
0 \to {\rm Mat} \to {\rm Der}(A) \ltimes {\rm Mat} \to {\rm Der}(A) \to 0.
$$
Then, the map $u \mapsto (u, D(u))$ is a Lie algebra homomorphism ${\rm Der}(A) \to {\rm Der}(A) \ltimes {\rm Mat}$.
\end{rem}

Taking the weight filtration into consideration, we have the following:

\begin{prop}        \label{prop:D_weights}
For any assignment of positive integral weights $w_j = {\rm wt}(z_j)$, 
the operator $D\colon {\rm End}(A) \to {\rm Mat}$ is a map of filtration degree zero.
In particular, $D$ maps ${\rm Der}^{\ge 0}(A)$, ${\rm Der}^+(A)$ and ${\rm Der}^0(A)$ to ${\rm Mat}^{\ge 0}$, ${\rm Mat}^+$ and ${\rm Mat}^0$, respectively.
Similarly, $D$ maps ${\rm Aut}^{\ge 0}(A)$, ${\rm Aut}^+(A)$ and ${\rm Aut}^0(A)$ to ${\rm Mat}^{*,\ge 0}$, ${\rm Mat}^{*,+}$ and ${\rm Mat}^{*,0}$ respectively.
\end{prop}

\begin{proof}
Let $E \in {\rm End}(A)$ be a homogeneous element of degree $k$.
Then $E(z_j)$ is of degree $k + w_j$ and $D(E)_{ij} = \partial_i E(z_j)$ is of degree $k + (w_j - w_i)$.
This shows that the matrix $D(E)$ is of degree $k$.
\end{proof}

\subsection{Non-commutative divergence cocycle}
In this subsection, we define the non-commutative divergence cocycle.

We start with the following observation:
\begin{prop} \label{prop:trace_Mat}
    There is an isomorphism $|{\rm Mat}|\cong |A| \otimes |A|$ 
    under which the class of $X \in {\rm Mat}$ corresponds to $|\cdot |^{\otimes 2}(\sum_{j=1}^n X_{jj})$.
\end{prop}

\begin{proof}
    Recall that $|A^{\rm op}| \cong |A|$, and that $|{\rm End}(H)| \cong \mathbb{K}$ with isomorphism given by the matrix trace. This implies 
    $$
    |{\rm Mat}| = |A^{\rm op}| \otimes |A| \otimes |{\rm End}(H)|  \cong |A| \otimes |A|,
    $$
    as required.
\end{proof}

By abuse of notation, we denote the canonical projection by the trace symbol ${\rm Tr}$:
\[
{\rm Tr}\colon {\rm Mat} \to |{\rm Mat}|  \cong |A| \otimes |A|.
\]

\begin{dfn} \label{dfn:ncd}
We define the {\em non-commutative divergence} as the composition of maps ${\rm Tr}$ and $D$:
$$
{\sf Div} = {\sf Div}_z := {\rm Tr} \circ D\colon {\rm Der}(A) \to |A| \otimes |A|, \hskip 0.3cm
u \mapsto {\rm Tr}(D(u)).
$$
\end{dfn}

\begin{prop}
The divergence map ${\sf Div}$ is a Lie algebra 1-cocycle on 
${\rm Der}(A)$  with values in $|A| \otimes |A|$. It is given by the following formula:
\begin{equation} \label{eq:div_def}
{\sf Div}(u)=\sum_{j=1}^n |\partial_j u(z_j)|.
\end{equation}
\end{prop}

\begin{proof}
The $1$-cocycle property follows from Proposition~\ref{prop:DuE}~(iv) and from the fact that the map ${\rm Tr}$ vanishes on commutators:
\begin{align*}
 {\sf Div}([u,v]) & = {\rm Tr}(D([u,v])) \\
 & = {\rm Tr}(u(D(v)) - v(D(u)) +  [D(u), D(v)]) \\
 & = u({\rm Tr}(D(v))) - v({\rm Tr}(D(u))) \\
 & = u({\sf Div}(v)) - v({\sf Div}(u)).
\end{align*}
Here we have used the fact 
that the map ${\rm Tr}$ is equivariant under the natural action of ${\rm Der}(A)$ on ${\rm Mat}$ and on $|A| \otimes |A|$. 
To prove equation \eqref{eq:div_def}, we compute
$$
{\sf Div}(u) = {\rm Tr}(D(u)) = |\cdot|^{\otimes 2} \big( \sum_{j=1}^n \partial_j u(z_j) \big)
$$
by using Proposition~\ref{prop:trace_Mat}.
This completes the proof.
\end{proof}

\begin{rem}
\begin{enumerate}
    \item[(i)]
The formula \eqref{eq:div_def} shows that ${\sf Div}$ is a map of weight zero.
    \item[(ii)]
Denote by ${\rm sw}\colon |A| \otimes |A| \to |A| \otimes |A|$ the switch map 
$|a| \otimes |b| \to |b| \otimes |a|$. Then, the composition map ${\rm sw} \circ {\sf Div}$ is also a Lie algebra 1-cocycle.
\end{enumerate}
\end{rem}

\begin{exple}
    Let $u \in {\rm Der}(A)$ be a linear derivarion. That is, $u(z_j) = \sum_k u_{jk} z_k$ with 
    $u_{jk} \in \K$. In this case, the divergence is given by the matrix trace:
    $$
{\sf Div}(u) = \sum_{j=1}^n |\partial_j u(z_j)| = \sum_{j=1}^n u_{jj} ({\bf 1} \otimes {\bf 1}).
    $$
\end{exple}

\begin{exple} \label{ex:innerderivation}
Let ${\rm inn}_a := -{\rm ad}_a$ be the inner derivation with generator $a \in A$.
That is, ${\rm inn}_a(x) = [x,a]$ for all $x \in A$.
Then, we have
\[
{\sf Div}({\rm inn}_a) = (n-1)\, {\bf 1} \wedge |a|.
\]
Indeed, we compute
\begin{align*}
    \textstyle\sum_{j=1}^n | \partial_j([z_j, a]) |
    &= \textstyle\sum_{j=1}^n
    | 1\otimes a + z_j \partial'_j(a) \otimes \partial''_j(a) - \partial'_j(a) \otimes \partial''_j(a)z_j - a \otimes 1| \\
    &= n\, {\bf 1} \wedge |a|
    + \textstyle\sum_{j=1}^n | \partial'_j(a)z_j \otimes \partial''_j(a) - \partial'_j(a) \otimes z_j \partial''_j(a) |,
\end{align*}
and the second term is equal to $|a \otimes 1 - 1 \otimes a| = |a| \wedge {\bf 1}$ by Example~\ref{ex:a11a}.
Hence ${\sf Div}({\rm inn}_a) = n\, {\bf 1} \wedge |a| + |a| \wedge {\bf 1} = (n-1)\, {\bf 1} \wedge |a|$, as required.
\end{exple}

\begin{rem}
It is natural to conjecture that every weight zero 1-cocycle on ${\rm Der}(A)$ with values in $|A| \otimes |A|$ is a linear combination of the maps ${\sf Div}$ and ${\rm sw} \circ {\sf Div}$.
In graded differential geometry, the (commutative) divergence is the unique weight zero 1-cocycle on the Lie algebra of graded vector fields with values in the ring of formal power series (see \cite[Theorem 1.2]{GF} for a similar result).
In the non-commutative setting, we do not know how to prove or disprove a similar statement. We will present a different uniqueness result in Section~\ref{sec:div_unique}.
\end{rem}

\subsection{Non-commutative log-Jacobian cocycle}

In this subsection, we introduce the non-commutative log-Jacobian cocycle which integrates the non-commutative divergence.
For this purpose, we make use of the following map.

\begin{prop} \label{prop:Trlog}
The map ${\rm Tr} \circ \log \colon {\rm Mat}^{*,+} \to |A| \otimes |A|$, where $\log$ is the bijection in equation~\eqref{eq:logbij}, 
is a group homomorphism under the additive group structure on $|A| \otimes |A|$.
\end{prop}

\begin{proof}
Let $X, Y \in {\rm Mat}^{*,+}$.
Since the map ${\rm Tr}$ vanishes on commutators, we compute
\begin{align*}
{\rm Tr} \log(XY) & = {\rm Tr}({\rm bch}(\log(X), \log(Y))) \\
& = {\rm Tr} \big( \log(X) + \log(Y) + \textstyle\frac{1}{2}[\log(X), \log(Y)] + \cdots \big) \\
& = {\rm Tr} \log(X) + {\rm Tr} \log(Y),
\end{align*}
as required.
\end{proof}

Now we introduce the {\em non-commutative log-Jacobian cocycle}.

\begin{dfn} \label{dfn:J}
Define ${\sf J} = {\sf J}_z \colon {\rm Aut}^{\ge 0}(A) \to |A| \otimes |A|$ by
\[
{\sf J}(F) = F_0 \left( {\rm Tr} \log D(F_+) \right),
\]
where $F= F_0F_+$ is the decomposition of $F$ according to \eqref{eq:Aut(w)dec}.
\end{dfn}

By definition, we have ${\sf J}(F) = F_0 ({\sf J}(F_+))$ for any $F \in {\rm Aut}^{\ge 0}(A)$.

We prove some basic properties of the map ${\sf J}$.
The key lemma is the following.

\begin{lem} \label{lem:trkey}
\begin{enumerate}
\item[$(i)$]
Let $X, Y \in {\rm Mat}^{*,0}$. For all $u \in {\rm Der}(A)$ and $m \ge 1$, the matrix $X u^m(Y)$ is upper triangular: $(X u^m(Y))_{ij} = 0$ if $w_i \ge w_j$.
In particular, we have
\[
{\rm Tr} (X u^m(Y) ) = 0.
\]
\item[$(ii)$]
Let $X \in {\rm Mat}^{*,0}$ and $F \in {\rm Aut}^+(A)$.
Then, $X F(X^{-1}), F(X^{-1})X \in {\rm Mat}^{*,+}$ and
\[
{\rm Tr} \log (X F(X^{-1})) = 
{\rm Tr} \log (F(X^{-1})X) = 0.
\]
\end{enumerate}
\end{lem}

\begin{proof}
(i) Suppose that $w_i \ge w_j$. We have
\[
(X u^m(Y))_{ij} = \sum_k X_{ik}\, u^m(Y_{kj}).
\]
If $w_k < w_j$, then $w_k - w_i \le w_k - w_j < 0$ hence $X_{ik} = 0 \in (A^{\rm op} \otimes A)_{w_k - w_i} = \{ 0 \}$.
If $w_k = w_j$, then $Y_{kj} \in (A^{\rm op} \otimes A)_{w_j - w_k} = \mathbb{K}$ is scalar and $u^m(Y_{kj}) = 0$ since $u$ is a derivation.
If $w_k > w_j$, then $Y_{kj} = 0 \in (A^{\rm op} \otimes A)_{w_j - w_k} = \{ 0 \}$.
Hence, for all $k$ we have $X_{ik} u^m(Y_{kj}) = 0$ and consequently
\[
(X u^m(Y))_{ij} = 0.
\]
Taking the trace of $X u^m(Y)$ yields ${\rm Tr} (X u^m(Y) ) = 0$.

(ii)
We have $X F(X^{-1}) \in {\rm Mat}^{*,+}$ since ${\rm gr}\, (X F(X^{-1})) = X ({\rm gr}\, F (X^{-1})) = X X^{-1} = I_n$.
Similarly, $F(X^{-1}) X \in {\rm Mat}^{*,+}$.
Since ${\rm Tr} \log (F(X^{-1}) X) = {\rm Tr} \log (X F(X^{-1}))$ by the conjugation invariance property of ${\rm Tr}$, it is sufficient to compute ${\rm Tr} \log (X F(X^{-1}))$.
Setting $u:= \log F$ and $Y = X^{-1}$, we have
\[
\log ( X F(X^{-1})) = \log \big( I_n + X u(Y) + \frac{1}{2!} X u^2(Y) + \cdots \big).
\]
This is upper triangular by (i) and the fact that the product of upper triangular matrices is again upper triangular.
Therefore, we conclude that ${\rm Tr} \log (X F(X^{-1})) = 0$.
\end{proof}

\begin{prop} \label{prop:FJG}
If $F \in {\rm Aut}^0(A)$ and $G \in {\rm Aut}^+(A)$, then $F G F^{-1} \in {\rm Aut}^+(A)$ and 
\[
{\sf J}(F G F^{-1}) = F ({\sf J}(G)).
\]
\end{prop}

\begin{proof}
Since the map $F \mapsto {\rm gr}\, F$ is multiplicative, ${\rm gr}\, G = {\rm id}_A$ implies that ${\rm gr}\, (FGF^{-1}) = {\rm id}_A$.
For any $F, G \in {\rm Aut}(A)$, we compute 
\[
D(FGF^{-1}) = 
D(F) F(D(G)) FGF^{-1}(D(F)^{-1})
\]
by using Proposition~\ref{prop:DuE}~(ii)(iii).
We set
\[
X_1 : = F (D(G)),
\quad
X_2 := F G F^{-1} (D(F)^{-1}) D(F).
\]
Then $X_1, X_2 \in {\rm Mat}^{*,+}$ and $D(FGF^{-1}) = D(F) X_1X_2 D(F)^{-1}$.
We compute 
\[
{\sf J}(F G F^{-1})
= {\rm Tr} \log D(F G F^{-1} ) 
= {\rm Tr} \log( X_1 X_2) 
= {\rm Tr} \log X_1 + {\rm Tr} \log X_2. \]
Here, we have used the conjugation invariance of the map ${\rm Tr}$ and Proposition~\ref{prop:Trlog}.
By Lemma~\ref{lem:trkey}~(ii), we have ${\rm Tr} \log X_2 = 0$.
Hence
\[
{\sf J}(FGF^{-1}) = {\rm Tr} \log X_1 = F({\rm Tr} \log D(G)) = F({\sf J}(G)),
\]
as required.
\end{proof} 

\begin{rem}
The map ${\sf J}$ is invariant under the right action of ${\rm Aut}^0(A)$:
if $F \in {\rm Aut}^{\ge 0}(A)$ and $G_0 \in {\rm Aut}^0(A)$, 
\[
{\sf J}(F G_0) = {\sf J}(F).
\]
Indeed, since $F G_0 = (F_0G_0)({G_0}^{-1} F_+ G_0)$, we compute
\[
{\sf J}(F G_0 ) = (F_0G_0) {\sf J}({G_0}^{-1} F_+ G_0)
= F_0G_0{G_0}^{-1} ({\sf J}(F_+))
= F_0 ({\sf J}(F_+)) = {\sf J}(F).
\]
Here, we have used Proposition~\ref{prop:FJG} in the second equality.
\end{rem}

\begin{prop} \label{prop:JDiv}
\begin{enumerate}
\item[$(i)$] The map ${\sf J}\colon {\rm Aut}^{\ge 0}(A) \to |A| \otimes |A|$ is a group $1$-cocycle.
\item[$(ii)$]
The restriction of ${\sf J}$ to ${\rm Aut}^+(A)$ is the integration of ${\sf Div} \colon {\rm Der}^+(A) \to |A| \otimes |A|$.
\end{enumerate}
\end{prop}

\begin{proof}
(i) We will have to show that for any $F, G \in {\rm Aut}^{\ge 0}(A)$ we have
\[
{\sf J}(FG) = {\sf J}(F) + F( {\sf J}(G)).
\]

{\it Step 1.}
First, we consider the case where $F, G \in {\rm Aut}^+(A)$.
Then we compute
\begin{align*}
{\sf J}(FG) & = {\rm Tr} \log D(FG) \\
& = {\rm Tr} \log (D(F)F(D(G))) \\
& = {\rm Tr} \log D(F) + {\rm Tr} \log F(D(G)) \\
& = {\sf J}(F) + F({\sf J}(G)).
\end{align*}
Here, we have used Proposition~\ref{prop:DuE} (ii) in the second line and Proposition~\ref{prop:Trlog} in the third line.

{\it Step 2.}
We next consider the general case.
Since $F = F_0 F_+$ and $G = G_0 G_+$, we have $FG = (F_0G_0) ({G_0}^{-1} F_+ G_0 G_+)$.
Hence
\begin{align*}
{\sf J}(FG) & = (F_0G_0) \left( {\sf J}({G_0}^{-1}F_+G_0 G_+) \right) \\
&= (F_0G_0) \left( {\sf J}({G_0}^{-1}F_+G_0) + {G_0}^{-1}F_+G_0 ({\sf J}(G_+)) \right) \\
&= (F_0G_0) \left( {G_0}^{-1} {\sf J}(F_+) \right) + F_0F_+ ( G_0 {\sf J}(G_+)) \\
& = {\sf J}(F) + F({\sf J}(G)).
\end{align*}
Here, we have used Step 1 in the second line and Proposition \ref{prop:FJG} in the third line. 
This completes the proof.

(ii) For $u \in {\rm Der}^+(A)$, we compute
\begin{align*}
\left. \textstyle\frac{d}{dt} \, {\sf J}(\exp(tu)) \right|_{t=0} 
& = \left. \textstyle\frac{d}{dt} \, {\rm Tr} \log(D(\exp(tu))) \right|_{t=0} \\
& = \left. \textstyle\frac{d}{dt} \, {\rm Tr} \log(I_n + t D(u) + O(t^2)) \right|_{t=0} \\
& = {\rm Tr}(D(u)) = {\sf Div}(u).
\end{align*}
This completes the proof.
\end{proof}

\begin{rem}
In formal (commutative) geometry, one defines the multiplicative Jacobian group 1-cocycle by formula ${\sf J}_{\rm mult}(F)={\rm det}(D(F))$, where $D(F)_{ij}=\partial_i F(z_j)$ is the matrix representing the derivative of the map $F$. Whenever $D(F)$ admits a logarithm, one can define an additive group 1-cocycle
$$
{\sf J}_{\rm add}(F)=\log({\sf J}_{\rm mult}(F)) = \log({\rm det}(D(F))) = 
{\rm Tr}(\log(D(F))).
$$
This observation motivates our definition of the map ${\sf J}$. The notion of trace admits a straightforward generalization to non-commutative context while extending the notion of determinant would require some rule for ordering matrix elements. This is the reason for us to use the additive log-Jacobian cocycle ${\sf J}_{\rm add}$ instead of the multiplicative Jacobian cocycle ${\sf J}_{\rm mult}$.
\end{rem}

\begin{prop}  \label{prop:change_variables_Div}
For all $F \in {\rm Aut}^{\ge 0}(A)$ and $u \in {\rm Der}(A)$, we have
\begin{equation*} 
F({\sf Div}_z({\rm Ad}_{F^{-1}} u)) =
{\sf Div}_z(u) + u({\sf J}_z(F)).
\end{equation*}
\end{prop}

\begin{rem}
If $F \in {\rm Aut}^+(A)$, this equation follows from Proposition~\ref{prop:bFu}.
We will show that it actually holds true for all $F \in {\rm Aut}^{\ge 0}(A)$.
\end{rem}
 
\begin{proof}[Proof of Proposition \ref{prop:change_variables_Div}]
We first compute by Proposition~\ref{prop:DuE}
\begin{equation} \label{eq:FDAd}
F(D({\rm Ad}_{F^{-1}}u)) = D(F)^{-1} D(u) D(F) + D(F)^{-1} u(D(F)).
\end{equation}
We will proceed with the following three steps.

{\it Step 1.}
Assume that $F\in {\rm Aut}^+(A)$.
By taking the trace of \eqref{eq:FDAd}, we have
\begin{align*}
F({\sf Div}({\rm Ad}_{F^{-1}}u)) &= 
{\rm Tr} \big( D(F)^{-1} D(u) D(F) + D(F)^{-1} u(D(F)) \big)  \\
& = {\rm Tr} (D(u)) + {\rm Tr} (D(F)^{-1} u(D(F))) \\
& = {\sf Div}(u) + u ({\rm Tr} \log D(F) ) \\
& = {\sf Div}(u) + u ({\sf J}(F)).
\end{align*}
Here, we have used Example \ref{ex:ulog} in the third line.

{\it Step 2.}
Assume that $F \in {\rm Aut}^0(A)$.
We show that 
\[
F({\sf Div}_z({\rm Ad}_{F^{-1}} u))
= {\sf Div}_z(u).
\]
Since $D(F) \in {\rm Mat}^{*,0}$ by Proposition~\ref{prop:D_weights}, Lemma~\ref{lem:trkey}~(i) implies that ${\rm Tr}(D(F)^{-1} u(D(F))) = 0$.
Hence taking the trace of the both sides of \eqref{eq:FDAd} gives the result.

{\it Step 3.}
Finally we consider the general case. 
Denote $F = F_0F_+$.
Then
\begin{align*}
F({\sf Div}_z({\rm Ad}_{F^{-1}} u))
&=
F_0F_+({\sf Div}_z ( {\rm Ad}_{{F_+}^{-1}} {\rm Ad}_{{F_0}^{-1}} u)) \\
& = F_0 \left( {\sf Div}_z( {\rm Ad}_{{F_0}^{-1}} u) + ({\rm Ad}_{{F_0}^{-1}} u) ({\sf J}(F_+)) \right) \\
& = {\sf Div}_z (u) + u (F_0({\sf J}(F_+))) \\
&= {\sf Div}_z (u) + u({\sf J}(F)).  
\end{align*}
Here we have used Step 1 in the second line and Step 2 in the third line.
This completes the proof.
\end{proof}

\subsection{Change of variables and examples}
\label{subsec:change_variables}

The divergence and log-Jacobian cocycles defined in the previous subsections depend on the choice of generators of $A$.
Now we fix $F \in {\rm Aut}^{\ge 0}(A)$, and define the new set of generators $Z_j=F(z_j)$. We let the weights of $Z_j$ be equal to the weights of $z_j$. We will denote the map $D$, the divergence and log-Jacobian cocycles corresponding to the generating sets $z_1, \dots, z_n$ and $Z_1, \dots, Z_n$ by $D_z, {\sf Div}_z, {\sf J}_z$ and $D_Z, {\sf Div}_Z, {\sf J}_Z$, respectively.

The following statement gives the transformation formula of the non-commutative log-Jacobian and divergence cocycles under the change of variables.

\begin{prop}  \label{prop:change_variables_unified}
\begin{enumerate}   
\item[$(i)$]
For $E \in \mathrm{End}(A)$, we have $D_Z(E) = F(D_z(F^{-1}EF))$.
\item[$(ii)$]
As in \eqref{eq:Aut(w)dec}, 
we decompose $G = G_{0,z}G_{+,z}$ and  $G = G_{0,Z}G_{+,Z}$ with respect to the generators $z_j$ and $Z_j$, respectively, for $G \in {\rm Aut}^{\ge 0}(A)$. Then we have 
$G_{0,Z} = F(F^{-1}GF)_{0,z}F^{-1}$ and $G_{+,Z} = F(F^{-1}GF)_{+,z}F^{-1}$.
\item[$(iii)$]
For $G \in {\rm Aut}^{\ge 0}(A)$, we have
$$
{\sf J}_Z(G)=F{\sf J}_z(F^{-1}GF) ={\sf J}_z(G) +G({\sf J}_z(F)) - {\sf J}_z(F). 
$$
\item[$(iv)$]
For $u \in {\rm Der}(A)$, we have 
\begin{equation*} 
 {\sf Div}_Z(u) = F({\sf Div}_z({\rm Ad}_{F^{-1}} u)) =
 {\sf Div}_z(u) + u({\sf J}_z(F)).
 \end{equation*}
\end{enumerate}
\end{prop}
\begin{proof}
(i) Denote $E(z_j)=E_j(z_1, \dots, z_n)$, $E(Z_j)=\tilde{E}_j(Z_1, \dots, Z_n)$ and compute
$$
F^{-1}EF(z_j)=F^{-1}E(Z_j)=F^{-1}(\tilde{E}_j(Z_1, \dots, Z_n))=
\tilde{E}_j(z_1, \dots, z_n).
$$
We observe that the maps $E$ and $F^{-1}EF$ are given by the same formulas in variables $Z_j$ and $z_j$, respectively. 
This implies the equality (i). 

(ii) If we take $E = G$ in the proof of (i), 
we observe that the maps $G$ and $F^{-1}GF$ are given by the same formulas in variables $Z_j$ and $z_j$, respectively. 
This implies the equalities (ii).

(iii) By (i) and (ii), we have 
\begin{align*}
{\sf J}_Z(G) &= F(F^{-1}GF)_{0,z}F^{-1} \big( \mathrm{Tr}\log D_Z(F(F^{-1}GF)_{+,z}F^{-1}) \big) \\
&= F(F^{-1}GF)_{0,z} (\mathrm{Tr}\log D_z((F^{-1}GF)_{+,z})) \\
&= F{\sf J}_z(F^{-1}GF).
\end{align*}
The second equality follows from the $1$-cocycle property of ${\sf J}_z$ (Proposition \ref{prop:JDiv} (i)).

(iv) Similarly we have 
\begin{align*}
{\sf Div}_Z(u) &= \mathrm{Tr}(D_Z(u)) = \mathrm{Tr}(FD_z(F^{-1}uF)) 
= F \big( \mathrm{Tr}(D_z(F^{-1}uF))\big) \\
& = F({\sf Div}_z(F^{-1}uF)) = F({\sf Div}_z({\rm Ad}_{F^{-1}} u)).
\end{align*}
The second equality follows from Proposition~\ref{prop:change_variables_Div}. 
\end{proof}

\begin{rem}   \label{rem:any_weights}
\begin{enumerate}
\item[(i)]
Note that Proposition \ref{prop:change_variables_unified} applies to any choice of weights $w_j$. In particular, if one chooses $w_j=1$ for all $j$, then ${\rm Aut}^{\ge 0}(A)$ is the group of automorphisms of $A$ preserving the filtration by polynomial degree. 

\item[(ii)]
By Propositions~\ref{prop:change_variables_unified} (iii)(iv), the cohomology classes of divergence and log-Jacobian cocycles are independent of the choice of generators of $A$. 
\end{enumerate}
\end{rem}

It is instructive to consider the case of $n=1$ in the non-commutative context. We use a shorthand notation $z_1 = z$.

\begin{prop} \label{prop:Div_J_n=1}
For $n=1$, let $u \in {\rm Der}(A)$ and $F\in {\rm Aut}^+(A)$. Then,
$$
{\sf Div}(u) =  \left| \frac{u(z)\otimes 1 - 1 \otimes u(z)}{z\otimes 1 - 1 \otimes z} \right|, \hskip 0.3cm
{\sf J}(F)=\left| \log\left( \frac{F(z) \otimes 1 - 1 \otimes F(z)}{z\otimes 1 - 1 \otimes z} \right)        \right|.
$$
\end{prop}

\begin{proof}
Note that 
$$
\partial_z (z^k)
= \sum_{l=0}^{k-1} z^l \otimes z^{k-l-1}=
\frac{z^k\otimes 1 - 1 \otimes z^k}{z\otimes 1 - 1 \otimes z}.
$$
This implies that 
\[
D(E) = \frac{E(z) \otimes 1 - 1 \otimes E(z)}{z\otimes 1 - 1 \otimes z}
\]
for any ${\rm End}(E)$.
From this, we obtain the desired formulas.
\end{proof}

\begin{exple}   \label{ex:Z=e^z-1}
For $Z=F(z) =e^z-1$, we obtain
$$
{\sf Div}_Z(u) = {\sf Div}_z(u) + u \left| \log\left( \frac{e^{z\otimes 1} - e^{1 \otimes z}}{z\otimes 1 - 1 \otimes z} \right) \right| .
$$
Note that we can treat $z\otimes 1$ and $1 \otimes z$ as two independent commuting variables. This allows us to rewrite the log-Jacobian on the right hand side as follows:
$$
\left| \log\left( \frac{e^{z\otimes 1} - e^{1 \otimes z}}{z\otimes 1 - 1 \otimes z} \right) \right| =
{\bf 1} \otimes |z| + \left| \log\left( \frac{e^{(z\otimes 1 - 1 \otimes z)} - 1}{z\otimes 1 - 1 \otimes z} \right) \right|= {\bf 1} \otimes |z| + |\tilde{\Delta}(r(z))|,
$$
where $\tilde{\Delta}$ is the twisted coproduct \eqref{eq:tDelta} and
\begin{equation}      \label{eq:r(s)}
r(x)=\log\left( \frac{e^x - 1}{x} \right)= \frac{x}{2} + \sum_{k=2}^\infty \frac{B_k}{k \cdot k!} \, x^k
= \frac{x}{2} + \frac{x^2}{24} - \frac{x^4}{2880} + \frac{x^6}{181440} + \cdots.
\end{equation}
Here $B_k$'s are the Bernoulli numbers.
\end{exple}

\subsection{Uniqueness of divergence cocycles on group algebras} \label{sec:div_unique}

Let $\Gamma$ be a free group of finite rank.
In this subsection, we define divergence cocycles on derivations of the group algebra of $\Gamma$.
We use them to define a family of 1-cocycles on the group ${\rm Aut}(\Gamma)$ of automorphisms of $\Gamma$ cohomologous to a generator of the twisted cohomology group $H^1({\rm Aut}(\Gamma), H_1(\Gamma,\mathbb{Z}))$.
We will see that the divergence cocycle on the derivations of the group algebra of $\Gamma$ is (essentially) unique.

As in Remark~\ref{rem:Magnus}, for a given choice of free generators $\gamma_1, \dots, \gamma_n$ of $\Gamma$ there is a Magnus map ({\em Magnus expansion}) which injects the group algebra of $\Gamma$ into  a (degree completed) free associative algebra:
$$
\theta_M \colon \mathbb{K} \Gamma \to \mathbb{K}\langle\langle Z_1, \dots, Z_n\rangle\rangle, \hskip 0.3cm
\gamma_j \mapsto 1 + Z_j.
$$
In addition, it is useful to introduce the exponential map ({\em exponential expansion}):
$$
\theta_{\rm exp} \colon \mathbb{K} \Gamma \to \mathbb{K}\langle\langle z_1, \dots, z_n\rangle\rangle, \hskip 0.3cm \gamma_j \mapsto e^{z_j}.
$$
Both the Magnus and the exponential expansions induce algebra isomorphisms between the (degree completed) free associative algebra and the completed group algebra:
$$
\mathbb{K}\langle\langle Z_1, \dots, Z_n\rangle\rangle
\xleftarrow[\hspace{0.5em} \cong \hspace{0.5em}]{\theta_M}
\widehat{\mathbb{K} \Gamma}
\xrightarrow[\hspace{0.5em} \cong \hspace{0.5em}]{\theta_{\exp}}
\mathbb{K}\langle\langle z_1, \dots, z_n\rangle\rangle.
$$
The Magnus generators $Z_j$ are related to the exponential generators $z_j$ by formula
$$
Z_j = e^{z_j} -1.
$$

Now we introduce $1$-cocycles on derivations of $\K \Gamma$.
First, partial derivatives $\partial_{Z_j}$ with respect to the $Z$-variables are naturally defined on the completed group algebra $\widehat{\mathbb{K}\Gamma}$. They restrict to $\mathbb{K} \Gamma$:
$$
\partial_{Z_j} (\gamma_k) = \delta_{jk}(1 \otimes 1), \hskip 0.3cm
\partial_{Z_j} (\gamma_k^{-1}) = - \delta_{jk}(\gamma_k^{-1} \otimes \gamma_k^{-1}).
$$
Hence, the divergence cocycle ${\sf Div}_Z$ also restricts to ${\rm Der}(\mathbb{K} \Gamma)$ and takes the form
\begin{equation} \label{eq:Mag_div}
{\sf Div}_Z(u)= \sum_{j=1}^n |\partial_{Z_j} u(Z_j)| = \sum_{j=1}^n |\partial_{Z_j} u(\gamma_j)|.
\end{equation}

Next, for each $1 \le j \le n$,  we consider the map ${\sf e}_j \colon {\rm Der}(\mathbb{K} \Gamma) \to |\mathbb{K} \Gamma|$ given by formula
$$
{\sf e}_j(u) = |u(\gamma_j) \gamma_j^{-1}|.
$$
\begin{prop} \label{prop:eju}
    \begin{enumerate}
        \item[$(i)$] The map ${\sf e}_j$ is a Lie algebra $1$-cocycle.
        \item[$(ii)$] The extension of ${\sf e}_j$ to ${\rm Der}(\widehat{\K \Gamma})$ is the coboundary of $|z_j|$: that is, it holds that ${\sf e}_j(u) = u(|z_j|)$ for any $u\in {\rm Der}(\widehat{\K \Gamma})$.
    \end{enumerate}
\end{prop}
\begin{proof}
(i)
Let $u, v \in {\rm Der}(\K\Gamma)$.
We compute
\[
u({\sf e}_j(v))
= u(|v(\gamma_j) \gamma_j^{-1}|)
= |u(v(\gamma_j)) \gamma_j^{-1}| - |v(\gamma_j)\gamma_j^{-1}u(\gamma_j)\gamma_j^{-1}|.
\]
The second term is symmetric in $u$ and $v$.
Hence,
\[
u({\sf e}_j(v)) - v({\sf e}_j(u)) = |u(v(\gamma_j))\gamma_j^{-1}| - |v(u(\gamma_j))\gamma_j^{-1}|
= |[u,v](\gamma_j) \gamma_j^{-1}|
= {\sf e}_j([u,v]).
\]

(ii) Let $u \in {\rm Der}(\widehat{\K\Gamma})$.
We compute
$$
{\sf e}_j(u)=|u(\gamma_j)\gamma_j^{-1}| = |u(e^{z_j})e^{-z_j}| = |u(z_j)e^{z_j}e^{-z_j}|=
|u(z_j)|=u(|z_j|).
$$
Here we have used Lemma~\ref{lem:ufaga} in the third equality.
\end{proof}

Finally, we introduce the main object of this section.
To a free generating set $\gamma=\{ \gamma_1, \dots, \gamma_n\}$ of $\Gamma$ we associate a Lie algebra 1-cocycle on ${\rm Der}(\mathbb{K} \Gamma)$ with values in $|\mathbb{K} \Gamma| \otimes |\mathbb{K} \Gamma|$ given by formula
\begin{equation} \label{dfn:c_gamma}
    {\sf Div}_\gamma(u)={\sf Div}_Z(u) - {\bf 1} \otimes \sum_{j=1}^n {\sf e}_j(u).
\end{equation}
The map ${\sf Div}_\gamma$ naturally extends to a 1-cocycle ${\rm Der}(\widehat{ \mathbb{K}\Gamma}) \to |\widehat{\mathbb{K}\Gamma}| \otimes |\widehat{\mathbb{K}\Gamma}|$ that we denote in the same way.

The extended map ${\sf Div}_{\gamma}$ can also be expressed in terms of the divergence in the exponential generators:
\begin{prop}  \label{prop:c_change_variables}
For any $u\in {\rm Der}(\widehat{\K \Gamma})$, we have
\begin{equation*} 
{\sf Div}_\gamma(u) = {\sf Div}_z(u) + u \big( |\tilde{\Delta}(\sum_{j=1}^n r(z_j))| \big),
\end{equation*}
where $r(x)$ is given by formula~\eqref{eq:r(s)}.
\end{prop}
\begin{proof}
Using Proposition~\ref{prop:eju}~(ii) and the calculation of Example~\ref{ex:Z=e^z-1}, we obtain
\begin{align*}
{\sf Div}_\gamma(u) &=  {\sf Div}_Z(u) - {\bf 1} \otimes \textstyle\sum_j u(|z_j|) \\
&=  {\sf Div}_z(u) + \textstyle\sum_j u \big( {\bf 1} \otimes |z_j| + |\tilde{\Delta}(r(z_j))| \big)
- {\bf 1} \otimes \textstyle\sum_j u(|z_j|)\\
&= {\sf Div}_z(u) + u(|\tilde{\Delta}(\textstyle\sum_j r(z_j)|),
\end{align*}
as required.
\end{proof}

Consider the group homomorphism $\Gamma \to |\widehat{\K \Gamma}|, \gamma \mapsto |\log \gamma|$, where $|\widehat{\K \Gamma}|$ is viewed as an abelian group under addition. The group homomorphism property follows from
\[
|\log (\gamma \gamma')| = |{\rm bch}(\log \gamma, \log \gamma')| =
|\log \gamma| + |\log \gamma'|,
\]
where we have used that commutators vanish under the $|\cdot |$-sign.

This group homomorphism factors through the natural map 
$\Gamma \to H_1(\Gamma, \mathbb{Z}) \cong \mathbb{Z}^n, \gamma \mapsto [\gamma]$. The induced linear map 
$H_1(\Gamma, \mathbb{Z}) \to |\widehat{\K\Gamma}|, [\gamma] \mapsto |\log \gamma|$ is in fact an isomorphism onto its image.
Indeed, by composing with $\theta_{\exp} \colon |\widehat{\K\Gamma}| \to |\K \langle \langle z_1, \ldots, z_n \rangle \rangle|$, one obtains an injection from $H_1(\Gamma, \mathbb{Z})$ to the linear part of $|\K \langle \langle z_1, \ldots, z_n \rangle \rangle|$:
\[
H_1(\Gamma, \mathbb{Z}) \cong \bigoplus_{j=1}^n \mathbb{Z} [\gamma_j] \to \bigoplus_{j=1}^n \K |z_j| \subset |\K \langle \langle z_1, \ldots, z_n \rangle \rangle|,
\quad [\gamma_j] \mapsto |\log \gamma_j| \mapsto |z_j|.
\]
Hence the image
\[
\Lambda := \{ |\log \gamma| \in |\widehat{\K \Gamma}| ; \gamma \in \Gamma \}
\]
is a lattice in $|\widehat{\K \Gamma}|$ naturally isomorphic to $H_1(\Gamma, \mathbb{Z})$.
We will often identify  elements of $H_1(\Gamma, \mathbb{Z})$ and their images in $\Lambda$.
We will also be interested in the image of $\Lambda$ under the twisted coproduct $\tilde{\Delta}$,
$$
\tilde{\Delta}(\Lambda)=\{ |\log \gamma| \otimes {\bf 1} - {\bf 1}\otimes |\log \gamma| ; \gamma \in \Gamma \} \subset |\widehat{\K\Gamma}| \otimes |\widehat{\K\Gamma}|.
$$

The main result of this subsection is the following theorem:

\begin{thm}  \label{thm:div_independent}
For every generating set $\gamma=\{ \gamma_1, \dots, \gamma_n\}$ of $\Gamma$
there is a unique additive group 1-cocycle $\mathfrak{s}_\gamma \colon {\rm Aut}(\Gamma) \to H_1(\Gamma, \mathbb{Z})$ such that for all $F \in {\rm Aut}(\Gamma)$ and $u\in {\rm Der}(\widehat{\mathbb{K}\Gamma})$ we have
$$
{\sf Div}_{F(\gamma)}(u) - {\sf Div}_\gamma(u) =u(\tilde{\Delta}(\mathfrak{s}_\gamma(F))).
$$
Furthermore, for $G \in {\rm Aut}(\Gamma)$, we have
$$
\mathfrak{s}_{G(\gamma)}(F) = \mathfrak{s}_\gamma(F) + F(\mathfrak{s}_\gamma(G)) -\mathfrak{s}_\gamma(G).
$$
\end{thm}

\begin{proof}
In what follows, we put all the weights $w_j = {\rm wt}(z_j)$ equal to $1$.
Then, the group ${\rm Aut}(\Gamma)$ naturally injects into ${\rm Aut}^{\ge 0}(\widehat{\K \Gamma})$.
We split the proof into several steps.

\vskip 0.2cm

{\em Step 1}. Recall (see e.g., \cite[\S 3.5]{MKS}) that the group of automorphisms of the free group $\Gamma$ is generated by permutations of generators, an inversion:
\begin{equation}  \label{eq:inversion}
\gamma_1 \mapsto \gamma_1^{-1},
\quad
\gamma_j \mapsto \gamma_j \,\, {\rm for} \,\, j\geq 2,
\end{equation}
and a slide move:
\begin{equation} \label{eq:slide}
\gamma_1 \mapsto \gamma_1 \gamma_2,
\quad
\gamma_j \mapsto \gamma_j \,\, {\rm for} \,\, j\geq 2.
\end{equation}

We will establish the fact that 
\begin{equation}  \label{eq:c-c_in}
{\sf Div}_{F(\gamma)} - {\sf Div}_\gamma \in d(\tilde{\Delta}(\Lambda))
\end{equation}
first on each of these generators, and then on all elements $F \in {\rm Aut}(\Gamma)$.
Observe that both the Magnus divergence ${\sf Div}_Z$ and the element $\sum_j |z_j|$ are invariant under permutations of generators $\gamma_j$. Hence, the cocycle ${\sf Div}_\gamma$ is also invariant under permutations. 

\vskip 0.2cm 

{\em Step 2}. Next, let $F \in {\rm Aut}(\Gamma)$ be the inversion  \eqref{eq:inversion}.
It gives rise to a linear transformation of the exponential generators
\begin{equation*}
z_1 \mapsto z'_1=\log(\gamma_1^{-1}) = - z_1, \hskip 0.3cm z_j \mapsto z'_j=z_j \,\,
{\rm for} \,\, j \geq 2,
\end{equation*}
which induces a map
$$
|z_1| \mapsto - |z_1|, \hskip 0.3cm |z_j| \mapsto |z_j| \,\,
{\rm for} \,\, j \geq 2.
$$
This implies
$$
\sum_{j=1}^n |z'_j| = \big( \sum_{j=1}^n |z_j| \big) - 2|z_1|.
$$
For new Magnus generators $Z'_j=F(Z_j)$, we have 
$$
Z'_1 = (1 + Z_1)^{-1} - 1 = \sum_{k=1}^\infty (-1)^k {Z_1}^k, \hskip 0.3cm
Z_j'=Z_j \,\, {\rm for} \,\, j\geq 2.
$$
The map $F$ admits a decomposition $F = F_0F_+$ of the form
\[
F_0(Z_1) = -Z_1, \quad
F_+(Z_1) = \sum_{k=1}^{\infty} {Z_1}^k, \quad
F_0(Z_j) = F_+(Z_j) = Z_j
\, \, \text{for $j\ge 2$}.
\]
Using Proposition~\ref{prop:Div_J_n=1}, we compute
\[
{\sf J}_Z(F_+) = \left| \log \left(
\frac{F_+(Z_1) \otimes 1 -1 \otimes F_+(Z_1)}{Z_1 \otimes 1 - 1 \otimes Z_1} \right) \right|
= \left| \log \left(
\frac{1}{(1 - Z_1) \otimes (1 - Z_1)} \right) \right|
\]
and
\begin{align*}
{\sf J}_Z(F) & = F_0 ({\sf J}_Z(F_+)) = \big| \log \big( {\textstyle \frac{1}{(1 + Z_1) \otimes (1 + Z_1)}} \big) \big| \\
& = \left| \log( e^{-z_1} \otimes e^{-z_1}) \right| =
-(|z_1| \otimes {\bf 1} + {\bf 1} \otimes |z_1|).
\end{align*}
Hence, by Proposition~\ref{prop:change_variables_unified}, we obtain
$$
{\sf Div}_{Z'}(u)={\sf Div}_Z(u) - u(|z_1| \otimes {\bf 1} + {\bf 1} \otimes |z_1|).
$$
Referring to the defining formula \eqref{dfn:c_gamma} of ${\sf Div}_{\gamma}$, we conclude that under the inversion
$$
{\sf Div}_{F(\gamma)}(u) - {\sf Div}_\gamma(u) = u(2({\bf 1} \otimes |z_1|) - |z_1| \otimes {\bf 1} - {\bf 1} \otimes |z_1|) = - u(\tilde{\Delta}(|z_1|)).
$$

\vskip 0.2cm

{\em Step 3}.  We now analyse the slide move  $G \in {\rm Aut}(\Gamma)$ given by equation \eqref{eq:slide}. We have the following transformation of exponential generators:
$$
z_1 \mapsto z'_1=\log(\gamma_1\gamma_2) = {\rm bch}(z_1, z_2), \hskip 0.3cm z_j \mapsto z'_j=z_j \,\,
{\rm for} \,\, j \geq 2
$$
which induces the map
$$
|z_1| \mapsto |{\rm bch}(z_1, z_2)| = |z_1| + |z_2|, \hskip 0.3cm |z_j| \mapsto |z_j| \,\,
{\rm for} \,\, j \geq 2,
$$
where we have used the fact that $| \cdot |$ vanishes on commutators. We obtain
$$
\sum_{j=1}^n |z'_j| = \big( \sum_{j=1}^n |z_j| \big) + |z_2|.
$$
For Magnus generators $Z'_j=G(Z_j)$, we compute
$$
Z'_1 = Z_1 + Z_2 + Z_1 Z_2, \hskip 0.3cm
Z_j'=Z_j \,\, {\rm for} \,\, j\geq 2.
$$
The map $G$ admits a decomposition
$G = G_0 G_+$ of the form
\[
G_0(Z_1) = Z_1 + Z_2, 
\quad
G_+(Z_1) = Z_1 + Z_1Z_2 - Z_2^2,
\quad
G_0(Z_j) = G_+(Z_j) = Z_j
\, \, \text{for $j\ge 2$}.
\]
Note that $D_{Z}(G_+) = (\partial_{Z_i} G_+(Z_j))_{ij}$ is an upper triangular matrix with diagonal entries 
$$
(1 \otimes (1+Z_2), 1\otimes 1, \dots, 1\otimes 1)=
(1 \otimes e^{z_2}, 1\otimes 1, \dots, 1\otimes 1)
$$
and hence $\log D_Z(G_+)$ has diagonal entries $(1 \otimes z_2, 0, \ldots, 0)$.
Thus ${\sf J}_Z(G_+) = {\bf 1} \otimes |z_2|$ and ${\sf J}_Z(G) = G_0 ({\sf J}_Z(G_+)) = {\bf 1} \otimes |z_2|$.
Therefore
$$
{\sf Div}_{Z'}(u)={\sf Div}_Z(u) + u({\bf 1} \otimes |z_2|)
$$
and
$$
{\sf Div}_{G(\gamma)}(u) - {\sf Div}_\gamma(u) = u( {\bf 1} \otimes |z_2| - {\bf 1} \otimes |z_2|) = 0.
$$

\vskip 0.2cm

{\em Step 4}. We will now define a map $\mathfrak{s}_\gamma \colon {\rm Aut}(\Gamma) \to H_1(\Gamma, \mathbb{Z}) \cong \Lambda$ and show that it is indeed an additive group 1-cocycle. 
First, observe that 
\begin{align*}
{\sf Div}_{F(\gamma)}(u) 
& = {\sf Div}_{F(Z)}(u) -  {\bf 1} \otimes \textstyle\sum_j u(|F(z_j)|) \\
& = F({\sf Div}_Z(F^{-1}uF) -{\bf 1} \otimes \textstyle\sum_j (F^{-1}u F)(|z_j|)) \\
& = F({\sf Div}_\gamma(F^{-1}uF)).
\end{align*}

Assume that for $F, G \in {\rm Aut}(\Gamma)$ we have
$$
{\sf Div}_{F(\gamma)}(u)-{\sf Div}_\gamma(u)=u(\tilde{\Delta}(x)), \hskip 0.3cm
{\sf Div}_{G(\gamma)}(u)-{\sf Div}_\gamma(u)=u(\tilde{\Delta}(y))
$$
for some $x, y \in \Lambda$. We compute
\begin{align}
 & {\sf Div}_{FG(\gamma)}(u) - {\sf Div}_\gamma(u) \nonumber \\
=\, & {\sf Div}_{FG(\gamma)}(u) - {\sf Div}_{F(\gamma)}(u) +{\sf Div}_{F(\gamma)}(u) - {\sf Div}_\gamma(u) \nonumber \\
=\, & F({\sf Div}_{G(\gamma)}(F^{-1}uF) -{\sf Div}_\gamma(F^{-1}uF)) + {\sf Div}_{F(\gamma)}(u) - {\sf Div}_\gamma(u) \nonumber \\
=\, & F(F^{-1}uF(\tilde{\Delta}(y))) + u(\tilde{\Delta}(x)) \nonumber \\
=\, & u(\tilde{\Delta}(x + F(y))).
\label{eq:m_cocycle}
\end{align}

Assign $\mathfrak{s}_\gamma(F)=0$ if $F$ is a permutation of the generators or the slide move, and $\mathfrak{s}_\gamma(F)=-|z_1|$ if $F$ is the inversion.
Equation \eqref{eq:m_cocycle} shows that for every $F \in {\rm Aut}(\Gamma)$ there is an element $x(F) \in \Lambda$ such that
$$
{\sf Div}_{F(\gamma)}(u)-{\sf Div}_\gamma(u)=u(\tilde{\Delta}(x(F))).
$$
Property \eqref{eq:c-c_in} has been established.
We claim that such an element is in fact unique.
Indeed, consider the derivation $\ell$ such that $\ell(z_j) = z_j$ for all $j$ (namely, $\ell$ is the Euler operator) and compute
$$
{\sf Div}_{F(\gamma)}(\ell)-{\sf Div}_\gamma(\ell)=\ell(\tilde{\Delta}(x(F))) = \tilde{\Delta}(x(F)).
$$
Since $\tilde{\Delta}$ is injective, this shows that $x(F)$ is uniquely determined by the left hand side. We can now put $\mathfrak{s}_\gamma(F)=x(F)$.
The uniqueness of $x(F)$ and equation \eqref{eq:m_cocycle} show that $\mathfrak{s}_\gamma$ is indeed an additive group 1-cocycle on ${\rm Aut}(\Gamma)$ with values in $\Lambda \cong H_1(\Gamma, \mathbb{Z})$.

\vskip 0.2cm

{\em Step 5.} Let $G \in {\rm Aut}(\Gamma)$.
For any $u \in {\rm Der}(\widehat{\K \Gamma})$, we compute
\begin{align*}
u(\tilde{\Delta}(\mathfrak{s}_{G(\gamma)}(F)) )
& = {\sf Div}_{FG(\gamma)}(u) - {\sf Div}_{G(\gamma)}(u) \\
& = ({\sf Div}_{FG(\gamma)}(u) - {\sf Div}_\gamma(u)) - ({\sf Div}_{G(\gamma)}(u) - {\sf Div}_\gamma(u)) \\
& = u(\tilde{\Delta}(\mathfrak{s}_\gamma(FG) - \mathfrak{s}_\gamma(G))) \\
& = u(\tilde{\Delta}(\mathfrak{s}_\gamma(F) + F(\mathfrak{s}_\gamma(G))- \mathfrak{s}_\gamma(G))).
\end{align*}
By choosing again $u$ to be the Euler operator $\ell$ and
using the injectivity of  $\tilde{\Delta}$, we conclude
$$
\mathfrak{s}_{G(\gamma)}(F) = \mathfrak{s}_\gamma(F) + F(\mathfrak{s}_\gamma(G)) -\mathfrak{s}_\gamma(G),
$$
as required.
\end{proof}

\begin{rem}
Theorem~\ref{thm:div_independent} shows that the cocycle ${\sf Div}_\gamma$ is independent of the generating set $\gamma$ modulo the lattice $\tilde{\Delta}(\Lambda)$. Therefore, the cocycles ${\sf Div}_\gamma$ for different generating sets $\gamma$ descend to a canonical cocycle 
$$
{\sf Div}_{\Gamma} \colon {\rm Der}(\widehat{\mathbb{K}\Gamma}) \to (|\widehat{\mathbb{K}\Gamma}| \otimes |\widehat{\mathbb{K}\Gamma}|)/\tilde{\Delta}(\Lambda).
$$
\end{rem}

\begin{rem}
In \cite{Satoh}, Satoh showed that for a free group $\Gamma$ of rank $n\ge 2$ one has 
$$
H^1({\rm Aut}(\Gamma), H_1(\Gamma, \mathbb{Z})) \cong \mathbb{Z}.
$$
The cocycle $\mathfrak{s}_\gamma$ is nontrivial, and represents a generator of this cohomology group. 
%
%
\end{rem}

\subsection{Divergence on derivations of free Lie algebras} \label{subsec:div_der_Lie}

In this subsection, we consider restriction of the non-commutative divergence cocycle to derivations of free Lie algebras.

Recall that the (degree completed) free Lie algebra $L:=L(z_1, \dots, z_n)$ can be identified with the subspace of primitive elements in $A = \K\langle \langle z_1, \ldots, z_n \rangle \rangle$ under the coproduct $\Delta$, and moreover $A$ is the universal enveloping algebra of $L$.
For future use, we introduce a version of partial derivatives
$d_j = d_{z_j}\colon A \to A$ given by
$$
d_j =({\rm id} \otimes \varepsilon) \circ \partial_j.
$$
Here and in what follows we denote by $\partial_j=\partial_{z_j}$ the partial derivatives with respect to the exponential variables.
It is easy to see that for every $a \in A$ one has
$$
a=\varepsilon(a) + \sum_{j=1}^n (d_j a)z_j.
$$

\begin{lem} \label{lem:dja=ddja}
For $a \in L$, one has
$
\partial_j a = \tilde{\Delta}(d_j a)
$.
\end{lem}

\begin{proof}
The maps $L \to A \otimes A$ defined by $a \mapsto \partial_j a$ and $a \mapsto \tilde{\Delta}(d_j a)$ satisfy the identity
$$
\mathfrak{d}([a,b]) = \mathfrak{d}(a)(1 \otimes b)+(a\otimes 1)\mathfrak{d}(b) - (b\otimes 1)\mathfrak{d}(a)-\mathfrak{d}(b)(1 \otimes a).
$$
Here, we have meant by $\mathfrak{d}$ the two maps, and used the fact that for $a \in L$ one has $\tilde{\Delta}(a)=a\otimes 1 - 1 \otimes a$, and for 
$b\in A$ one has $\tilde{\Delta}(ab)=(a \otimes 1) \tilde{\Delta}(b) - \tilde{\Delta}(b)(1 \otimes a)$.
The two maps agree on the exponential generators: $\partial_j z_k = \tilde{\Delta}(d_j z_k)=\delta_{jk} (1\otimes 1)$, and we conclude that $\partial_j a = \tilde{\Delta}(d_j a)$ for all $a \in L$.
\end{proof}

Let ${\rm Der}(L)$ be the Lie algebra of continuous derivations of $A$.
It can be identified with the Lie algebra of Hopf derivations of $A$:
\[
{\rm Der}(L) \cong {\rm Der}(A, \Delta)
= \{ u \in {\rm Der}(A); u \circ \Delta = \Delta \circ u \}.
\]
Here, any derivation $u \in {\rm Der}(A)$ acts on $A\otimes A$ by $u\otimes {\rm id} + {\rm id}\otimes u$.
We introduce a map ${\sf div}_z \colon {\rm Der}(L) \to |A|$ given by formula (see \cite{AT12}):
\begin{equation} \label{eq:div_z} 
{\sf div}_z(u) = \sum_{j=1}^n |d_j u(z_j)|.
\end{equation}
It is related to the non-commutative divergence in the following way (see \cite{genus0}):

\begin{prop}  \label{prop:Div_div}
Let $u \in {\rm Der}(L)$. Then,
\begin{equation*}
{\sf Div}_z(u) = \tilde{\Delta}({\sf div}_z(u)), 
\hskip 0.3cm
{\sf div}_z(u) = ({\rm id} \otimes \varepsilon)({\sf Div}_z(u)).
\end{equation*}
Furthermore, the map ${\sf div}_z \colon {\rm Der}(L) \to |A|$ is a Lie algebra 1-cocycle.
\end{prop}

\begin{proof}
To prove the first equation, we compute, using Lemma \ref{lem:dja=ddja},
$$
\tilde{\Delta}({\sf div}_z(u)) = \tilde{\Delta}\sum_{j=1}^n |d_j u(z_j)| =
\sum_{j=1}^n |\tilde{\Delta}(d_j u(z_j))| =\sum_{j=1}^n |\partial_j u(z_j)| = {\sf Div}_z(u).
$$
In the other direction, we have
$$
({\rm id} \otimes \varepsilon)({\sf Div}_z(u))=({\rm id} \otimes \varepsilon) \sum_{j=1}^n |\partial_j u(z_j)| = \sum_{j=1}^n |({\rm id} \otimes \varepsilon) \partial_j u(z_j)|= \sum_{j=1}^n |d_j u(z_j)| = {\sf div}_z(u),
$$
as required.
The map $({\rm id} \otimes \varepsilon) \colon |A| \otimes |A| \to |A|$ is a homomorphism of ${\rm Der}(L)$-modules. Since the non-commutative divergence is a 1-cocycle, so is the map ${\sf div}_z$.
\end{proof}

Given positive integral weights $w_j = {\rm wt}(z_j)$, the free Lie algebra $L$ is equipped with the weight filtration.
Let ${\rm Der}^+(L) = {\rm Der}(L) \cap {\rm Der}^+(A)$ be the Lie algebra of derivations of positive filtration degree.
Its exponentiation is the group ${\rm Aut}^+(L) = {\rm Aut}(L) \cap {\rm Aut}^+(A)$ of positive automorphisms of $L$.
The Lie algebra 1-cocycle ${\sf div}_z \colon {\rm Der}(L) \to |A|$ restricted to ${\rm Der}^+(L)$ integrates to a group 1-cocycle 
${\sf j}_z \colon {\rm Aut}^+(L) \to |A|$.

\begin{prop} \label{prop:J_and_j}
Let $F \in {\rm Aut}^+(L)$.
\begin{enumerate}
\item[$(i)$] We have
${\sf J}_z(F) = \tilde{\Delta}({\sf j}_z(F))$ and ${\sf j}_z(F) = ({\rm id} \otimes \varepsilon)({\sf J}_z(F))$.
\item[$(ii)$]
We have 
\[
{\sf j}_z(F) = {\rm Tr} \log \mathsf{D}(F).
\]
Here, $\mathsf{D}(F) = (d_i F(z_j))_{ij} \in {\rm Mat}_n(A^{\rm op})$ with $d_i F(z_j)$ regarded as an element of $A^{\rm op}$ through the identity algebra anti-homomorphism $A \to A^{\rm op}$.
\end{enumerate} 
\end{prop}

\begin{proof}
(i) Both $\tilde{\Delta}\colon |A| \to |A| \otimes |A|$ and $({\rm id} \otimes \varepsilon)\colon |A| \otimes |A| \to |A|$ are homomorphisms of ${\rm Aut}(L)$-modules. Hence the assertion follows from Proposition~\ref{prop:Div_div}.

(ii)
This follows from applying ${\rm id} \otimes \varepsilon$ to the defining formula of ${\sf J}(F)$ in Definition~\ref{dfn:J}.
\end{proof}

\begin{rem}
In Proposition~\ref{prop:J_and_j}, it might be a bit cumbersome to work with a matrix with entries in $A^{\rm op}$.
To solve it, introduce yet another version of partial derivatives $d^{\flat}_j := (\varepsilon \otimes {\rm id}) \circ \partial_j \colon A \to A$ and a variant of twisted coproduct $\tilde{\Delta}^{\flat} := (\iota \otimes {\rm id}) \circ \Delta$.
Then, for any $a\in L$ we have $d_i a = \iota(d_i^{\flat} a)$ and $\partial_j a = \tilde{\Delta}^{\flat}(d_j^{\flat}a)$, similarly to Lemma~\ref{lem:dja=ddja}.
Then, replacing $d_j$ with $d_j^{\flat}$, we obtain a Lie algebra $1$-cocycle ${\sf div}_z^{\flat} \colon {\rm Der}(L) \to |A|$ 
and the corresponding integration ${\sf j}_z^{\flat}\colon {\rm Aut}^+(L) \to |A|$ which satisfies ${\sf div}_z(u) = \iota({\sf div}_z^{\flat}(u))$ and ${\sf j}_z(F) = \iota ({\sf j}_z^{\flat}(F))$.
Then we obtain ${\sf j}_z^{\flat}(F) = (\varepsilon \otimes {\rm id})({\sf J}_z(F)) = {\rm Tr} \log \mathsf{D}^{\flat}(F)$, where $\mathsf{D}^{\flat}(F):= (d_i^{\flat} F(z_j))_{ij} \in {\rm Mat}_n(A)$.
Therefore another expression of ${\sf j}_z(F)$ is as follows:
\[
{\sf j}_z(F) = \iota ({\rm Tr} \log \mathsf{D}^{\flat}(F)).
\]
We have chosen $d_j$ to adjust to the convention in \cite{AT12}.
\end{rem}

\subsection{Cocycles on Hopf derivations of group algebras}

Let $\Gamma$ be a free group with $n$ free generators $\gamma_1, \dots, \gamma_n$. Recall that the group algebra $\mathbb{K}\Gamma$ and the completed group algebra $\widehat{\mathbb{K}\Gamma}$ are Hopf algebras with coproduct $\Delta(\gamma_j) = \gamma_j \otimes \gamma_j$, augmentation $\varepsilon(\gamma_j) = 1$ and antipode $\iota(\gamma_j) = \gamma_j^{-1}$.

We will consider the Lie algebra of Hopf derivations of $\widehat{\K \Gamma}$:
\[
{\rm Der}(\widehat{\mathbb{K}\Gamma}, \Delta) =
\{ u \in {\rm Der}(\widehat{\mathbb{K}\Gamma}); u \circ \Delta = \Delta \circ u \}.
\]
Here, any derivation $u \in {\rm Der}(\widehat{\K \Gamma})$ acts on $\widehat{\K \Gamma} \otimes \widehat{\K \Gamma}$ by $u\otimes {\rm id} + {\rm id} \otimes u$.
For completeness, we record the following elementary properties of Hopf derivations.
\begin{lem} \label{lem:ucommutes}
Let $u \in {\rm Der}(\widehat{\mathbb{K}\Gamma}, \Delta)$.
Then $u$ commutes with $\iota$, and $\varepsilon \circ u = 0$.
In particular, $u$ commutes with the twisted coproduct $\tilde{\Delta} = ({\rm id}\otimes \iota) \circ \Delta$:
\[
u \circ \tilde{\Delta} = \tilde{\Delta} \circ u.
\]
\end{lem}
\begin{proof}
Note that the antipode $\iota$ acts on the primitive elements in $\widehat{\K \Gamma}$ as minus the identity, and that the primitive elements are in the kernel of the augmentation $\varepsilon$.
Since the derivation $u \in {\rm Der}(\widehat{\mathbb{K}\Gamma}, \Delta)$ maps primitives to primitives, the equations $u \circ \iota = \iota \circ u$ and $\varepsilon \circ u = 0$ hold on the primitive elements.
As $u\circ \iota$, $\iota \circ u$ and $\varepsilon \circ u$ are (anti-)derivations, these equations hold true on the whole $\widehat{\K \Gamma}$ as well.
\end{proof}

Furthermore, under the exponential expansion $\theta_{\rm exp}$ as in Section~\ref{sec:div_unique}, we have an isomorphism of Hopf algebras $\widehat{\mathbb{K}\Gamma} \cong \mathbb{K}\langle\langle z_1, \dots, z_n \rangle\rangle = A$ and
$$
{\rm Der}(\widehat{\mathbb{K}\Gamma}, \Delta) \cong
{\rm Der}(L).
$$
Similar to the notation ${\sf div}_z(u)$, it is convenient to introduce the notation
$$
{\sf div}_\gamma(u) = \sum_{j=1}^n |d_{Z_j} u(Z_j)|,
$$
where $d_{Z_j} = ({\rm id} \otimes \varepsilon) \circ \partial_{Z_j}$.

\begin{prop}  \label{prop:c_b}
For each free generating system $\{ \gamma_1, \dots, \gamma_n\}$ of $\Gamma$ the map 
$$
{\sf div}_\gamma \colon {\rm Der}(\widehat{\mathbb{K}\Gamma}, \Delta) \to |\widehat{\mathbb{K}\Gamma}|
$$
is a Lie algebra 1-cocycle. Furthermore, we have 
$$
{\sf div}_\gamma(u) = {\sf div}_z(u) + u(|\sum_{j=1}^n r(z_j)|)
$$
and ${\sf Div}_\gamma(u) = \tilde{\Delta}({\sf div}_\gamma(u))$.
\end{prop}

\begin{proof}
By Proposition~\ref{prop:c_change_variables}, Proposition~\ref{prop:Div_div} and Lemma~\ref{lem:ucommutes}, we have
$$
{\sf Div}_\gamma(u) = {\sf Div}_z(u) + u(|\tilde{\Delta}( \sum_j r(z_j))|)
= \tilde{\Delta}({\sf div}_z(u) + u(\sum_j |r(z_j)|)).
$$
By applying $({\rm id} \otimes \varepsilon)$ to this equation, we get
\begin{align*}
{\sf div}_z(u) + u(\textstyle\sum_j |r(z_j)|) & =  
({\rm id} \otimes \varepsilon)(\textstyle\sum_j |\partial_{Z_j} u(Z_j)| - {\bf 1} \otimes \textstyle\sum_j u(|z_j|)) \\
& = \textstyle\sum_j |d_{Z_j} u(Z_j)| \\
& = {\sf div}_\gamma(u).
\end{align*}
Here we have used Lemma~\ref{lem:ucommutes} to compute $\varepsilon (u(|z_j|)) = 0$ and that
$({\rm id} \otimes \varepsilon) \circ \partial_{Z_j} = d_{Z_j}$.
This formula shows that the map ${\sf div}_\gamma$ differs from ${\sf div}_z$ by a $1$-coboundary.
Since ${\sf div}_z$ is a Lie algebra $1$-cocycle by Proposition \ref{prop:Div_div}, we conclude that ${\sf div}_\gamma$ is a $1$-cocycle as well.
\end{proof}

Let ${\rm Aut}^+(\widehat{\K \Gamma})$ be the group of positive automorphisms of $\widehat{\K \Gamma}$ (with respect to the exponential generators), and let ${\rm Aut}^+(\widehat{\K \Gamma}, \Delta)$ be the subgroup of ${\rm Aut}^+(\widehat{\K \Gamma})$ consisting of Hopf autormorphisms of $\widehat{\K \Gamma}$. 
Denote by ${\sf J}_\gamma \colon {\rm Aut}^+(\widehat{\mathbb{K}\Gamma}) \to |\widehat{\mathbb{K}\Gamma}| \otimes |\widehat{\mathbb{K}\Gamma}|$ and
by ${\sf j}_\gamma\colon {\rm Aut}^+(\widehat{\mathbb{K}\Gamma}, \Delta) \to |\widehat{\mathbb{K}\Gamma}|$ the group $1$-cocycles obtained by integrating ${\sf Div}_{\gamma}$ and ${\sf div}_{\gamma}$, respectively.
They are related by formula
$$
{\sf J}_\gamma(F)=\tilde{\Delta}({\sf j}_\gamma(F)).
$$

\subsection{Tangential derivations and other cocycles} \label{subse:tngder}

The content of this and the following subsections will be needed when we work with surfaces with more than one boundary component.

Let $\mathcal{T}= \{ t_1, \ldots, t_k \} \subset A$ be a finite subset of $A$. We consider the following space of {\em tangential derivations}:
$$
{\rm tDer}_{\mathcal{T}}(A) = (u, u_1, \dots u_k) \in {\rm Der}(A) \times A^k; u(t_j)=[t_j, u_j]\}.
$$
We will use the notation $\tilde{u} = (u,u_1,\ldots,u_k)$ for tangential derivations.
Note that the components $u_j$, which we call the $j$th tangential component of $\tilde{u}$, are determined by $u$ up to the centralizer of $t_j$ in $A$.
The space ${\rm tDer}_\mathcal{T}(A)$ carries a Lie bracket: for $\tilde{u} = (u,u_1,\ldots,u_k)$ and $\tilde{v} = (v,v_1,\ldots,v_k)$,
$$
[\tilde{u}, \tilde{v}] = 
([u,v], \dots, u(v_j)-v(u_j)+(u_jv_j-v_ju_j), \dots).
$$
The projection ${\rm tDer}_{\mathcal{T}}(A) \to {\rm Der}(A), \tilde{u} \mapsto u$ is a Lie homomorphism. Hence, ${\rm Der}(A)$-modules automatically get a structure of ${\rm tDer}_\mathcal{T}(A)$-modules.

\begin{rem}
If $\mathcal{T}$ consists of homogeneous elements in $A$, then the Lie algebra ${\rm tDer}_\mathcal{T}(A)$ is graded. When it does not lead to confusion, we drop the symbol  $\mathcal{T}$ in the notation.
\end{rem}

\begin{prop} \label{prop:tgccocycles_lie}
For each $j \in \{1, \ldots, k\}$, the map ${\sf b}_j\colon {\rm tDer}_\mathcal{T}(A) \to |A|$ defined by formula
\[ 
    {\sf b}_j(\tilde{u}) = |u_j|
\]
is a Lie algebra 1-cocycle.
\end{prop}

\begin{proof}
We compute
$$
{\sf b}_j([\tilde{u}, \tilde{v}])=|u(v_j)-v(u_j)+(u_jv_j-v_ju_j)|=
u(|v_j|)-v(|u_j|)=u({\sf b}_j(\tilde{v}))-v({\sf b}_j(\tilde{u})),
$$
as required.
\end{proof}

Let $L = L(z_1,\ldots, z_n)$ be the (complete) free Lie algebra generated by $z_1, \ldots, z_n$.
If all elements $t_j$ belong to $L$, one can define the following Lie subalgebra of ${\rm tDer}_{\mathcal{T}}(A)$:
$$
{\rm tDer}_\mathcal{T}(L) = \{ (u, u_1, \dots, u_k) \in {\rm Der}(L) \times L^k; u(t_j)=[t_j, u_j] \}.
$$
Note that when restricted to ${\rm tDer}_\mathcal{T}(L)$, the cocycle ${\sf b}_j(\tilde{u})=|u_j|$ picks only the terms in $u_j$ which are linear in $z_i$'s since Lie brackets and multibrackets vanish under the $|\cdot|$-sign.

Given positive integral weights $w_j = {\rm wt}(z_j)$, one can consider the space ${\rm Der}^+(L)$ of positive derivations as in Section~\ref{subsec:div_der_Lie}.
We define a pro-nilpotent Lie algebra
\[
{\rm tDer}_{\mathcal{T}}^+(L) :=
\{ (u,u_1,\ldots,u_k) \in {\rm tDer}_{\mathcal{T}}(L) ; u \in {\rm Der}^+(L) \}
\]
and its exponentiation
${\rm tAut}_{\mathcal{T}}^+(L):= \exp({\rm tDer}_{\mathcal{T}}^+(L))$.
Another description of ${\rm tAut}_{\mathcal{T}}^+(L)$ is given  as follows:
\[
{\rm tAut}_{\mathcal{T}}^+(L) = 
\{ (F,f_1,\ldots, f_k) \in {\rm Aut}^+(L) \times \exp(L)^k ; F(t_j) = {f_j}^{-1}t_j f_j \}.
\]
For more details, see Appendix~\ref{subsec:tgtl}.
The projection to the first component $\tilde{F} = (F,f_1,\ldots,f_k) \mapsto F$ is a group homomorphism, and the product of ${\rm tAut}^+_{\mathcal{T}}(L)$ is defined by formula
$$
\tilde{F}\tilde{G}=(FG, \dots, f_j F(g_j), \dots).
$$

\begin{prop} \label{prop:tgccocycles_grp}
For each $j \in \{ 1,\ldots,k \}$, the expression 
$$
{\sf c}_j(\tilde{F}) = |\log(f_j)|
$$
for $\tilde{F} = (F,f_1,\ldots,f_k)$ is a group 1-cocycle on ${\rm tAut}^+_{\mathcal{T}}(L)$. 
Furthermore, it integrates the restriction of the cocycle ${\sf b}_j$ to ${\rm tDer}^+_{\mathcal{T}}(L)$.
\end{prop}

\begin{proof}
First, we check the cocycle property of ${\sf c}_j$:
\begin{align*}
{\sf c}_j(\tilde{F} \tilde{G}) =  |\log(f_j F(g_j))| 
& = |{\rm bch}(\log(f_j), \log(F(g_j)))| \\
& = |\log(f_j)| + |\log(F(g_j))| \\
& = {\sf c}_j(\tilde{F})+ \tilde{F}({\sf c}_j(\tilde{G})).
\end{align*}
Next, let $\tilde{u} = (u,u_1,\ldots,u_k) \in {\rm tDer}^+_{\mathcal{T}}(L)$.
Then $\exp(t \tilde{u}) = (\exp(tu), \dots, \exp(tu_j + O(t^2)), \dots)$ (see \eqref{eq:fu} in Appendix~\ref{subsec:tgtl}). We compute
$$
\left. \frac{d}{dt} \, {\sf c}_j(\exp(t \tilde{u})) \right|_{t=0} =
\left. \frac{d}{dt} |t u_j + O(t^2)| \right|_{t=0} = |u_j| = {\sf b}_j(\tilde{u}),
$$
as required.
\end{proof}

\subsection{Tangential derivations for surfaces} \label{subsec:tdersurface}

In this section, we apply the construction in the previous subsections to the completed group algebra of the fundamental group $\pi = \pi_1(\Sigma, *)$.

Let $\gamma =\{ \gamma_1, \ldots, \gamma_n \}$ be a peripheral system for $\pi$, i.e., a set of elements in $\pi$ such that $\gamma_j$ is homotopic to the $j$th boundary component of $\Sigma$ with positive orientation. 
For instance, such a $\gamma$ can be chosen as part of some standard generating system for $\pi$.
We consider the following Lie algebra:
\[
{\rm tDer}_{\gamma}(\K \pi) :=
\{
(u, u_1,\ldots, u_n) \in {\rm Der}(\K \pi) \times (\K \pi)^n ; u(\gamma_j) = [\gamma_j, u_j]
\}.
\]

\begin{prop} \label{prop:tDerSS'}
For two peripheral systems $\gamma=\{ \gamma_1, \dots, \gamma_n\}$ and $\gamma' = \{ \gamma'_1, \ldots, \gamma'_n \}$, there is a canonical Lie algebra isomorphism
\[
{\rm tDer}_{\gamma}(\K \pi)
\cong {\rm tDer}_{\gamma'}(\K \pi).
\]
\end{prop}

\begin{proof}
Since $\gamma_j'$ and $\gamma_j$ are freely homotopic, there is an element $\delta_j \in \pi$ such that $\gamma'_j = \delta_j \gamma_j \delta_j^{-1}$.
We consider the following map:
\begin{equation} \label{eq:tDerS}
{\rm tDer}_{\gamma}(\K \pi)
\to {\rm tDer}_{\gamma'}(\K \pi),
\quad
(u, u_1, \ldots, u_n) \mapsto
(u, u'_1, \ldots, u'_n),
\end{equation}
where $u'_j = \delta_j u_j \delta_j^{-1} - u(\delta_j) \delta_j^{-1}$.

First we check that $(u, u'_1, \dots, u'_n)$ is a tangential derivation:
\begin{align*}
u(\gamma'_j) =  u (\delta_j \gamma_j  \delta_j^{-1}) 
&=  u(\delta_j) \gamma_j \delta_j^{-1} +
\delta_j u(\gamma_j) \delta_j^{-1}
+ \delta_j \gamma_j u(\delta_j^{-1}) \\
& = u(\delta_j) \gamma_j \delta_j^{-1} +
\delta_j [\gamma_j, u_j] \delta_j^{-1} - \delta_j \gamma_j \delta_j^{-1} u(\delta_j) \delta_j^{-1} \\
& = [\gamma'_j, \delta_j u_j \delta_j^{-1} - u(\delta_j) \delta_j^{-1} ] \\
& = [\gamma'_j, u'_j].
\end{align*}
Next we show that the map \eqref{eq:tDerS} is well defined.
Since the centralizer of $\gamma_j$ in $\pi$ is an infinite cyclic group generated by $\gamma_j$, the element $\delta_j$ is determined only up to multiplication by powers of $\gamma_j$ from the right.
If $\delta'_j = \delta_j \gamma_j$, then we compute
\begin{align*}
\delta'_j u_j (\delta'_j)^{-1} - u(\delta'_j) (\delta'_j)^{-1} 
& = (\delta_j \gamma_j) u_j (\gamma_j^{-1} \delta_j^{-1})
- (u(\delta_j) \gamma_j + \delta_j u(\gamma_j) ) (\gamma_j^{-1} \delta_j^{-1}) \\
&= \delta_j \gamma_j u_j \gamma_j^{-1} \delta_j^{-1}
- (u(\delta_j) \gamma_j + \delta_j [\gamma_j, u_j]) \gamma_j^{-1} \delta_j^{-1} \\
&= \delta_j u_j \delta_j^{-1} - u(\delta_j) \delta_j^{-1}.
\end{align*}
This shows that the map \eqref{eq:tDerS} is independent of the choice of $\{ \delta_j\}_j$.
Finally, it is easy to see that the map \eqref{eq:tDerS} is an isomorphism of Lie algebras.
\end{proof}

Proposition \ref{prop:tDerSS'} implies that the Lie algebra ${\rm tDer}_{\gamma}(\K \pi)$ is uniquely determined by the surface $\Sigma$ and the numbering of its boundary components.
We will drop $\gamma$ from the notation and simply use ${\rm tDer}(\K \pi)$.

We also consider the completed version of ${\rm tDer}(\K\pi)$:
\[
{\rm tDer}(\widehat{\K \pi})=
{\rm tDer}_{\gamma}(\widehat{\K \pi}) :=
\{
(u, u_1,\ldots, u_n) \in {\rm Der}(\widehat{\K \pi}) \times (\widehat{\K \pi})^n ; u(\gamma_j) = [\gamma_j, u_j]
\}.
\]
Note that $u(\gamma_j) = [\gamma_j, u_j]$ is equivalent to $u(\log \gamma_j) = [\log \gamma_j, u_j]$.
This Lie algebra will play an important role in the subsequent sections.
We make a couple of remarks:
\begin{itemize}
    \item
The Lie algebra ${\rm tDer}(\widehat{\K \pi})$ is naturally filtered by the weight filtration on $\widehat{\K \pi}$, and its associated graded is isomorphic to ${\rm tDer}_{\{z_1,\ldots,z_n\}}(A)$, where $A = \widehat{T}({\rm gr}^{\rm wt}\, H)$.
     \item
Let $\{\alpha_i, \beta_i, \gamma_j\}$ be a standard generating system of $\pi$ and let $\gamma = \{\gamma_1, \ldots, \gamma_n \}$.
The exponential expansion
\[
\theta_{\exp}\colon \alpha_i \mapsto e^{x_i},
\quad \beta_i \mapsto e^{y_i},
\quad \gamma_j \mapsto e^{z_j}
\]
gives rise to a Lie algebra isomorphism
\[
{\rm tDer}(\widehat{\K \pi})
\xrightarrow[\hspace{0.5em} \cong \hspace{0.5em}]{\theta_{\exp}}
{\rm tDer}_{\{ z_1, \ldots, z_n \}}(A).
\]
\end{itemize}

We end this subsection with definition of several related objects.
First let 
\[
{\rm tDer}(\widehat{\K \pi}, \Delta)
= \{ (u,u_1,\ldots,u_n) \in {\rm tDer}(\widehat{\K \pi}) ; \text{$u \in {\rm Der}(\widehat{\K \pi},\Delta)$ and $u_j$ is primitive in $\widehat{\K \pi}$} \}
\]
be the Lie algebra of tangential Hopf derivations of $\widehat{\K \pi}$.
Next, using the weight filtration on $\widehat{\K \pi}$, we define the Lie algebras
\begin{align*}
    {\rm tDer}^+(\widehat{\K \pi})
    & = \{ (u,u_1,\ldots,u_n) \in {\rm tDer}(\widehat{\K \pi}) ; u \in {\rm Der}^+(\widehat{\K \pi}) \}, \\
    {\rm tDer}^+(\widehat{\K \pi}, \Delta) & = {\rm tDer}^+(\widehat{\K \pi}) \cap {\rm tDer}(\widehat{\K \pi}, \Delta).
    \end{align*}
Finally these two Lie algebras exponentiate to the following groups
\[
{\rm tAut}^+(\widehat{\K \pi})
= \exp({\rm tDer}^+(\widehat{\K \pi})),
\hspace{2em}
{\rm tAut}^+(\widehat{\K \pi}, \Delta)
= \exp({\rm tDer}^+(\widehat{\K \pi}, \Delta)).
\]

\subsection{Cocycles with spectator variable} \label{subsec:spectator}

In this section, we introduce an algebraic construction which will be used to analyze the coaction maps $\mu^f_r$ and $\mu^f_l$.

For $A = \K\langle \langle z_1, \ldots, z_n \rangle \rangle$, let $A\langle s \rangle:= \K\langle \langle z_1, \ldots, z_n, s \rangle \rangle$ be the algebra obtained by adding to $A$ an extra free non-commutative variable $s$ of weight zero.
We consider the Lie algebra of derivations of $A\langle s \rangle$ which vanish on $s$:
$$
{\rm Der}_s(A\langle s \rangle)=\{ u \in {\rm Der}(A\langle s \rangle); u(s)=0 \}.
$$
The non-commutative divergence cocycle ${\sf Div} \colon {\rm Der}_s(A\langle s \rangle) \to |A\langle s \rangle| \otimes |A\langle s \rangle|$ is given by the standard formula
$$
{\sf Div}(u) = \sum_{j=1}^n |\partial_j u(z_j)|.
$$
Note that there is no term $|\partial_s u(s)|$ on the right hand side since $u(s)=0$.

In what follows, we will focus on derivations of degree one in $s$, namely, derivations whose values at any elements of $A$ are linear in $s$. They are of the form
\begin{equation}  \label{eq:a_s_b}
u(z_j) = a'_j s a''_j,
\quad j=1,\ldots,n,
\end{equation}
where $a'_j \otimes a''_j \in A \otimes A$, adopting the Sweedler notation.
The space ${\rm Der}^{s-{\rm lin}}(A\langle s \rangle)$ of derivations of $A\langle s \rangle$ linear in $s$ can be identified with the space ${\rm Der}(A, A \otimes A)$ of derivations of $A \to A \otimes A$ by formally replacing $s$ with the symbol $\otimes$.
In more detail, to a derivation $u$ as in \eqref{eq:a_s_b}, one associates a derivation
\begin{equation} \label{eq:pa_u}
\partial_u: A \to A \otimes A, \quad
z_j \mapsto \partial_u(z_j) := \partial_s u(z_j) = a'_j \otimes a''_j.
\end{equation}
to obtain an isomorphism $\partial\colon {\rm Der}^{s-{\rm lin}}(A\langle s \rangle) \overset{\cong}{\to} {\rm Der}(A, A\otimes A), u \mapsto \partial_u$.
The space ${\rm Der}^{s-{\rm lin}}(A\langle s \rangle)$
is naturally an $A \otimes A^{\rm op}$-module with the action given by
\begin{equation} \label{eq:s-linbimodule}
(x \otimes y)(u) \colon z_j \mapsto (a'_j y) s (x a''_j).
\end{equation}
Formally this action is given by replacing $s$ with $ysx$, and matches a natural $A\otimes A^{\rm op}$-module structure on ${\rm Der}(A, A\otimes A)$.
When restricted to derivations linear in $s$, the divergence map takes values in the vector space
\begin{equation}  \label{eq:linear-s}
\big( |A| \otimes |s A| \big) \oplus \big( |s A| \otimes |A| \big).
\end{equation}
It is convenient to split the divergence map on ${\rm Der}^{s-{\rm lin}}(A\langle s \rangle)$ into two maps using the identification $|sA| \cong A$ as follows:
$$
{\sf Div}_r \colon {\rm Der}^{s-{\rm lin}}(A\langle s \rangle) \to |A| \otimes A, 
\hskip 0.3cm
{\sf Div}_l \colon {\rm Der}^{s-{\rm lin}}(A\langle s \rangle) \to A \otimes |A|.
$$
Using the Sweedler notation, one can define these maps by the following equation: for all $u \in {\rm Der}^{s-{\rm lin}}(A\langle s \rangle)$, 
\[
{\sf Div}(u) = |{\sf Div}'_r(u) | \otimes | s {\sf Div}''_r(u)| + 
| s {\sf Div}'_l(u) | \otimes |{\sf Div}''_l(u) |.
\]

\begin{prop}  \label{prop:Div_l_r}
For $u \in {\rm Der}^{s-{\rm lin}}(A\langle s \rangle)$ and $x, y \in A$, we have
\begin{equation*}
\begin{array}{lll}  
{\sf Div}_r((x \otimes y)u) 
& = & (1 \otimes x) \big( {\sf Div}_r(u)(1 \otimes y) + (|\cdot | \otimes {\rm id})\partial_u y \big), \\
{\sf Div}_l((x \otimes y)u) 
& = & \big( (x \otimes 1){\sf Div}_l(u) + 
({\rm id} \otimes | \cdot|)\partial_u x \big) (y \otimes 1).
\end{array}
\end{equation*}
\end{prop}

\begin{proof}
Let $u$ be given as in \eqref{eq:a_s_b}.
We compute
\begin{align*}
{\sf Div}((x \otimes y)u)
& = \textstyle\sum_j |\partial_j((x\otimes y)u)(z_j)| \\
& = \textstyle\sum_j
(|a'_j ysx(\partial'_j a''_j)| \otimes |\partial''_j a''_j| +
|a'_j ys(\partial'_j x)| \otimes |(\partial''_j x) a''_j|) + \cdots,
\end{align*}
where $\cdots$ stand for terms where $s$ is in the second factor of the tensor product. We can now read the formula for the map ${\sf Div}_l$:
$$
{\sf Div}_l((x\otimes y)u) = \sum_j(x(\partial'_j a''_j)a'_j y \otimes |\partial''_j a''_j| + (\partial'_j x)a'_j y \otimes |(\partial''_j x) a''_j|).
$$
The two terms on the right hand side reproduce the two terms in the formula for ${\sf Div}_l$.
The formula for ${\sf Div}_r$ can be verified in a similar fashion.
\end{proof}

Next consider a $1$-coboundary $d{\bf m} \colon {\rm Der}(A) \to |A| \otimes |A|$ of some ${\bf m}=|m'| \otimes |m''| \in |A| \otimes |A|$. It admits a natural extension to ${\rm Der}_s(A\langle s \rangle)$ given by the same formula. On derivations linear in $s$, it takes values in the space \eqref{eq:linear-s}. We again identify $|s A| \cong A$, and we
denote the components of the map $d{\bf m}$ by 
\begin{equation*}
d{\bf m}_r \colon {\rm Der}^{s-{\rm lin}}(A\langle s \rangle) \to |A| \otimes A,  
\hskip 0.3cm
d{\bf m}_l \colon {\rm Der}^{s-{\rm lin}}(A\langle s \rangle) \to A \otimes |A|.
\end{equation*}

\begin{prop}  \label{prop:trivial_l_r}
For $u \in {\rm Der}^{s-{\rm lin}}(A\langle s \rangle)$ and $x, y \in A$, we have
\[
\begin{array}{lll}
d{\bf m}_r((x \otimes y)u)
& = & (1 \otimes x)d{\bf m}_r(u)(1 \otimes y), \\
d{\bf m}_l((x \otimes y)u)
& = & (x \otimes 1)d{\bf m}_l(u)(y \otimes 1).
\end{array}
\]
\end{prop}

\begin{proof}
Let $u$ be given as in \eqref{eq:a_s_b}.
We compute
$$
d{\bf m}((x \otimes y)u)=
((x \otimes y)u) (|m'| \otimes |m''|) =
\sum_j |(\partial'_j m') a'_j ysx a''_j (\partial''_j m')|
\otimes |m''| + \cdots,
$$
where $\cdots$ stands for the action of the derivation $(x \otimes y)u$ on $|m''|$. We can now read the result for the map $d{\bf m}_l$:
$$
d{\bf m}_l((x \otimes y)u) = x a''_j (\partial''_j m')(\partial'_j m') a'_j y \otimes |m''| 
= (x \otimes 1)dm_l(u)(y \otimes 1),
$$
as required. The calculation for the map $d{\bf m}_r$ is similar.
\end{proof}

Let $\mathcal{T} = \{ t_1, \ldots, t_k \} \subset A$.
One can extend the cocycles ${\sf b}_j \colon {\rm tDer}_{\mathcal{T}}(A) \to |A|$ in Proposition~\ref{prop:tgccocycles_lie} to tangential derivations of the algebra with spectator variable $A\langle s \rangle$. We define
$$
{\rm tDer}_\mathcal{T}(A\langle s \rangle)=\{ (u, u_1, \dots, u_k) \in {\rm Der}_s(A\langle s \rangle) \times (A\langle s \rangle)^k; u(t_j)=[t_j, u_j] \} .
$$
Again we will be interested in tangential derivations which are linear in $s$:
$$
u(z_j)= a'_j s a''_j,
\quad j = 1,\ldots, n, 
\quad \text{and} \quad
u_j= u'_j s u''_j, 
\quad j = 1,\ldots,k.
$$
Let ${\rm tDer}_\mathcal{T}^{s-{\rm lin}}(A\langle s \rangle)$ be the space of $s$-linear tangential derivations, i.e., the space of elements $\tilde{u} = (u, u_1,\ldots,u_k) \in {\rm tDer}_{\mathcal{T}}(A\langle s \rangle)$ such that $u \in {\rm Der}^{s-{\rm lin}}(A\langle s \rangle)$ and $u_j$ is linear in $s$ for all $j = 1,\ldots,k$.
As a natural lift of the action \eqref{eq:s-linbimodule}, it carries an action of $A \otimes A^{\rm op}$ via the substitution $s \mapsto ysx$:
\begin{equation} \label{eq:s-linbimoduletan}
((x\otimes y)u)(z_j)= (a'_jy)s(x a''_j), \hskip 0.6cm
((x\otimes y)u)_j= (u'_j y)s(x u''_j).
\end{equation}

As before, the linear in $s$ part of $|A\langle s \rangle|$ can be identified with $A$.
Hence, the restriction of the cocycle ${\sf b}_j$ to ${\rm tDer}_\mathcal{T}^{s-{\rm lin}}(A\langle s \rangle)$ defines the following map with values in $A$:
\[
{\sf b}_j^{s-{\rm lin}}\colon {\rm tDer}_{\mathcal{T}}^{s-{\rm lin}}(A\langle s \rangle) \to A, 
\quad
\tilde{u} \overset{{\sf b}_j}{\mapsto} |u'_j s u''_j| \mapsto u''_j u'_j.
\]

\begin{lem}  \label{prop:c_i_action}
For all $\tilde{u} \in {\rm tDer}_{\mathcal{T}}^{s-{\rm lin}}(A\langle s \rangle)$ and $x,y \in A$, one has
$$
{\sf b}_j^{s-{\rm lin}}((x \otimes y)\tilde{u}) = 
x {\sf b}_j^{s-{\rm lin}}(\tilde{u}) y.
$$
\end{lem}

\begin{proof}
This is straightforward:
$
{\sf b}_j^{s-{\rm lin}}((x \otimes y)\tilde{u}) =
x u''_j u'_j y = x( u''_j u'_j) y = x {\sf b}_j^{s-{\rm lin}}(\tilde{u}) y
$.
\end{proof}

The next theorem summarizes several statements of this section.

\begin{thm}  \label{thm:c_l_r}
Let $\mathfrak{b}\colon {\rm tDer}_\mathcal{T}(A) \to |A| \otimes |A|$ be a 1-cocycle which is cohomologous to the sum of the pull-back of the non-commutative divergence ${\sf Div}$ and a linear combination of the cocycles ${\bf 1} \otimes {\sf b}_j$ and ${\sf b}_j \otimes {\bf 1}$, $j=1,\ldots,k$.
For its natural extension to ${\rm tDer}_\mathcal{T}(A\langle s \rangle)$, we denote the components of $\mathfrak{b}$ by
\[
\mathfrak{b}_r\colon {\rm tDer}_\mathcal{T}^{s-{\rm lin}}(A\langle s \rangle) \to |A| \otimes A,
\hspace{2em}
\mathfrak{b}_l\colon {\rm tDer}_\mathcal{T}^{s-{\rm lin}}(A\langle s \rangle) \to A \otimes |A|.
\]
In more detail, these maps are defined by the following equation: for all $\tilde{u} \in {\rm tDer}_\mathcal{T}^{s-{\rm lin}}(A\langle s \rangle)$,
\[
\mathfrak{b}({\tilde{u}}) = 
| \mathfrak{b}'_r(\tilde{u})| \otimes |s\mathfrak{b}''_r(\tilde{u})| + 
|s\mathfrak{b}'_l(\tilde{u})| \otimes |\mathfrak{b}''_l(\tilde{u})|.
\]
Then, the maps $\mathfrak{b}_r$ and $\mathfrak{b}_l$ satisfy the properties
\begin{equation}       \label{eq:c_l_r_long}
\begin{array}{lll}
\mathfrak{b}_r((x \otimes y)\tilde{u}) & = & 
(1\otimes x)\big( \mathfrak{b}_r(\tilde{u})(1\otimes y)+(|\cdot| \otimes {\rm id})\partial_u y \big), \\
\mathfrak{b}_l((x\otimes y)\tilde{u}) 
& = & \big( (x\otimes 1)\mathfrak{b}_l(\tilde{u})+({\rm id}\otimes |\cdot|)\partial_u x \big)(y\otimes 1)
\end{array}
\end{equation}
for all $\tilde{u} \in {\rm tDer}_\mathcal{T}^{s-{\rm lin}}(A\langle s \rangle)$ and $x, y \in A$.
\end{thm}

\begin{proof}
For 1-cocycles cohomologous to the pull-back of the non-commutative divergence this statement is a direct consequence of Propositions~\ref{prop:Div_l_r} and \ref{prop:trivial_l_r}.
It remains to check the effect of adding a linear combination of the cocycles ${\bf 1} \otimes {\sf b}_j$ and ${\sf b}_j \otimes {\bf 1}$. 
Consider the cocycle ${\sf b}_j \otimes {\bf 1}$.
Its contribution to the component $\mathfrak{b}_r$ vanishes for all $\tilde{u}$.
For the component $\mathfrak{b}_l$, we compute using Lemma~\ref{prop:c_i_action}:
$$
({\sf b}_j \otimes {\bf 1})_l((x\otimes y)\tilde{u})= {\sf b}_j^{s-{\rm lin}}((x \otimes y)\tilde{u}) \otimes {\bf 1} =
(x {\sf b}_j^{s-{\rm lin}}(\tilde{u})y) \otimes {\bf 1})=
(x\otimes 1)({\sf b}_j \otimes {\bf 1})_l(\tilde{u})(y\otimes 1),
$$
which implies the statement of the theorem. A computation for the cocycle ${\bf 1} \otimes {\sf b}_j$ is similar.
\end{proof}

\begin{rem}
By putting $s=1$, we obtain Lie algebra homomorphisms
\[
{\rm Der}^{s-{\rm lin}}(A\langle s \rangle)
\to {\rm Der}(A),
\hspace{2em}
{\rm tDer}^{s-{\rm lin}}_{\mathcal{T}}(A\langle s \rangle)
\to {\rm tDer}_{\mathcal{T}}(A).
\]
These maps will be used in the next section.
\end{rem}

\section{Turaev cobracket as non-commutative divergence}      \label{sec:Turaev}

In this section, we explain how the Turaev cobracket and the coaction maps $\mu^f_r$, $\mu^f_l$ can be reprensented as compositions of a canonical $1$-cocycle obtained by modifying the non-commutative divergence cocycle and operations induced by the double bracket $\kappa$.
We also describe their associated graded versions.

\subsection{Multiplicative properties} \label{subsec:mulprop}

Let $A=\mathbb{K}\langle\langle z_1, \dots z_n\rangle\rangle$ be the ring of non-commutative formal power series in variables $z_1, \dots, z_n$ (with some grading or filtration). In this subsection, we give a way to construct maps $A \to |A| \otimes A$ and $A \to A \otimes |A|$ with certain multiplicative properties.

Assume that $A$ is equipped with a double bracket $\Pi \colon A \otimes A \to A \otimes A$. This double bracket has a canonical extension to the algebra $A\langle s \rangle = \K\langle \langle z_1, \ldots, z_n, s\rangle \rangle$ with the free extra variable $s$
defined by the property that $\Pi(s, a)=\Pi(a, s)=0$
for all $a \in A\langle s \rangle$. By abuse of notation, we also denote this double bracket by $\Pi$. 

\begin{dfn} \label{dfn:Ham}
We define a map ${\rm Ham}^{\Pi}$ as follows:
\begin{equation*}
{\rm Ham}^{\Pi} \colon A \rightarrow {\rm Der}^{s-{\rm lin}}(A\langle s \rangle), \quad
a \mapsto {\rm Ham}^{\Pi}_a = \sigma^{\Pi}(|sa|) = \{ |sa|, \cdot \}_{\Pi}.
\end{equation*}
\end{dfn}
Formally ${\rm Ham}^{\Pi}$ is obtained from $\Pi$ by replacing $\otimes$ with $s$, or equivalently equals the composite of the map $A \to {\rm Der}(A, A\otimes A), a \mapsto \Pi(a, \cdot)$ and the inverse of the isomorphism $\partial\colon {\rm Der}^{s-{\rm lin}}(A\langle s \rangle) \to {\rm Der}(A, A\otimes A)$ given in \eqref{eq:pa_u}.
In more detail, for any $a, b\in A$ we have
$$
{\rm Ham}^{\Pi}_a(b)=\Pi(a,b)' s \Pi(a,b)''
$$
where $\Pi(a,b)=\Pi(a,b)' \otimes \Pi(a,b)''$, and we have
$
\partial_{{\rm Ham}^{\Pi}_a}(b) = \Pi(a,b)
$.

Since the derivation ${\rm Ham}^{\Pi}_a$ is linear in $s$, it maps $|A|$ to the $s$-linear part of $|A\langle s \rangle|$, that is $|sA| \cong A$.
The following lemma describes this action.
\begin{lem} \label{lem:Hama|b|}
For any $a, b \in A$, we have
$
{\rm Ham}_a^{\Pi}(|b|) = 
| s \{ a, |b| \}_{\Pi}|
$.
\end{lem}
\begin{proof}
We compute
\[
{\rm Ham}_a^{\Pi}(|b|)
= | \{ |sa|, b\}_{\Pi} |
= | \Pi(a,b)' s \Pi(a,b)''|
= | s \Pi(a,b)''\Pi(a,b)' | 
= | s \{ a, |b| \}_{\Pi}|.
\]
\end{proof}

Recall from \eqref{eq:s-linbimodule} that the space ${\rm Der}^{s-{\rm lin}}(A\langle s \rangle)$ has an $A\otimes A^{\rm op}$-module structure.

\begin{prop}  \label{prop:h_{ab}}
For $a, b \in A$, we have
$$
{\rm Ham}^{\Pi}_{ab} = (1\otimes b){\rm Ham}^{\Pi}_a + (a\otimes 1){\rm Ham}^{\Pi}_b.
$$
\end{prop}

\begin{proof}  \label{prop:m=ch}
The direct computation
\begin{align*}
{\rm Ham}^{\Pi}_{ab}(c)  & =  \Pi(ab, c)' s \Pi(ab, c)'' \\
& = (\Pi(a, c)'b)s \Pi(a, c)'' + \Pi(b, c)' s (a\Pi(b, c)'') \\
& = \big( (1\otimes b){\rm Ham}^{\Pi}_a + (a\otimes 1){\rm Ham}^{\Pi}_b \big)(c)
\end{align*}
follows from defining properties of double brackets.
\end{proof}

Let $\mathcal{T} = \{ t_1, \ldots, t_k \} \subset A$.
In all our applications, $\mathcal{T}$ is supposed to be a subset of a free generating set of the algebra $A$, i.e., $\mathcal{T} \subset \{ z_1, \ldots, z_n\}$.
For such a choice of $\mathcal{T}$, the centralizer of $t_j$ in the algebra $A\langle s \rangle$ (which is a non-commutative formal power series ring in $n+1$ variables) is $\K [[t_j]]$,
and its intersection with the set of elements linear in $s$ is trivial.
In what follows, we will always assume this condition.

\begin{dfn}
A double bracket $\Pi\colon A \otimes A \to A \otimes A$ is called {\em tangential to $\mathcal{T}$} if there exists a map $\widetilde{{\rm Ham}^{\Pi}} \colon A \to {\rm tDer}_{\mathcal{T}}^{s-{\rm lin}}(A\langle s \rangle), a \mapsto \widetilde{{\rm Ham}_a^{\Pi}}$ which lifts the map ${\rm Ham}^{\Pi}$.
\end{dfn}

\begin{lem} \label{lem:lift_Ham}
\begin{enumerate}
    \item[$(i)$]
Let $\Pi \colon A \otimes A \to A \otimes A$ be a double bracket tangential to $\mathcal{T}$.
Then, the lift $\widetilde{{\rm Ham}^{\Pi}}$ of ${\rm Ham}^{\Pi}$ is uniquely determined by $\Pi$.
\item[$(ii)$]
Let $\mathcal{Z}$ be a free generating set of $A$.
A double bracket $\Pi$ is tangential to $\mathcal{T}$ if and only if for any generator $z \in \mathcal{Z}$ and $1\le j \le k$, there exists an element $u_j(z,s) \in A \langle s \rangle$ linear in $s$ such that ${\rm Ham}^{\Pi}_z(t_j) = [t_j, u_j(z,s)]$.
\item[$(iii)$]
If $\Pi$ is a double bracket tangential to $\mathcal{T}$, the composition 
\[
A \xrightarrow{\widetilde{{\rm Ham}^{\Pi}}}
{\rm tDer}_{\mathcal{T}}^{s-{\rm lin}}(A\langle s \rangle)
\xrightarrow{s=1}
{\rm tDer}_{\mathcal{T}}(A)
\]
descends to a map $\tilde{\sigma} = \tilde{\sigma}^{\Pi}\colon |A| \to {\rm tDer}_{\mathcal{T}}(A), |a| \mapsto {\widetilde{{\rm Ham}^{\Pi}_a}}|_{s=1}$.
The map $\tilde{\sigma}$ lifts the map $\sigma = \sigma^{\Pi}\colon |A| \to {\rm Der}(A)$.
\end{enumerate}
\end{lem}

\begin{proof}
(i) The existence of a lift $\widetilde{{\rm Ham}^{\Pi}}$ means that for each $a\in A$ and $1 \le j \le k$ there exists an element $u_j(a,s) \in A\langle s \rangle$ linear in $s$ such that ${\rm Ham}^{\Pi}_a(t_j) = [t_j, u_j(a,s)]$.
The element $u_j(a,s)$ is determined by the derivation ${\rm Ham}^{\Pi}_a$ modulo the centralizer of $t_j$ in $A \langle s \rangle$.
Since by assumptions the intersection of the centralizer of $t_j$ in $A\langle s \rangle$ with the set of elements linear in $s$ is trivial, it follows that $u_j(a,s)$ is uniquely determined and $u_j(a,s)$ is the $j$th tangential component of $\widetilde{{\rm Ham}^{\Pi}_a}$.

(ii) The ``only if'' part is clear.
To prove the converse, let $a, b \in A$ and assume that 
\[
{\rm Ham}^{\Pi}_a(t_j) = [t_j, u'_js u''_j],
\quad \quad
{\rm Ham}^{\Pi}_b(t_j) = [t_j, v'_js v''_j],
\]
where we use the Sweedler notation.
Then, by Proposition \ref{prop:h_{ab}} we have
\[
{\rm Ham}^{\Pi}_{ab}(t_j) =
\left( (1\otimes b){\rm Ham}^{\Pi}_a + (a \otimes 1){\rm Ham}^{\Pi}_b \right) (t_j)
= [t_j, u'_j b s u''_j + v'_j s a v''_j].
\]
Hence the property established on generators of $A$ extends to all elements of $A$.

(iii)
The proof of (ii) shows the following: for any $a,b\in A$ the $j$th tangential component of $\widetilde{{\rm Ham}^{\Pi}_{ab}}$ is given by $u'_j bsu''_j + v'_j sa v''_j$.
Putting $s=1$, we obtain $u'_j b u''_j + v'_j a v''_j$, which is symmetric in $a$ and $b$.
Hence the composition of $\widetilde{{\rm Ham}^{\Pi}}$ with evaluation at $s=1$ vanishes on commutators in $A$.
Hence we obtain the induced map $\tilde{\sigma}$, as required.
By construction, the projection of $\tilde{\sigma}^{\Pi}(|a|) = {\widetilde{{\rm Ham}^{\Pi}_a}}|_{s=1} \in {\rm tDer}_{\mathcal{T}}(A)$ onto ${\rm Der}(A)$ is ${{\rm Ham}^{\Pi}_a}|_{s=1} = \sigma^{\Pi}(|a|)$.
\end{proof}

Recall from \eqref{eq:s-linbimoduletan} that the space ${\rm tDer}_{\mathcal{T}}^{s-{\rm lin}}(A \langle s \rangle )$ has a structure of $A \otimes A^{\rm op}$-module.

\begin{prop} \label{prop:h_{ab}refined}
Let $\Pi$ be a double bracket tangential to $\mathcal{T}$.
For all $a, b \in A$, we have
\[
\widetilde{{\rm Ham}^{\Pi}_{ab}}
= (1 \otimes b) \widetilde{{\rm Ham}^{\Pi}_a} + (a\otimes 1)\widetilde{{\rm Ham}^{\Pi}_b}.
\]
\end{prop}
\begin{proof}
The proof of Lemma~\ref{lem:lift_Ham}~(ii) shows the following: if the $j$th tangnential components of $\widetilde{{\rm Ham}^{\Pi}_a}$ and $\widetilde{{\rm Ham}^{\Pi}_b}$ are $u'_jsu''_j$ and $v'_jsv''_j$, respectively, then the $j$th tangential component of $\widetilde{{\rm Ham}^{\Pi}_{ab}}$ is equal to $u'_jbsu''_j + v'_jsav''_j$.
This fact together with Proposition~\ref{prop:h_{ab}} implies the required result.
\end{proof}

The main result of this subsection is the following statement:
\begin{prop} \label{prop:mrab}
Let $\Pi$ be a double bracket on $A$ tangential to $\mathcal{T}$, and assume that $\mathfrak{b}_r\colon {\rm tDer}_{\mathcal{T}}^{s-{\rm lin}}(A\langle s \rangle) \to |A| \otimes A$ and $\mathfrak{b}_l\colon {\rm tDer}_{\mathcal{T}}^{s-{\rm lin}}(A\langle s \rangle) \to A \otimes |A|$ are maps with properties~\eqref{eq:c_l_r_long}.
Then the composition maps $m_r = \mathfrak{b}_r \circ \widetilde{{\rm Ham}^{\Pi}} \colon A \to |A| \otimes A$ and $m_l= \mathfrak{b}_l \circ \widetilde{{\rm Ham}^{\Pi}}\colon A \to A \otimes |A|$ satisfy the following multiplicative properties:
\begin{equation} \label{eq:multiplicative}
\begin{array}{lll}
m_r(ab) & = & m_r(a)(1 \otimes b) + (1\otimes a)m_r(b)+ (|\cdot| \otimes 1) \Pi(a,b), \\
m_l(ab) & = & m_l(a)(b \otimes 1) + (a\otimes 1)m_l(b)+ (1 \otimes |\cdot|) \Pi(b,a). \\
\end{array}
\end{equation}
\end{prop}

\begin{proof}
The proof is by direct computation:
\begin{align*}
m_r(ab) & =  \mathfrak{b}_r(\widetilde{{\rm Ham}^{\Pi}_{ab}}) \\
& = \mathfrak{b}_r \big( (1\otimes b)\widetilde{{\rm Ham}^{\Pi}_a} + (a\otimes 1)\widetilde{{\rm Ham}^{\Pi}_b} \big) \\
& = \mathfrak{b}_r(\widetilde{{\rm Ham}^{\Pi}_a})(1 \otimes b) + (|\cdot| \otimes 1)\partial_{{\rm Ham}^{\Pi}_a}b + (1 \otimes a)\mathfrak{b}_r(\widetilde{{\rm Ham}^{\Pi}_b}) \\
& = m_r(a)(1 \otimes b) + (1\otimes a)m_r(b)+ (|\cdot| \otimes 1) \Pi(a,b).
\end{align*}
Here we have used Proposition \ref{prop:h_{ab}refined} in the second line and \eqref{eq:c_l_r_long} in the third line.
The multiplicative property of the map $m_l$ is proved by a similar calculation.
\end{proof}

\subsection{The cocycle ${\sf Div}^f$} \label{subsec:cocyclec^f}
In this subsection, for each framing $f$ on the surface $\Sigma$ with nonepmty boundary we introduce a 1-cocycle ${\sf Div}^f$ on the Lie algebra of tangential derivations of the completed group algebra $\widehat{\mathbb{K}\pi}$, where  $\pi=\pi_1(\Sigma)$.

In more detail, let $\Sigma$ be an oriented surface of genus $g$ with $n+1$ boundary components.
Choose a standard generating system $\alpha_i, \beta_i, \gamma_j$ with $i=1, \dots, g, j=1, \dots, n$.
We use this notation (when it does not lead to confusion)  both for simple closed curves and for generators of $\pi$.
We define Magnus generators $X_i, Y_i, Z_j$
$$
X_i = \alpha_i -1, \hskip 0.3cm Y_i=\beta_i -1, \hskip 0.3cm Z_j= \gamma_j -1
$$
and exponential generators $x_i, y_i, z_j$
$$
x_i = \log (\alpha_i), \hskip 0.3cm y_i = \log (\beta_i), \hskip 0.3cm 
z_j = \log (\gamma_j).
$$
As was explained in Section~\ref{sec:div_unique}, the completed group algebra $\widehat{\mathbb{K}\pi}$ is isomorphic to the degree completed free associative algebra with the generators $X_i, Y_i, Z_j$ and to the degree completed free associative algebra with the generators $x_i, y_i, z_j$:
$$
\mathbb{K}\langle\langle X_i, Y_i, Z_j \rangle\rangle
\xleftarrow[\hspace{0.5em} \cong \hspace{0.5em}]{\theta_M}
\widehat{\mathbb{K} \pi}
\xrightarrow[\hspace{0.5em} \cong \hspace{0.5em}]{\theta_{\exp}}
\mathbb{K}\langle\langle x_i, y_i, z_j \rangle\rangle.
$$

Let $f$ be a framing of $\Sigma$.
We use it to define the following element
\begin{equation} \label{eq:def_pf}
{\bf p}^f=\sum_{i=1}^g |{\rm rot}^f(\alpha_i) y_i - {\rm rot}^f(\beta_i) x_i| \in \Lambda \subset |\widehat{\K \pi}|
\end{equation}
and the 1-cocycle ${\sf b}^f \colon {\rm tDer}(\widehat{\K \pi}) \to |\widehat{\K \pi}|$ 
on the Lie algebra of tangential derivations ${\rm tDer}(\widehat{\K \pi})$ (see Section~\ref{subsec:tdersurface}), defined by the following formula:
\begin{equation} \label{eq:def_bf}
{\sf b}^f(\tilde{u}) = \sum_{j=1}^n {\rm rot}^f(\gamma_j) {\sf b}_j(\tilde{u}) = \sum_{j=1}^n {\rm rot}^f(\gamma_j) |u_j|,
\quad \quad
\tilde{u} = (u, u_1,\ldots, u_n).
\end{equation}

Finally we can define one of the main objects of this paper:
\begin{dfn} \label{dfn:cfabc}
For a standard generating system $\{ \alpha_i, \beta_i, \gamma_j\}$ and a framing $f$ on $\Sigma$, we define a 1-cocycle
$$
{\sf Div}^f_{\alpha, \beta, \gamma}\colon {\rm tDer}(\widehat{\K \pi}) \to |\widehat{\K \pi}| \otimes |\widehat{\K \pi}|
$$
by formula
\begin{equation} \label{eq:c^f_definition}
{\sf Div}^f_{\alpha, \beta, \gamma}(\tilde{u}) =
{\sf Div}_{\alpha, \beta, \gamma}(u) + {\bf 1} \wedge (u({\bf p}^f) + {\sf b}^f(\tilde{u})),
\end{equation}
where ${\sf Div}_{\alpha, \beta, \gamma}$ stands for the 1-cocycle in equation~\eqref{dfn:c_gamma} determined by the generating system $\{ \alpha_i, \beta_i, \gamma_j \}$.
\end{dfn}

Note that the $1$-coboundaries $u \mapsto u(|z_j|)$ vanish on ${\rm tDer}(\widehat{\K \pi})$.
Indeed, we have $u(|z_j|)=|[z_j, u_j]|=0$.
Using this fact and equation~\eqref{dfn:c_gamma}, one has
\begin{equation}  \label{eq:c^f_details}
{\sf Div}^f_{\alpha, \beta, \gamma}(\tilde{u}) = {\sf Div}_{X,Y,Z}(u) - {\bf 1} \otimes u(\sum_i (|x_i| + |y_i|))
+ {\bf 1} \wedge (u({\bf p}^f) + {\sf b}^f(\tilde{u})),
\end{equation}
where ${\sf Div}_{X,Y,Z}$ is the Magnus divergence~\eqref{eq:Mag_div} with respect to the variables $X_i, Y_i$ and $Z_j$.
Using Proposition~\ref{prop:c_change_variables}, one also obtains another expression
\begin{equation} \label{eq:c^f_details2}
{\sf Div}^f_{\alpha, \beta, \gamma}(\tilde{u}) = 
{\sf Div}_{x,y,z}(u) +
{\bf 1} \wedge {\sf b}^f(\tilde{u}) + u ( \tilde{\Delta}({\bf r} - {\bf p}^f) ).
\end{equation}
Here, ${\sf Div}_{x,y,z}$ is the divergence map with respect to the exponential generators $x_i, y_i$ and $z_j$ and the element ${\bf r} \in |\widehat{\K \pi}|$ is defined by 
\begin{equation} \label{dfn:bfr}
{\bf r} := \sum_{i=1}^g | r(x_i) + r(y_i) |.
\end{equation} 
To obtain \eqref{eq:c^f_details2} from \eqref{eq:c^f_definition}, one uses the fact that ${\bf p}^f$ is in the span of $|x_i|$ and $|y_i|$ which implies that ${\bf 1} \wedge {\bf p}^f = - \tilde{\Delta}({\bf p}^f)$ and the fact that the $1$-coboundaries $u \mapsto u (\tilde{\Delta}(|r(z_j)|))$ vanish on ${\rm tDer}(\widehat{\K \pi})$.

\begin{rem}
When $g=0$, we drop the symbols $x,y$ in ${\sf Div}_{x,y,z}$ and simply denote ${\sf Div}_{z}$.
When $n = 0$, we denote ${\sf Div}_{x,y}$ instead of ${\sf Div}_{x,y,z}$.
We will use the same style of abbreviation when it is applicable.
\end{rem}

\subsection{The map ${\rm Ham}^\kappa$}

Recall that the completed group algebra $\widehat{\mathbb{K} \pi}$ carries a canonical double bracket $\kappa$ described in Section~\ref{subsec:bra}.
This double bracket induces a map 
$$
{\rm Ham}^\kappa \colon \widehat{\K \pi} \to {\rm Der}^{s-{\rm lin}}(\widehat{\K \pi}\langle s \rangle)
$$
introduced in Section~\ref{subsec:mulprop}.
We can compute values of the map ${\rm Ham}^\kappa$ on the standard generators $\alpha_i, \beta_i, \gamma_j$ by using Proposition~\ref{prop:kapvalue}.
For instance, if $k < i$, 
\[
{\rm Ham}^\kappa_{\alpha_i}(\alpha_k) = 
\alpha_k s \alpha_i + \alpha_i s \alpha_k - \alpha_k \alpha_i s - s \alpha_i \alpha_k 
= [\alpha_k, [s, \alpha_i]]
\]
and ${\rm Ham}^\kappa_{\alpha_i}(\alpha_i) = \alpha_i s\alpha_i - s {\alpha_i}^2 = [\alpha_i, s \alpha_i]$.
Besides, ${\rm Ham}^{\kappa}_{\alpha}(\beta) = 0$ if $\kappa(\alpha, \beta) = 0$.

The nontrivial values of the map ${\rm Ham}^{\kappa}_{\alpha_i}$ on the standard generators are as follows:
\begin{equation} \label{eq:Hamalpha}
\begin{array}{llll}
{\rm Ham}^\kappa_{\alpha_i}(\alpha_k) & = & 
[\alpha_k, [s,\alpha_i]] & {\rm for} \,\, k < i, \\
{\rm Ham}^\kappa_{\alpha_i}(\beta_k) & = & 
[\beta_k, [s,\alpha_i]]
& {\rm for} \,\, k < i, \\
{\rm Ham}^\kappa_{\alpha_i}(\alpha_i) & = & 
[\alpha_i, s\alpha_i], & \\
{\rm Ham}^\kappa_{\alpha_i}(\beta_i) & = &
\beta_i s \alpha_i. & 
\end{array}
\end{equation}
For ${\rm Ham}^\kappa_{\beta_i}$, the nontrivial values are
\begin{equation} \label{eq:Hambeta}
\begin{array}{llll}
{\rm Ham}^\kappa_{\beta_i}(\alpha_k) & = &
[\alpha_k, [s,\beta_i]] & {\rm for} \,\, k< i, \\
{\rm Ham}^\kappa_{\beta_i}(\beta_k) & = &
[\beta_k, [s,\beta_i]] & {\rm for} \,\, k< i, \\
{\rm Ham}^\kappa_{\beta_i}(\alpha_i) & = & 
\beta_i s \alpha_i - \alpha_i \beta_i s - s \beta_i \alpha_i, & \\
{\rm Ham}^\kappa_{\beta_i}(\beta_i) & = &
[\beta_i, -\beta_i s]. & \\
\end{array}
\end{equation}
Finally the nontrivial values of ${\rm Ham}^\kappa_{\gamma_j}$ are
\begin{equation} \label{eq:Hamgamma}
\begin{array}{llll}
{\rm Ham}^\kappa_{\gamma_j}(\alpha_i) & = &
[\alpha_i, [s,\gamma_j]] & {\rm for} \,\, {\rm all} \,\, i, \\
{\rm Ham}^\kappa_{\gamma_j}(\beta_i) & = &
[\beta_i, [s,\gamma_j]] & {\rm for} \,\, {\rm all} \,\, i, \\
{\rm Ham}^\kappa_{\gamma_j}(\gamma_k) & = &
[\gamma_k, [s,\gamma_j]] &
{\rm for} \,\, k < j, \\
{\rm Ham}^\kappa_{\gamma_j}(\gamma_j) & = & 
[\gamma_j, s\gamma_j]. & \\
\end{array}
\end{equation}

Let $\mathcal{T} = \{ z_1, \ldots, z_n \}$, and consider the Lie algebras ${\rm tDer}_{\mathcal{T}}(\widehat{\K \pi})$ and ${\rm tDer}^{s-{\rm lin}}_{\mathcal{T}}(\widehat{\K \pi}\langle s \rangle)$ following constructions in Sections~\ref{subse:tngder} and \ref{subsec:spectator}.
With a more intrinsic description in Section~\ref{subsec:tdersurface} in mind, we will drop the symbol $\mathcal{T}$ in the notation and simply denote ${\rm tDer}(\widehat{\K \pi})$ and ${\rm tDer}^{s-{\rm lin}}(\widehat{\K \pi} \langle s \rangle)$, respectively.

The double bracket $\kappa$ has the following special property:

\begin{prop} \label{prop:lifthk}
The double bracket $\kappa$ is tangential to $\{ z_1, \ldots, z_n \}$.
In other words, the map ${\rm Ham}^{\kappa}$ admits a lift $\widetilde{{\rm Ham}^{\kappa}} \colon \widehat{\K \pi} \to {\rm tDer}^{s-{\rm lin}}(\widehat{\K \pi}\langle s \rangle)$.
\end{prop}
\begin{proof}
According to Lemma~\ref{lem:lift_Ham} (ii), it is sufficient to prove that for any $\delta \in \{ \alpha_i, \beta_i, \gamma_j\}$, there exist elements $u_j(\delta, s) \in \widehat{\K \pi}\langle s \rangle$ linear in $s$ such that ${\rm Ham}^{\kappa}_{\delta}(z_j) = [z_j, u_j(\delta,s)]$.

Since ${\rm Ham}^\kappa_{\alpha_i}(\gamma_j)={\rm Ham}^\kappa_{\beta_i}(\gamma_j)=0$, this condition is verified on the generators $\alpha_i$ and $\beta_i$ (with $u_j(\alpha_i,s) = u_j(\beta_i,s) = 0$).
For the generators $\gamma_j$, by \eqref{eq:Hamgamma} we have for $k < j$
$$
{\rm Ham}^\kappa_{\gamma_j}(\gamma_k)=[\gamma_k, [s, \gamma_j]], \hskip 0.3cm
{\rm Ham}^\gamma_{\gamma_j}(z_k)=[z_k, [s, \gamma_j]],
$$
for $k=j$, we obtain
$$
{\rm Ham}^\kappa_{\gamma_j}(\gamma_j)=[\gamma_j, s\gamma_j], \hskip 0.3cm
{\rm Ham}^\kappa_{\gamma_j}(z_j)=[z_j, s\gamma_j], 
$$
and for $k>j$, ${\rm Ham}^{\kappa}_{\gamma_j}(\gamma_k) = {\rm Ham}^{\kappa}_{\gamma_j}(z_k) = 0$.
This completes the proof.
\end{proof}

In the next subsection, we use the following calculation on derivations in the image of the maps ${\rm Ham}^{\kappa}$ and $\widetilde{{\rm Ham}^{\kappa}}$.
\begin{lem} \label{lem:Ham_computation_trace}
\begin{enumerate}
\item[(i)]
The action of the derivations ${\rm Ham}^{\kappa}_{\alpha_i}, {\rm Ham}^{\kappa}_{\beta_i},{\rm Ham}^{\kappa}_{\gamma_j}$ on the elements $|x_i|, |y_i|, |z_j|$, which form a basis of $\Lambda \otimes_{\mathbb{Z}} \K \subset |\widehat{\K \pi}|$, is given as follows: 
\[
{\rm Ham}^{\kappa}_{\alpha_i}(|y_i|) = |s\alpha_i|, 
\quad {\rm Ham}^{\kappa}_{\beta_i}(|x_i|) = -|s\beta_i|,
\]
and all the other values are trivial.
\item[(ii)]
The values of the $1$-cocycles ${\sf b}_1, \ldots, {\sf b}_n$ on the tangential lifts $\widetilde{{\rm Ham}^{\kappa}_{\alpha_i}}, \widetilde{{\rm Ham}^{\kappa}_{\beta_i}}, \widetilde{{\rm Ham}^{\kappa}_{\gamma_j}}$ are given as follows:
\[
{\sf b}_j(\widetilde{{\rm Ham}^{\kappa}_{\gamma_j}}) = |s\gamma_j|,
\]
and all the other values are trivial.
\end{enumerate}
\end{lem}

\begin{proof}
(i) First we consider ${\rm Ham}^{\kappa}_{\alpha_i}$.
For the value ${\rm Ham}_{\alpha_i}^{\kappa}(|y_i|)$, we compute as follows: 
\[
{\rm Ham}_{\alpha_i}^{\kappa}(|y_i|) = {\rm Ham}_{\alpha_i}^{\kappa}(|\log \beta_i|) = | {\rm Ham}_{\alpha_i}^{\kappa}(\beta_i) \beta_i^{-1} | = |\beta_i s \alpha_i \beta_i^{-1}| = |s\alpha_i|.
\]
Here we have used Example \ref{ex:ulog} in the second equality.
By \eqref{eq:Hamalpha}, we have ${\rm Ham}^{\kappa}_{\alpha_i}(|x_k|) = |[x_k, [s,\alpha_i]]| = 0$ for $k<i$, and similarly we can compute all the other values which are trivial.
The calculations for ${\rm Ham}^{\kappa}_{\beta_i}$ and ${\rm Ham}^{\kappa}_{\gamma_j}$ are similar. 

(ii) Since ${\rm Ham}^{\kappa}_{\alpha_i}(\gamma_j) = 0$, the value ${\sf b}_j(\widetilde{{\rm Ham}^{\kappa}_{\alpha_i}})$ vanishes.
Similarly, ${\sf b}_j(\widetilde{{\rm Ham}^{\kappa}_{\beta_i}})=0$.
For $\widetilde{{\rm Ham}^{\kappa}_{\gamma_j}}$, the proof of Proposition~\ref{prop:lifthk} yields the following: ${\sf b}_k(\widetilde{{\rm Ham}^\kappa_{\gamma_j}}) = |[s, \gamma_j]|=0
$ for $k < j$, ${\sf b}_j(\widetilde{{\rm Ham}^\kappa_{\gamma_j}})=|s\gamma_j|$, and
${\sf b}_k(\widetilde{{\rm Ham}^\kappa_{\gamma_j}}) = 0$ for $k > j$.
This completes the proof.
\end{proof}

We show that the composition of $\widetilde{{\rm Ham}^{\kappa}}$ with the $s=1$ evaluation map becomes a Lie algebra homomorphism.

\begin{prop} \label{prop:sigmatilde}
The map $\sigma \colon |\widehat{\K \pi}| \to {\rm Der}(\widehat{\K \pi})$ lifts to a Lie algebra homomorphism
\[
\tilde{\sigma}\colon |\widehat{\K \pi}| \to {\rm tDer}(\widehat{\K \pi}).
\]
\end{prop}
\begin{proof}
By Lemma~\ref{lem:lift_Ham}~(iii), $\sigma$ lifts to the map 
\[
\tilde{\sigma} \colon |\widehat{\K \pi}| \to {\rm tDer}(\widehat{\K \pi}),
\quad
|a| \mapsto \widetilde{{\rm Ham}^{\kappa}_a}|_{s=1}.
\]

To show that $\tilde{\sigma}$ is a Lie homomorphism, 
let $a, b \in \widehat{\K \pi}$ and write ${\rm Ham}^{\kappa}_a(z_j) = [z_j, u'_j s u''_j]$ and ${\rm Ham}^{\kappa}_b(z_j) = [z_j, v'_j s v''_j]$.
Then the $j$th tangential components of $\tilde{\sigma}(|a|)$ and $\tilde{\sigma}(|b|)$ are $u'_ju''_j$ and $v'_jv''_j$, respectively.
Since $[|a|, |b|] = | \sigma(|a|)b|$, the $j$th tangential component of $\tilde{\sigma}([|a|, |b|])$ can be seen from the bracket $\{ |s \sigma(|a|)(b)|, z_j \}$.
Using Proposition~\ref{prop:ksabc}, we compute
\begin{align*}
  \{ |s \sigma(|a|)b|, z_j \} & = \kappa'(\sigma(|a|)(b),z_j)s\kappa''(\sigma(|a|)(b),z_j) \\
  & = \sigma(|a|) (\kappa'(b,z_j)s \kappa''(b,z_j))
  - \kappa'(b, \sigma(|a|)(z_j))s \kappa''(b, \sigma(|a|)(z_j)) \\
  & = \sigma(|a|)( [z_j, v'_jsv''_j])
  - \sigma(|sb|)([z_j, u'_ju''_j]) \\
  & = [ [z_j, u'_ju''_j], v'_jsv''_j] + [z_j, \sigma(|a|)(v'_jsv''_j)] \\
  & \hspace{1em} - [[z_j, v'_jsv''_j], u'_ju''_j] - [z_j, \sigma(|sb|)(u'_ju''_j)] .
\end{align*}
By the Jacobi identity, the first and third term combined yields $[z_j, [u'_ju''_j, v'_jsv''_j]]$.
Thus the tangential component of $\widetilde{{\rm Ham}^{\kappa}}_{\hspace{-0.4em} \sigma(|a|)b} \in {\rm tDer}^{s-{\rm lin}}(\widehat{\K \pi}\langle s \rangle)$ is equal to 
\[
[u'_ju''_j, v'_jsv''_j] + \sigma(|a|)(v'_jsv''_j) - \sigma(|sb|)(u'_ju''_j).
\]
By putting $s=1$, the $j$th tangential component of $\tilde{\sigma}([|a|, |b|])$ is equal to
\[
[u'_ju''_j, v'_jv''_j] + \sigma(|a|)(v'_jv''_j) - \sigma(|b|)(u'_ju''_j),
\]
which coincides with the $j$th tangential component of $[\tilde{\sigma}(|a|), \tilde{\sigma}(|b|)]$.
This proves that $\tilde{\sigma}$ is a Lie homomorphism.
\end{proof}

Propositions~\ref{prop:lifthk} and \ref{prop:sigmatilde} summarize in the following commutative diagram:
\begin{equation} \label{rem:tangential_lift}
\xymatrix{
\widehat{\K \pi}
\ar[r]_(0.3){{\rm Ham}^{\kappa}} \ar[d]_{|\cdot|}
\ar@/^/@<2ex>[rr]^{\widetilde{{\rm Ham}^{\kappa}}} & {\rm Der}^{s-{\rm lin}}(\widehat{\K \pi}\langle s \rangle)
\ar[d]^{s=1}
& {\rm tDer}^{s-{\rm lin}}
(\widehat{\K \pi}\langle s \rangle)
\ar[l] \ar[d]_{s=1}
\\
|\widehat{\K \pi}|
\ar[r]^(0.4){\sigma}
\ar@/_/@<-2ex>[rr]_{\tilde{\sigma}}
& {\rm Der}(\widehat{\K\pi})
& {\rm tDer}(\widehat{\K \pi})
\ar[l]
}
\end{equation}
By taking the associated graded of Proposition \ref{prop:lifthk}, we see that the double bracket $\kappa_{\rm gr}$ is tangential to $\{ z_1, \ldots, z_n \} \subset A = {\rm gr}^{\rm wt}\, \K \pi$ and induces tangential lifts
\[
\widetilde{{\rm Ham}^{\kappa_{\rm gr}}} \colon A \to 
{\rm tDer}^{s-{\rm lin}}(A\langle s \rangle), \qquad
\tilde{\sigma}_{\rm gr} \colon |A| \to {\rm tDer}_{\{z_1,\ldots,z_n\}}(A).
\]
Also, by taking the associated graded of the diagram \eqref{rem:tangential_lift}, we obtain the following commutative diagram:
\[
\xymatrix{
A
\ar[r]_(0.3){{\rm Ham}^{\kappa_{\rm gr}}} \ar[d]_{|\cdot|}
\ar@/^/@<2ex>[rr]^{\widetilde{{\rm Ham}^{\kappa_{\rm gr}}}} & {\rm Der}^{s-{\rm lin}}(A\langle s \rangle)
\ar[d]^{s=1}
& {\rm tDer}^{s-{\rm lin}}
(A\langle s \rangle)
\ar[l] \ar[d]_{s=1}
\\
|A|
\ar[r]^(0.4){\sigma_{\rm gr}}
\ar@/_/@<-2ex>[rr]_{\tilde{\sigma}_{\rm gr}}
& {\rm Der}(A)
& {\rm tDer}_{\{z_1,\ldots,z_n\}}(A)
\ar[l]
}
\]

\subsection{Maps $\mu^f$ from the cocycle ${\sf Div}^f$}

In this subsection, we establish one of the main results of this paper: we show that the maps $\mu^f_r$ and $\mu^f_l$ determined by topology of self-intersections of curves on the framed surface can be represented as compositions of the cocycle ${\sf Div}^f_{\alpha, \beta, \gamma}$ and of the map $\widetilde{{\rm Ham}^\kappa}$ determined by intersections of pairs of curves.

We will abbreviate ${\sf Div}^f_{\alpha, \beta, \gamma}$ by ${\sf Div}^f$.
(As we will see soon in Theorem \ref{thm:Div^f}, the cocycle ${\sf Div}^f_{\alpha, \beta, \gamma}$ does not depend on the choice of standard generating system for $\pi$. Thus this abbreviation will be justified.)
Theorem~\ref{thm:c_l_r} applies to the $1$-cocycle ${\sf Div}^f$.
Thus it naturally extends to the 1-cocycle ${\rm tDer}(\widehat{\K \pi}\langle s \rangle) \to |\widehat{\K \pi}\langle s \rangle| \otimes |\widehat{\K \pi}\langle s \rangle|$ (which we also denote by ${\sf Div}^f$) and gives rise to maps
$$
{\sf Div}^f_r \colon {\rm tDer}^{s-{\rm lin}}(\widehat{\K \pi}\langle s \rangle) \to |\widehat{\K \pi}| \otimes \widehat{\K \pi},
\hskip 0.3cm
{\sf Div}^f_l \colon {\rm tDer}^{s-{\rm lin}}(\widehat{\K \pi}\langle s \rangle) \to \widehat{\K \pi} \otimes |\widehat{\K\pi}|.
$$

The following theorem expresses the coaction maps $\mu^f_r$ and $\mu^f_l$ as compositions of maps ${\sf Div}^f_r$ and ${\sf Div}^f_l$ induced by the cocycle ${\sf Div}^f$ and of the map $\widetilde{{\rm Ham}}^\kappa$ induced by the double bracket $\kappa$.

\begin{thm}   \label{thm:mu=ch}
We have,
\begin{equation} \label{eq:mu=ch}
\mu^f_r={\sf Div}^f_r \circ \widetilde{{\rm Ham}^\kappa},
\hspace{2em} 
\mu^f_l = {\sf Div}^f_l \circ \widetilde{{\rm Ham}^\kappa}.
\end{equation}
\end{thm}

\begin{proof}
First observe that the 1-cocycle ${\sf Div}^f$ is cohomologous to the sum of the divergence cocycle and a linear combination of cocycles ${\bf 1} \wedge {\sf b}_j$. Hence, by Theorem~\ref{thm:c_l_r}, the maps ${\sf Div}^f_r$ and ${\sf Div}^f_l$ satisfy equations \eqref{eq:c_l_r_long}.
Then, by Proposition~\ref{prop:mrab}, the maps
$$
m^f_r={\sf Div}^f_r \circ \widetilde{{\rm Ham}^\kappa},
\hspace{2em}
m^f_l={\sf Div}^f_l \circ \widetilde{{\rm Ham}^\kappa}
$$
satisfy the multiplicative properties~\eqref{eq:multiplicative} with $\Pi=\kappa$. Since the maps $\mu^f_r$ and $\mu^f_l$ have exactly the same multiplicative properties 
(see Proposition \ref{prop:prodmu}), it is sufficient to verify equalities
$\mu^f_r=m^f_r$ and $\mu^f_l=m^f_l$ on generators of $\pi$. The values of $\mu^f_r$ and $\mu^f_l$ on generators were computed in Theorem \ref{thm:values_mu^f}. So, it remains to compute the values of $m^f_r$ and $m^f_l$ on the generators of $\pi$.
From equation \eqref{eq:c^f_details}, we obtain the following: for any $\gamma \in \pi$,
\begin{align}
m^f_r(\gamma) & =  
{\sf Div}_r({\rm Ham}_{\gamma}^{\kappa})
- {\bf 1} \otimes {\rm Ham}_{\gamma}^{\kappa}(\sum_i |x_i| + |y_i|) + {\bf 1} \otimes {\rm Ham}_{\gamma}^{\kappa}({\bf p}^f) + {\bf 1} \otimes {\sf b}^f(\widetilde{{\rm Ham}_{\gamma}^{\kappa}}),  \nonumber \\
m^f_l(\gamma) & = 
{\sf Div}_l({\rm Ham}_{\gamma}^{\kappa}) - {\rm Ham}_{\gamma}^{\kappa}({\bf p}^f) \otimes {\bf 1} - {\sf b}^f(\widetilde{{\rm Ham}_{\gamma}^{\kappa}}) \otimes {\bf 1}.
\label{eq:m^f_explicit}
\end{align}
Here ${\sf Div}_r$ and ${\sf Div}_l$ are the components of the Magnus divergence ${\sf Div} = {\sf Div}_{X,Y,Z}$.
Also, we use the identification $|s \widehat{\K\pi}| \cong \widehat{\K\pi}$ to regard ${\rm Ham}^{\kappa}_{\gamma}(|x_i|), {\rm Ham}^{\kappa}_{\gamma}(|y_i|)$ and ${\sf b}^f(\widetilde{{\rm Ham}_{\gamma}^{\kappa}})$ as elements in $\widehat{\K \pi}$.

Using formulas \eqref{eq:Hamalpha} for ${\rm Ham}^\kappa_{\alpha_i}$, we compute 
%
$$
{\sf Div}({\rm Ham}^\kappa_{\alpha_i}) = 
| \partial_{X_i} {\rm Ham}_{\alpha_i}^{\kappa}(\alpha_i) + \partial_{Y_i} {\rm Ham}_{\alpha_i}^{\kappa}(\beta_i) |
= 2({\bf 1} \otimes |s\alpha_i|) - |s| \otimes |\alpha_i|
$$
which yields
$$
{\sf Div}_r({\rm Ham}^\kappa_{\alpha_i})=2({\bf 1} \otimes \alpha_i), \hskip 0.3cm
{\sf Div}_l({\rm Ham}^\kappa_{\alpha_i}) = - 1 \otimes |\alpha_i|.
$$
By Lemma~\ref{lem:Ham_computation_trace}, we see that the derivation ${\rm Ham}^{\kappa}_{\alpha_i}$ sends $|y_i|$ to $|s\alpha_i| \in |s \widehat{\K \pi}|$ (which corresponds to $\alpha_i \in \widehat{\K \pi}$) and all the other elements in the basis of $\Lambda \otimes_{\mathbb{Z}} \K$ to zero.
Also ${\sf b}_j(\widetilde{{\rm Ham}^{\kappa}_{\alpha_i}}) = 0$.
By \eqref{eq:m^f_explicit} we conclude that 
\begin{align*}
m_r^f(\alpha_i) & = {\sf Div}_r({\rm Ham}^\kappa_{\alpha_i}) + ({\rm rot}^f(\alpha_i) -1) ({\bf 1} \otimes {\rm Ham}^\kappa_{\alpha_i}|y_i|)  \\
& = 2({\bf 1} \otimes \alpha_i) +({\rm rot}^f(\alpha_i) -1) ({\bf 1} \otimes \alpha_i) \\
& = ({\rm rot}^f(\alpha_i) +1) ({\bf 1} \otimes \alpha_i)
\end{align*}
which coincides with the value of $\mu^f_r(\alpha_i)$.
In a similar fashion, we compute
\[
m^f_l(\alpha_i) = {\sf Div}_l({\rm Ham}^\kappa_{\alpha_i}) - {\rm rot}^f(\alpha_i)({\rm Ham}^\kappa_{\alpha_i} |y_i| \otimes {\bf 1})
= - 1 \otimes |\alpha_i| - {\rm rot}^f(\alpha_1)(\alpha_i \otimes {\bf 1})
\]
which coincides with the value of $\mu^f_l(\alpha_i)$.

For ${\rm Ham}^\kappa_{\beta_i}$ the
computation of the Magnus divergence yields
$$
{\sf Div}({\rm Ham}^\kappa_{\beta_i}) = 
|\partial_{X_i} {\rm Ham}^{\kappa}_{\beta_i}(\alpha_i) + \partial_{Y_i} {\rm Ham}^{\kappa}_{\beta_i}(\beta_i) | = 
|\beta_i s| \otimes {\bf 1} - {\bf 1} \otimes |\beta_i s| - |\beta_i| \otimes |s|
$$
and
$$
{\sf Div}_r({\rm Ham}^\kappa_{\beta_i})= - {\bf 1} \otimes \beta_i - |\beta_i| \otimes 1, \hskip 0.3cm
{\sf Div}_l({\rm Ham}^\kappa_{\beta_i})= \beta_i \otimes {\bf 1}.
$$
Again, by using Lemma~\ref{lem:Ham_computation_trace}, we conclude that
\[
m^f_r(\beta_i) = 
{\sf Div}_r({\rm Ham}^\kappa_{\beta_i}) - ({\rm rot}^f(\beta_i) + 1)({\bf 1} \otimes {\rm Ham}^\kappa_{\beta_i}|x_i|)
= {\rm rot}^f(\beta_i)({\bf 1} \otimes \beta_i) - |\beta_i| \otimes 1,
\]
and this expression coincides with the value of $\mu^f_r(\beta_i)$.
Similarly we compute
\[
m^f_l(\beta_i) = 
{\sf Div}_l({\rm Ham}^\kappa_{\beta_i}) + {\rm rot}^f(\beta_i)({\rm Ham}^\kappa_{\beta_i}|x_i| \otimes {\bf 1}) 
= (1 - {\rm rot}^f(\beta_i))(\beta_i \otimes {\bf 1})
\]
which coincides with the value of $\mu^f_l(\beta_i)$.

Finally the divergence of ${\rm Ham}^\kappa_{\gamma_j}$ takes the form
$$
{\sf Div}({\rm Ham}^\kappa_{\gamma_j}) = |\partial_{Z_j} {\rm Ham}^{\kappa}_{\gamma_j}(\gamma_j)| = {\bf 1} \otimes |s\gamma_j| - |s| \otimes |\gamma_j|,
$$
and
$$
{\sf Div}_r({\rm Ham}^\kappa_{\gamma_j})={\bf 1} \otimes \gamma_j, \hskip 0.3cm
{\sf Div}_l({\rm Ham}^\kappa_{\gamma_j})= - 1\otimes |\gamma_j|.
$$
Using Lemma~\ref{lem:Ham_computation_trace} again, we conclude that
\[
m^f_r(\gamma_j) = {\sf Div}_r({\rm Ham}^\kappa_{\gamma_j}) + {\rm rot}^f(\gamma_j)({\bf 1} \otimes {\sf b}_j(\widetilde{{\rm Ham}^\kappa_{\gamma_j}}))
= ({\rm rot}^f(\gamma_j) + 1)({\bf 1} \otimes \gamma_j).
\]
This expression coincides with $\mu^f_r(\gamma_j)$. In a similar way, we have
\[
m^f_l(\gamma_j) = {\sf Div}_l({\rm Ham}^\kappa_{\gamma_j}) - {\rm rot}^f(\gamma_j)({\sf b}_j(\widetilde{{\rm Ham}^\kappa_{\gamma_j}})\otimes {\bf 1})
= - {\rm rot}^f(\gamma_j)(\gamma_j \otimes {\bf 1}) - 1 \otimes |\gamma_j|,
\]
which coincides with the value of $\mu^f_l(\gamma_j)$.
\end{proof}

The next result establishes independence of the cocycle ${\sf Div}^f_{\alpha, \beta, \gamma}$ of the standard generating system.

\begin{thm}      \label{thm:Div^f}
For all standard generating systems of $\pi$, the cocycles
$$
{\sf Div}^f_{\alpha, \beta, \gamma} \colon {\rm tDer}(\widehat{\K \pi}) \to |\widehat{\K \pi}| \otimes |\widehat{\K \pi}|
$$
coincide with each other and define a canonical cocycle ${\sf Div}^f$ which depends only on the framing $f$.
\end{thm}

\begin{proof}
By Theorem \ref{thm:div_independent}, the cocycles ${\sf Div}_{\alpha, \beta,\gamma}$ for different generating systems differ by exact cocycles of the form
$$
u(\tilde{\Delta}(\lambda)) = u(\lambda) \otimes {\bf 1} - {\bf 1} \otimes u(\lambda),
$$
where 
$$
\lambda \in \Lambda = \mathbb{Z}\langle |x_i|, |y_i|, |z_j|; i=1,\dots,g, j=1, \dots, n \rangle.
$$
Together with the defining formula \eqref{eq:c^f_definition} of ${\sf Div}^f_{\alpha, \beta, \gamma}$, 
this shows that for two generating system $\{ \alpha_i, \beta_i, \gamma_j \}$ and $\{ \alpha'_i, \beta'_i, \gamma'_j \}$, we have
$$
\mathfrak{b}(\tilde{u})={\sf Div}^f_{\alpha', \beta',\gamma'}(\tilde{u}) - {\sf Div}^f_{\alpha, \beta, \gamma}(\tilde{u}) =
{\bf 1} \wedge (u \sum_i (k_i |x_i| + l_i |y_i|) + \sum_j(m_j u(|z_j|) + n_j {\sf b}_j(\tilde{u}))).
$$
Here $k_i, l_i, m_j, n_j \in \mathbb{Z}$ are integers. Recall that for $\tilde{u} \in {\rm tDer}(\widehat{\K \pi})$ we have $u(|z_j|)=|[z_j, u_j]|=0$. 
Hence $\sum_j m_j u(|z_j|) =0$, and the corresponding term in $\mathfrak{b}(\tilde{u})$ can be dropped.

By Theorem \ref{thm:mu=ch}, we have 
$\mathfrak{b}_r \circ \widetilde{{\rm Ham}^\kappa} = \mu^f_r - \mu^f_r=0$.
Hence, using Lemma~\ref{lem:Ham_computation_trace}, we compute
$$
0=(\mathfrak{b}_r \circ \widetilde{{\rm Ham}^\kappa})(\sum_i (\alpha_i + \beta_i) + \sum_j \gamma_j) = 
{\bf 1} \otimes (\sum_i (l_i \alpha_i - k_i \beta_i) + \sum_j n_j \gamma_j).
$$
The right hand side vanishes if and only if $k_i=l_i=n_j=0$ for all $i$ and $j$. Hence, all cocycles ${\sf Div}^f_{\alpha, \beta, \gamma}$ coincide, as required.
\end{proof}

\subsection{Further properties of the cocycle ${\sf Div}^f$} \label{subsec:furtherc_f}

In this subsection, we discuss further properties of the cocycle ${\sf Div}^f$, and in particular its relation to the Turaev cobracket and its graded version.

First we note that the cocycle ${\sf Div}^f$ admits a unique integration to the group 1-cocycle
\[
{\sf J}^f \colon {\rm tAut}^+(\widehat{\mathbb{K}\pi}) \to 
|\widehat{\mathbb{K}\pi}| \otimes |\widehat{\mathbb{K}\pi}|,
\quad
\tilde{F} \mapsto {\sf J}_{x,y,z}(F) + {\bf 1} \wedge {\sf c}^f(\tilde{F}) + F( \tilde{\Delta}({\bf r} - {\bf p}^f)) - \tilde{\Delta}({\bf r} - {\bf p}^f).
\]
Here, ${\sf J}_{x,y,z}$ is the non-commutative log-Jacobian cocycle in Definition~\ref{dfn:J} determined by the exponential variables $\{ x_i, y_i, z_j\}$, and ${\sf c}^f$ is the integration of the Lie $1$-cocycle ${\sf b}^f$:
\[
{\sf c}^f \colon {\rm tAut}^+(\widehat{\K \pi}, \Delta) \to |\widehat{\K\pi}|,
\quad \tilde{F} \mapsto \sum_j {\rm rot}^f(\gamma_j) {\sf c}_j(\tilde{F}) = \sum_j {\rm rot}^f(\gamma_j) |\log f_j|.
\]

Next we consider restrictions of ${\sf Div}^f$ to tangential Hopf derivations of $\widehat{\mathbb{K}\pi}$.

\begin{prop} \label{prop:div^f}
There is a unique Lie algebra 1-cocycle ${\sf div}^f \colon {\rm tDer}(\widehat{\mathbb{K}\pi}, \Delta) \to |\widehat{\mathbb{K}\pi}|$ such that ${\sf Div}^f = \tilde{\Delta} \circ {\sf div}^f$. It is given by formula
\[
{\sf div}^f(\tilde{u}) = {\sf div}_{\alpha,\beta,\gamma}(u) - u({\bf p}^f) - {\sf b}^f(\tilde{u})
= {\sf div}_{x,y,z} - {\sf b}^f(\tilde{u}) + u({\bf r} - {\bf p}^f).
\]
Here ${\sf div}_{\alpha, \beta, \gamma}$ is the Lie algebra $1$-cocycle in Proposition~\ref{prop:c_b} determined by the standard generating system $\{\alpha_i, \beta_i, \gamma_j\}$, and ${\sf div}_{x,y,z}$ is the map in \eqref{eq:div_z} determined by the exponential variables $\{ x_i, y_i, z_j \}$.
\end{prop}

\begin{proof}
Recall from \eqref{eq:c^f_definition} that 
$$
{\sf Div}^f(\tilde{u}) = {\sf Div}_{\alpha, \beta, \gamma}(u) + {\bf 1} \wedge (u({\bf p}^f) + {\sf b}^f(\tilde{u}))
$$
and observe that for $\tilde{u} \in {\rm tDer}(\widehat{\mathbb{K}\pi}, \Delta)$, we have 
$$
{\bf 1} \wedge (u({\bf p}^f) + {\sf b}^f(\tilde{u})) = - \tilde{\Delta}(u({\bf p}^f) + {\sf b}^f(\tilde{u})).
$$
Furthermore, by Proposition \ref{prop:c_b}, we have
${\sf Div}_{\alpha, \beta, \gamma}(u) = \tilde{\Delta}({\sf div}_{\alpha,\beta,\gamma}(u))$.
Hence ${\sf Div}^f(\tilde{u})=\tilde{\Delta}({\sf div}^f(\tilde{u}))$, where
$$
{\sf div}^f(\tilde{u}) = {\sf div}_{\alpha,\beta,\gamma}(u) - u({\bf p}^f) - {\sf b}^f(\tilde{u}).
$$
Similarly the second expression of ${\sf div}^f(\tilde{u})$ follows from \eqref{eq:c^f_details2}.
The uniqueness and the cocycle property of ${\sf div}^f$ follow from the injectivity of $\tilde{\Delta}$. The cocycle ${\sf div}^f$ depends only on the framing $f$ since so does the cocycle ${\sf Div}^f$.
\end{proof}

Note that the cocycle ${\sf div}^f$ admits a unique integration to a group 1-cocycle 
\[
{\sf j}^f \colon {\rm tAut}^+(\widehat{\mathbb{K}\pi}, \Delta) \to 
|\widehat{\mathbb{K}\pi}|,
\quad
\tilde{F} \mapsto {\sf j}_{x,y,z}(F) - {\sf c}^f(\tilde{F}) + F({\bf r} - {\bf p}^f) - ({\bf r} - {\bf p}^f), 
\]
and for $F\in {\rm tAut}^+(\widehat{\mathbb{K}\pi}, \Delta)$ we have ${\sf J}^f(F) = \tilde{\Delta}({\sf j}^f(F))$.
Here, ${\sf j}_{x,y,z}$ is the group $1$-cocycle in Proposition~\ref{prop:J_and_j} determined by the variables $\{x_i, y_i, z_j\}$.

Recall that the map $\sigma \colon |\widehat{\mathbb{K} \pi}| \to {\rm Der}(\widehat{\mathbb{K} \pi})$ admits a canonical tangential lift $\tilde{\sigma}$
(see Proposition~\ref{prop:sigmatilde}).

\begin{thm} \label{prop:delta=c_sigma}
We have
\begin{equation} \label{eq:delta=c_sigma}
    \delta^f_{\rm Turaev}= {\sf Div}^f \circ \tilde{\sigma}, 
\end{equation}
where $\tilde{\sigma} \colon |\widehat{\mathbb{K} \pi}| \to {\rm tDer}(\widehat{\mathbb{K} \pi})$ is the map induced by the double bracket $\kappa$.
\end{thm}

\begin{proof}
By Proposition \ref{prop:prodmu} (ii), the framed version of the Turaev cobracket is given by formula
$$
\delta^f_{\rm Turaev}(|a|) = 
({\rm id} \otimes |\cdot|) \mu^f_r(a) +(|\cdot| \otimes {\rm id}) \mu^f_l(a),
$$
where $a \in \widehat{\mathbb{K} \pi}$.
Using Theorem \ref{thm:mu=ch}, we compute
$$
\delta^f_{\rm Turaev}(|a|) =
\big( ({\rm id} \otimes |\cdot|) {\sf Div}^f_r + (|\cdot| \otimes {\rm id}) {\sf Div}^f_l \big) \circ \widetilde{{\rm Ham}^{\kappa}_a} 
= ({\sf Div}^f \circ \tilde{\sigma})(|a|).
$$
Here we have used the following facts: we recall that $\widetilde{{\rm Ham}^{\kappa}_a}$ is a tangential derivation linear in $s$. Then ${\sf Div}^f(\widetilde{{\rm Ham}^{\kappa}_a})$ takes values in 
$$
\big( |\widehat{\K \pi}| \otimes |s \widehat{\K \pi}| \big) \oplus
\big( |s \widehat{\K \pi}| \otimes |\widehat{\K \pi}| \big)
.
$$
By definition, we identify $|s \widehat{\K \pi}| \cong \widehat{\K \pi}$.
This gives rise to two components of the cocycle ${\sf Div}^f$, ${\sf Div}^f_r$ and ${\sf Div}^f_l$.
Composing ${\sf Div}^f_r$ with ${\rm id} \otimes |\cdot|$ and ${\sf Div}^f_l$ with $|\cdot| \otimes {\rm id}$ amounts to evaluation at $s=1$ and replacing $\widetilde{{\rm Ham}^\kappa}$ with $\tilde{\sigma}$ (see the commutative diagram \eqref{rem:tangential_lift}).
Adding the two components gives back the value of the cocycle ${\sf Div}^f$.
\end{proof}

Next we examine the graded version of the cocycle ${\sf Div}^f$.
For this puropose, we work with the explicit formula of ${\sf Div}^f$ given in \eqref{eq:c^f_details}.

\begin{prop}   \label{prop:c^f_gr}
The cocycle ${\sf Div}^f$ is of filtration degree zero, and
the associated graded map
${\sf Div}^f_{\rm gr} \colon {\rm tDer}_{\{ z_1, \ldots, z_n \}}(A) \to |A| \otimes |A|$ is given by
\begin{equation}   \label{eq:c^f_gr}
{\sf Div}^f_{\rm gr}={\sf Div}_{x,y,z} + {\bf 1} \wedge {\sf b}^f.
\end{equation}
Moreover, when restricted to tangential Hopf derivations of $A$, the cocycle
${\sf Div}^f_{\rm gr}$ acquires the form ${\sf Div}^f_{\rm gr}(\tilde{u}) = \tilde{\Delta}({\sf div}^f_{\rm gr}(\tilde{u}))$, where ${\sf div}^f_{\rm gr}\colon {\rm tDer}_{\{ z_1, \ldots, z_n\}}(L) \to |A|$ is given by
$$
{\sf div}^f_{\rm gr}(\tilde{u}) = {\sf div}_{x,y,z}(u) - {\sf b}^f(\tilde{u}).
$$
\end{prop}

\begin{proof}
Note that we can use the Magnus or the exponential (or any other) variables on $\widehat{\mathbb{K} \pi}$ to define the associated graded. With this in mind, we observe that in terms of Magnus variables the map ${\sf Div}_{X,Y,Z}$ is homogeneous of degree zero. Hence, so is its associated graded ${\sf Div}_{x,y,z}$ (where, as usual, we denote generators of the associated graded ${\rm gr}^{\rm wt}\, \widehat{\mathbb{K} \pi}$ by $x_i,y_i,z_j$). Furthermore, in terms of the exponential variables, the cocycle ${\sf b}^f$ is also homogeneous of degree zero. By abuse of notation, we denote its graded version by the same symbol. Finally the degree zero graded components of all trivial cocycles vanish, and in particular this applies to $u(\sum_i( |x_i| + |y_i|))$ and to $u({\bf p}^f)$. To conclude, we observe that the right hand side of equation~\eqref{eq:c^f_gr} is a non-vanishing cocycle. Hence the filtration degree of ${\sf Div}^f$ is indeed zero.
\end{proof}

The following statement is a graded version of Theorems~\ref{thm:mu=ch} and \ref{prop:delta=c_sigma}:

\begin{prop} \label{prop:delta=c_sigma_gr}
We have
\begin{equation} \label{eq:mu=ch_gr}
 \mu^f_{r, {\rm gr}} = {\sf Div}^f_{r, {\rm gr}} \circ \widetilde{{\rm Ham}^{\kappa_{\rm gr}}},
 \hspace{2em}
 \mu^f_{l, {\rm gr}} = {\sf Div}^f_{l, {\rm gr}} \circ \widetilde{{\rm Ham}^{\kappa_{\rm gr}}},
\end{equation}
\begin{equation} \label{eq:delta=c_sigma_gr}
    \delta^f_{\rm gr} = {\sf Div}^f_{\rm gr} \circ \tilde{\sigma}_{\rm gr}. 
\end{equation}
\end{prop}

\begin{proof}
Taking the associated graded of \eqref{eq:mu=ch} and \eqref{eq:delta=c_sigma} yields \eqref{eq:mu=ch_gr} and \eqref{eq:delta=c_sigma_gr}, respectively.
\end{proof}

For the group $1$-cocycles ${\sf J}^f_{\rm gr} \colon {\rm tAut}^+_{\{ z_1, \ldots, z_n \}}(A) \to |A| \otimes |A|$ and ${\sf j}^f_{\rm gr} \colon {\rm tAut}^+_{ \{ z_1, \ldots, z_n \}}(L) \to |A|$ corresponding to ${\sf Div}^f_{\rm gr}$ and ${\sf div}^f_{\rm gr}$, we have
${\sf J}^f_{\rm gr}(\tilde{F}) = \tilde{\Delta}({\sf j}^f_{\rm gr}(\tilde{F}))$, where
$$
{\sf j}^f_{\rm gr}(\tilde{F}) = {\sf j}_{x,y,z}(F) - {\sf c}^f(\tilde{F}).
$$

\section{Expansions and the Kashiwara-Vergne problem}  \label{sec:expansions+KV}

The goal of this section is to establish Theorem~\ref{thm:intro4} and to introduce higher genus Kashiwara-Vergne problems. The main tool is the theory of homomorphic expansions in the sense of Bar-Natan and Dancso \cite{BND13}.

We collect the notations used throughout this section.
Let $\Sigma$ be a surface of genus $g$ with $n+1$ boundary components.
We fix a standard generating system $\alpha_i, \beta_i, \gamma_j$ with $i=1,\ldots,g$, $j=1,\ldots,n$ for $\pi = \pi_1(\Sigma,*)$
(see Section~\ref{subsec:stdada}).
As was shown in Proposition~\ref{prop:grh}, the associated graded of the group algebra ${\rm gr}^{\rm wt}\, \K \pi$ with respect to the weight filtration is canonically isomorphic to the (completed) free associative algebra $A = \widehat{T}({\rm gr}^{\rm wt}\, H)$ with generators $x_i,y_i$ for $i=1,\ldots,g$ and $z_j$ for $j=1,\ldots,n$.
Under the weight filtration, these generators have the following degrees:
\[
\deg (x_i) = \deg(y_i) = 1,
\quad
\deg (z_j) = 2.
\]
We denote by
\[
L = L(x_i,y_i,z_j) = L({\rm gr}^{\rm wt}\, H)
\]
the (completed) free Lie algebra with generators $x_i,y_i,z_j$.
Recall from Section~\ref{subsec:div_der_Lie} that the Lie algebra ${\rm Der}^+(L)$ of positive derivations of $L$ integrates to the group ${\rm Aut}^+(L) = {\rm Aut}^+(A, \Delta)$.
As was explained in Section~\ref{subse:tngder}, there is also a tangential version of these Lie algebra and group, denoted by ${\rm tDer}^+(L) = {\rm tDer}^+_{\{ z_1, \ldots, z_n\}}(L)$ and ${\rm tAut}^+(L) = {\rm tAut}^+_{\{ z_1, \ldots, z_n\}}(L)$.

\subsection{Expansions: definitions and basic properties}

In this subsection, we define the notion of an expansion and discuss some special properties of expansions.

\begin{dfn}
An isomorphism of filtered algebras $\theta \colon \widehat{\K \pi} \to A$ is called an {\em expansion} if it descends to identity on the associated graded: ${\rm gr}\, \theta = {\rm id}$.
\end{dfn}

An expansion being an algebra isomorphism, it is sufficient to define it on the generators. The definition implies
$$
\theta(\alpha_i)\equiv_2 1 + x_i, \hskip 0.3cm
\theta(\beta_i) \equiv_2 1 + y_i, \hskip 0.3cm
\theta(\gamma_j)\equiv_3 1+ z_j,
$$
where $\equiv_m$ stands for the equivalence modulo $A_{\geq m}$.
Up to now, we were implicitly using the {\em Magnus expansion}
$$
\theta_M(\alpha_i)=1+x_i, \hskip 0.3cm
\theta_M(\beta_i) = 1 + y_i, \hskip 0.3cm
\theta_M(\gamma_j)=1+z_j
$$
and the {\em exponential expansion}
$$
\theta_{\rm exp}(\alpha_i)=e^{x_i}, \hskip 0.3cm
\theta_{\rm exp}(\beta_i) = e^{y_i}, \hskip 0.3cm
\theta_{\rm exp}(\gamma_j)=e^{z_j}.
$$

Recall that both $\widehat{\K \pi}$ and $A$ are filtered Hopf algebras. 

\begin{dfn}
An expansion is called {\em group-like}  if it is an isomorphism of filtered Hopf algebras.
\end{dfn}

It is easy to see that the exponential expansion is group-like whereas the Magnus expansion is not.
The following lemma describes group-like expansions using automorphisms of $L$.

\begin{prop}
Let $\theta \colon \widehat{\K\pi} \to A$ be a group-like expansion. Then there is a unique element $F \in {\rm Aut}^+(L)$ such that $\theta= F \circ \theta_{\rm exp}$.
\end{prop}

\begin{proof}
The map $F= \theta \circ \theta_{\rm exp}^{-1}$ is a Hopf algebra automorphism of $A$. Hence, it is an element of ${\rm Aut}(L)$.  By definition of an expansion, the associated graded of $\theta$ and $\theta_{\rm exp}$ are the identity morphisms. Hence, 
$$
F(x_i)= \theta(\theta_{\rm exp}^{-1}(x_i)) = \theta(\log(\alpha_i)) \equiv_2 \theta(\alpha_i -1) \equiv_2 x_i
$$
and $F(x_i) \in x_i + A_{\geq 2}$. Similar arguments show that $F(y_i) \in y_i + A_{\geq 2}$ and $F(z_j) \in z_j + A_{\geq 3}$. 
Therefore, $F \in {\rm Aut}^+(L)$, as required.
\end{proof}

Recall from Section~\ref{subsec:filtonKpi} that there is a distinguished element
\begin{equation} \label{eq:recall_omega}
\omega = {\rm gr}(\gamma_0 - 1) = \sum_{i=1}^g [x_i, y_i] + \sum_{j=1}^n z_j
\in A,
\end{equation}
where $\gamma_0 \in \pi$ is the loop around the $0$th boundary of $\Sigma$.

\begin{lem}
Let $\theta \colon \widehat{\K\pi} \to A$ be an expansion. Then we have
$\theta(\gamma_0) \equiv_3 1 + \omega$.
\end{lem}

\begin{proof}
Since ${\rm gr}\, \theta = {\rm id}$, we have $\theta(\gamma_0 - 1) \equiv_3 {\rm gr}(\gamma_0 -1) = \omega$.
\end{proof}

We continue with more definitions concerning expansions:

\begin{dfn}
\label{dfn:exptanb}
\begin{enumerate}
\item[(i)]
A group-like expansion $\theta \colon \widehat{\K\pi} \to A$ is called \emph{tangential} if for $j=1,\ldots,n$ there are group-like elements $f_j \in A$ such that
$\theta(\gamma_j)=f_j^{-1} \exp(z_j) f_j$.
\item[(ii)]
A tangential expansion $\theta$ is called {\em weakly special} if there is a group-like element $f_0 \in A$ such that $\theta(\gamma_0)=f_0^{-1} \exp(\omega)f_0$.
\item[(iii)] An expansion is called {\em special} if $\theta(\gamma_0)=\exp(\omega)$.
\end{enumerate}
\end{dfn}

%

We will use the following notation:
\begin{equation} \label{eq:xi_def}
\xi := \log \theta_{\rm exp}(\gamma_0) = \log \big( \prod_{i=1}^{g} (e^{x_i} e^{y_i} e^{-x_i} e^{-y_i})  \, \prod_{j=1}^{n} e^{z_j} \big)
\in A.
\end{equation}
It is convenient to state  conditions for a group-like expansion $\theta = F \circ \theta_{\rm exp}$ to be tangential, weakly special or special in terms of the automorphism $F \in {\rm Aut}^+(L)$.

\begin{prop}       \label{prop:properties_expansions}
\begin{enumerate}
    \item[$(i)$]
An expansion $\theta = \theta_F = F \circ \theta_{\rm exp}$ is tangential if and only if $F$ satisfies equations
$F(z_j)=f_j^{-1}z_jf_j$
for some group-like elements $f_j \in A$, $j=1, \dots, n$, i.e., $F$ admits a tangential lift $\tilde{F} = (F, f_1,\ldots, f_n) \in {\rm tAut}^+(L)$.
    \item[$(ii)$]
The expansion $\theta$ is weakly special if and only if $F$ admits a tangential lift and in addition satisfies the equation
$F(\xi) = f_0^{-1} \omega f_0$
for some group-like element $f_0 \in A$. 
    \item[$(iii)$]
The expansion $\theta$ is special
if and only if $F$ admits a tangential lift and
$F(\xi)=\omega$.
\end{enumerate}
\end{prop}

\begin{proof}
Assume that $\theta = F \circ \theta_{\rm exp}$ is tangential. Then,
$$
F(e^{z_j})=F(\theta_{\rm exp}(\gamma_j)) = \theta(\gamma_j) = f_j^{-1} e^{z_j} f_j
$$
which implies $f(z_j) = f_j^{-1} z_j f_j$.
In the other direction, we rearrange the equality above as follows:
$$
\theta(\gamma_j) = F(\theta_{\rm exp}(\gamma_j)) = F(e^{z_j}) =
\exp(f_j^{-1} z_j f_j) = f_j^{-1} e^{z_j} f_j
$$
which shows that $\theta$ is indeed a tangential expansion. This proves (i).
The proof of (ii) and (iii) are similar, so we omit them.
%
%
%
\end{proof}

\begin{rem}   \label{rem:special_inner}
A weakly special expansion can always be presented as a composition $\theta = {\rm Ad}_{f_0}^{-1} \circ \theta'$ of an inner automorphism of $A$ and of a special expansion $\theta'$. 
\end{rem}

\begin{rem}
If $n=0$, special expansions are called \emph{symplectic expansions} \cite{Mas12}.
If $g=0$, the terminology is due to \cite{Mas15} (but the notion had already appeared implicitly in \cite{AET, AT12}).
The condition in Definition \ref{dfn:exptanb} (ii) was given in \cite[\S 7.2]{KK16}.
Special expansions exist for any $g$ and $n$. This can be deduced from existence of symplectic expansions and 
of special expansions for $g=0$ by using gluing arguments for expansions.
\end{rem}

\subsection{Special expansions and double brackets}
In this subsection, we discuss various relations between special expansions and double brackets. 

First we recall the following key result (see Lemma 6.2 and Lemma 6.3 in \cite{MT13} and Theorem 2.31 in \cite{Naef}):

\begin{thm}  \label{thm:Florian}
\begin{enumerate}
\item[$(i)$] An element $\tilde{F} \in {\rm tAut}^+(A)$ preserves the double bracket 
$\kappa_{\rm gr}$ if and only if $F(\omega) = \omega$.

\item[$(ii)$] An element $\tilde{F} \in {\rm tAut}^+(\widehat{\K\pi})$ preserves the 
double bracket $\kappa$ if and only if $F(\gamma_0)=\gamma_0$.
\end{enumerate}
\end{thm}

\begin{rem}  \label{rem:omega=>h,sigma}
The following statements are direct corollaries of Theorem \ref{thm:Florian}:
\begin{enumerate}
\item[(i)]
Let $\tilde{F} \in {\rm tAut}^+(A)$ such that $F(\omega) = \omega$. Then, $F$ preserves the maps ${\rm Ham}^{\kappa_{\rm gr}}$ and $\sigma_{\rm gr}$ and the Lie bracket $[\cdot, \cdot]_{\rm gr}$.

\item[(ii)]
Let $\tilde{F} \in {\rm tAut}^+(\widehat{\K\pi})$ such that $F(\gamma_0)=\gamma_0$.
Then, $F$ preserves the maps ${\rm Ham}^\kappa$ and $\sigma$ and the Lie bracket $[\cdot, \cdot]_{\rm Goldman}$. 
\end{enumerate}
\end{rem}

For the future use we will need the following notation. Recall that the generating function of Bernoulli numbers has the form
$$
\frac{x}{e^x-1} =\sum_{k=0}^\infty \frac{B_k}{k!} \, x^k.
$$
We will consider the function 
$$
\phi(x) = \frac{1}{e^x -1} - \frac{1}{x} = \sum_{k=1}^\infty \frac{B_k}{k!} \, x^{k-1}
= - \frac{1}{2} + \frac{x}{12} - \frac{x^3}{720} + \frac{x^5}{30240} + \cdots
$$
and the element
\begin{equation}   \label{eq:phi}
\phi=\tilde{\Delta}(\phi(\omega)) \in A \otimes A.
\end{equation}
In the Sweedler notation, $\phi =\phi' \otimes \phi''$. Note that $\phi'\phi'' = \phi(0)=-1/2$.

For each expansion $\theta$, the double bracket $\kappa$ induces a double bracket $\kappa_\theta$ on $A$ such that $\theta$ becomes a $\K$-algebra homomorphism preserving double brackets.
In other words, $\kappa_{\theta}$ is a unique map which makes the following diagram commutative: 
\[
\xymatrix{
\widehat{\K \pi} \otimes \widehat{\K \pi} 
\ar[r]^{\kappa} \ar[d]_{\theta^{\otimes 2}}
&
\widehat{\K \pi} \otimes \widehat{\K \pi}
\ar[d]^{\theta^{\otimes 2}} \\
A \otimes A 
\ar[r]_{\kappa_{\theta}}
&
A \otimes A.
}
\]

The following result describes the induced map $\kappa_\theta$ for a special expansion $\theta$. It follows from the corresponding result by Massuyeau-Turaev \cite{MT13} for the homotopy intersection form $\eta = (\varepsilon \otimes {\rm id})\kappa$.

\begin{thm}[\cite{KK15, KK16, MTpre}]
\label{thm:kappa_f}
Let $\theta \colon \widehat{\K\pi} \to A$ be a special expansion. Then,
$$
\kappa_\theta=\kappa_{\rm gr} + \kappa_\phi,
$$
where $\kappa_{\phi}$ is the double bracket defined by 
$$
\kappa_{\phi}(a,b) = \phi'' a \otimes \phi' b + b \phi'' \otimes a \phi' - \phi'' \otimes a \phi' b - b \phi'' a \otimes \phi'.
$$
\end{thm}

\begin{rem}
Note that $\kappa_\theta \neq \kappa_{\rm gr}$. This is because (in terminology of van den Bergh) $\kappa_{\rm gr}$ is a double Poisson bracket and $\kappa$ (and hence $\kappa_\theta$) is a double quasi-Poisson bracket. In this paper, we do not introduce the two versions of the Jacobi identity for double brackets, and we do not develop this theme.
\end{rem}

Theorem \ref{thm:kappa_f} implies the following important statement:

\begin{prop}  \label{prop:sigma=sigma}
Let $\theta\colon \widehat{\K\pi} \to A$ be a special expansion
and let $\tilde{\sigma}_{\theta} \colon |A| \to {\rm tDer}(A) = {\rm tDer}_{\{z_1,\ldots, z_n\}}(A)$ be the induced map of $\tilde{\sigma}$ by $\theta$, i.e., the unique map which makes the diagram
\[
\xymatrix{
|\widehat{\K \pi}|
\ar[r]^{\tilde{\sigma}} \ar[d]_{\theta}
&
{\rm tDer}(\widehat{\K \pi})
\ar[d]^{\theta} \\
|A|
\ar[r]_{\tilde{\sigma}_{\theta}}
&
{\rm tDer}(A)
}
\]
commute. 
Then $\tilde{\sigma}_{\theta}$ coincides with the canonical tangential lift of $\sigma_{\rm gr}$: 
$\tilde{\sigma}_{\theta} = \tilde{\sigma}_{\rm gr}$.
\end{prop}

\begin{proof}
First, note that $\tilde{\sigma}_{\theta}$ is the composition of the map $\widetilde{{\rm Ham}^{\kappa_{\theta}}}$ and the evaluation at $s=1$.
We compute
\begin{equation*}   
{\rm Ham}^{\kappa_\phi}_a \colon b \mapsto \{ |sa|, b\}_{\kappa_{\phi}} = \phi'' as \phi' b + b \phi'' s a \phi' - \phi'' sa \phi' b - b \phi'' as \phi' =
[b, \phi'' [s,a] \phi'].
\end{equation*}
for all $a, b \in A$.
This shows that the double bracket $\kappa_{\phi}$ is tangential.
Since $(\phi'' [s,a] \phi')|_{s = 1} = 0$,
the map $\sigma_{\phi}(|a|)$ and its tangential lift vanish.
Therefore we have an equality of tangential lifts
$$
\tilde{\sigma}_\theta(|a|) = \widetilde{{\rm Ham}^{\kappa_\theta}_a}|_{s=1} =
\widetilde{{\rm Ham}^{\kappa_{\rm gr}}_a}|_{s=1} = \tilde{\sigma}_{\rm gr}(|a|),
$$
as required.
\end{proof}

\begin{rem}
The proof of Proposition~\ref{prop:sigma=sigma} shows that the image of ${\rm Ham}^{\kappa_\phi}$ is contained in the subspace of inner derivations linear in $s$.
\end{rem}

We now consider the map $\tilde{\sigma}_{\rm gr} \colon |A| \to {\rm tDer}(A)$ in more detail.
Its image and kernel are described by the following statement:

\begin{prop}  \label{prop:im_sigma_gr}
The image of the map $\tilde{\sigma}_{\rm gr} \colon |A| \to {\rm tDer}(A)$
is given by
$$
{\rm im}(\tilde{\sigma}_{\rm gr}) =
\{ \tilde{u} \in {\rm tDer}(A); u(\omega)=0\}, 
$$
and the kernel is given by
$$
\ker(\tilde{\sigma}_{\rm gr}) = \K {\bf 1}.
$$
\end{prop}

\begin{proof}
See Appendix \ref{subsec:bra_kappa_gr}.
\end{proof}

Next we can state similar results on ${\rm ker}(\tilde{\sigma})$ and ${\rm im}(\tilde{\sigma})$:

\begin{prop} \label{prop:ker_sigma}
The image of the map $\tilde{\sigma}\colon |\widehat{\K \pi}| \to {\rm tDer}(\widehat{\K \pi})$ is given by
\[
{\rm im} (\tilde{\sigma}) = 
\{ \tilde{u} \in {\rm tDer}(\widehat{\K \pi}) ; u(\xi) = 0 \}, 
\]
and its kernel is given by  
\[
\ker (\tilde{\sigma}) = \K {\bf 1}.
\]
The kernel of the map $\sigma \colon |\widehat{\K \pi}| \to {\rm Der}(\widehat{\K \pi})$ is given by 
\[
\ker (\sigma) = \K {\bf 1} \oplus 
\sum_{j=1}^n |\K [[ \log \gamma_j]]_{\ge 1} |.
\]
\end{prop}

\begin{proof}
Using some special expansion, identify $|\widehat{\K\pi}| \cong |A|$ and
${\rm tDer}(\widehat{\K\pi}) \cong {\rm tDer}(A)$. Under this identification,
$\xi$ maps to $\omega$ and $\log \gamma_j$ to a conjugate of $z_j$.
Now the statements follow from Proposition \ref{prop:ker_sigma_gr} and Proposition \ref{prop:im_sigma_gr}. 
\end{proof}

Finally we complete our presentation with the following theorem:

\begin{thm}[\cite{AKKN_new, KK14, KK16, MT13, MTpre}] 
\label{thm:toMT}
Let $\theta \colon \widehat{\K\pi} \to A$ be a group-like expansion. Then the following statements are equivalent:
\begin{enumerate}
\item[$(i)$]
$\theta$ is weakly special;

\item[$(ii)$]
$\theta$ induces an isomorphism of filtered Lie algebras 
$(|\widehat{\K \pi}|, [\cdot, \cdot]_{\rm Goldman})  \to (|A|, [\cdot, \cdot]_\gr)$.
\end{enumerate}
\end{thm}

\begin{proof}
Theorem \ref{thm:toMT} summarizes several results. The assertion $(i) \Rightarrow (ii)$ was established in \cite{KK14, KK16, MT13, MTpre}. The assertion $(ii) \Rightarrow (i)$ is the result by the authors \cite{AKKN_new}.
\end{proof}

In particular, if $\theta = F\circ \theta_{\exp}$ is a group-like expansion which solves the formality problem for the Goldman Lie algebra, then it follows that $\theta$ is tangential and there is a lift of $F$ to a tangential automorphism $\tilde{F}$.

\subsection{Homomorphic expansions} \label{subsec:homoexp}
In this subsection, we define and study {\em homomorphic expansions} for the Goldman-Turaev Lie bialgebra. 
Similar to special expansions, their properties can be encoded in those of automorphsims $\tilde{F} \in {\rm tAut}^+(L)$.

Let $f$ be a framing on $\Sigma$. Recall that the Lie algebras $(|\widehat{\mathbb{K}\pi}|, [\cdot, \cdot]_{\rm Goldman})$ and $(|A|, [\cdot,  \cdot]_{\rm gr})$ carry the Lie cobrackets $\delta^f_{\rm Turaev}$ and $\delta^f_{\rm gr}$, respectively, which promote them to involutive Lie bialgebras.

\begin{dfn}
A group-like expansion $\theta$ is called {\em homomorphic} (with respect to the framing $f$) if it induces an isomorphism of involutive Lie bialgebras
$$
\theta \colon (|\widehat{\K \pi}|, [\cdot, \cdot]_{\rm Goldman}, \delta^f_{\rm Turaev}) \to (|A|, [\cdot, \cdot]_\gr, \delta^f_\gr).
$$
\end{dfn}

Let ${\delta}^f_{\theta}$ be the induced map of the framed Turaev cobracket by the expansion $\theta$:
\[
\xymatrix{
|\widehat{\K \pi}| 
\ar[r]^{\delta^f} \ar[d]_{\theta}
&
|\widehat{\K \pi}| \otimes |\widehat{\K \pi}|
\ar[d]^{\theta^{\otimes 2}} \\
A 
\ar[r]^{\delta^f_{\theta}}
&
A \otimes A.
}
\]
Then $\theta$ is homomorphic if and only if the induced map of the Goldman bracket by $\theta$ equals $[\cdot, \cdot]_{\rm gr}$ and
\[
\delta^f_{\theta} = \delta^f_{\rm gr}.
\]

The existence problem for homomorphic expansions is also called the Goldman-Turaev (GT) formality problem. This is the central question addressed in this paper. Its positive or negative solution may depend on the framing $f$.

\begin{rem}
If the GT formality problem admits a positive solution for some choice of framing, this implies a positive solution of the GT formality problem for the canonical GT Lie bialgebra $|\widehat{\K \pi}|/\K {\bf 1}$.
\end{rem}

\begin{rem}  \label{rem:inner}
If $\theta$ is a group-like homomorphic expansion and $f_0 \in \exp(L)$, then the expansion $\theta' = {\rm Ad}_{f_0} \circ \theta$ is also group-like and homomorphic.
Indeed, the inner automorphism ${\rm Ad}_{f_0}$ is an automorphism of the free Lie algebra $L$. Hence, it preserves the group-like property. Furthermore, ${\rm Ad}_{f_0}$ acts trivially on $|A|$. Therefore, the induced maps
$\theta, \theta' \colon |\widehat{\mathbb{K}\pi}| \to |A|$ coincide.
\end{rem}

The exponential expansion $\theta_{\exp}$ associated with the generating system $\{\alpha_i, \beta_i, \gamma_j \}$ identifies $\widehat{\K \pi}$ and $A = \widehat{T}({\rm gr}^{\rm wt}\, H)$, through which we regard elements ${\bf p}^f$ and ${\bf r}$ defined in \eqref{eq:def_pf} and \eqref{dfn:bfr} as elements in $|A|$.
Also, we regard the $1$-cocycle ${\sf div}^f$ as a map ${\rm tDer}(L) \to |A|$.
By Propositions~\ref{prop:div^f}~and~\ref{prop:c^f_gr}, we have
\[
{\sf div}^f_{\rm gr}(\tilde{u}) - {\sf div}^f(\tilde{u})
= u({\bf p}^f - {\bf r}).
\]
By applying equation~\eqref{eq:D^f_first} to cocycles 
${\sf div}^f$ and ${\sf div}^f_{\rm gr}$, we get
\begin{align}
R^f(\tilde{F}):= R_{{\sf div}^f, {\sf div}^f_{\rm gr}}(\tilde{F})  & = {\sf j}^f_{\rm gr}(\tilde{F}) + F({\bf r}-{\bf p}^f) \nonumber \\
& = {\sf j}_{x,y,z}(F) - {\sf c}^f(\tilde{F}) + F({\bf r}-{\bf p}^f) \label{eq:R^f_def} \\
& = {\sf j}^f(\tilde{F}) + {\bf r} - {\bf p}^f. \nonumber
\end{align}

\begin{rem}
When $g = 0$, we have $R^f(\tilde{F}) = {\sf j}^f_{\rm gr}(\tilde{F}) = {\sf j}^f(\tilde{F}) = {\sf j}_{x,y,z}(F) - {\sf c}^f(\tilde{F})$.
\end{rem}

The map $R^f\colon {\rm tAut}^+(L) \to |A|$ will play a 
central role in what follows.
\begin{prop} \label{prop:transform_R^f}
The map $R^f$ has the following property: for any $\tilde{F}, \tilde{G} \in {\rm tAut}^+(L)$, 
\begin{equation*}
R^f(\tilde{F} \tilde{G}) = 
{\sf j}^f_{\rm gr}(\tilde{F}) + F(R^f(\tilde{G})) =
R^f(\tilde{F}) + F({\sf j}^f(\tilde{G})).
\end{equation*}
\end{prop}
\begin{proof}
This is a direct consequence of Proposition~\ref{prop:transform_R}.
\end{proof}

Let $\theta = F \circ \theta_{\rm exp}$ be a tangential expansion.
Then we obtain the induced map ${\sf Div}^f_{\theta}$.
Namely, the diagram
\[
\xymatrix{
{\rm tDer}(\widehat{\K \pi})
\ar[r]^{{\sf Div}^f}
\ar[d]_{\theta}
&
|\widehat{\K\pi}| \otimes |\widehat{\K \pi}|
\ar[d]^{\theta^{\otimes 2}}
\\
{\rm tDer}(A)
\ar[r]_{{\sf Div}^f_{\theta}}
&
|A| \otimes |A|
}
\]
commutes. 
In a similar way, we obtain the induced map ${\sf div}^f_{\theta}\colon {\rm tDer}(L) \to |A|$.

\begin{prop}  \label{prop:c_transfer}
Let $\theta = \theta_F = F \circ \theta_{\rm exp}$ be a tangential expansion determined by $\tilde{F} \in {\rm tAut}^+(L)$. Then, for all $\tilde{u} \in {\rm tDer}(A)$, we have
$$
{\sf Div}^f_\theta(\tilde{u})={\sf Div}^f_{\rm gr}(\tilde{u}) + u(\tilde{\Delta}(R^f(\tilde{F}))),
$$
and for all $\tilde{u} \in {\rm tDer}(L)$ we have
\[
{\sf div}^f_{\theta}(\tilde{u}) = {\sf div}^f_{\rm gr}(\tilde{u}) + u(R^f(\tilde{F})).
\]
\end{prop}

\begin{proof}
Since $\tilde{\Delta}$ is injective and commutes with elements in ${\rm tDer}(L)$, the assertion for ${\sf div}^f$ follows from the assertion for ${\sf Div}^f|_{{\rm tDer}(L)} = \tilde{\Delta} \circ {\sf div}^f$.

For the cocycle ${\sf Div}^f$, we first consider the induced map by the exponential expansion $\theta_{\rm exp}$.
Let ${\sf Div}^f_{\exp} \colon {\rm tDer}(A) \to |A| \otimes |A|$ be the resulting map.
By \eqref{eq:c^f_details2}, we obtain
$$
{\sf Div}^f_{\rm exp}(\tilde{u}) = {\sf Div}_{x,y,z}(u) + {\bf 1} \wedge {\sf b}^f(\tilde{u}) + u(\tilde{\Delta}({\bf r} - {\bf p}^f)).
$$

Next we apply the automorphism $\tilde{F}$. The transformation property of 1-cocycles (Proposition~\ref{prop:bFu}) yields that for any $\tilde{u} \in {\rm tDer}(A)$
\begin{align*}
{\sf Div}^f_\theta(\tilde{u}) 
& = {\sf Div}_{x,y,z}(u) + {\bf 1} \wedge {\sf b}^f(\tilde{u}) + u(\tilde{\Delta}({\bf r} - {\bf p}^f)) \\
& \hspace{1em} + u({\sf J}_{x,y,z}(F) + {\bf 1} \wedge {\sf c}^f(\tilde{F}) + F(\tilde{\Delta}({\bf r} - {\bf p}^f)) - \tilde{\Delta}({\bf r} - {\bf p}^f) ) \\
& = {\sf Div}_{x,y,z}(u) + {\bf 1} \wedge {\sf b}^f(\tilde{u}) + 
u(\tilde{\Delta}({\sf j}_{x,y,z}(F)- {\sf c}^f(\tilde{F})+ F({\bf r} -{\bf p}^f))) \\
& = {\sf Div}^f_{\rm gr}(\tilde{u}) + u(\tilde{\Delta}(R^f(\tilde{F}))).
\end{align*}
Here we have used the fact that for $\tilde{F} \in {\rm tAut}^+(L)$ the element ${\sf c}^f(\tilde{F})$ is in the span of $|x_i|, |y_i|$ and $|z_j|$, and therefore ${\bf 1} \wedge {\sf c}^f(\tilde{F}) = - \tilde{\Delta}({\sf c}^f(\tilde{F}))$. We have also used that by Proposition~\ref{prop:J_and_j}~(i) it holds that ${\sf J}_{x,y,z}(F)=\tilde{\Delta}({\sf j}_{x,y,z}(F))$.
\end{proof}

As for the operations $\mu^f_r, \mu^f_l$, a tangential expansion $\theta = F \circ \theta_{\rm exp}$ induces the following maps:
$$
\mu^f_{r, \theta} \colon A \to |A| \otimes A, \hskip 0.3cm
\mu^f_{l, \theta} \colon A \to A \otimes |A|.
$$
\begin{prop} \label{prop:mu_transfer}
Let $\theta = F \circ \theta_{\rm exp}$ be a special expansion. Then, for any $a\in A$
\begin{equation*}  
\begin{array}{lll}
\mu^f_{r, \theta}(a) & = & \mu^f_{r, {\rm gr}}(a) + 
R' \otimes \{ a,  R''\}_{\rm gr}
+ |\phi''| \otimes a \phi'  - |\phi'' a| \otimes \phi', \\
\mu^f_{l, \theta}(a) & = & \mu^f_{l, {\rm gr}}(a) + \{ a, R' \}_{\rm gr}
\otimes  R'' + \phi'' a \otimes |\phi'| - \phi'' \otimes |a \phi'|.
\end{array}
\end{equation*}
where
$\tilde{\Delta}(R^f(\tilde{F})) =
R' \otimes R''$, and 
$\phi=\phi' \otimes \phi''$ is given by equation \eqref{eq:phi}.
\end{prop}

\begin{proof}
Since $\mu^f_r = {\sf Div}^f_r \circ \widetilde{{\rm Ham}^{\kappa}}$ and $\mu^f_l = {\sf Div}^f_l \circ \widetilde{{\rm Ham}^\kappa}$, the induced maps $\mu^f_{r,\theta}$ and $\mu^f_{l,\theta}$ are given as follows:
\begin{align*}
\mu^f_{r, \theta} &= {\sf Div}^f_{r, \theta} \circ \widetilde{{\rm Ham}^{\kappa_\theta}} =
{\sf Div}^f_{r, \theta} \circ (\widetilde{{\rm Ham}^{\kappa_{\rm gr}}}+ \widetilde{{\rm Ham}^{\kappa_\phi}}), \\
\mu^f_{l, \theta} &= {\sf Div}^f_{l, \theta} \circ \widetilde{{\rm Ham}^{\kappa_\theta}} =
{\sf Div}^f_{l, \theta} \circ (\widetilde{{\rm Ham}^{\kappa_{\rm gr}}}+ \widetilde{{\rm Ham}^{\kappa_\phi}}).
\end{align*}
The first terms in the right hand side yields
\begin{align*}
{\sf Div}^f_{r, \theta} \circ \widetilde{{\rm Ham}^{\kappa_{\rm gr}}}(a) &=
 \mu^f_{r, {\rm gr}}(a) + R' \otimes \{ a, R''\}_{\rm gr}, \\
{\sf Div}^f_{l, \theta} \circ \widetilde{{\rm Ham}^{\kappa_{\rm gr}}}(a) &=
 \mu^f_{l, \rm gr}(a) + \{ a, R' \}_{\rm gr} \otimes R''.
\end{align*} 
Here we have used Proposition \ref{prop:c_transfer} and \eqref{eq:mu=ch_gr}.
We have also used Lemma \ref{lem:Hama|b|} to compute
\begin{align*}
{\rm Ham}^{\kappa_{\rm gr}}_a(\tilde{\Delta}(R^f(\tilde{F}))) 
& = {\rm Ham}^{\kappa_{\rm gr}}_a(R') \otimes R'' + R' \otimes {\rm Ham}^{\kappa_{\rm gr}}_a(R'') \\
& = |s\{ a, R'\}_{\rm gr}| \otimes R'' + R' \otimes |s\{ a, R''\}_{\rm gr}|.
\end{align*}

Next we compute
$$
{\sf Div}^f_{\theta} \circ  \widetilde{{\rm Ham}^{\kappa_\phi}}(a) =
{\sf Div}^f_{\rm gr} \circ \widetilde{{\rm Ham}^{\kappa_\phi}}(a) + {\rm Ham}^{\kappa_\phi}_a (\tilde{\Delta}(R^f(\tilde{F}))).
$$
As we saw in the proof of Proposition \ref{prop:sigma=sigma}, the map ${\rm Ham}^{\kappa_\phi}_a$ acts as an inner derivation with generator $\phi''[s,a]\phi'$ on $x_i,y_i,z_j$.
Hence it acts trivially on the trace space $|A|$, and the second term on the right hand side in the formula above vanishes: ${\rm Ham}^{\kappa_\phi}_a (\tilde{\Delta}(R^f(\tilde{F})))=0$.
For the first term, we have
$$
{\sf Div}^f_{\rm gr} \circ \widetilde{{\rm Ham}^{\kappa_\phi}}(a) = {\sf Div}_{x,y,z}({\rm Ham}^{\kappa_\phi}_a) + {\bf 1} \wedge {\sf b}^f(\widetilde{{\rm Ham}^{\kappa_\phi}_a}).
$$
The contribution of the cocycle ${\sf b}^f$ vanishes:
$$
{\sf b}^f(\widetilde{{\rm Ham}^{\kappa_\phi}_a}) = 
\sum_j {\rm rot}^f(\gamma_j)  |\phi''[s,a]\phi'|
= \sum_j {\rm rot}^f(\gamma_j)  |s[a,\phi'\phi'']| = 0.
$$
Here we have used the fact that $\phi'\phi'' = -1/2$ and $[a, \phi'\phi'']=0$ for all $a \in A$.
To compute the divergence term, let us consider the derivation $u_{\phi,a}$ of $A\langle s \rangle$ which vanishes on $x_i, y_i, z_j$ and maps $s$ to $[s, \phi''[s,a]\phi']$.
Then ${\rm Ham}^{\kappa_{\phi}}_a + u_{\phi,a}$ is the inner derivation of $A\langle s \rangle$ with generator $\phi''[s,a]\phi'$.
By Example~\ref{ex:innerderivation}, the divergence of this inner derivation is $(2g + n)\, {\bf 1} \wedge |\phi''[s,a]\phi'| =0$.
Hence 
\begin{align*}
{\sf Div}_{x,y,z}({\rm Ham}^{\kappa_{\phi}}_a) & =
- | \partial_s(u_{\phi,a}(s)) | \\
& = - | \partial_s([s,\phi''[s,a]\phi']) | \\
& = 
|\phi''| \otimes |a\phi's| - |\phi''a| \otimes |\phi's|
+ |s\phi''a| \otimes |\phi'| - |s \phi''| \otimes |a\phi'|
\end{align*}
and we obtain
\begin{align*}
    {\sf Div}^f_{r, {\rm gr}}\circ \widetilde{{\rm Ham}^{\kappa_\phi}}(a)
    &= |\phi''| \otimes a \phi' - |\phi'' a| \otimes \phi', \\
{\sf Div}^f_{l, {\rm gr}}\circ \widetilde{{\rm Ham}^{\kappa_\phi}}(a)
    &= \phi''a \otimes |\phi'| - \phi'' \otimes |a\phi'|.
\end{align*}
Putting things together, we obtain the desired result.
\end{proof}

The following proposition is one of the main results of this subsection:

\begin{prop}  \label{prop:homomorphic=center}
A special expansion $\theta = F \circ \theta_{\rm exp}$ is homomorphic if and only if
\begin{equation*}    
R^f(\tilde{F}) \in Z(|A|, [\cdot, \cdot]_{\rm gr}).
\end{equation*}
\end{prop}

\begin{proof}
Recall from equation \eqref{eq:bul**bul} that 
$$
\delta^f_{\rm Turaev} = ({\rm id} \otimes |\cdot|) \mu^f_r + (|\cdot| \otimes {\rm id}) \mu^f_l.
$$
Therefore, the induced map $\delta^f_{\theta}$ of the framed Turaev cobracket by $\theta$ is given by 
\begin{align*}
\delta^f_\theta(|a|) & = ({\rm id} \otimes |\cdot|) \mu^f_{r, \theta}(a) + (|\cdot| \otimes {\rm id}) \mu^f_{l, \theta}(a) \\
& = ({\rm id} \otimes |\cdot|) \mu^f_{r, {\rm gr}}(a) + (|\cdot| \otimes {\rm id}) \mu^f_{l, {\rm gr}}(a) + [|a|, \tilde{\Delta}(R^f(\tilde{F}))]_{\rm gr} \\
& \hspace{1em} + |\phi''| \otimes |a \phi'| - |\phi'' a| \otimes |\phi'| + |\phi'' a| \otimes |\phi'| - |\phi''| \otimes |a\phi'| \\
& = \delta^f_{\rm gr}(|a|) + [|a|, \tilde{\Delta}(R^f(\tilde{F}))]_{\rm gr}.
\end{align*}
Here we have used Proposition~\ref{prop:mu_transfer} in the second line.

Assume that $R^f(\tilde{F}) \in Z(|A|, [\cdot, \cdot]_{\rm gr})$.
Then, by Proposition~\ref{prop:tiDelta_center} we have
$
\tilde{\Delta}(R^f(\tilde{F})) \in 
Z(|A|, [\cdot, \cdot]_{\rm gr})^{\otimes 2}
$
and
$$
\delta^f_\theta(|a|) = \delta^f_{\rm gr}(|a|) + [|a|, \tilde{\Delta}(R^f(\tilde{F}))]_{\rm gr} = \delta^f_{\rm gr}(|a|)
$$
for all $|a| \in |A|$, as required. 

In the other direction, assume that $\delta^f_\theta=\delta^f_{\rm gr}$. Then,
$$
[|a|, \tilde{\Delta}(R^f(\tilde{F}))]_{\rm gr} =
 \delta^f_\theta(|a|) - \delta^f_{\rm gr}(|a|)=0
$$
for all $|a| \in |A|$. Assume that the degree of $a$ is greater than or equal to 2. Then,
\begin{align*}
[|a|, \tilde{\Delta}(R^f(\tilde{F}))]_{\rm gr} & =  [|a|, R']_{\rm gr} \otimes R'' 
+ R' \otimes [|a|, R'']_{\rm gr} \\
& = [|a|, R^f(\tilde{F})]_{\rm gr} \otimes {\bf 1}
+ (\text{terms in $|A| \otimes |A|_{\geq 1}$}).
\end{align*}
Here we have used the fact that $[|a|, R'']_{\rm gr}$ is of degree at least 1 if $|a|$ is of degree at least 2.
We conclude that $[|a|, R^f(\tilde{F}))]_{\rm gr} =0$ for all $|a|$ of degree at least 2. By Proposition~\ref{prop:inner_derivation}, this implies that
$R^f(\tilde{F}) \in Z(|A|, [\cdot, \cdot]_{\rm gr})$, as required.
\end{proof}

\subsection{Kashiwara-Vergne problems in higher genera}
\label{subsec:KV}

In this subsection, we introduce Kashiwara-Vergne (KV) problems associated to framed oriented surfaces. 

Recall the elements $\omega, \xi \in A$ and $R^f(\tilde{F}) \in |A|$ (see \eqref{eq:xi_def}, \eqref{eq:recall_omega} and \eqref{eq:R^f_def}).

\begin{dfn}[KV Problem of type $(g, n+1)$ with framing $f$]      \label{dfn:KV_problem}
Find an element $\tilde{F} \in {\rm tAut}^+(L)$ satisfying the following two conditions:
$$
\begin{array}{ll}
{\rm KV}\, {\rm I}_{\gn} :& F(\xi) =  \omega,  \\
{\rm KV}\, {\rm II}_{\gn}:&  R^f(\tilde{F})= \sum_{j=1}^n |h_j(z_j)|  - |h (\omega)| \\
& \text{for some $h_j \in s\mathbb{K}[[s]]$, $j= 1,\ldots, n$ and $h \in s^2\mathbb{K}[[s]]$}.
\end{array}
$$
\end{dfn}

This definition is motivated by the following result:
\begin{thm} \label{thm:GThomomorphic}
Let $\Sigma$ be a surface of genus $g$ with $n+1$ boundary components and a framing $f$.
Let $\theta'= F' \circ \theta_{\rm exp}$ be a group-like expansion determined by $F' \in {\rm Aut}^+(L)$.
Then $\theta'$ is homomorphic with respect to the framing $f$ if and only if there is an  element $f_0 \in \exp(L)$ and a lift $\tilde{F}' \in {\rm tAut}^+(L)$ of $F'$ such that $\tilde{F}={\rm Ad}_{f_0} \circ \tilde{F}' \in {\rm tAut}^+(L)$ is a solution of the KV problem of type $(g,n+1)$ with framing $f$.
Here ${\rm Ad}_{f_0}$ is lifted to a tangential automorphism with all the tangential components being $f_0^{-1}$.
\end{thm}

\begin{proof}
Let $\theta' = F' \circ \theta_{\rm exp}$ be a group-like homomorphic expansion. Then, by Theorem~\ref{thm:toMT}, $\theta'$ is a weakly special expansion.
Thus, $F'$ admits a tangential lift $\tilde{F}'$ and there is an element $f_0 \in \exp (L)$ such that $\theta'(\gamma_0) = f_0^{-1} \exp(\omega) f_0$.
Therefore, the expansion
$\theta={\rm Ad}_{f_0} \circ \theta'$ is special. Indeed,
$$
\theta(\gamma_0)={\rm Ad}_{f_0}(\theta'(\gamma_0)) =
{\rm Ad}_{f_0}(f_0^{-1} \exp(\omega) f_0)=\exp(\omega).
$$
Set $\tilde{F}:= {\rm Ad}_{f_0} \circ \tilde{F}' \in {\rm tAut}^+(L)$.
The special expansion $\theta = F \circ \theta_{\exp}$ is group-like and homomorphic since these properties are not affected by composition with inner automorphisms.
By Proposition \ref{prop:homomorphic=center} we have $R^f(\tilde{F}) \in Z(|A|, [\cdot, \cdot]_{\rm gr})$. By Theorem \ref{thm:center}, this expression must be of the form
$$
R^f(\tilde{F}) = \sum_{j=1}^n |h_j(z_j)| - |h(\omega)|
$$
for some formal power series $h_j, h \in \mathbb{K}[[s]]$.
Since $R^f(\tilde{F})$ has no constant term, 
we can choose $h_j, h \in s\mathbb{K}[[s]]$. Furthermore, since 
$|\omega| = | \sum_i [x_i, y_i] + \sum_j z_j | = \sum_j |z_j|$, we can also choose $h \in s^2\mathbb{K}[[s]]$.
In conclusion, $\tilde{F}$ is a solution of the KV problem, as required.

In the other direction, let $\theta = F \circ \theta_{\rm exp}$ be a group-like expansion with $\tilde{F}$ being a solution of the KV problem. By Proposition~\ref{prop:properties_expansions}, the expansion $\theta$ is special, and by Proposition~\ref{prop:homomorphic=center} this expansion is homomorphic. The properties of an expansion being group-like and homomorphic are not affected by inner automorphisms with $f_0 \in \exp(L)$. Hence the expansion $\theta'={\rm Ad}_{f_0}^{-1} \circ \theta$ is group-like and homomorphic, as required.
\end{proof}

Sometimes it turns out to be more convenient to state the generalized KV problem in terms of the element $\tilde{F}^{-1}$ (instead of $\tilde{F}$):

\begin{prop}  \label{prop:F^-1}
The (generalized) KV problem is equivalent to the following set of equations for the element $\tilde{F}^{-1} \in {\rm tAut}^+(L)$:
\begin{equation*}
\begin{array}{ll}
(i) & F^{-1}(\omega) = \xi, \\
(ii) & {\sf j}^f_{\rm gr}(\tilde{F}^{-1})= {\bf r} - {\bf p}^f + |h(\xi)| - \sum_j |h_j(z_j)| \\
& \text{for some $h_j \in s\mathbb{K}[[s]]$, $j=1,\ldots, n$ and $h \in s^2\mathbb{K}[[s]]$.}
\end{array}
\end{equation*}
\end{prop}

\begin{proof}
First the first KV equation $F(\xi)=\omega$ implies $F^{-1}(\omega) = \xi$.

Next we apply $F^{-1}$ to the both sides of the second KV equation. On the left hand side, we obtain
$$
F^{-1}(R^f(\tilde{F})) = 
F^{-1}({\sf j}^f_{\rm gr}(\tilde{F})) + {\bf r} - {\bf p}^f = 
-{\sf j}^f_{\rm gr}(\tilde{F}^{-1}) + {\bf r} - {\bf p}^f.
$$
On the right hand side, we get
$$
F^{-1}(\sum_j |h_j(z_j)| - |h(\omega)|)  = 
\sum_j |h_j({f'_j}^{-1} z_j f'_j)| - |h(\xi)| 
 = 
\sum_j |h_j(z_j)| - |h(\xi)|, 
$$
where we write $\tilde{F}^{-1} = (F^{-1}, f'_1, \ldots, f'_n)$.
(Remark that $f'_j = (F^{-1}(f_j))^{-1}$ if $\tilde{F} = (F, f_1,\ldots, f_n)$.)
Putting things together yields the desired result. The argument in the opposite direction works in a similar way.
\end{proof}

We continue with several remarks on the generalized KV problem.

\begin{rem}  \label{rem:Duflo_genus_zero}
The formal power series $h_1,\ldots,h_n$ and $h$ are called {\em Duflo functions}. In the case of $g=0$, they always agree modulo linear parts with even part given by 
\begin{equation}       \label{eq:h_even}
h_{\rm even}(s) = \frac{1}{2} \log \left( \frac{e^{s/2} - e^{-s/2}}{s}\right)
\end{equation}
(\cite[Theorem 8.7, Proposition 8.5 and 8.11]{genus0}).
At present, we do not know whether this holds in the case of positive genus.
\end{rem}

\begin{rem}  \label{rem:framing_genus_0}
In the case of $g=0$, KV problems with different framings are all equivalent to each other. Indeed, the cocycle ${\sf c}^f(\tilde{F})$ takes values in the span of $|z_j|$ for $j=1, \dots, n$. All these terms can be absorbed in linear parts of Duflo functions $|h_j(z_j)|$.
Note that we consider framing as part of the input data of a KV problem while the Duflo functions are viewed as an output together with the tangential automorphism $\tilde{F} \in {\rm tAut}^+(L)$.
\end{rem}

\begin{rem}    \label{rem:KV02}
The simplest and somewhat special example is the KV problem of type $(0,2)$. In that case, the free Lie algebra in one variable $L(z)$ is abelian, $f_1=\exp(az)$ for some $a \in \mathbb{K}$, and $F={\rm id}$. The first KV equation is automatically satisfied since $\xi = \omega = z$. The second KV equation reads
$$
{\sf j}^f_{\rm gr}(\tilde{F})={\sf j}_z(F) - {\sf c}^f(\tilde{F}) = - {\rm rot}^f(\gamma_1) a|z| = |h_1(z) - h(z)|.
$$
Hence,  $h(s) \in s^2\mathbb{K}[[s]]$ is an arbitrary formal power series starting in degree 2, and $h_1(s)$ coincides with $h(s)$ up to a linear term.
\end{rem}

\subsection{The Kashiawar-Vergne problem of type $(0,3)$.}

In this subsection, we focus on the KV problem of type $(0,3)$ which plays a crucial role in our considerations.
First recall the relation between the generalized KV problem and the classical KV problem:

\begin{rem}  \label{rem:KVnorm}
Let $\Sigma$ be a sphere with three boundary components, $\gamma=\{ \gamma_1, \gamma_2\}$ a standard generating system of $\pi$ and $f$ the adapted framing (that is, ${\rm rot}^f(\gamma_1) = {\rm rot}^f(\gamma_2) = -1$). Then the corresponding generalized KV problem coincides with the classical KV problem in the formulation of \cite{AT12}. 
Furthermore consider a tangential derivation
$\bar{t} = (-{\rm ad}_{z_1 + z_2}, z_2,z_1) \in {\rm tDer}^+(L)$, 
and
let $\tilde{F}$ be a solution of the KV problem of type $(0,3)$ in the adapted framing. Then, for any $m\in \K$, the tangential automorphism $\exp(m \bar{t})\tilde{F}$ is also a solution in the adapted framing with the same Duflo functions. Indeed,
$$
{\sf div}^f_{\rm gr}(\bar{t}) = {\sf div}_z( - {\rm ad}_{z_1 + z_2}) - {\sf b}^f(\bar{t}) = -|z_1 + z_2| + |z_1 +z_2|=0.
$$
\end{rem}

The classical KV problem is known to admit solutions:

\begin{thm}[Proposition 6.2, Theorem 7.5 and Theorem 9.6 in \cite{AT12}] \label{thm:KVclassical}
The KV problem of type $(0,3)$ in the adapted framing admits solutions.
The Duflo functions $h_1, h_2, h$ of any solution agree modulo the linear part, and the even part is given by \eqref{eq:h_even}. 
Conversely, any functions $h_1, h_2, h$ satisfying these conditions are realized as Duflo functions of a solution of the KV problem of type $(0,3)$ in the adapted framing. 
\end{thm}

We will need the following more detailed properties of the generalized KV problem of type $(0,3)$ in different framings.
In fact, for any framing we can determine which functions appear as the Duflo functions of the corresponding KV problem. The difference between different framings is reflected in linear parts of Duflo functions.

To state the result, we introduce the following terminology.
Fixing a framing $f$, we say that a triad of functions $h_1, h_2 \in s\K[[s]]$, $h \in s^2\K[[s]]$ is {\em admissible} if they arise as Duflo functions of a solution for the KV problem of type $(0,3)$ in the framing $f$.
\begin{prop}  \label{prop:KV03fh}
For any framing, the KV problem of type $(0,3)$ in the corresponding framing have solutions.
Any admissible triad of functions satisfies the following condition:
\begin{align} 
& \text{The functions $h, h_1$ and $h_2$ agree modulo the linear part,} \nonumber \\
& \text{and the even part is given by equation \eqref{eq:h_even}}. \label{Dufloadm}
\end{align}
Moreover, we have the following:
\begin{enumerate}
\item[$(i)$] If $({\rm rot}^f(\gamma_1), {\rm rot}^f(\gamma_2)) =(0, -1)$, then $(h_1, h_2, h)$ satisfying \eqref{Dufloadm} is admissible if and only if the linear part of $h_1$ is zero.
\item[$(ii)$] If $({\rm rot}^f(\gamma_1), {\rm rot}^f(\gamma_2)) =(-1, 0)$, then $(h_1, h_2, h)$ satisfying \eqref{Dufloadm} is admissible if and only if the linear part of $h_2$ is zero.
\item[$(iii)$] If $({\rm rot}^f(\gamma_1), {\rm rot}^f(\gamma_2)) =(0,0)$, then $(h_1, h_2, h)$ satisfying \eqref{Dufloadm} is admissible if and only if $h_2(s) - h_1(s) = (1/2)s$. 
\item[$(iv)$] For all the other values of ${\rm rot}^f(\gamma_1)$ and ${\rm rot}^f(\gamma_2)$, any $(h_1, h_2, h)$ satisfying \eqref{Dufloadm} is admissible.
\end{enumerate}
\end{prop}

\begin{proof}
Let $\tilde{F}=(F, f_1, f_2)$ be a solution of the KV problem of type $(0,3)$ in framing $f$ with Duflo functions $h_1, h_2, h$.
As in Remark~\ref{rem:framing_genus_0}, we can view $\tilde{F}$ as a solution of the KV problem of type $(0,3)$ in the adapted framing by changing linear parts of $h_1, h_2, h$. By the second statement of Theorem \ref{thm:KVclassical}, $(h_1, h_2, h)$ satisfies the condition \eqref{Dufloadm}.

Let us analyze the linear parts of $h_1, h_2, h$.
We denote 
$$
\log(f_1) = az_1 + bz_2 + O(z^2), \hskip 0.3cm
\log(f_2)= cz_1 + dz_2 + O(z^2),
$$
and consider the first KV equation modulo terms of degree 3:
\begin{align*}
z_1 + z_2 & =  F(\log(e^{z_1}e^{z_2}))   \\
& = \log(f_1^{-1} e^{z_1} f_1 f_2^{-1} e^{z_2} f_2) \\
& \equiv_3 z_1 + z_2 + \frac{1}{2}[z_1, z_2] + [z_1, az_1 +bz_2]
+ [z_2, cz_1 + dz_2] \\
& = z_1 + z_2 + \big( \frac{1}{2} + b - c \big)[z_1, z_2].
\end{align*}
Hence $c-b=1/2$. We now consider the second KV equation
$$
{\sf j}^f_{\rm gr}(\tilde{F}) = {\sf j}_z(F) - {\sf c}^f(\tilde{F}) = |h_1(z_1) + h_2(z_2) - h(z_1 + z_2)|
$$
modulo terms of degree 2.
From the expression of $\log (f_1)$ and $\log (f_2)$ above, we have
$$
{\sf j}_z(F) \equiv_2 |d_1([z_1, az_1 + b z_2]) + d_2([z_2, cz_1 + dz_2])| = - b|z_2| - c|z_1|,
$$
and
$$
{\sf c}^f(\tilde{F})
= {\rm rot}^f(\gamma_1)|az_1+bz_2| + {\rm rot}^f(\gamma_2)|cz_1 +dz_2|.
$$
We denote $h_1(s)-h(s)=ks, h_2(s)-h(s)=ls$, and obtain the following equality:
$$
-(({\rm rot}^f(\gamma_2) +1) c + {\rm rot}^f(\gamma_1)a)|z_1|
-({\rm rot}^f(\gamma_2) d + ({\rm rot}^f(\gamma_1)+1)b)|z_2|=
k|z_1| + l|z_2|,
$$
which implies
\begin{equation}  \label{eq:k,l}
\begin{array}{lll}
  k & =   &  -{\rm rot}^f(\gamma_1)a - ({\rm rot}^f(\gamma_2) +1) c,  \\
  l & =   & - ({\rm rot}^f(\gamma_1)+1)b -{\rm rot}^f(\gamma_2) d. 
\end{array}
\end{equation}
If ${\rm rot}^f(\gamma_1) = 0, {\rm rot}^f(\gamma_2)=-1$, we conclude that $k=0$. If ${\rm rot}^f(\gamma_1) = -1, {\rm rot}^f(\gamma_2)=0$, we have $l=0$. If ${\rm rot}^f(\gamma_1)={\rm rot}^f(\gamma_2)=0$, we obtain
$$
l-k = c-b = \frac{1}{2}.
$$
For all other values of $({\rm rot}^f(\gamma_1), {\rm rot}^f(\gamma_2))$, equations 
\eqref{eq:k,l} define a full rank linear map, and the image of this map is $\mathbb{K}^2$.

In the other direction, let $\tilde{F}=(F, f_1, f_2)$ be a solution of the KV problem in the adapted framing with $h_1=h_2=h \in s^2\mathbb{K}[[s]]$. The latter condition implies $k=l=0$. Since in the adapted framing ${\rm rot}^f(\gamma_1) = {\rm rot}^f(\gamma_2) = -1$, we have $a=k=0, d=l=0$ and 
$$
\log(f_1) = bz_2 + O(z^2), \hskip 0.3cm 
\log(f_2)=cz_1 + O(z^2)
$$
with $c-b=1/2$.
Then $\tilde{F}' = \exp(m\bar{t})\tilde{F}$ is also a solution of the KV problem in the adapted framing with the same Duflo functions (see Remark \ref{rem:KVnorm}) and with $b'=b+m, c'=c+m$. Furthermore define $\tilde{F}''=(F', \exp(pz_1)f'_1,
\exp(qz_2)f'_2)$. This tangential automorphism is a solution of the first KV equation since $F''(\xi)= F'(\xi) =\omega$. Also note that ${\sf j}_z(F'')={\sf j}_z(F')={\sf j}_z(F)$. For the second KV equation, we compute
\begin{align*}
{\sf j}^f_{\rm gr}(\tilde{F}'') & =  {\sf j}_z(F'') - {\sf c}^f(\tilde{F}'') \\
& =  
{\sf j}_z(F') - {\rm rot}^f(\gamma_1)|\log(f''_1)| -
{\rm rot}^f(\gamma_2) |\log(f''_2)| \\
& = |h(z_1) + h(z_2) - h(z_1+z_2)| - (b+m)|z_2| - (c+m)|z_1| \\
& \hspace{1em} - {\rm rot}^f(\gamma_1)(p|z_1|+(b+m)|z_2|) - 
{\rm rot}^f(\gamma_2)((c+m)|z_1| + q|z_2|) \\
& = |h(z_1) + h(z_2) - h(z_1+z_2)| - (({\rm rot}^f(\gamma_2) +1)(c+m) + {\rm rot}^f(\gamma_1)p)|z_1| \\
& \hspace{1em} - (({\rm rot}^f(\gamma_1)+1)(b+m) + {\rm rot}^f(\gamma_2)q)|z_2|.
\end{align*}
If ${\rm rot}^f(\gamma_1) = 0, {\rm rot}^f(\gamma_2)=-1$, the extra linear part on the right hand side reads $(q-b-m)|z_2|$, where $(q-b-m)$ is an arbitrary element of $\K$ since $q$ and $m$ are. Similarly,
if ${\rm rot}^f(\gamma_1) = -1, {\rm rot}^f(\gamma_2)=0$, the linear term is given by  $(p-c-m)|z_1|$ with arbitrary $(p-c-m) \in \K$.
If ${\rm rot}^f(\gamma_1)={\rm rot}^f(\gamma_2)=0$, we have
$-((c+m)|z_1| -(b+m)|z_2|)$, and the difference between coefficients in front of $|z_2|$ and $|z_1|$ is $c-b=1/2$. Finally, if ${\rm rot}^f(\gamma_1)$ and ${\rm rot}^f(\gamma_2)$ take some other values, the extra linear term is in the image of a full rank linear map, and hence it is an arbitrary linear combination of $|z_1|$ and $|z_2|$.
\end{proof}

\subsection{Comparison of different framings} \label{subsec:comparison_diff_framings}
In this subsection, we compare the generalized KV problems in different framings. It turns out that with the exception of genus one, the KV problem in an arbitrary framing is always equivalent to the KV problem in the adapted framing.

We introduce an element
$$
{\bf q}^f=\sum_{j=1}^n q^f(z_j) |z_j|=\sum_{j=1}^n ({\rm rot}^f(\gamma_j) +1) |z_j| \in |A|.
$$
The pair $\chi^f=({\bf p}^f, {\bf q}^f)$ characterizes the framing $f$ with respect to the adapted framing associated to the generating system $\alpha_i, \beta_i, \gamma_j$. In particular, the adapted framing corresponds to $(0,0)$.

We will need the following notation.
We denote by 
$$
|A|_1=\mathbb{K}\langle |x_i|, |y_i|; i=1, \dots, g\rangle
$$
the vector space of degree one elements in $|A|$.
For any ${\bf x} \in |A|$, we denote by $[{\bf x}]_1 \in |A|_1$ the degree $1$ part of ${\bf x}$.

\begin{prop} \label{prop:KVcase1}
For $g \geq 2$, the KV problem of type $\gn$ for any framing $f$ is equivalent to the KV problem for the adapted framing.
\end{prop}

\begin{proof}
Let $\tilde{F} \in {\rm tAut}^+(L)$ be a solution of the KV problem for the framing $f$ characterized by $\chi^f=({\bf p}^f, {\bf q}^f)$.  We will show that there is a solution $\tilde{F}'$ of the KV problem for the adapted framing.
It will be convenient to use the KV problem as stated in Proposition~\ref{prop:F^-1}.

First consider a tangential derivation $\tilde{u} = (u,u_1,\ldots,u_n) \in {\rm tDer}^+(L)$ with $u_j=0$ for all $j$ and with non-vanishing components:
$$
u\colon x_1 \mapsto [x_2, x_1], \quad y_1 \mapsto [x_2, y_1], \quad
y_2 \mapsto [y_1, x_1].
$$
Observe that we have $u(\omega) = 0$ since
$$
u(\omega) = [x_1, [x_2, y_1]] + [[x_2, x_1], y_1] + [x_2, [y_1, x_1]]=0.
$$
Furthermore
$$
{\sf div}_{x,y,z}(u) = |d_{x_1}([x_2, x_1]) + d_{y_1}([x_2, y_1]) + d_{y_2}([y_1, x_1])| = 2|x_2|.
$$
By permuting labels, we conclude that for any ${\bf a} \in |A|_1$ there exists some $\tilde{u} \in {\rm tDer}^+(L)$ with vanishing tangential components such that $u(\omega) = 0$ and ${\sf div}_{x,y,z}(u) = {\bf a}$.
Then ${\sf b}^f(\tilde{u}) = {\sf b}^{\rm adp}(\tilde{u}) = 0$ and hence ${\sf c}^f(e^{\tilde{u}}) = {\sf c}^{\rm adp}(e^{\tilde{u}}) = 0$.
Moreover 
$$
{\sf j}^f_{\rm gr}(e^{\tilde{u}}) = 
{\sf j}_{x,y,z}(e^u)=\frac{e^u-1}{u}\, {\sf div}_{x,y,z}(u) = {\sf div}_{x,y,z}(u) = {\bf a}.
$$
Here we have used the facts that ${\sf div}_{x,y,z}(u)= {\bf a} \in |A|_1$ and that the images of all $x_i$ and $y_i$ generators under the action of $u$ are in $[L, L]$ and hence vanish under the $|\cdot |$-sign.

Let $\tilde{F} \in {\rm tAut}^+(L)$ be a solution of the KV problem for framing $f$. 
Applying the above construction, take $\tilde{u}\in {\rm tDer}^+(L)$ such that $u_j=0$ for all $j$, $u(\omega) = 0$ and
\[
{\sf div}_{x,y,z}(u) = {\bf a} := [{\sf c}^{\rm adp}(\tilde{F}^{-1})-
{\sf c}^f(\tilde{F}^{-1})]_1 + {\bf p}^f \in |A|_1.
\]
Set $\tilde{F}' = e^{-\tilde{u}} \tilde{F} \in {\rm tAut}^+(L)$. Since $u(\omega)=0$, $\tilde{F}'$ satisfies the first KV equation:
$$
F'(\xi) = e^{-u}(F(\xi)) = e^{-u}(\omega) = \omega.
$$
For the second KV equation, we compute 
%
\begin{align*}
{\sf j}^{\rm adp}_{\rm gr}(\tilde{F}'^{-1}) & =  {\sf j}^f_{\rm gr}(\tilde{F}'^{-1}) - ({\sf c}^{\rm adp}(\tilde{F}'^{-1}) - {\sf c}^f(\tilde{F}'^{-1})) \\
& = {\sf j}^f_{\rm gr}(\tilde{F}^{-1}) + F^{-1}({\sf j}^f_{\rm gr}(e^{\tilde{u}})) - ({\sf c}^{\rm adp}(\tilde{F}^{-1}) - {\sf c}^f(\tilde{F}^{-1})) \\
& = {\bf r} - {\bf p}^f + |h(\xi)| - \textstyle\sum_j |h_j(z_j)| + {\bf a} + \textstyle\sum_j a_j |z_j| - ({\bf a} - {\bf p}^f + \textstyle\sum_j \lambda_j |z_j|) \\
& = {\bf r} + |h(\xi)| - \textstyle\sum_j |h_j(z_j) + ( \lambda_j - a_j) z_j|.
\end{align*}
Here we have used the fact that ${\sf c}^f(e^{\tilde{u}}) = {\sf c}^{\rm adp}(e^{\tilde{u}}) = 0$ in the second line.
Also $\sum_j \lambda_j |z_j|$ is the degree $2$ part of the expression ${\sf c}^{\rm adp}(\tilde{F}^{-1})-
{\sf c}^f(\tilde{F}^{-1})$, and note that one can write $F^{-1}({\sf j}^f_{\rm gr}(e^{\tilde{u}})) = F^{-1}({\bf a}) = {\bf a} + \sum_j a_j |z_j|$ for some $a_j \in \K$.
Hence $\tilde{F}'$ is a solution of the KV problem for adapted framing, as required. 

Similarly one can run the argument in the other direction and show that a solution $\tilde{F}'$ of the KV problem with adapted framing can be modified to a solution $\tilde{F}=e^{\tilde{u}} \tilde{F}'$ of the KV problem with framing $f$.
\end{proof}

\begin{prop} \label{prop:KVcase2}
If $g=1$, $n\geq 1$ and ${\bf q}^f \neq 0$,
then the KV problem with framing $f$ is equivalent to the KV problem with the adapted framing.
\end{prop}

\begin{proof}
Without loss of generality, we may assume that $q^f(z_1) = {\rm rot}^f(\gamma_1) + 1 \neq 0$.
Denote $x = x_1$ and $y = y_1$ for simplicity, and consider the derivation $u$ of degree one of the form:
$$
u\colon x \mapsto  - l z_1, \hskip 0.3cm
y \mapsto k z_1, \hskip 0.3cm 
z_1 \mapsto [z_1, kx + ly]
$$
with $k,l \in \mathbb{K}$.
It lifts to a tangential derivation $\tilde{u} = (u, u_1, \ldots, u_n)$ with $u_1=kx+ly$ and $u_j=0$ for $j\geq 2$. This derivation has a property
$$
u(\omega)=u([x,y] + \sum_j z_j) = 
[-l z_1, y] + [x, k z_1] + [z_1, kx + ly] = 0.
$$
Furthermore
${\sf div}_{x,y,z}(u) = -|kx+ly|$, ${\sf b}^f(\tilde{u})= {\rm rot}^f(\gamma_1) |kx+ly|$
and
$$
{\sf div}^f_{\rm gr}(\tilde{u})={\sf div}_{x,y,z}(u) - {\sf b}^f(\tilde{u})= - ({\rm rot}^f(\gamma_1) +1) |kx +ly|
= - q^f(z_1) |kx + ly|
$$
which is an arbitrary element of $|A|_1$. 
Using the fact that
$
u(kx+ly) = 0,
$
we conclude that
$$
{\sf j}^f_{\rm gr}(e^{\tilde{u}}) = 
\frac{e^u -1}{u} ({\sf div}_{\rm gr}^f(\tilde{u}) )
= - q^f(z_1) |kx +ly|.
$$
Starting from this point, the proof of the previous proposition applies verbatim.
\end{proof}

\begin{prop}  \label{prop:KVcase3}
If $g=1$ and ${\bf q}^f = 0$, then the KV problem has a solution for at most one value of ${\bf p}^f$.
\end{prop}
\begin{proof}
Let $\tilde{F}$ be a solution of the KV problem with framing $\chi = ({\bf p}, 0)$ and $\tilde{F}'$ be a solution of the KV problem with framing $\chi' =({\bf p}', 0)$. Denote by $\tilde{u}$ the unique tangential derivation such that
$e^{\tilde{u}} =\tilde{F} \tilde{F}'^{-1}$. Note that
$$
e^u(\omega)=F(F'^{-1}(\omega))=F(\xi)=\omega
$$
which implies $u(\omega)=0$.

Let $\tilde{u}^1$ be the degree $1$ part of the derivation $\tilde{u}$.
It is of the form
$$
u^1(x) = k [x,y] + \sum_j \mu_j z_j, \hskip 0.3cm
u^1(y)= l[x,y] + \sum_j \nu_j z_j, \hskip 0.3cm
u^1(z_j) = [z_j, u^1_j],
$$
where $u^1_j$'s are linear combinations of $x$ and $y$. The derivation $u^1$ satisfies the condition $u^1(\omega)=0$ which reads
$$
0 = u^1(\omega) = k[[x,y],y] + l[x, [x,y]] + (\text{terms including $z_j$'s}).
$$
Therefore $k=l=0$.

Next we compute the following expression:
$$
[{\sf j}^f_{\rm gr}(e^{\tilde{u}})]_1 =
{\sf div}_{x,y,z}(u^1) - {\sf b}^f(\tilde{u}^1) = - \sum_j |u^1_j| + \sum_j |u^1_j| =0,
$$
where we have used fact $q^f(z_j)={\rm rot}^f(\gamma_j) + 1 =0$ and
$
{\sf b}^f(\tilde{u}_1) = \sum_j {\rm rot}^f(\gamma_j) |u^1_j| = - \sum_j |u^1_j|
$.
Using the second KV equation, we can transform the left hand side as follows:
\begin{align*}
[{\sf j}^f_{\rm gr}(e^{\tilde{u}})]_1 
& = [{\sf j}^f_{\rm gr}(\tilde{F} \tilde{F}'^{-1})]_1 \\ 
& = [F({\sf j}^f_{\rm gr}(\tilde{F}'^{-1}) - {\sf j}^f_{\rm gr}(\tilde{F}^{-1}))]_1 \\
& = [F((|h'(\xi)| - \textstyle\sum_j |h'_j(z_j)| - {\bf p}') -(|h(\xi)| - \textstyle\sum_j |h_j(z_j)| - {\bf p})) ]_1 \\
& = {\bf p} - {\bf p}'.
\end{align*}
Here we have used the fact that variables $z_j$ are of degree 2 and cannot contribute in the degree 1 part. The same argument applies to the expression $\xi$.

By comparing two displayed formulas above, we conclude that ${\bf p} - {\bf p}'=0$, as required.
\end{proof}

\section{Solving KV}
\label{sec:solveKV}
In this section, we show that in many cases the higher genus KV problems admit solutions.

We denote by  ${\rm SolKV}_{(g,n+1)}^f$ the space of solutions of the KV problem of type $(g, n+1)$ with framing $f$.
Furthermore, denote by ${\rm SolKV}_\gn$ the space of solutions of the KV problem with the adapted framing, i.e.,
${\rm SolKV}_\gn = {\rm SolKV}_\gn^{f^{\rm adp}}$.

\begin{thm}
\label{thm:KVsolve}
The KV problem admits solutions in the following cases:
\begin{itemize}
    \item ${\rm SolKV}_{(g,n+1)}^f \neq \emptyset$ for $g\geq 2$ and for any framing $f$.
    \item ${\rm SolKV}_{(1,n+1)}^f \neq \emptyset$ if and only if ${\bf q}^f \neq 0$ or 
    ${\bf q}^f=0$ and ${\bf p}^f=0$.
\end{itemize}
\end{thm}

In view of Propositions \ref{prop:KVcase1}, \ref{prop:KVcase2} and \ref{prop:KVcase3}, this is equivalent to showing that the KV problem of type $\gn$ admits solutions for the adapted framing. 

In order to emphasize the genus and the number of boundary components of the surface under consideration, we will use the following notation:
for a surface $\Sigma_{g,n+1}$ of genus $g$ with $n+1$ boundary components, we write
\[
{\rm tDer}_{(g,n+1)}= {\rm tDer}_{\{ z_1, \ldots, z_n \}}(L(x_i, y_i, z_j))
\]
where $i$ and $j$ run through $1, \dots, g$ and $1, \ldots, n$, respectively, and we denote by ${\rm tDer}^+_{(g,n+1)}$ the Lie subalgebra of ${\rm tDer}_{(g,n+1)}$ that consists of derivations of positive degree.
In a similar fashion, we write
\[
{\rm tAut}^+_{(g,n+1)} := {\rm tAut}^+_{\{ z_1, \ldots, z_n \}}(L(x_i, y_i, z_j)).
\]

\subsection{Elliptic case} \label{subsec:elliptic}
 Figure \ref{fig:elliptic} shows that one can construct a surface $\Sigma_{1,1}$ by gluing an annulus to the first and second boundaries of $\Sigma_{0,3}$.
 If we consider the framing on $\Sigma_{0,3}$ given by ${\rm rot}^f(\gamma_1) = {\rm rot}^f(\gamma_2) = 0$ (that is, ${\bf q}^f = |z_1| + |z_2|$), then it extends to the adapted framing on $\Sigma_{1,1}$.
 Inspired by this observation and by algebraic constructions of \cite{Enriquez}, we will build solutions of the KV problem of type $(1,1)$ starting from solutions of the KV problem of type $(0,3)$.

\begin{figure}
\begin{center}
\input{fig_elliptic.tex}
\end{center}
\caption{Elliptic case}
\label{fig:elliptic}
\end{figure}

For the surface $\Sigma_{1,1}$, we simply denote $\alpha = \alpha_1$ and $\beta = \beta_1$ for the standard generators, and $x = x_1$ and $y = y_1$ for the exponential generators.

In Figure \ref{fig:elliptic}, the inclusion homomorphism $\pi_1(\Sigma_{0,3}) \to \pi_1(\Sigma_{1,1})$ sends  $\gamma_1$ to $\alpha \beta \alpha^{-1}$ and $\gamma_2$ to $\beta^{-1}$, respectively.
With this in mind, we introduce a Lie algebra homomorphism $\psi \colon L(z_1, z_2) \to L(x,y)$ defined as follows: 
$$
\psi \colon
\begin{cases}
z_1 \mapsto \psi_1(x,y) = e^x y e^{-x}=
y + [x,y] + \frac{1}{2}[x,[x,y]] + \dots, \\ 
z_2 \mapsto \psi_2(x,y)=-y.
\end{cases}
$$
Furthermore define a map
$\mathcal{E} \colon {\rm tAut}^+_{(0,3)} \to {\rm tAut}^+_{(1,1)}$ given by formula
$$
    \tilde{F} = (F, f_1, f_2) \longmapsto \mathcal{E}(\tilde{F}) = F^\text{ell} \colon \begin{cases}
    \alpha \mapsto
\psi(f_1)^{-1} \alpha \psi(f_2), \\
    \beta \mapsto
\psi(f_2)^{-1} \beta \psi(f_2),
    \end{cases}
$$
where $\alpha=e^x$ and $\beta=e^y$, and by abuse of notation we denote by $\psi$ the lift of the Lie algebra homomorphism $\psi$ to the corresponding enveloping algebras.
Note that the transformation property of $\beta$ implies
$$
F^{\rm ell}(y)=\psi(f_2)^{-1} y \psi(f_2).
$$

The associated Lie algebra map $\mathcal{E}\colon {\rm tDer}^+_{(0,3)} \rightarrow {\rm tDer}^+_{(1,1)}$ is of the form
$$
    \tilde{u} = (u, u_1, u_2) \longmapsto u^\text{ell} \colon \begin{cases}
    \alpha \mapsto \alpha \psi(u_2) - \psi(u_1) \alpha = [\alpha, \psi(u_2)] + \psi(u_2-u_1) \alpha, \\
    \beta \mapsto [\beta, \psi(u_2)].
    \end{cases}
$$
It is instructive to write $u^{\rm ell}$ in terms of the exponential generators $x,y$:
\begin{equation} \label{eq:u^ellxy}
u^{\rm ell}(x) = \frac{{\rm ad}_x}{1 - e^{-{\rm ad}_x}} (\psi(u_2 - u_1)) + [x, \psi(u_1)],
\hskip 0.3cm
u^{\rm ell}(y)=[y, \psi(u_2)].
\end{equation}

\begin{prop} \label{prop:Fellhom}
The map $\mathcal{E}\colon \tilde{F} \mapsto F^\text{ell}$ is a group homomorphism, and we have $F^{\rm ell} \circ \psi = \psi \circ F$ for any $\tilde{F} \in {\rm tAut}^+_{(0,3)}$. 
Furthermore the map $\mathcal{E}\colon \tilde{u} \mapsto u^\text{ell}$ is a Lie algebra homomorphism, and $u^{\rm ell} \circ \psi = \psi \circ \tilde{u}$ for any $\tilde{u} \in {\rm tDer}^+_{(0,3)}$.
\end{prop}

\begin{proof}
First we prove $F^{\rm ell} \circ \psi = \psi \circ F$.
Observe that
\begin{align*}
F^{\rm ell}(\psi_1) = F^{\rm ell}(\alpha y \alpha^{-1})
 & = (\psi(f_1)^{-1} \alpha\,  \psi(f_2))(\psi(f_2)^{-1} y\,  \psi(f_2))(\psi(f_2)^{-1} \alpha^{-1} \psi(f_1)) \\
 & = \psi(f_1)^{-1} \alpha y \alpha^{-1} \psi(f_1) \\
 & = \psi({f_1}^{-1} z_1 f_1) = \psi(F(z_1)),
\end{align*}
\[
 F^{\rm ell}(\psi_2) = F^{\rm ell}(-y) = \psi(f_2)^{-1} (-y) \psi(f_2) 
 = \psi({f_2}^{-1} z_2 f_2) = \psi(F(z_2)).
\]
Since $z_1$ and $z_2$ generate the free Lie algebra $L(z_1, z_2)$,
we conclude that $F^{\rm ell} \circ \psi = \psi \circ F$. 

Let $\tilde{F}, \tilde{G} \in {\rm tAut}^+_{(0,3)}$.
By the group law for ${\rm tAut}^+_{(0,3)}$, the tangential components of $\tilde{K} = \tilde{F} \circ \tilde{G}$ is of the form
$k_j = f_j F(g_j)$.
We compute
\begin{align*}
  F^{\rm ell} \circ G^{\rm ell}(\alpha)& = F^{\rm ell}(\psi(g_1)^{-1} \alpha\, \psi(g_2)) \\
  & = \psi(F(g_1))^{-1}
  \psi(f_1)^{-1} \alpha \, \psi(f_2)
  \psi(F(g_2)) \\
  & = \psi(k_1)^{-1} \alpha \, \psi(k_2) = K^{\rm ell}(\alpha), \\
  F^{\rm ell} \circ G^{\rm ell}(\beta) & = F^{\rm ell}(\psi(g_2)^{-1}
  \beta\, \psi(g_2)) \\
  & = \psi(F(g_2))^{-1}
  \psi(f_2)^{-1} \beta\, \psi(f_2)
  \psi(F(g_2)) \\
  & = \psi(k_2)^{-1} \beta\, \psi(k_2) = K^{\rm ell}(\beta).
\end{align*}
This completes the proof of the group homomorphism property of the map $\tilde{F} \mapsto F^{\rm ell}$. The corresponding statements for Lie algebras follow by differentiation.
\end{proof}

\begin{rem} \label{rem:kerI}
The kernel of the map $\mathcal{E}\colon \tilde{F} \mapsto F^{\rm ell}$ is equal to
\[
\{ ({\rm id}, e^{\lambda z_1}, e^{-\lambda z_2}) \in {\rm tAut}^+_{(0,3)} ; \lambda \in \K \} \cong \K.
\]
To see this, suppose that $\tilde{F} = (F, f_1, f_2) \in {\rm tAut}^+_{(0,3)}$ satisfies $F^{\rm ell} = {\rm id}$.
For any $a\in A(z_1,z_2)$, we have
$
\psi( F(a) ) = F^{\rm ell}( \psi(a) ) = \psi(a)
$.
Since $\psi$ is injective, we conclude that $F = {\rm id}$.
Therefore, since $F(z_1) = f_1^{-1} z_1 f_1$ and $F(z_2) = f_2^{-1} z_2 f_2$, there exist scalars $\lambda_1, \lambda_2 \in \K$ such that
\[
f_1 = \exp(\lambda_1 z_1), \qquad
f_2 = \exp(\lambda_2 z_2).
\]
Hence $\psi(f_1) = \exp (e^x (\lambda_1y) e^{-x}) = e^x e^{\lambda_1y} e^{-x}$ and $\psi(f_2) = \exp( -\lambda_2 y)$.
Getting these equations back to the defining formula of $F^{\rm ell}$, we obtain
\[
e^x = F^{\rm ell}(e^x) = \psi(f_1)^{-1} e^x \psi(f_2) = 
(e^x e^{ -\lambda_1 y} e^{-x}) e^x e^{-c_2y} = e^x e^{-(\lambda_1+\lambda_2)y}.
\]
Therefore we obtain $\lambda_1 + \lambda_2 = 0$.
Conversely it is easy to check that elements of the form $({\rm id}, \exp(\lambda z_1), \exp(-\lambda z_2))$ are in the kernel.
\end{rem}

The next proposition describes the relation between non-commutative divergence under the Lie homomorphism $\mathcal{E}\colon \tilde{u} \mapsto u^{\rm ell}$.

\begin{prop} \label{prop:ell_Lie}
For all $\tilde{u} \in {\rm tDer}^+_{(0,3)}$, we have
\[
\psi ({\sf div}_{z}(\tilde{u})) = {\sf div}_{x,y}(u^{\rm ell}) + 
u^{\rm ell}(|r(x)|).
\]
Here, ${\sf div}_z$ and ${\sf div}_{x,y}$ are divergence cocycles on ${\rm tDer}^+_{(0,3)}$ and ${\rm tDer}^+_{(1,1)} = {\rm Der}^+_{(1,1)}$, respectively.
\end{prop}

\begin{rem}
Since $u^{\rm ell}(|r(y)|) = 0$, one can write the right hand side of the formula above as
${\sf div}_{x,y}(u^{\rm ell}) + u^{\rm ell}(|r(x) + r(y)|)$, which looks more symmetric in $x$ and $y$.
\end{rem}
\begin{proof}[Proof of Proposition \ref{prop:ell_Lie}]
Let $\tilde{u} = (u,u_1,u_2) \in {\rm tDer}^+_{(0,3)}$.
We compute
\[
{\sf div}_{z}(\tilde{u}) = 
| d_1([z_1, u_1]) + d_2([z_2,u_2]) |
= | z_1 (d_1 u_1) - u_1 + z_2 (d_2u_2) - u_2 |.
\]
Next we compute ${\sf div}_{x,y}(u^{\rm ell})$.
Referring to formula~\eqref{eq:u^ellxy}, we have
$
{\sf div}_{x,y}(u^{\rm ell}) = 
\Psi_1 + \Psi_2,
$
where
\[
\Psi_1 = | d_x \big( \textstyle\frac{{\rm ad}_x}{1 - e^{-{\rm ad}_x}} (\psi(u_2 - u_1) ) \big) |,
\quad
\Psi_2 =  |d_x ( [x, \psi(u_1)] + d_y( [y, \psi(u_2)])|.
\]
Since $\psi(z_1) = e^x y e^{-x}$ and $\psi_2 = -y$, we have
\[
d_x(\psi(z_1)) = e^x y ( \textstyle\frac{e^{-x} - 1}{x} ),
\quad
d_x(\psi(z_2)) = 0,
\quad
d_y(\psi(z_1)) = e^x,
\quad
d_y(\psi(z_2)) = -1.
\]
Therefore, using the chain rule formula (see Lemma \ref{ex:chainrule_Lie}), for any $a\in L(z_1, z_2)$ we have
\begin{align}
& d_x(\psi(a)) = \psi(d_1 a) d_x ( \psi(z_1)) + \psi(d_2 a) d_x (\psi(z_2))
= \psi( d_1 a) e^x y ( \textstyle\frac{e^{-x} - 1}{x} ), \nonumber \\
& d_y(\psi(a)) = \psi(d_1 a) d_y( \psi(z_1)) + \psi(d_2 a) d_y(\psi(z_2))
= \psi( d_1 a) e^x - \psi(d_2 a). \label{lem:dxypsiu}
\end{align}
Using Lemma~\ref{lem:dsfadxa} and \eqref{lem:dxypsiu} and the fact that $s/(1-e^{-s}) = s r'(s) + 1$, we compute
\begin{align*}
\Psi_1 
& = | \textstyle\frac{x}{1 - e^{-x}} d_x ( \psi(u_2 - u_1)) - r'(x) \psi(u_2 - u_1) | \\
& = - | \psi(d_1(u_2 - u_1)) e^x y + r'(x) \psi(u_2 - u_1) |.
\end{align*}
Using \eqref{lem:dxypsiu}, we compute
\begin{align*}
\Psi_2 & = 
| x d_x(\psi(u_1)) - \psi(u_1) + y d_y (\psi(u_2)) - \psi(u_2) | \\
& = | x \psi(d_1 u_1) e^x y( \textstyle\frac{e^{-x} - 1}{x}) - \psi(u_1) | + | y(\psi(d_1 u_2) e^x - \psi(d_2 u_2)) - \psi(u_2) |.
\end{align*}
Therefore we conclude that
\begin{align*}
{\sf div}_{x,y}(u^{\rm ell})  & = 
- | r'(x) \psi ( u_2 - u_1) | 
+ |\psi(d_1 u_1) e^x ye^{-x} - \psi(u_1) - y \psi(d_2 u_2) - \psi(u_2) | \\
& = - | r'(x) \psi ( u_2 - u_1) |  + \psi( {\sf div}_{z_1,z_2}(\tilde{u})).
\end{align*}
Finally, using formula~\eqref{lem:dxypsiu} again, we compute
\[
u^{\rm ell}(| r(x) |) = 
| u^{\rm ell}(x) r'(x)  |
= |\psi(u_2 - u_1) r'(x) |.
\]
Therefore
$
{\sf div}_{x,y}(u^{\rm ell}) = - u^{\rm ell}(|r(x)|) + \psi({\sf div}_{z}(\tilde{u}))$, as was to be shown.
\end{proof}

As a corollary of Proposition~\ref{prop:ell_Lie}, for any $\tilde{F} \in {\rm tAut}^+_{(0,3)}$ we obtain
\begin{equation}
\label{eq:psijF}
\psi({\sf j}_{z_1, z_2}(F)) =  {\sf j}_{x,y}(F^{\rm ell}) + F^{\rm ell}(|r(x)|) -
|r(x)|. 
\end{equation}

Define an automorphism of $\zeta \in {\rm Aut}^+_{(1,1)}$ by 
\begin{equation} \label{eq:auto_zeta}
\zeta \colon x \mapsto x, \hskip 0.3cm y \mapsto \frac{{\rm ad}_x}{e^{\ad_x}-1}\, y.
\end{equation}
\begin{lem} \label{lem:zeta_properties}
We have $\zeta(\psi(z_1 + z_2)) = [x,y]$ and ${\sf j}_{x,y}(\zeta) = - |r(x)|$.
\end{lem}
\begin{proof}
We compute  
\[
\zeta(\psi(z_1+z_2)) = \zeta(e^x ye^{-x} -y) = \frac{{\rm ad}_x}{e^{\ad_x}-1}\, (e^{{\rm ad}_x}-1)\, y = [x,y],
\]
which proves the first equality.
To prove the second equality, notice that the derivation $u:= \log(\zeta)$ is given by $u(x) = 0$ and $u(y) = \log(\frac{{\rm ad}_x}{e^{\ad_x} -1 }) y=- r(\ad_x) y$, and it holds that ${\sf div}_{x,y}(u) = -|r(x)|$. 
Hence
\[
{\sf j}_{x,y}(\zeta) = \frac{e^{u} - 1}{u} {\sf div}_{x,y}(u) = -|r(x)|.
\]
\end{proof}

\begin{thm}  \label{thm:KVelliptic}
Let $\tilde{F} = (F, f_1, f_2) \in {\rm tAut}^+_{(0,3)}$ be a solution of the KV problem of type $(0,3)$ with Duflo functions $h_1, h_2, h$ associated to the framing $f$ with ${\rm rot}^f(\gamma_1) = {\rm rot}^f(\gamma_2) = 0$.
Then $\hat{F} = \zeta F^{\rm ell} \in {\rm tAut}^+_{(1,1)}$ is a solution of the KV problem of type $(1,1)$ with Duflo function $h$ associated to the adapted framing ${\rm rot}^{\rm adp}(\alpha) = {\rm rot}^{\rm adp}(\beta)=0$.
\end{thm}

\begin{proof}
First we show that $\hat{F}$ verifies the first KV equation:
\begin{align*}
\hat{F}(\log(e^xe^ye^{-x}e^{-y})) 
& = \zeta(F^{\rm ell}(\log(e^{\psi_1}e^{\psi_2}))) \\ 
& =  \zeta(\psi(F(\log(e^{z_1}e^{z_2})))) \\
& = \zeta(\psi(z_1+z_2)) = [x,y].
\end{align*}
Here we have used Lemma~\ref{lem:zeta_properties} in the last line.

Next we proceed with establishing the second KV equation for $\hat{F}$.
By Proposition~\ref{prop:KV03fh}, the second KV equation for $\tilde{F}$ reads 
\[
{\sf j}_{z_1, z_2}(F) = |h(z_1) + h(z_2) - h(z_1+z_2)| + a |z_1| + b |z_2|
\]
for some $a, b \in \K$ with $b - a = 1/2$.
Also, we have $h(s) + h(-s)=2 h_{\rm even}(s) = \frac{1}{2}s - r(s)$.

Now we compute $\zeta(\psi({\sf j}_{z_1, z_2}(F)))$ in two ways.
First, using the results in the preceding paragraph, we compute
\begin{align*}
\zeta(\psi({\sf j}_{z_1, z_2}(F))) & =  \zeta(\psi(|h(z_1) + h(z_2) - h(z_1 + z_2)| + a|z_1| + b|z_2|))\\
& = \zeta(|h(y) + h(-y) + (a - b)y|) - |h([x, y])| \\
& = - \zeta(|r(y)|) - |h([x, y])| \\
& = - \hat{F}(|r(y)|) - |h([x, y])|.
\end{align*}
Here we have used that $\hat{F}(|r(y)|) = \zeta(F^{\rm ell}(|r(y)|)) = \zeta(|r(y)|)$.
Second, using equation~\eqref{eq:psijF} and Lemma~\ref{lem:zeta_properties}, we compute 
\begin{align*}
\zeta(\psi({\sf j}_{z_1, z_2}(F))) & =  \zeta \big( {\sf j}_{x,y}(F^{\rm ell}) + F^{\rm ell}(|r(x)|) -|r(x)| \big) \\
& = {\sf j}_{x,y}(\zeta F^{\rm ell}) - {\sf j}_{x,y}(\zeta) + \hat{F}(|r(x)|) - |r(x)| \\
& = {\sf j}_{x,y}(\hat{F}) + \hat{F}(|r(x)|).
\end{align*}

Putting things together, we obtain the second KV equation for the element $\hat{F} \in {\rm Aut}^+_{(1,1)}$:
$$
{\sf j}_{x,y}(\hat{F}) + \hat{F}(|r(x)+r(y)|)=-|h([x, y])|.
$$
This completes the proof.
\end{proof}

\subsection{Gluing}

\begin{figure}
\begin{center}
\input{fig_gluing.tex}
\end{center}
\caption{Gluing}
\label{fig:glue}
\end{figure}

As shown on Figure \ref{fig:glue}, one can glue surfaces $\Sigma_{g_1,n_1+1}$, $\Sigma_{g_2,n_2+1}$ and $\Sigma_{0,3}$ to get a surface $\Sigma$ of genus $g_1+g_2$ with $n_1+n_2+1$ boundary components.
The purpose of this subsection is to establish an algebraic counterpart of this procedure: to construct solutions of the KV problem of type $(g_1+g_2, n_1+n_2+1)$ starting from a triple of solutions of types $(g_1, n_1+1), (g_2, n_2+1)$ and $(0,3)$.

\begin{rem} \label{rem:glue_rotation}
The framings $f_1$ and $f_2$ on $\Sigma_1$ and $\Sigma_2$  uniquely determine the framing $f_0$ on $\Sigma_{0,3}$ and the framing $f$ on the glued surface $\Sigma=\Sigma_{g_1+g_2,n_1+n_2+1}$. The framing $f_0$ is characterized by rotation numbers of the boundary loops $\gamma_{10}$ and $\gamma_{20}$:
\begin{equation} \label{eq:rot_gluing}
{\rm rot}^{f_0}(\gamma_{k0})=2g_k - 1 +n_k + \sum_{j=1}^{n_k} {\rm rot}^{f_k}(\gamma_{kj}), 
\qquad
k = 1, 2.
\end{equation}
If $f_1$ and $f_2$ are adapted framings,
we have 
 $\rot^{f_0}(\gamma_{10})=2g_1-1$ and $\rot^{f_0}(\gamma_{20})=2g_2-1$.
\end{rem}

Inspired by the gluing of Figure \ref{fig:glue},  we denote the generators of the free Lie algebra associated to $\Sigma=\Sigma_{g_1 + g_2, n_1 + n_2 +1}$by  $x_{1i},y_{1i},z_{1j}$ with $i=1, \dots, g_1, j=1, \dots, n_1$ and $x_{2i},y_{2i},z_{2j}$ with $i=1, \dots, g_2, j=1, \dots, n_2$, respectively. We define a Lie algebra homomorphism
\begin{align*}
{\rm tDer}_{(g_1,n_1+1)} \oplus {\rm tDer}_{(g_2,n_2+1)} &\longrightarrow {\rm tDer}_{(g_1+g_2, n_1+n_2+1)},
\quad
(\tilde{u}, \tilde{v}) \mapsto \tilde{u} + \tilde{v}
\end{align*}
by
$$
(u+v) \colon
\begin{cases}
x_{1i} \mapsto u(x_{1i}), \quad y_{1i} \mapsto u(y_{1i}), \quad z_{1j} \mapsto u(z_{1j}), \\
x_{2i} \mapsto v(x_{2i}), \quad y_{2i} \mapsto v(y_{2i}), \quad z_{2j} \mapsto v(z_{2j}),
\end{cases}
$$
and for the tangential components, we have
$$
(\tilde{u} + \tilde{v})_{1j}=u_j, \hskip 0.3cm
(\tilde{u} + \tilde{v})_{2j} = v_j.
$$
This Lie algebra homomorphism lifts to a group homomorphism
\[
{\rm tAut}^+_{(g_1,n_1+1)} \times {\rm tAut}^+_{(g_2,n_2+1)} \to {\rm tAut}^+_{(g_1+g_2,n_1+n_2+1)}, \quad 
(\tilde{F}_1, \tilde{F}_2) \mapsto \tilde{F}_1 \times \tilde{F}_2.
\]
Naturally the cocycles ${\sf div}$ and ${\sf b}^f$ have the property
$$
{\sf div}(u+v) = {\sf div}(u) + {\sf div}(v), \hskip 0.3cm
{\sf b}^f(\tilde{u}+\tilde{v}) = {\sf b}^{f_1}(\tilde{u}) + {\sf b}^{f_2}(\tilde{v}).
$$
These relations between the Lie algebra 1-cocycles lift to similar relations between the group 1-cocycles ${\sf j}$ and ${\sf c}^f$.

Furthermore we introduce a map
\[
\mathcal{P}\colon {\rm tDer}_{(0,3)} \rightarrow {\rm tDer}_{(g_1+g_2, n_1+n_2+1)}
\]
as follows: for $\tilde{u}=(u, u_1, u_2)\in \tDer_{(0,3)}$, $\mathcal{P}(\tilde{u})$ is the derivation which sends
\[
\mathcal{P}(\tilde{u}) \colon
\begin{cases}
x_{1i} \mapsto [x_{1i},u_1(\omega_1,\omega_2)], \quad y_{1i} \mapsto [y_{1i},u_1(\omega_1,\omega_2)], \quad
z_{1j} \mapsto [z_{1j}, u_1(\omega_1,\omega_2)], \\
x_{2i} \mapsto [x_{2i},u_2(\omega_1,\omega_2)], \quad y_{2i} \mapsto [y_{2i},u_2(\omega_1,\omega_2)], \quad
z_{2j} \mapsto [z_{2j}, u_2(\omega_1,\omega_2)],
\end{cases}
\]
where
\[
\omega_1 = \sum_{i=1}^{g_1} [x_{1i},y_{1i}] + \sum_{j=1}^{n_1} z_{1j}, \quad \omega_2 = \sum_{i=1}^{g_2} [x_{2i},y_{2i}] + \sum_{j=1}^{n_2} z_{2j}.
\]
The tangential components of $\mathcal{P}(\tilde{u})$ corresponding to $z_{1j}$ are equal to $u_1(\omega_1, \omega_2)$, and the tangential components corresponding to $z_{2j}$ are equal of $u_2(\omega_1, \omega_2)$.

\begin{prop}
The map $\mathcal{P}$ is a Lie algebra homomorphism.
\end{prop}
\begin{proof}
This follows from a direct computation using that for $k=1,2$ we have $\mathcal{P}(\tilde{u})(\omega_k) = [\omega_k, u_k(\omega_1,\omega_2)]$.
\end{proof}

The map $\mathcal{P}$ lifts to a group homomorphism
$\mathcal{P}\colon {\rm tAut}^+_{(0,3)} \to {\rm tAut}^+_{(g_1 + g_2, n_1 + n_2 +1)}$.
For any $a = a(z_1, z_2) \in L(z_1, z_2)$, it holds that $\mathcal{P}(\tilde{F})(a(\omega_1, \omega_2))
= (F(a))(\omega_1, \omega_2)$.

In what follows, ${\sf div}$ stands for ${\sf div}_{x,y,z}$, where $x_i, y_i, z_j$ are the exponential generators.

\begin{prop}  \label{prop:divP}
For all $\tilde{u} \in {\rm tDer}_{(0,3)}$, we have
$$
{\sf div}(\mathcal{P}(\tilde{u})) = \left({\sf div}(u) + (1 - 2g_1 - n_1)|u_1| + (1 - 2g_2 - n_2)|u_2|\right)_{z_k = \omega_k}.
$$
Moreover, for any $\tilde{F} \in {\rm tAut}^+_{(0,3)}$ we have
\[
{\sf j}(\mathcal{P}(\tilde{F})) = 
\big( {\sf j}(F) + (1 - 2g_1 - n_1) {\sf c}_1(\tilde{F}) + (1 -2g_2 - n_2) {\sf c}_2(\tilde{F}) \big)_{z_k = \omega_k}.
\]
\end{prop}

\begin{proof}
Note that $d_{x_{ki}} \omega_k = - y_{ki}$, $d_{y_{ki}} \omega_k = x_{ki}$, $d_{z_{kj}} \omega_k = 1$, and $d_{x_{ki}} \omega_{k'} = d_{y_{ki}} \omega_{k'} = d_{z_{kj}} \omega_{k'} = 0$ if $k \neq k'$. 
Using the chain rule formula of partial derivatives (Lemma~\ref{ex:chainrule_Lie}), we compute
\begin{align*}
{\sf div}(\mathcal{P}(\tilde{u})) & = 
\textstyle\sum_{k=1}^2 \big( \textstyle\sum_{i=1}^{g_k} |x_{ki}(d_k u_k)(\omega_1, \omega_2)(-y_{ki}) - u_k | \\
& \hspace{1em} + \textstyle\sum_{i=1}^{g_k}
| y_{ki} (d_k u_k)(\omega_1, \omega_2) x_{ki} - u_k| + 
\textstyle\sum_{j=1}^{n_k} |z_{kj} (d_k u_k)(\omega_1, \omega_2) - u_k| \big) \\
& = 
|\omega_1 (d_1 u_1)(\omega_1, \omega_2) + \omega_2 (d_2 u_2)(\omega_1, \omega_2)| -
(2g_1 + n_1)|u_1| - (2g_2 + n_2)|u_2|.
\end{align*}
Notice that
\[
{\sf div}(u) = | d_1([z_1, u_1]) + d_2([z_2, u_2]) | =
| z_1(d_1 u_1) + z_2(d_2 u_2)| - |u_1| - |u_2|.
\]
Hence the first assertion follows.
The second assertion follows by integration.
\end{proof}

\begin{prop}\label{prop:BfunderP}
For $\tilde{F} \in {\rm tAut}^+_{(0,3)}$, we have
$
{\sf j}^f_{\rm gr}(\mathcal{P}(\tilde{F})) = {\sf j}^{f_0}_{\rm gr}(\tilde{F})|_{z_k = \omega_k}
$.
\end{prop}

\begin{proof}
Observe that
$$
{\sf b}^f(\mathcal{P}(\tilde{u})) = \sum_{k=1}^2 |u_k(\omega_1, \omega_2)| \sum_{j=1}^{n_k} \, {\rm rot}(\gamma_{kj}) =  
\big( \sum_{j=1}^{n_1} {\rm rot}^f(\gamma_{1j}) |u_1| + \sum_{j=1}^{n_2} {\rm rot}^f(\gamma_{2j}) |u_2|\big)_{z_k = \omega_k}.
$$
Then
$$
{\sf c}^f(\mathcal{P}(\tilde{F}))  =  \big( \sum_{j=1}^{n_1} {\rm rot}^f(\gamma_{1j}) {\sf c}_1(\tilde{F}) +
\sum_{j=1}^{n_2} ({\rm rot}^f(\gamma_{2j}) {\sf c}_2(\tilde{F}) \big)_{z_k = \omega_k}.
$$
Using Proposition~\ref{prop:divP} and equation~\eqref{eq:rot_gluing}, we obtain
$$
{\sf j}^f_{\rm gr}(\mathcal{P}(\tilde{F})) = {\sf j}(\mathcal{P}(\tilde{F})) - {\sf c}^f(\mathcal{P}(\tilde{F})) =
({\sf j}(F) - {\sf c}^{f_0}(\tilde{F}))_{z_k = \omega_k} =  {\sf j}^{f_0}_{\rm gr}(\tilde{F})|_{z_k = \omega_k},
$$
as required.
\end{proof}

We are now ready to state the relation between solutions of the KV problem associated to the surface $\Sigma$ and to its pieces $\Sigma_{g_1, n_1+1}, \Sigma_{g_2, n_2+1}$ and $\Sigma_{0,3}$.

\begin{thm}   \label{thm:KVgluing}
Suppose $n_1=0$ or $g_2=0$.
Let
$
\tilde{F}_1 \in {\rm SolKV}_{(g_1, n_1+1)}
$
and
$
\tilde{F}_2 \in {\rm SolKV}_{(g_2, n_2+1)}
$
be solutions of KV problems in the adapted framing with Duflo functions 
$h^{(k)}_1,\ldots,h^{(k)}_{n_k}$, $h^{(k)}$ for $\tilde{F}_k$ ($k=1,2$), and 
let 
$
\tilde{F} \in {\rm SolKV}_{(0,3)}^{f_0}
$
be a solution of the KV problem with framing $f_0$ with ${\bf q}^{f_0} =2g_1 |z_1| + 2g_2 |z_2|$ and the Duflo functions $h^{(1)}, h^{(2)}, h$.
Then
\[
\mathcal{P}(\tilde{F}) (\tilde{F}_1 \times \tilde{F}_2)
\]
is a solution of the KV problem of type $(g_1 + g_2, n_1+n_2 +1)$ in the adapted framing with Duflo functions $h^{(1)}_1,\ldots,h^{(1)}_{n_1},h^{(2)}_1,\ldots,h^{(2)}_{n_2},h$.
\end{thm}

\begin{proof}
First observe that both $\tilde{F}_1 \times \tilde{F}_2$ and $\mathcal{P}(\tilde{F})$ 
are elements of ${\rm tAut}^+_{(g_1 + g_2, n_1 +n_2 +1)}$. Hence so is their product $\mathcal{P}(\tilde{F}) (\tilde{F}_1 \times \tilde{F}_2)$.

Since by assumption either $n_1=0$ or $g_2=0$, we observe that 
\[
\xi = \log \big(  \prod_{i=1}^{g_1} (e^{x_{1i}} e^{y_{1i}} e^{-x_{1i}} e^{-y_{1i}}) \, \prod_{j=1}^{n_1} e^{z_{1j}}  \, \prod_{i=1}^{g_2} (e^{x_{2i}} e^{y_{2i}} e^{-x_{2i}} e^{-y_{2i}})  \, \prod_{j=1}^{n_2} e^{z_{2j}} \big)
= \log(e^{\xi_{1}} e^{\xi_{2}}),
\]
where $\xi_k$ is the element \eqref{eq:xi_def} for the surface $\Sigma_{g_k, n_k+1}$ ($k = 1,2$).
Hence
$$
\mathcal{P}(\tilde{F}) (F_1 \times F_2)(\xi)=\mathcal{P}(\tilde{F})(\log(e^{\omega_1}e^{\omega_2})) = F(\log(e^{z_1}e^{z_2}))|_{z_k = \omega_k}
= \omega_1 + \omega_2 = \omega
$$
which proves the first KV equation.

For the second KV equation, note that
$$
\begin{array}{rll}
R^{\rm adp}(\tilde{F}_1) & = & 
\sum_j |h_j^{(1)}(z_{1j})| - |h^{(1)}(\omega_1)|, \\
R^{\rm adp}(\tilde{F}_2) & = & 
\sum_j |h_j^{(2)}(z_{2j})| - |h^{(2)}(\omega_2)|, \\
R^{f_0}(\tilde{F}) = {\sf j}^{f_0}_{\rm gr}(\tilde{F}) & = & |h^{(1)}(z_1)| + |h^{(2)}(z_2)| - |h(z_1 + z_2)|.
\end{array}
$$
We compute
\begin{align*}
& R^{\rm adp}(\mathcal{P}(\tilde{F})(\tilde{F}_1 \times \tilde{F}_2)) \\
=\, &  
{\sf j}^{\rm adp}_{\rm gr}(\mathcal{P}(\tilde{F})) 
+ \mathcal{P}(\tilde{F})(R^{\rm adp}(\tilde{F}_1 \times \tilde{F}_2)) \\
=\, &  
{\sf j}^{f_0}_{\rm gr}(\tilde{F})|_{z_k = \omega_k} 
+ \mathcal{P}(\tilde{F})\big( \textstyle\sum_j |h_j^{(1)}(z_{1j})| - |h^{(1)}(\omega_1)| + \textstyle\sum_j |h_j^{(2)}(z_{2j})| - |h^{(2)}(\omega_2)| \big) \\
=\, & 
|h^{(1)}(\omega_1)| + |h^{(2)}(\omega_2)| - |h(\omega_1 + \omega_2)| \\
& + \textstyle\sum_j |h_j^{(1)}(z_{1j})| - |h^{(1)}(\omega_1)| + \textstyle\sum_j |h_j^{(2)}(z_{2j})| - |h^{(2)}(\omega_2)| \\
=\, &  
\textstyle\sum_j |h_j^{(1)}(z_{1j})| + \textstyle\sum_j |h_j^{(2)}(z_{2j})| - |h(\omega)|.
\end{align*}
Here we have used Proposition~\ref{prop:transform_R^f} in the second line and Proposition~\ref{prop:BfunderP} in the third line.
Furthermore, we have used the facts that ${\bf r}={\bf r}_1 + {\bf r}_2$ which means $R^{\rm adp}(\tilde{F}_1 \times \tilde{F}_2) = R^{\rm adp}(\tilde{F}_1) + R^{\rm adp}(\tilde{F}_2)$ and that the action of $\mathcal{P}(\tilde{F})$ preserves $|h^{(1)}(\omega_1)|$ and $|h^{(2)}(\omega_2)|$.
This calculation completes the proof.
\end{proof}

\subsection{Proof of Theorem \ref{thm:KVsolve}}

We are now ready to prove the main theorem of this section:

\vskip 0.2cm

\begin{proof}[Proof of Theorem \ref{thm:KVsolve}]

We first treat the case of adapted framing and then proceed to general framings.

\vskip 0.5em
{\em Step 1}.
Let $h \in s^2\mathbb{K}[[s]]$ be a formal power series with even part given by equation \eqref{eq:h_even}.
Then, by Theorem \ref{thm:KVclassical}, there is a normalized solution of the KV problem of type $(0,3)$ in adapted framing  with Duflo functions $h_1=h_2=h$. By Theorem \ref{thm:KVelliptic}, it gives rise to a solution of the KV problem of type $(1,1)$ in adapted framing with the same Duflo function $h$.

Theorem \ref{thm:KVgluing} allows us to construct genus zero solutions of KV problems in adapted framing  using the initial solution of the KV problem of type $(0,3)$ with Duflo functions $h_1=h_2=h$ by successive gluing. This observation gives rise to solutions of KV problems of type $(0, n+1)$ for any $n\geq 2$ with Duflo functions $h_1=\dots=h_n=h$. Note that this result was also established in \cite{genus0}, Lemma 7.3.

Similarly Theorem \ref{thm:KVgluing} allows us to construct solutions of KV problems of type $(g,1)$ for any $g$ by successive gluing of solutions of type $(1,1)$ and $(g-1, 1)$. Note that in order to make this gluing possible, one needs a solution of the KV problem of type $(0,3)$ in the framing with ${\rm rot}^f(\gamma_1) = 2 -1 =1, {\rm rot}^f(\gamma_2)=2(g-1)-1 = 2g-3$ with Duflo functions $h_1=h_2=h$. Since both these rotation numbers are non vanishing, such solutions exist by Proposition  \ref{prop:KV03fh}. 

Again by Theorem \ref{thm:KVgluing}, one can glue a solution of the KV problem of type $(0, n+1)$ and a solution of the KV problem of type $(g,1)$ to obtain a solution of the KV problem of type $(g, n+1)$ for any $g$ and $n$. The existence of a solution of type $(0,3)$ needed for gluing is guaranteed by  Proposition  \ref{prop:KV03fh}, as in the previous paragraph.

We conclude that the KV problem of type $(g,n+1)$ in adapted framing admits solutions for any $g$ and $n$.

\vskip 0.5em
{\em Step 2}.
Proposition \ref{prop:KVcase1} shows that for $g\geq 2$ the KV problems in different framings are all equivalent to each other. Hence, in this case, the KV problem admits solutions for any framing, as required.

In the case of $g=1$ and ${\bf q}^f \neq 0$, by Proposition \ref{prop:KVcase2}, the KV problem of type $(1, n+1)$ in framing $f$ is equivalent to the KV problem in adapted framing. Hence, by the arguments above, it admits solutions.

Finally, in the case of $g=1$ and ${\bf q}^f=0$ by Proposition \ref{prop:KVcase3}, the KV problem admits solutions for at most one value of ${\bf p}^f$. Since the KV problem always admits solutions in adapted framing, we conclude that this value of ${\bf p}^f$ is zero. Hence, in this case, the KV problem does not admit solutions for ${\bf p}^f\neq 0$. 
 
This completes the proof.
\end{proof}

\section{Uniqueness} \label{sec:uniqueness}

In this section, we discuss the uniqueness issue for solutions of the KV problems.

\subsection{Kasiwara-Vergne groups}
The purpose of this section is to introduce and study symmetry groups for the higher genus KV problems.
Recall the definition of the cocycles ${\sf j}^f, {\sf j}^f_{\rm gr} \colon {\rm tAut}^+(L) \to |A|$:
$$
\begin{array}{lll}
{\sf j}^f(\tilde{G}) & = & {\sf j}_{x,y,z}(G) - {\sf c}^f(\tilde{G}) + G({\bf r} -{\bf p}^f) - ({\bf r}-{\bf p}^f), \\
{\sf j}^f_{\rm gr}(\tilde{G}) & = & {\sf j}_{x,y,z}(G) - {\sf c}^f(\tilde{G}).
\end{array}
$$
Using these cocycles, we define the following subsets of ${\rm tAut}^+(L)$:
\begin{dfn}      \label{dfn:KV_KRV}
For a framing $f$ on a surface of genus $g$ with $n+1$ boundary components, let
\begin{align*}
{\rm KV}^f_{(g, n+1)} & := \big\{ \tilde{G} \in {\rm tAut}^+(L);
\text{$G(\xi) = \xi$ and} \\ 
    & \hspace{2em} \text{${\sf j}^f(\tilde{G}) = \textstyle\sum_j |h_j(z_j)| - |h(\xi)|$ for some $h_j \in s\mathbb{K}[[s]]$, $h\in s^2\mathbb{K}[[s]]$} \big\}, \\
{\rm KRV}_{(g, n+1)}^f & := \big\{ \tilde{G} \in {\rm tAut}^+(L);
\text{$G(\omega) = \omega$ and} \\
 & \hspace{2em} \text{${\sf j}_{\rm gr}^f(\tilde{G}) = \textstyle\sum_j |h_j(z_j)| - |h(\omega)|$ for some $h_j \in s\mathbb{K}[[s]]$, $h\in s^2\mathbb{K}[[s]]$} \big\}.
\end{align*}
\end{dfn}

The functions $h_j$ ($j=1,\ldots,n$) and $h$ are uniquely determined by $\tilde{G}$ and called the Duflo functions for $\tilde{G}$, for which we use the simplified notation $h_j^G$ and $h^G$ (dropping $\tilde{\, }$).
When it is clear from the context, we will drop the superscript $(g, n+1)$ in the notation.

In fact, these subsets are subgroups of ${\rm tAut}^+(L)$:
\begin{prop} \label{prop:KVKRVsubgroup}
The sets ${\rm KV}^f$ and ${\rm KRV}^f$ are subgroups of ${\rm tAut}^+(L)$.
The maps $h_j\colon {\rm KV}^f, {\rm KRV}^f \to s\mathbb{K}[[s]]$ and $h\colon {\rm KV}^f, {\rm KRV}^f \to s^2\mathbb{K}[[s]]$ are group homomorphisms under the additive group structure on the target.
\end{prop}

\begin{proof}
For $\tilde{F}, \tilde{G} \in {\rm KV}^f$, we compute
$$
(FG)(\xi) = F(G(\xi)) = F(\xi) = \xi,
$$
and
\begin{align*}
{\sf j}^f(\tilde{F}\tilde{G}) & = {\sf j}^f(\tilde{F}) + F({\sf j}^f(\tilde{G})) \\
& = \textstyle\sum_j |h^F_j(z_j)| - |h^F(\xi)| + F(\textstyle\sum_j |h^G_j(z_j)| - |h^G(\xi)|) \\
& = \textstyle\sum_j |h^F_j(z_j) + h^G_j(z_j)| - |h^F(\xi) + h^G(\xi)|.
\end{align*}
Here we have used the facts that $F(|h^G(z_j)|) = |h^G(z_j)|$ since $\tilde{F}$ is tangential and that $F(|h^G(\xi)|) = |h^G(F(\xi))| = |h^G(\xi)|$.
Hence $\tilde{F}\tilde{G} \in {\rm KV}^f$.
Similarly one can prove that $\tilde{F} \in {\rm KV}^f$ implies $\tilde{F}^{-1} \in {\rm KV}^f$.
The last displayed formula shows that the maps $h_j$ and $h$ are indeed group homomorphisms under the additive group structure on $\mathbb{K}[[s]]$.
The proof for ${\rm KRV}^f$ is similar.
\end{proof}

\begin{rem}       \label{rem:f_independent}
There are several interesting situations when the Kashiwara-Vergne groups ${\rm KV}^f$ and ${\rm KRV}^f$ do not depend on the choice of framing. First, in the genus zero case (similar to Remark~\ref{rem:framing_genus_0}), the groups ${\rm KV}^f, {\rm KRV}^f$ are independent of the framing $f$. Indeed, the framing dependent term 
$$
{\sf c}^f(\tilde{G}) = \sum_{j=1}^n {\rm rot}^f(\gamma_j) {\sf c}_j(\tilde{G}) \in \mathbb{K}\langle |z_j|; j=1, \dots, n\rangle
$$
takes values in the span of $|z_j|$'s, and can be eliminated by changing linear parts of $h_j$'s.
Next the groups ${\rm KV}^f, {\rm KRV}^f$ are independent of framing in the case of surfaces of genus $g$ with one boundary component (that is, $n=0$). For ${\rm KRV}^f$ groups, this is obvious since ${\sf c}^f=0$. For the groups ${\rm KV}^f$, this follows from the fact that $L$ has no generators of weight 2, and $G({\bf p}^f)={\bf p}^f$ for all $G \in {\rm Aut}^+(L)$.
\end{rem}

The importance of the groups ${\rm KV}^f$ and ${\rm KRV}^f$ stems from the following theorem:
\begin{thm}   \label{thm:KV_KRV}
Assume that the set ${\rm SolKV}_{(g, n+1)}^f$ is nonempty. Then it carries a free and transitive action of the group ${\rm KV}_{(g, n+1)}^f$ by right translations and of the group ${\rm KRV}_{(g, n+1)}^f$ by left translations:
\[
{\rm KRV}^f_{(g,n+1)}
\curvearrowright
{\rm SolKV}_{(g, n+1)}^f
\curvearrowleft
{\rm KV}^f_{(g,n+1)}.
\]
Every element $\tilde{F} \in {\rm SolKV}_{(g, n+1)}^f$ defines a group isomorphism ${\rm KRV}^f_{(g,n+1)} \to {\rm KV}^f_{(g,n+1)}$ 
by 
$\tilde{G} \mapsto \tilde{H}=\tilde{F}^{-1} \tilde{G} \tilde{F}$. This isomorphism has the property $h^H=h^G, h^H_j = h^G_j$ for $j=1, \dots, n$.
\end{thm}

\begin{proof}
Assume that ${\rm SolKV}^f = {\rm SolKV}^f_{(g,n+1)} \neq \emptyset$, and let $\tilde{F} \in {\rm SolKV}^f$ and $\tilde{G} \in {\rm KV}^f$. Then
$\tilde{F} \tilde{G} \in {\rm SolKV}^f$. Indeed,
$$
(FG)(\xi) = F(G(\xi)) = F(\xi) = \omega,
$$
and
\begin{align*}
R^f(\tilde{F} \tilde{G}) & =  R^f(\tilde{F}) + F({\sf j}^f(\tilde{G})) \\
& = \textstyle\sum_j |h^F_j(z_j)| - |h^F(\omega)| + F(\textstyle\sum_j |h^G_j(z_j)| - |h^G(\xi)|)  \\
& = \textstyle\sum_j |h^F_j(z_j) + h^G_j(z_j)| - |h^F(\omega) + h^G(\omega)|.
\end{align*}
Here we have used Proposition \ref{prop:transform_R^f} and the facts that 
$\tilde{F}$ is tangential and that $F(\xi) = \omega$.

In a similar fashion, let $\tilde{F} \in {\rm SolKV}^f$ and $\tilde{K} \in {\rm KRV}^f$. In order to show that $\tilde{K} \tilde{F} \in {\rm SolKV}^f$, we compute
$$
(KF)(\xi) = K(F(\xi)) = K(\omega)= \omega,
$$
and
\begin{align*}
R^f(\tilde{K} \tilde{F}) & =  {\sf j}^f_{\rm gr}(\tilde{K}) + K(R^f(\tilde{F})) \\
& = \textstyle\sum_j |h^K_j(z_j)| - |h^K(\omega)| + K(\textstyle\sum_j |h^F(z_j)| - |h^F(\omega)|)  \\
& = \textstyle\sum_j |h^K_j(z_j) + h^F_j(z_j)| - |h^K(\omega) + h^F(\omega)|.
\end{align*}

Hence ${\rm KV}^f$ acts on ${\rm SolKV}^f$ by right translations and ${\rm KRV}^f$ by left translations, as required. These actions are free since so are the actions of ${\rm tAut}^+(L)$ on itself by left and right translations. 

These actions are transitive. Indeed, let $\tilde{F}_1, \tilde{F}_2 \in {\rm SolKV}^f$.
Then, $\tilde{G}=\tilde{F}_1^{-1} \tilde{F}_2$ has the following properties:
$$
G(\xi) = F_1^{-1}(F_2(\xi)) = F_1^{-1}(\omega) = \xi,
$$
and
\begin{align*}
{\sf j}^f(\tilde{G}) & = {\sf j}^f(F_1^{-1}) + F_1^{-1}({\sf j}^f(F_2))  \\
& = F_1^{-1}({\sf j}^f(F_2) - {\sf j}^f(F_1)) \\ 
& = F_1^{-1} \big( (R^f(\tilde{F}_2) +{\bf p}^f -{\bf r}) - (R^f(\tilde{F}_1) +{\bf p}^f -{\bf r}) \big) \\
& = F_1^{-1}(\textstyle\sum_j|h^{F_2}_j(z_j) - h^{F_1}_j(z_j)| - |h^{F_2}(\omega) - h^{F_1}(\omega)|) \\
& = \textstyle\sum_j|h^{F_2}_j(z_j) - h^{F_1}_j(z_j)| - |h^{F_2}(\xi) - h^{F_1}(\xi)|.
\end{align*}
Therefore $\tilde{G} \in {\rm KV}^f$ as required. 
Similarly, one can prove that $\tilde{K} = \tilde{F}_2\tilde{F}_1^{-1}$ belongs to ${\rm KRV}^f$.

Let $\tilde{F} \in {\rm SolKV}^f$. It defines a group isomorphism ${\rm KRV}^f \to {\rm KV}^f$ 
by 
$\tilde{K} \mapsto \tilde{G} =\tilde{F}^{-1} \tilde{K} \tilde{F}$. Indeed, $\tilde{G}$ is the unique solution of the equation
$\tilde{K} \tilde{F} = \tilde{F} \tilde{G}$. Since $\tilde{K} \in {\rm KRV}^f, \tilde{F} \in {\rm SolKV}^f$, the product $\tilde{K} \tilde{F} \in {\rm SolKV}^f$. Therefore $\tilde{G} = \tilde{F}^{-1}(\tilde{K} \tilde{F}) \in {\rm KV}^f$. The map $\tilde{K} \mapsto \tilde{G}$ is conjugation by $\tilde{F}$. Therefore it is an injective group homomorphism. A similar argument shows that the inverse map $\tilde{G} \mapsto \tilde{K} = \tilde{F} \tilde{G} \tilde{F}^{-1}$ is an injective group homomorphism ${\rm KV}^f \to {\rm KRV}^f$. Hence both maps are bijections, as required.

The equality $\tilde{K} \tilde{F} = \tilde{F} \tilde{G}$ and calculations above imply
$$
h^{K}_j + h^{F}_j = h^{F}_j + h^{G}_j, \hskip 0.3cm
h^{K} + h^{F} = h^{F} + h^{G}.
$$
Hence we conclude that $h_j^G=h_j^K$ for all $j=1, \dots, n$ and $h^H = h^G$.
\end{proof}

For future use, we define the following subgroups of the KV groups:
$$
{\rm KV}_0^f=\{ \tilde{F} \in {\rm KV}^f; h^F=0\},
\hskip 0.3cm
{\rm KRV}_0^f=\{ \tilde{F} \in {\rm KRV}^f; h^F=0\}.
$$
Since the maps $h \colon {\rm KV}^f, {\rm KRV}^f \to s^2\mathbb{K}[[s]]$ are group homomorphisms, the subsets ${\rm KV}_0^f \subset {\rm KV}^f, {\rm KRV}_0^f \subset {\rm KRV}^f$ are indeed subgroups. By Theorem \ref{thm:KV_KRV}, every group isomorphism ${\rm KRV}^f \to {\rm KV}^f$ induced by an element $\tilde{F} \in {\rm SolKV}^f$ restricts to a group isomorphism ${\rm KRV}_0^f \to {\rm KV}_0^f$.

\begin{rem} \label{rem:Duflo_g0_KVgroups}
Let $\tilde{F} \in {\rm SolKV}^f$ for a KV problem in genus zero. Recall that in this case by Remark~\ref{rem:Duflo_genus_zero} all Duflo functions
$h_j, h$ agree modulo linear terms with even part given by \eqref{eq:h_even}.
This implies that Duflo functions of elements $\tilde{G} \in {\rm KRV}^f, \tilde{H} \in {\rm KV}^f$ also agree modulo linear terms and have vanishing even parts. Indeed, we can represent $\tilde{G} = \tilde{F}_2 \tilde{F}_1^{-1}$ for some $\tilde{F}_1, \tilde{F}_2 \in {\rm SolKV}^f$. 
Then,
$$
h^G_j = h^{F_2}_j - h^{F_1}_j, \hskip 0.3cm
h^G=h^{F_2} - h^{F_1},
$$
and the conclusion follows. The argument for $\tilde{H}$ is similar. 
\end{rem}

\begin{rem}
Note that the groups ${\rm KRV}^f$ depend only on rotations numbers
${\rm rot}^f(\gamma_j)$ corresponding to boundary components of the surface, and they do not depend on the element ${\bf p}^f$ which encodes all other rotation numbers. 
Assuming ${\rm SolKV}^f \neq \emptyset$, Theorem~\ref{thm:KV_KRV} implies ${\rm KV}^f \cong {\rm KRV}^f$, and the same consideration applies to the groups ${\rm KV}^f$ even though ${\bf p}^f$ enters in their definition.
\end{rem}

\begin{rem}
Note that elements $\tilde{F}=(F, f_1, \dots, f_n) \in {\rm tAut}^+(L)$ with $F = {\rm id}, f_j=e^{\lambda_jz_j}$ for all $\lambda_j \in \mathbb{K}$ always belong to ${\rm KRV}^f$ and to ${\rm KRV}_0^f$. Indeed, $F(\omega) =\omega$ since $F = {\rm id}$, and
$$
{\sf j}^f_{\rm gr}(\tilde{F}) = {\sf j}_{x,y,z}(F) - \sum_{j=1}^n {\rm rot}^f(\gamma_j) {\sf c}_j(\tilde{F}) = -\sum_{j=1}^n \lambda_j\, {\rm rot}^f(\gamma_j) |z_j|,
$$
as required.
\end{rem}

\subsection{Gluing of surfaces}
The purpose of this subsection is to establish relations between the KV groups under gluing of surfaces, similar to the relations between solutions of the KV problems.
We start with the relation between the KV groups in genus zero and genus one:
\begin{prop}  \label{prop:KRV_genus0genus1}
The maps $\mathcal{E}\colon \tilde{G} \mapsto G^{\rm ell}$ and ${\rm Ad}_{\zeta} \circ \mathcal{E} \colon \tilde{G} \mapsto \zeta G^{\rm ell} \zeta^{-1}$ define group homomorphisms preserving Duflo functions ${\rm KV}_{(0,3)} \to {\rm KV}_{(1,1)}$ and ${\rm KRV}_{(0, 3)} \to {\rm KRV}_{(1,1)}$, respectively.
Here $\zeta \in {\rm Aut}^+_{(1,1)}$ is the element defined in \eqref{eq:auto_zeta}.
\end{prop}
\begin{proof}
The map $\mathcal{E}\colon \tilde{G} \mapsto G^{\rm ell}$ is a group homomorphism by Proposition~\ref{prop:Fellhom}.
In order to show that it maps ${\rm KV}_{(0,3)}$ to ${\rm KV}_{(1,1)}$, let $\tilde{G} \in {\rm KV}_{(0,3)}$, and represent it in the form
$\tilde{G}=\tilde{F}_1^{-1} \tilde{F}_2$ with $\tilde{F}_1, \tilde{F}_2 \in {\rm SolKV}_{(0,3)}^f$, where the framing $f$ is specified by ${\rm rot}^f(\gamma_1) = {\rm rot}^f(\gamma_2) = 0$.
Here we have used the fact that the group ${\rm KV}_{(0,3)}$ is independent of framing (see Remark~\ref{rem:f_independent}).
By the proof of Theorem \ref{thm:KV_KRV}, the Duflo function of $\tilde{G}$ is given by $h^G=h^{F_2}-h^{F_1}$, where $h^{F_1}$ and $h^{F_2}$ are the Duflo functions of $\tilde{F}_1$ and $\tilde{F}_2$, respectively.

We can decompose $\mathcal{E}(\tilde{G})$ into a product of two factors:
$$
\mathcal{E}(\tilde{G}) = G^{\rm ell} = 
(\zeta F_1^{\rm ell})^{-1} (\zeta F_2^{\rm ell}).
$$
By Theorem~\ref{thm:KVelliptic}, we have $\zeta F_1^{\rm ell}, \zeta F_2^{\rm ell} \in {\rm SolKV}_{(1,1)}^{\rm adp}$ with the same Duflo functions as $\tilde{F}_1$ and $\tilde{F}_2$, respectively.
Then Theorem~\ref{thm:KV_KRV} implies that $\mathcal{E}(\tilde{G})$ is an element of ${\rm KV}_{(1,1)}$ with the Duflo function $h^{\mathcal{E}(\tilde{G})} = h^{F_2} - h^{F_1} = h^G$, as required.

Next consider the map ${\rm Ad}_{\zeta} \circ \mathcal{E} \colon \tilde{G} \mapsto \zeta G^{\rm ell} \zeta^{-1}$.
It is a group homomorphism because so is the map $\tilde{G} \mapsto G^{\rm ell}$.
To show that it maps ${\rm KRV}_{(0,3)}$ to ${\rm KRV}_{(1,1)}$, let $\tilde{G} \in {\rm KRV}_{(0.3)}$ and represent it in the form $\tilde{G} = \tilde{F}_2 \tilde{F}_1^{-1}$ with $\tilde{F}_1, \tilde{F}_2 \in {\rm SolKV}_{(0,3)}^f$.
Then, $\zeta F_1^{\rm ell}, \zeta F_2^{\rm ell} \in {\rm SolKV}_{(1,1)}^{\rm adp}$ and 
\[
{\rm Ad}_{\zeta}\circ \mathcal{E} (\tilde{G})
= \zeta G^{\rm ell} \zeta^{-1} = (\zeta F_2^{\rm ell})(\zeta F_1^{\rm ell})^{-1} \in {\rm KRV}_{(1,1)},
\]
as required.
The argument for Duflo functions works exactly as in the case of ${\rm KV}$ groups.
\end{proof}

We now consider gluing of surfaces.  For any non-negative integers $g_1,n_1, g_2, n_2$ and framings $f_1, f_2, f_0$ on the surfaces $\Sigma_1, \Sigma_2, \Sigma_{0,3}$ as in Remark~\ref{rem:glue_rotation}, we define a subgroup in the product of ${\rm KRV}$ groups
$$
{\rm KRV}^{\rm prod} = \{ (\tilde{G}_1, \tilde{G}_2, \tilde{G}) \in {\rm KRV}^{f_1}_{(g_1, n_1+1)}\times {\rm KRV}_{(g_2, n_2+1)}^{f_2} \times  {\rm KRV}^{f_0}_{(0,3)}; \, h^{G_1}=h^{G_2}=h^{G}=h \} .
$$
We denote the framing of the glued surface by $f$. Its restrictions to the three components are $f_0$, $f_1$ and $f_2$, respectively.
The following proposition gives a gluing construction for ${\rm KRV}$ groups:

\begin{prop}  \label{prop:KRV_gluing}
There is a group homomorphism 
${\rm KRV}^{\rm prod} \to {\rm KRV}_{(g_1+g_2, n_1+n_2+1)}^f$ defined by formula
$$
(\tilde{G}_1, \tilde{G}_2, \tilde{G}) \mapsto (\tilde{G}_1 \times \tilde{G}_2) \mathcal{P}(\tilde{G}).
$$
Here $\mathcal{P}$ is the map in Proposition~\ref{prop:BfunderP}.
\end{prop}

\begin{proof}
For the first equation in the definition of ${\rm KRV}_{(g_1+g_2, n_1+n_2+1)}^f$ observe that $G_1, G_2$ and $\mathcal{P}(G)$ preserve $\omega=\omega_1+\omega_2$. For the second equation, we compute
\begin{align*}
{\sf j}^f_{\rm gr}((\tilde{G}_1 \times \tilde{G}_2) \mathcal{P}(\tilde{G})) & = 
{\sf j}^{f}_{\rm gr}(\tilde{G}_1) + {\sf j}^{f}_{\rm gr}(\tilde{G}_2) +
(G_1 \times G_2)({\sf j}^{f}_{\rm gr}(\mathcal{P}(\tilde{G}))) \\
& = {\sf j}^{f_1}_{\rm gr}(\tilde{G}_1) + {\sf j}^{f_2}_{\rm gr}(\tilde{G}_2) +
(G_1 \times G_2)({\sf j}^{f_0}_{\rm gr}(\tilde{G})|_{z_k=\omega_k}) \\
& = \textstyle\sum_{j=1}^{n_1} |h^{G_1}_j(z_{1j})| - |h(\omega_1)| +
\textstyle\sum_{j=1}^{n_2} |h^{G_2}_j(z_{2j})| - |h(\omega_2)| \\
& \hspace{1em} + 
(G_1 \times G_2)(|h^G_1(\omega_1) + h^G_2(\omega_2) - h(\omega)|)   \\
& = \textstyle\sum_{j=1}^{n_1} |h^1_j(z_{1j})| + \textstyle\sum_{j=1}^{n_2} |h^2_j(z_{2j})| - |h(\omega)|.
\end{align*}
Here we have used Proposition \ref{prop:BfunderP} in the second line,
the assumption that $h^{G_1}=h^{G_2}=h^{G}$,  and the fact that 
$|h^{G}_1(\omega_1) - h(\omega_1)|$ and $|h^{G}_2(\omega_2) - h(\omega_2)|$ 
are linear combinations of expressions $|z_{1j}|$ and $|z_{2j}|$. The new Duflo functions $h^1_j$ and $h^2_j$ are obtained by adding these linear terms to $h^{G_1}_j$ and $h^{G_2}_j$, respectively.

To show that the map $(\tilde{G}_1, \tilde{G}_2, \tilde{G}) \mapsto (\tilde{G}_1 \times \tilde{G}_2) \mathcal{P}(\tilde{G})$ is a group homomorphism recall that $\mathcal{P}$ is a group homomorphism and note that the factors $(\tilde{G}_1 \times \tilde{G}_2)$ and $\mathcal{P}(\tilde{G})$ commute since $G_1(\omega_1) = \omega_1$ and $G_2(\omega_2) = \omega_2$.
\end{proof}

There is also a gluing construction for ${\rm KV}$ groups. In order to describe it, we need the following pair-of-paints map
\[
\mathcal{P}^{\diamond}
\colon \taut^+_{(0,3)} \rightarrow \taut^+_{(g_1+g_2, n_1+n_2+1)}
\]
that sends an element $\tilde{F} = (F,f_1,f_2) \in \taut^+_{(0,3)}$ to the automorphism
\[
\mathcal{P}^{\diamond}(\tilde{F}) \colon
  \begin{cases}   
         x_{ki} \mapsto f_k(\xi_1, \xi_2)^{-1} x_{ki} f_k(\xi_1, \xi_2) \\
        y_{ki} \mapsto f_k(\xi_1, \xi_2)^{-1} y_{ki} f_k(\xi_1, \xi_2) \\
        z_{kj} \mapsto f_k(\xi_1, \xi_2)^{-1} z_{kj} f_k(\xi_1, \xi_2)
  \end{cases}
  \qquad k = 1,2,
\]
where
\begin{align*}
    \xi_1 &= \log \big(\prod_{i=1}^{g_1} (e^{x_{1i}} e^{y_{1i}} e^{-x_{1i}} e^{-y_{1i}})  \, \prod_{j=1}^{n_1} e^{z_{1j}} \big) ,
    &\xi_2 &= \log \big(\prod_{i=1}^{g_2} (e^{x_{2i}} e^{y_{2i}} e^{-x_{2i}} e^{-y_{2i}})  \, \prod_{j=1}^{n_2} e^{z_{2j}} \big).
\end{align*}
Let $\tilde{F}_1$ and $\tilde{F}_2$ be solutions of the equations $\KV {\rm I}_{(g_1, n_1+1)}$ and  $\KV {\rm I}_{(g_2, n_2+1)}$, respectively. Then we have the identity
\begin{align} \label{eq:pantsintertwine}
    (\tilde{F}_1 \times \tilde{F}_2) \circ \mathcal{P}^{\diamond}(\tilde{G}) = \mathcal{P}(\tilde{G}) \circ (\tilde{F}_1 \times \tilde{F}_2)
\end{align}
for all $\tilde{G} \in \taut^+_{(0,3)}$.
Keeping the same setting as the definition of the group ${\rm KRV}^{\rm prod}$, we consider the group
$$
{\rm KV}^{\rm prod} = \{ (\tilde{G}_1, \tilde{G}_2, \tilde{G}) \in {\rm KV}^{f_1}_{(g_1, n_1+1)}\times {\rm KV}_{(g_2, n_2+1)}^{f_2} \times  {\rm KV}_{(0,3)}^{f_0}; \, h^{G_1}=h^{G_2}=h^{G}=h\} .
$$

\begin{prop}      \label{prop:glue_kv}
Assume that $n_1=0$ or $g_2=0$. Then there is a group homomorphism ${\rm KV}^{\rm prod} \to {\rm KV}^{f}_{(g_1+g_2, n_1+n_2 +1)}$ given by formula
$$
(\tilde{G}_1, \tilde{G}_2, \tilde{G}) \mapsto (\tilde{G}_1 \times \tilde{G}_2) \mathcal{P}^{\diamond}(\tilde{G}).
$$
\end{prop}

\begin{proof}
First recall that if $n_1=0$ or $g_2=0$ we have $\xi = \log(e^{\xi_1} e^{\xi_2})$. 
This implies
\begin{align*}
(G_1 \times G_2) \mathcal{P}^{\diamond}(G)(\xi)
& = (G_1 \times G_2) \mathcal{P}^{\diamond}(G)(\log(e^{\xi_1} e^{\xi_2})) \\
& = 
(G_1 \times G_2)(\log(e^{\xi_1} e^{\xi_2})) \\
& = 
\log(e^{G_1(\xi_1)}e^{G_2(\xi_2)}) \\
& = \log(e^{\xi_1} e^{\xi_2}) = \xi.
\end{align*}
Here we have used that $G_1(\xi_1)=\xi_1, G_2(\xi_2) = \xi_2,
G(\log(e^{z_1}e^{z_2})) = \log(e^{z_1}e^{z_2})$.

In order to confirm the second defining condition of ${\rm KV}^f_{(g_1 + g_2, n_1 + n_2 + 1)}$, we use the following relation: for any $\tilde{G} \in {\rm tAut}^+_{(0,3)}$ we have
\begin{equation} \label{eq:BfunderP}
{\sf j}^f(\mathcal{P}^{\diamond}(\tilde{G})) = 
{\sf j}^{f_0}(\tilde{G})|_{z_k = \xi_k}.
\end{equation}
To see this, let $\tilde{F}_1$ and $\tilde{F}_2$ be (any) solutions of equations $\KV {\rm I}_{(g_1, n_1+1)}$ and  $\KV {\rm I}_{(g_2, n_2+1)}$, respectively. 
Applying ${\sf j}^f$ to the both sides of \eqref{eq:pantsintertwine}, we obtain
\[
{\sf j}^f(\tilde{F}_1 \times \tilde{F}_2) + 
(\tilde{F}_1 \times \tilde{F}_2) ({\sf j}^f(\mathcal{P}^{\diamond}(\tilde{G}))) 
= {\sf j}^f(\mathcal{P}(\tilde{G})) + \mathcal{P}(\tilde{G}) ({\sf j}^f(\tilde{F}_1 \times \tilde{F}_2)).
\]
Since for each $k = 1,2$ the map $\mathcal{P}(\tilde{G})$ preserves the trace of any formal power series in $x_{ki}, y_{ki}, z_{kj}$, it preserves ${\sf j}^f(\tilde{F}_1 \times \tilde{F}_2) = {\sf j}^{f_1}(\tilde{F}_1) + {\sf j}^{f_2}(\tilde{F}_2)$ and ${\bf r} - {\bf p}^f$.
Therefore we obtain $(\tilde{F}_1 \times \tilde{F}_2) ({\sf j}^f(\mathcal{P}^{\diamond}(\tilde{G}))) = {\sf j}^f(\mathcal{P}(\tilde{G})) = {\sf j}^f_{\rm gr}(\mathcal{P}(\tilde{G}))$.
Thus 
\[
{\sf j}^f(\mathcal{P}^{\diamond}(\tilde{G})) = (\tilde{F}_1 \times \tilde{F}_2)^{-1}({\sf j}^f_{\rm gr} (\mathcal{P}(\tilde{G}))) = 
(\tilde{F}_1 \times \tilde{F}_2)^{-1}({\sf j}^{f_0}_{\rm gr}(\tilde{G})|_{z_k = \omega_k})
= {\sf j}^{f_0}(\tilde{G})|_{z_k = \xi_k}.
\]
Here we have used Proposition~\ref{prop:BfunderP} and the fact that $\tilde{F}_k^{-1}(\omega_k) = \xi_k$.
Also, note that ${\sf j}^{f_0}_{\rm gr} = {\sf j}^{f_0}$ since the genus is zero.
This proves \eqref{eq:BfunderP}.

Now using equation \eqref{eq:BfunderP} we compute
\begin{align*}
{\sf j}^f((\tilde{G}_1 \times \tilde{G}_2) \mathcal{P}^{\diamond}(\tilde{G})) & =  
{\sf j}^f(\tilde{G}_1) + {\sf j}^f(\tilde{G}_2) + (G_1 \times G_2)({\sf j}^f(\mathcal{P}^{\diamond}(\tilde{G}))) \\
& = {\sf j}^{f_1}(\tilde{G}_1) + {\sf j}^{f_2}(\tilde{G}_2) + (G_1 \times G_2)({\sf j}^{f_0}(\tilde{G})|_{z_k=\xi_k})\\
& = \textstyle\sum_{j=1}^{n_1} |h^{G_1}_j(z_{1j})| - h(\xi_1) + \textstyle\sum_{j=1}^{n_2} |h^{G_2}_j(z_{2j})| - h(\xi_2) \\
& \hspace{1em} + (G_1 \times G_2)(|h^G_1(\xi_1)| + |h^G_2(\xi_2)| - |h(\xi)|) \\
& = \textstyle\sum_{j=1}^{n_1} |h^1_j(z_{1j})| + \textstyle\sum_{j=1}^{n_2} |h^2_j(z_{2j})| - |h(\xi)|.
\end{align*}
Here we have used the assumption $h^{G_1}=h^{G_2}=h^G=h$ and the fact that 
the expressions $h^{G}_1-h, h^{G}_2-h$ are linear functions in $z_1, z_2$.

This completes the proof.
\end{proof}

\subsection{Kashiwara-Vergne Lie algebras}

The groups ${\rm KV}^f_{(g, n+1)}, {\rm KRV}^f_{(g, n+1)}$ are pro-unipotent. The corresponding Lie algebras are given by
\begin{align*}
\mathfrak{kv}^f_{(g, n+1)} & :=  
\big\{ \tilde{u} \in {\rm tDer}^+(L); \text{$u(\xi)=0$ and} \\ 
&  \hspace{2em} \text{${\sf div}^f(\tilde{u}) =
\textstyle\sum_j|h_j(z_j)| - |h(\xi)|$ for some $h_j  \in s\mathbb{K}[[s]]$, $h \in s^2\K[[s]]$} \big\}. \\
\mathfrak{krv}^f_{(g, n+1)} & := 
\big\{ \tilde{u} \in {\rm tDer}^+(L); \text{$u(\omega)=0$ and} \\
& \hspace{2em} \text{${\sf div}^f_{\rm gr}(\tilde{u}) =
\textstyle\sum_j|h_j(z_j)| - |h(\omega)|$ for some $h_j \in s\mathbb{K}[[s]]$, $h \in s^2\K[[s]]$} \big\}.
\end{align*}
More precisely, this correspondence is established by the following proposition:

\begin{prop} \label{prop:KVlog}
Let $\tilde{G} \in {\rm tAut}^+(L)$.
Then, $\tilde{G} \in {\rm KV}^f$ if and only if $\tilde{u}=\log(\tilde{G}) \in \mathfrak{kv}^f$, and $\tilde{G} \in {\rm KRV}^f$ if and only if $\tilde{u}=\log(\tilde{G}) \in \mathfrak{krv}^f$.
\end{prop}

\begin{proof}
Let $\tilde{u} \in \mathfrak{kv}^f$. Then  $\tilde{G}=e^{\tilde{u}}$ has the following property: $\tilde{G}(\xi) = \exp(\tilde{u})(\xi) = \xi$ since
$\tilde{u}(\xi) =0$. Furthermore
$$
{\sf j}^f(\tilde{G}) = \frac{e^{\tilde{u}} - 1}{\tilde{u}} \, {\sf div}^f(\tilde{u}) =
\frac{e^{\tilde{u}} - 1}{\tilde{u}}(\sum_j |h_j(z_j)| - |h(\xi)|) =
\sum_j |h_j(z_j)| - |h(\xi)|.
$$
Here we have used the facts that $\tilde{u}$ acts by zero on $|h_j(z_j)|$ (because $\tilde{u}$ is a tangential derivation) and on $|h(\xi)|$ (because $\tilde{u}$ vanishes on $\xi$).

In the other direction, assume that $\tilde{G} \in {\rm KV}^f$, and define
$$
\tilde{u} = \log(\tilde{G}) = \log(1 + (\tilde{G}-1)).
$$
Since $(\tilde{G}-1)(\xi)=0$, we conclude that $\tilde{u}(\xi)=0$. Furthermore
$\tilde{G}=e^{\tilde{u}}$ implies
$$
{\sf div}^f(\tilde{u}) =\frac{\tilde{u}}{e^{\tilde{u}} -1} {\sf j}^f(\tilde{G}) =
\frac{\tilde{u}}{e^{\tilde{u}} -1}(\sum_j |h_j(z_j)| - |h(\xi)|) =
\sum_j |h_j(z_j)| - |h(\xi)|,
$$
as required. This completes the proof of the relation between ${\rm KV}^f$ and $\mathfrak{kv}^f$. The proof for ${\rm KRV}^f$ and $\mathfrak{krv}^f$ is similar.
\end{proof}

Similarly to the case of KV groups, we introduce the following Lie subalgebras of Kashiwara-Vergne Lie algebras:
\[
\mathfrak{kv}^f_0 = 
\{ \tilde{u} \in \mathfrak{kv}^f ; h^{u} = 0 \},
\qquad
\mathfrak{krv}^f_0 = 
\{ \tilde{u} \in \mathfrak{krv}^f ;
h^{u} = 0 \}.
\]
Under the correspondence shown in Proposition~\ref{prop:KVlog}, these are Lie algebras of the groups ${\rm KV}^f_0$ and ${\rm KRV}^f_0$, respectively.

In the genus zero case, the Lie algebras $\mathfrak{kv}^f_{(0, n+1)}, \mathfrak{krv}^f_{(0, n+1)}$ for different framings are isomorphic. 
In this case, we will drop the subscript $f$ which indicates the framing. 
The Lie algebras $\mathfrak{kv}^f$ are filtered by the weight. The Lie algebras $\mathfrak{krv}^f$ are their associated graded. This follows from the fact that under the weight filtration ${\rm gr}(\xi)=\omega$.
In what follows, we consider two important examples: the graded Lie algebras 
$\mathfrak{krv}_{(0,3)}$ and $\mathfrak{krv}_{(1,1)}$. 

In more explicit terms, 
we define 
the Lie algebra $\mathfrak{krv}_{(0,3)}$ 
by 
\begin{align*}
\mathfrak{krv}_{(0,3)} & = \big\{ (u_1, u_2) \in L(z_1, z_2)^{2};
\text{$[z_1, u_1] + [z_2, u_2]=0$ and $\exists h \in s^2\K[[s]]$ s.t.} \\
& \hspace{2em} \text{$|z_1 (d_1 u_1) + z_2 (d_2 u_2)| \equiv |h(z_1) +h(z_2)-
h(z_1 +z_2)|$ modulo $\K |z_1| \oplus \K |z_2|$} \big\}.
\end{align*}

\begin{exple} \label{ex:t_another}
In weight 2 (that is, linear in $z_1, z_2$) the tangential derivation
$\tilde{u}=(u, u_1, u_2)$ with
$$
u_1 = az_1 + bz_2, \hskip 0.3cm
u_2 = cz_1 + dz_2
$$
is in $\mathfrak{krv}_{(0,3)}$ if and only if $b=c$. Such tangential derivations form a 3-dimensional abelian Lie algebra.
In particular, we have an inner derivation
\begin{equation*}
\tilde{t} = (-{\rm ad}_{z_1 + z_2}, z_1 + z_2, z_1 + z_2) \in \mathfrak{krv}_{(0,3)}.
\end{equation*}
(Note that tangential derivations $\tilde{t}$ and $\bar{t}$ (see Remark~\ref{rem:KVnorm}) correspond to the same derivation $t = -{\rm ad}_{z_1 + z_2}$ but they have different tangential components.)
\end{exple}

The following result was established in \cite{AT12} (Theorem 4.6):
\begin{thm}  \label{thm:grt&krv}
There is an injective Lie homomorphism
$$
\nu_{(0,3)}\colon \mathfrak{grt}_1 \to \mathfrak{krv}_{(0,3)},
$$
from the Grothendieck-Teichm\"{u}ller Lie algebra $\mathfrak{grt}_1$ to the Kashiwara-Vergne Lie algebra $\mathfrak{krv}_{(0,3)}$, where the image of $\psi \in \mathfrak{grt}_1 \subset L(z_1, z_2)$
is given by
\[
\nu_{(0,3)}(\psi) = (u, u_1, u_2),
\]
where $u_1 = \psi(z_1, -z_1-z_2)$, $u_2=\psi(z_2, -z_1-z_2)$ and $u(z_j) = [z_j,u_j]$ ($j= 1,2$).
\end{thm}

\begin{rem} \label{rem:grtkrv}
Conjecturally, the map $\nu_{(0,3)}$ is an isomorphism starting from weight 3 (that is, starting from degree 2). By the results of \cite{Brown}, the Lie algebra $\mathfrak{grt}_1$ contains a free Lie algebra with the Drinfeld-Ihara generators $\sigma_3, \sigma_5, \dots$ in all odd degrees starting from 3. Also conjecturally, this inclusion is actually surjective. If both conjectures hold true, then
we have 
$$
\mathfrak{krv}_{(0,3)}/{\rm linear \,\, terms} \cong  L(\sigma_3, \sigma_5, \dots).
$$
\end{rem}

In the rest of this section, we will consider the Lie algebra $\mathfrak{krv}_{(1,1)}$ which corresponds to the second nontrivial case (besides $g=0, n=2$) in the existence proof. In this case, the definition above can be written in a more explicit form: 
\begin{align*}
\mathfrak{krv}_{(1,1)} & =
\big\{ u \in {\rm Der}^+(L(x,y)); \text{$u([x,y]) = 0$ and $\exists h \in s^2\K[[s]]$ s.t.} \\
& \hspace{2em} \text{$|d_x(u(x))+d_y(u(y))|=-|h([x,y])|$} \big\}.
\end{align*}
The following proposition establishes a link between the Lie algebras
$\mathfrak{krv}_{(0,3)}$ and $\mathfrak{krv}_{(1,1)}$:

\begin{prop} \label{prop:krv_genus0genus1}
The map $I = {\rm Ad}_{\zeta} \circ \mathcal{E}\colon \tilde{u} \mapsto {\rm Ad}_{\zeta}(u^\text{ell})$ defines a Lie homomorphism
$\mathfrak{krv}_{(0,3)} \to \mathfrak{krv}_{(1,1)}$. Furthermore the Duflo functions of $\tilde{u}$ and $I(u)$ coincide.
\end{prop}

\begin{proof}
This is a direct consequence of Proposition \ref{prop:KRV_genus0genus1}. 
\end{proof}

\begin{exple} \label{exple:tildet}
It is instructive to determine the image under the map $I$ of the element 
$\tilde{t} \in \mathfrak{krv}_{(0,3)}$. We have
$$
t^{\rm ell} \colon x \mapsto [x, (e^{{\rm ad}_x} -1 )y], \hskip 0.3cm
y \mapsto [y, (e^{{\rm ad}_x} -1 )y],
$$
where we have used that $\psi(z_1 + z_2)=(e^{{\rm ad}_x} -1 )y$.
We compute
$$
I(\tilde{t}) \colon x  \mapsto \zeta(t^{\rm ell}(\zeta^{-1}(x))) 
= \zeta(t^{\rm ell}(x))
= \zeta([x, (e^{{\rm ad}_x} -1 )y])
= [x, [x,y]],
$$
where we have used the fact that $\zeta((e^{{\rm ad}_x} -1 )y)=[x,y]$.
Similarly we compute:
$$
I(\tilde{t}) \colon y  \mapsto \zeta(t^{\rm ell}(\zeta^{-1}(y)))
= \zeta([\zeta^{-1}(y), (e^{{\rm ad}_x} -1 )y])
= [y, [x,y]].
$$
Hence $I(\tilde{t})$ is an inner derivation with generator $\omega=[x,y]$.
\end{exple}

By Remarks~\ref{rem:kerI} and \ref{rem:grtkrv}, the intersection of $\nu_{(0,3)}(\mathfrak{grt}_1)$ and ${\rm Ker}(\mathcal{E})$ is trivial. 
Proposition~\ref{prop:krv_genus0genus1} together with Theorem~\ref{thm:grt&krv} imply that there is an injective Lie algebra homomorphism
$$
\nu_{(1,1)}=I \circ \nu_{(0,3)} \colon 
 \mathfrak{grt}_1 \to \mathfrak{krv}_{(0,3)} \to \mathfrak{krv}_{(1,1)}.
$$

Recall from \cite{Enriquez, Pollack} the definition of symplectic derivation $\delta_{2n} \in {\rm Der}^+(L(x,y))$: for $n\geq 1$, $\delta_{2n}$ is the unique derivation of $L(x,y)$ of degree $2n$ such that 
$$
\delta_{2n}(x) = {\rm ad}_x^{2n}(y), \hskip 0.3cm
\delta_{2n}([x,y])=0.
$$
These conditions uniquely determine the value of $\delta_{2n}(y)$:
$$
\delta_{2n}(y)= \sum_{i=0}^{n-1} (-1)^{i} [{\rm ad}_x^i(y), {\rm ad}_x^{2n-1-i}(y)].
$$

\begin{rem}
There is another expression of $\delta_{2n}$ in terms of the map $\sigma_{\rm gr}$:
\[
\delta_{2n} = - \frac{1}{2}\sigma_{\rm gr}(|y\, {\rm ad}_x^{2n}(y)|).
\]
\end{rem}

\begin{prop}  \label{prop:delta_2n}
For all $n \geq 1$, we have $\delta_{2n} \in \mathfrak{krv}_{(1,1)}$.
\end{prop}

\begin{proof}
By definition, we have $\delta_{2n}([x,y])=0$ which is the first defining property of elements in $\mathfrak{krv}_{(1,1)}$. We need to compute
$$
{\sf div}(\delta_{2n}) = 
|d_x \delta_{2n}(x) + d_y \delta_{2n}(y)|.
$$
For the first term, using Lemma~\ref{lem:dsfadxa} we compute
$$
|d_x \delta_{2n}(x)| = | d_x ({\rm ad}_x^{2n}(y))| = |x^{2n} (d_x y) - x^{2n-1}y| = -|x^{2n-1}y|.
$$
For the second term, the direct computation shows
$$
|d_y \delta_{2n}(y)| =
\sum_{i=0}^{n-1} (-1)^i |{\rm ad}_x^i(y) x^{2n-1-i} - {\rm ad}_x^{2n-1-i}(y) x^i| = |y x^{2n-1}|.
$$
In conclusion, ${\sf div}(\delta_{2n}) =|yx^{2n-1}| - |x^{2n-1} y| =0$ which together with $\delta_{2n}([x,y])=0$ implies that $\delta_{2n} \in \mathfrak{krv}_{(1,1)}$.
\end{proof}

\subsection{${\rm GRT}_1$ and ${\rm GT}_1$ subgroups} \label{subsec:GRTGT}

In this subsection, we construct group homomorphisms from the Grothendieck-Teichm\"{u}ller groups ${\rm GRT}_1$ and ${\rm GT}_1$ to the Kashiwara-Vergne groups.

Recall that the injective Lie algebra homomorphism $\nu_{(0,3)} \colon \mathfrak{grt}_1 \to \mathfrak{krv}_{(0,3)}$ integrates to an injective group homomorphism
$$
\nu_{(0,3)} \colon {\rm GRT}_1 \to {\rm KRV}_{(0,3)}.
$$
Furthermore, by Theorem 9 in \cite{AET}, there is a natural injective group homomorphism
$$
\tilde{\nu}_{(0,3)} \colon {\rm GT}_1 \to {\rm KV}_{(0, 3)}.
$$

Let $\Sigma$ be a surface of genus $g$ with $n+1$ boundary components equipped with a framing $f$. We will consider pair of pants decompositions of $\Sigma$ parametrized by the set ${\bf T}$ of planar binary oriented (in the direction of the root) rooted trees 
$T \in {\bf T}$ with $g+n$ leaves labeled by symbols $[x_i, y_i]$ for $i=1, \dots, g$ and $z_j$ for $j=1, \dots, n$.
We give a linear ordering to the leaves of $T$ by arranging them and the root in the clockwise manner: the leaf just after the root becomes the first leaf.
Furthermore we consider the subset ${\bf T}^{\rm ord} \subset {\bf T}$ which consists of trees with the leaves labeled by $[x_1,y_1], \ldots, [x_g, y_g], z_1, \ldots, z_n$ in this order.
In Figure \ref{fig:trees}, the tree on the left hand side is an element in ${\bf T}$ but not in ${\bf T}^{\rm ord}$, while the tree on the right hand side is an element in ${\bf T}^{\rm ord}$.

\begin{figure}
\begin{center}
{\unitlength 0.1in%
\begin{picture}(39.0000,9.8000)(3.0000,-14.0000)%
%
\special{pn 13}%
\special{pa 1000 1400}%
\special{pa 1000 1200}%
\special{fp}%
%
\special{pn 13}%
\special{pa 1000 1200}%
\special{pa 400 600}%
\special{fp}%
%
\special{pn 13}%
\special{pa 1000 1200}%
\special{pa 1600 600}%
\special{fp}%
%
\special{pn 13}%
\special{pa 1200 1000}%
\special{pa 800 600}%
\special{fp}%
%
\special{pn 13}%
\special{pa 1000 800}%
\special{pa 1200 600}%
\special{fp}%
%
\special{pn 4}%
\special{sh 1}%
\special{ar 1000 1400 16 16 0 6.2831853}%
%
\special{pn 13}%
\special{pa 3600 1400}%
\special{pa 3600 1200}%
\special{fp}%
%
\special{pn 13}%
\special{pa 3600 1200}%
\special{pa 3000 600}%
\special{fp}%
%
\special{pn 13}%
\special{pa 3600 1200}%
\special{pa 4200 600}%
\special{fp}%
%
\special{pn 13}%
\special{pa 3800 1000}%
\special{pa 3400 600}%
\special{fp}%
%
\special{pn 13}%
\special{pa 3600 800}%
\special{pa 3800 600}%
\special{fp}%
%
\special{pn 4}%
\special{sh 1}%
\special{ar 3600 1400 16 16 0 6.2831853}%
\put(5.5000,-5.5000){\makebox(0,0)[lb]{$[x_1,y_1]$}}%
\put(10.5000,-5.5000){\makebox(0,0)[lb]{$[x_2,y_2]$}}%
\put(3.0000,-5.5000){\makebox(0,0)[lb]{$z_2$}}%
\put(16.0000,-5.5000){\makebox(0,0)[lb]{$z_1$}}%
\put(27.5000,-5.5000){\makebox(0,0)[lb]{$[x_1,y_1]$}}%
\put(32.5000,-5.5000){\makebox(0,0)[lb]{$[x_2,y_2]$}}%
\put(38.0000,-5.5000){\makebox(0,0)[lb]{$z_1$}}%
\put(41.0000,-5.5000){\makebox(0,0)[lb]{$z_2$}}%
\end{picture}}%
\end{center}
\caption{Trees ($g = 2, n = 2$)}
\label{fig:trees}
\end{figure}

The main result of this subsection is the following theorem:
\begin{thm} \label{thm:treeGTKV}
    For each tree $T \in {\bf T}$ there is an injective group homomorphism 
    $$
\nu^T \colon {\rm GRT}_1 \to {\rm KRV}^f_{(g, n+1)}.
    $$
    Furthermore, for each $T \in {\bf T}^{\rm ord}$, there is an injective group homomorphism
    $$
\tilde{\nu}^T \colon {\rm GT}_1 \to {\rm KV}^f_{(g, n+1)}.
    $$
\end{thm}

\begin{proof}
For a tree $T \in {\bf T}$, we label the inner vertices on $T$ by sums of expressions corresponding to the incoming branches, so that the root is labeled by the element $\omega$. We associate to $T$ a group homomorphism
$$
a^T \colon \left({\rm KRV}_{(0,3)}\right)^{g+n}_h \to {\rm KRV}_{(g, n+1)}^f.
$$
Here $\left({\rm KRV}_{(0,3)}\right)^{g+n}_h$  is the subgroup of the direct product which consists of elements with the same Duflo function $h$. The group homomorphism $a^T$  is constructed using the group homomorphism $I \colon {\rm KRV}_{(0,3)} \to {\rm KRV}_{(1,1)}$ of Proposition \ref{prop:KRV_genus0genus1} for each leaf labeled by $[x_i, y_i]$ and then using the group homomorphisms of Proposition \ref{prop:KRV_gluing} for each inner vertex of the tree $T$. 
Here we are using the facts that the groups ${\rm KRV}_{(0,3)}$ and ${\rm KRV}_{(1,1)}$ are independent of the choice of framing. 
Finally we define $\nu^T$ as a composition 
$$
\nu^T = a^T \circ \nu_{(0,3)}^{g+n}.
$$
Here we have used the fact that the map $\nu_{(0,3)}^{g+n}$ lands in the subgroup 
$\left({\rm KRV}_{(0,3)}\right)^{g+n}_h$ since all the Duflo functions are the same. 

    The construction of $\tilde{\nu}^T$ is similar. In more detail, 
    $$
\tilde{\nu}^T = \tilde{a}^T \circ \tilde{\nu}_{(0,3)}^{g+n},
    $$
    where $\tilde{a}^T \colon ({\rm KV}_{(0,3)})^{g+n}_h \to {\rm KV}^f_{(g, n+1)}$ is constructed using Propositions~\ref{prop:KRV_genus0genus1} and \ref{prop:glue_kv}. 
    Note that the gluing of Proposition \ref{prop:glue_kv} works only if $n_1=0$ or $g_2=0$, and this implies restrictions on the choice of labeled trees 
    $T \in {\bf T}^{\rm ord} \subset {\bf T}$.
\end{proof}

\begin{rem}
Let ${\rm KV}^{\rm small}_{(g,n+1)}$ be the set of elements $\tilde{G} = (G, g_1, \ldots, g_n) \in {\rm tAut}^+(L)$ such that
\begin{itemize}
\item $G(\xi) = \xi$,
\item $G(x_i)$ and $G(y_i)$ do not contain terms linear in $z_j$,
\item $\log g_j$ does not contain terms linear in $x_i, y_i$, and
\item ${\sf j}_{x,y,z}(G) + G({\bf r}) - {\bf r} = \sum_j |h_j(z_j)| - |h(\xi)|$ for some $h_j \in s \K[[s]], h \in s^2 \K[[s]]$.
\end{itemize}
Similarly to Proposition \ref{prop:KVKRVsubgroup}, one can check that ${\rm KV}^{\rm small}_{(g,n+1)}$ is a subgroup of ${\rm tAut}^+(L)$.
Any $\tilde{G} \in {\rm KV}^{\rm small}_{(g,n+1)}$ preserves ${\bf p}^f$ and $|\log g_i|$ is in the linear span of $z_j$'s.
Therefore ${\rm KV}^{\rm small}_{(g,n+1)}$ is a subgroup of ${\rm KV}^f_{(g,n+1)}$ for any framing $f$.
Similarly let ${\rm KRV}^{\rm small}_{(g,n+1)}$ be the set of elements $\tilde{G} \in {\rm tAut}^+(L)$ such that
\begin{itemize}
\item $G(\omega) = \omega$,
\item $\log g_j$ does not contain terms linear in $x_i, y_i$, and
\item ${\sf j}_{x,y,z}(G) = \sum_j |h_j(z_j)| - |h(\omega)|$ for some $h_j \in s\K[[s]], h \in s^2 \K[[s]]$.
\end{itemize}
Then ${\rm KRV}^{\rm small}_{(g,n+1)}$ is a subgroup of ${\rm KRV}^f_{(g,n+1)}$ for any framing $f$.

Inspecting the proof of Theorem \ref{thm:treeGTKV}, we see that the maps $a^T$ and $\tilde{a}^T$ (and hence $\nu^T$ and $\tilde{\nu}^T$) take values in these small version of the KV groups.
\end{rem}

\subsection{Kashiwara-Vergne groups as automorphism groups}

Recall that the completed group algebra $\widehat{\K\pi}$ carries the double bracket $\kappa$ and coaction maps $\mu^f_r, \mu^f_l$ (which depend on the framing $f$). 
The (completed) space of cyclic words $|\widehat{\K\pi}|$ carries a natural Lie bialgebra structure
$$
\hat{\mathfrak{g}}(\Sigma, f) = (|\widehat{\K\pi}|, [\cdot, \cdot], \delta^f),
$$
where $[\cdot, \cdot]$ is the Goldman bracket and $\delta^f$ is the Turaev cobracket for the framing $f$. 
Similarly the associated graded $A={\rm gr} \, \widehat{\K\pi}$ carries the graded double bracket $\kappa_{\rm gr}$ and the coaction maps $\mu^f_{r, {\rm gr}}, \mu^f_{l, {\rm gr}}$, and
$|A| = {\rm gr} \, |\widehat{\K\pi}|$ carries a graded Lie bialgebra structure
$$
{\rm gr}\, \widehat{\mathfrak{g}}(\Sigma, f)=
(|A|, [\cdot, \cdot]_{\rm gr}, \delta^f_{\rm gr}).
$$

We will be interested in several automorphism groups of these objects. In particular, we will consider the group
$
{\rm tAut^+}(\widehat{\K\pi}, \Delta, \kappa, \delta^f).
$
This is a group of Hopf automorphisms of the completed group algebra $\widehat{\K\pi}$ whose associated graded is the identity which have two additional properties. First they preserve the double bracket $\kappa$, and second, the induced endomorphism of $|\widehat{\K\pi}|$ preserves the cobracket $\delta^f$. In a similar fashion, we define a group ${\rm tAut}^+(L, \kappa_{\rm gr}, \delta^f_{\rm gr})$.
We will also consider the group 
$
{\rm tAut}^+(\widehat{\K\pi}, \Delta, \kappa, \mu^f_r, \mu^f_l)
$
which preserves the natural structures on $\widehat{\K\pi}$, and the group
${\rm tAut}^+(L, \kappa_{\rm gr}, \mu^f_{r, {\rm gr}}, \mu^f_{l, {\rm gr}})$.

The KV groups are related to these automorphism groups by the following result:
\begin{thm}   \label{thm:KRV_acts}
There are natural group isomorphisms
$$
\begin{array}{ll}
{\rm tAut}^+(\widehat{\K\pi}, \Delta, \kappa, \delta^f) \cong {\rm KV}^f, &
{\rm tAut}^+(L, \kappa_{\rm gr}, \delta^f_{\rm gr}) \cong {\rm KRV}^f, \\
{\rm tAut}^+(\widehat{\K\pi}, \Delta, \kappa, \mu^f_r, \mu^f_l) \cong {\rm KV}_0^f, &
{\rm tAut}^+(L, \kappa_{\rm gr}, \mu^f_{r, {\rm gr}}, \mu^f_{l, {\rm gr}}) \cong {\rm KRV}_0^f. 
\end{array}
$$
Here, for the $KV$ groups we fix the isomorphism $\widehat{\K\pi} \cong A$ determined by the exponential expansion $\theta_{\exp}$ associated to a standard generating system $\{ \alpha_i, \beta_i, \gamma_j \}$.
\end{thm}

\begin{proof}
Recall that by definition
${\rm KV}_0^f \subset {\rm KV}^f \subset {\rm tAut}^+(\widehat{\K\pi}, \Delta)$ and 
${\rm KRV}_0^f \subset {\rm KRV}^f \subset {\rm tAut}^+(L)$.
Hence one can directly compare the groups in question as subgroups of 
${\rm tAut}^+(L)$ or of ${\rm tAut}^+(\widehat{\K\pi}, \Delta)$.
Furthermore recall that by Theorem \ref{thm:Florian}
$$
\begin{array}{l}
{\rm tAut}^+(\widehat{\K\pi}, \Delta, \kappa) =
\{ \tilde{F} \in {\rm tAut}^+(\widehat{\K\pi}, \Delta); F(\xi)=\xi\}, \\
{\rm tAut}^+(L, \kappa_{\rm gr}) =
\{ \tilde{F} \in {\rm tAut}^+(L, \kappa_{\rm gr}); F(\omega)=\omega\}.
\end{array}
$$
Hence it remains to consider the behavior of coaction and cobracket maps under the action of tangential automorphisms.

Let $\tilde{G} \in {\rm tAut}^+(L, \kappa_{\rm gr})$
and apply it to the graded cobracket $\delta^f_{\rm gr} = {\sf Div}^f_{\rm gr} \circ \tilde{\sigma}_{\gr}$
(see equation \eqref{eq:delta=c_sigma_gr}).
Since $\tilde{G}$ preserves $\kappa_{\rm gr}$, it also preserves the map $\sigma_{\rm gr}$.
Moreover the map ${\sf Div}^f_{\rm gr}$ transforms as a $1$-cocycle as follows:
$$
\tilde{G}: {\sf Div}^f_{\rm gr} \mapsto {\sf Div}^f_{\rm gr} + d({\sf J}^f_{\rm gr}(\tilde{G})) = {\sf Div}^f_{\rm gr} + d(\tilde{\Delta}({\sf j}^f_{\rm gr}(\tilde{G}))).
$$
Similar to the proof of Proposition \ref{prop:homomorphic=center}, $\delta^f_{\rm gr}$ is preserved by $\tilde{G}$ if and only if
$$
{\sf j}^f_{\rm gr}(\tilde{G}) \in Z(|A|, [\cdot, \cdot]_{\rm gr}),
$$
and this is equivalent to 
$$
{\sf j}^f_{\rm gr}(\tilde{G}) = \sum_{j=1}^n |h_j(z_j)| - |h(\omega)|,
$$
as required. Hence we have proved the isomorphism ${\rm tAut}^+(L, \kappa_{\rm gr}, \delta^f_{\rm gr}) \cong {\rm KRV}^f$. The proof of the isomorphism ${\rm tAut}^+(\widehat{\K\pi}, \Delta, \kappa, \delta^f) \cong {\rm KV}^f$ is similar.

Next we consider the behavior under the action of $\tilde{G}$ of the map
$\mu^f_{l, {\rm gr}} = {\sf Div}^f_{l, {\rm gr}} \circ \widetilde{{\rm Ham}^{\kappa_{\rm gr}}}$. Again, we use the fact that $\tilde{G}$ preserves the double bracket $\kappa_{\rm gr}$, and hence the map $\widetilde{{\rm Ham}^{\kappa_{\rm gr}}}$. By using the transformation property of ${\sf Div}^f_{\rm gr}$, we compute for $a \in A$:
\begin{equation} \label{eq:mu_tildeG}
\begin{array}{lll}
(\mu^f_{l, {\rm gr}})^{\tilde{G}}(a) & = & \big( {\sf Div}^f_{\rm gr}(\{ |sa|, \cdot\}_{\rm gr}) +
\{ |sa|, {\sf J}^f_{\rm gr}(\tilde{G})\}_{\rm gr} \big)_{A \otimes |A|} \\
& = & \big( {\sf Div}^f_{\rm gr}(\{ |sa|, \cdot\}_{\rm gr}) +
\{ |sa|, \tilde{\Delta}({\sf j}^f_{\rm gr}(\tilde{G}))\}_{\rm gr} \big)_{A\otimes |A|} \\
& = & 
\mu^f_{l, {\rm gr}}(a) + \{ a, B'\}_{\rm gr} \otimes B''.
\end{array}
\end{equation}
Here the subscript $A\otimes |A|$ stands for taking the projection onto $|sA| \otimes |A| \cong A \otimes |A|$ and $B=\tilde{\Delta}({\sf j}^f_{\rm gr}(\tilde{G})) = B' \otimes B''$.
Therefore $\tilde{G}$ preserves the coaction map $\mu^f_{l, {\rm gr}}$ if and only if the expression
$\{ a, B'\}_{\rm gr} \otimes B''$ vanishes. Taking its component in $A \otimes {\bf 1}$, we conclude that $\{a, {\sf j}^f_{\rm gr}(\tilde{G})\}_{\rm gr}=0$ for all $a \in A$. This implies ${\sf j}^f_{\rm gr}(\tilde{G}) \in \ker(\sigma_{\rm gr})$.
By Proposition \ref{prop:ker_sigma_gr}, this is equivalent to
\begin{equation}  \label{eq:B=sum_h}
{\sf j}^f_{\rm gr}(\tilde{G}) = \sum_{j=1}^n |h_j(z_j)|,
\end{equation}
as required.

In the other direction, if ${\sf j}^f_{\rm gr}(\tilde{G}) \in \ker(\sigma_{\rm gr})$, then ${\sf j}^f_{\rm gr}(\tilde{G})$ is of the form \eqref{eq:B=sum_h}. This implies
$B=\tilde{\Delta}({\sf j}^f_{\rm gr}(\tilde{G})) \in \ker(\sigma_{\rm gr}) \otimes \ker(\sigma_{\rm gr})$ and $\{ a, B'\}_{\rm gr} \otimes B'' =0$ for all $a \in A$. By the transformation property \eqref{eq:mu_tildeG}, this implies
$(\mu^f_{l, {\rm gr}})^{\tilde{G}} = \mu^f_{l, {\rm gr}}$, as required.
The argument for $\mu^f_{r, {\rm gr}}$ is similar to the above.

In conclusion, we have established an isomorphism ${\rm tAut}^+(L, \kappa_{\rm gr}, \mu^f_{r, {\rm gr}}, \mu^f_{l, {\rm gr}}) \cong {\rm KRV}_0^f$.
The proof of the isomorphism ${\rm tAut}^+(\widehat{\K\pi}, \Delta, \kappa, \mu^f_r, \mu^f_l) \cong {\rm KV}_0^f$ is similar.
\end{proof}

\begin{cor}
For each tree $T \in {\bf T}$ there is an action of the group ${\rm GRT}_1$ by automorphisms of the graded Lie bialgebra ${\rm gr} \, \widehat{\mathfrak{g}}(\Sigma, f)$.
Furthermore, for each tree $T \in {\bf T}^{\rm ord}$, there is an action of the group ${\rm GT}_1$ by automorphims of the Lie bialgebra $\widehat{\mathfrak{g}}(\Sigma, f)$.
\end{cor}

\subsection{The mapping class group orbits of framings and KV problems} \label{subsec:orbitsKV}

Notice that the mapping class group $\mathcal{M} = \mathcal{M}(\Sigma)$ acts on the set of homotopy classes of framings on $\Sigma$ in a natural way.
In this section, we continue the discussion in Section~\ref{subsec:comparison_diff_framings} and compare the KV groups and the KV problems for different $\mathcal{M}$-orbits of framings.

\begin{prop} \label{prop:ff'equiv}
Suppose that two framings $f$ and $f'$ on $\Sigma$ are in the same $\mathcal{M}$-orbit.
Then there are isomorphisms of the KV groups, ${\rm KV}^f \cong {\rm KV}^{f'}$ and ${\rm KRV}^f \cong {\rm KRV}^{f'}$, and the KV problems for $f$ and for $f'$ are equivalent.
\end{prop}

\begin{proof}
Let $\varphi \in \mathcal{M}$ be a mapping class sending $f$ to $f'$.
Since $\varphi$ preserves the conjugacy class of $\gamma_j$ for each $j$, there exist elements $\delta_1, \ldots, \delta_n \in \pi$ such that $\varphi(\gamma_j) = \delta_j^{-1} \gamma_j \delta_j$.
Extend $\varphi$ to a Hopf algebra automorphism of $\widehat{\K \pi}$ and take its tangential lift $\tilde{\varphi} = (\varphi, \delta_1,\ldots,\delta_n) \in {\rm tAut}(\widehat{\K \pi}, \Delta)$.
(See Remark~\ref{rem:tAut(A)} (i) for the definition of the group ${\rm tAut}(\widehat{\K \pi}, \Delta)$.)
Since $\varphi$ induces an isomorphism
\[
(\widehat{\K \pi}, \Delta, \kappa, \delta^f) 
\overset{\varphi}{\to}
(\widehat{\K \pi}, \Delta, \kappa, \delta^{f'}),
\]
it follows that for any $\tilde{G} \in {\rm KV}^f$ the composition $\tilde{\varphi} \circ \tilde{G} \circ \tilde{\varphi}^{-1}$ induces an automorphism of $(\widehat{\K \pi}, \Delta, \kappa, \delta^{f'})$ and thus is an element of ${\rm KV}^{f'}$ by Theorem~\ref{thm:KRV_acts}.
In this way, we obtain a group isomorphism
\[
{\rm KV}^f \cong {\rm KV}^{f'}, \quad
\tilde{G} \mapsto \tilde{\varphi} \circ \tilde{G} \circ \tilde{\varphi}^{-1},
\]
and similarly
\[
{\rm KRV}^f \cong {\rm KRV}^{f'}, \quad
\tilde{G} \mapsto {\rm gr}\, \tilde{\varphi} \circ \tilde{G} \circ {\rm gr}\, (\tilde{\varphi})^{-1}.
\]
Finally the following bijective map
\[
{\rm tAut}^+(\widehat{\K \pi}, \Delta) \to 
{\rm tAut}^+(\widehat{\K \pi}, \Delta), \quad
\tilde{F} \mapsto  \tilde{F}_{\tilde{\varphi}} :={\rm gr}\, \tilde{\varphi} \circ \tilde{F} \circ \tilde{\varphi}^{-1}
\]
restricts to a bijective correspondence from ${\rm SolKV}^f$ to ${\rm SolKV}^{f'}$.
\end{proof}

In some cases (in fact, in almost all cases), framings in different $\mathcal{M}$-orbits have the isomorphic KV groups and the equivalent KV problems.
For the case of $g = 0$, this is always the case: see Remarks~\ref{rem:framing_genus_0} and \ref{rem:f_independent}.

In \cite{Ka17} the $\mathcal{M}$-orbits of the homotopy classes of framings were computed.
With this result in mind, let us examine the case of $g \ge 1$.
First note that the action of $\mathcal{M}$ on homotopy classes of framings preserves the rotation numbers along boundary components, so it preserves the element ${\bf q}^f=\sum_j ({\rm rot}^f(\gamma_j) + 1) |z_j|$ as well.
If $g \ge 2$ or $g = 1$ and ${\bf q}^f \neq 0$, the arguments used in the proof of Propositions~\ref{prop:KVcase1} and \ref{prop:KVcase2} show that any framing (with ${\bf q}^f \neq 0$ when $g=1$) produces KV groups isomorphic to those associated with the adapted framing, and a similar statement holds for KV problems.

The remaining is the case of $g = 1$ and ${\bf q}^f = 0$.
Recall from Theorem~\ref{thm:KVsolve} that the KV problem admits solutions only for the adapted framing.
As was shown in \cite[Theorem 1.3]{Ka17}, 
the set of $\mathcal{M}$-orbits of framings with $g = 1$ and ${\bf q}^f = 0$ 
has a bijective correspondence with 
the infinite set $\mathbb{Z}_{\ge 0}$, which is given by taking the non-negative integer which generates the ideal of $\mathbb{Z}$ spanned by the rotation numbers of all non-separating simple closed curves.
The adapted framing corresponds to $0 \in \mathbb{Z}_{\ge 0}$, and it is the unique framing whose rotation numbers vanish on all non-separating simple closed curves.
Hence it is preserved by the action of $\mathcal{M}(\Sigma)$.
In summary, when $g = 1$ and ${\bf q}^f = 0$, the orbit $\{ f^{\rm adp} \}$ is the unique $\mathcal{M}$-orbit which produces KV problems admitting solutions.

\section{Application to Johnson obstructions}
\label{sec:Johnson_obstruction}

In this section, we give an application of the formality of the Turaev cobracket to the study of the Johnson homomorphism for the mapping class group.
As explained in Remark~\ref{rem:Torelli_n=0}, we restrict our attention to the case of a surface of genus $g$ with one boundary component, $\Sigma = \Sigma_{g,1}$. 
Recall that $L = L(H)$ is the free Lie algebra generated by $H = H_1(\Sigma, \K)$.

\subsection{Johnson homomorphism}      \label{sec:Johnson}

Let us briefly recall some basic facts about the Johnson homomorphism. 
For more detail, see, for example \cite{Joh80, Joh83, MoAQ, KK16}. 
As defined in Section~\ref{subsec:mcgauto}, we have the Johnson filtration $\{ \mathcal{I}(k) \}_k$ of the Torelli group $\mathcal{I} = \mathcal{I}(\Sigma)$ and the associated graded Lie algebra 
\[
\Lie_{\rm gr}(\mathcal{I}) = \bigoplus_{k=1}^{\infty} \mathcal{I}(k)/\mathcal{I}(k+1).
\]
Let $\{ \Gamma_m \pi \}_{m \ge 1}$ be the lower central series of $\pi$.
It is a central filtration of $\pi$ defined inductively by $\Gamma_1 \pi = \pi$ and $\Gamma_{m+1} \pi = [\Gamma_m \pi, \pi]$ for $m\ge 1$.
The associated graded $\mathbb{Z}$-module 
\[
\Lie_{\rm gr} (\pi) = \bigoplus_{m=1}^{\infty} \Gamma_m \pi/ \Gamma_{m+1} \pi
\]
is a graded Lie algebra where the Lie bracket is induced from the group commutator $[\alpha, \beta] = \alpha \beta \alpha^{-1} \beta^{-1}$.
It is freely generated by the degree one part $\Gamma_1 \pi/\Gamma_2 \pi \cong H_1(\Sigma, \mathbb{Z})$, and there is a graded Lie algebra isomorphism
\begin{equation} \label{eq:Liegrpi}
\Lie_{\rm gr}(\pi) \otimes_{\mathbb{Z}} \K \cong L
\end{equation}
whose degree $1$ part coincides with the canonical identification $(\Gamma_1 \pi / \Gamma_2 \pi) \otimes_{\Z} \K \cong H$.
The boundary loop $\gamma_0 \in \Gamma_2 \pi$ defines an element of degree $2$, and it corresponds to $\omega$.

\begin{rem} \label{rem:onLiegrpi}
For any $m \ge 1$ and $\gamma \in \pi$, the condition $\gamma \in \Gamma_m \pi$ is equivalent to $\gamma - 1 \in \K \pi (m)$.
The map $\Gamma_m \pi \ni \gamma \mapsto (\gamma-1) \in \K \pi(m)$ induces a canonical embedding ${\rm Lie}_{\rm gr}(\pi) \subset {\rm gr}\, \K \pi$ which is compatible with the isomorphism \eqref{eq:Liegrpi}.
\end{rem}

The Dehn-Nielsen action of the mapping class group $\mathcal{M} = \mathcal{M}(\Sigma)$ on $\pi$ induces a Lie action of the graded Lie algebra ${\rm Lie}_{\rm gr}(\mathcal{I})$ on ${\rm Lie}_{\rm gr} (\pi)$.
The {\em Johnson homomorphism} (in the formulation of Morita \cite{MoAQ} refining the original definition by Johnson \cite{Joh80, Joh83}) describes it as an injective homomorphism of graded Lie algebras
\begin{equation} \label{eq:tau}
\tau^{\mathbb{Z}} \colon 
\Lie_{\rm gr}(\mathcal{I}) \rightarrow 
{\rm Der}^+_{[\gamma_0]} (\Lie_{\rm gr}(\pi)), 
\end{equation} 
where the target is the Lie algebra of positive derivations of $\Lie_{\rm gr} (\pi)$ that annihilates $[\gamma_0] \in \Gamma_2 \pi/\Gamma_3 \pi$, the class of $\gamma_0$. 
The degree $k$ part of $\tau^{\mathbb{Z}}$ is defined as follows. 
Note that every $\varphi \in \mathcal{I}(k)$ satisfies $\varphi(\gamma) \gamma^{-1} \in \Gamma_{k+l} \pi$ for any $l \ge 1$ and $\gamma \in \Gamma_l \pi$. 
With this in mind, the derivation $\tau^{\mathbb{Z}}([\varphi])$ is given by the following formula:  
\[
\Gamma_l\pi / \Gamma_{l+1} \pi \ni [\gamma] \mapsto [\varphi(\gamma)\gamma^{-1}] \in 
\Gamma_{k+l} \pi / \Gamma_{k+l+1} \pi.
\]

The map $\tau^{\mathbb{Z}}$ is equivariant under the action of the integral symplectic group ${\rm Sp}_{2g}(\mathbb{Z}) \cong {\rm Aut}(H_1(\Sigma, \Z), \langle \cdot, \cdot \rangle) \cong \mathcal{M}/\mathcal{I}$. 
To make use of representation theory of the symplectic group, one usually works over a field of characteristic $0$. 
By applying $\otimes_{\mathbb{Z}} \K$ to \eqref{eq:tau} we obtain the following injective homomorphism of graded Lie algebras over $\K$, which we also call the Johnson homomorphism:  
\begin{equation} \label{eq:tauK}
\tau = (\tau_k)_k \colon
\Lie_{\rm gr}(\mathcal{I}) \otimes_{\mathbb{Z}} \K \rightarrow
{\rm Der}^+_{\omega}(L)
= \bigoplus_{k=1}^{\infty} {\rm Der}^k_{\omega} (L).
\end{equation}
Here, the target is the Lie subalgebra of ${\rm Der}^+(L)$ that consist of derivations annihilating $\omega$, and we identify it with ${\rm Der}^+_{[\gamma_0]} (\Lie_{\rm gr}(\pi)) \otimes_{\Z} \K$ by using 
\eqref{eq:Liegrpi}.

The central problem concerning the Johnson homomorphism is to determine its image and cokernel. 
On the one hand, there is a deep result of Hain:

\begin{thm}[\cite{Hain97}] \label{thm:HainJohnson}
The image of $\tau$ is generated by the degree $1$ part ${\rm Im}\, \tau_1 = {\rm Der}^1_{\omega}(L)$. 
\end{thm}

On the other hand, the structure of the cokernel is rather mysterious, and its study is an active area of research.
In \cite{MoAQ}, Morita showed that the cokernel is nontrivial. 
More precisely, he introduced a surjective homomorphism 
${\rm Tr}_k \colon {\rm Der}^k_{\omega}(L) \to S^k H$ such that ${\rm Tr}_k \circ \tau_k=0$ 
for  each odd $k\ge 3$.
Here, $S^k H$ is the $k$th symmetric product of $H$.
After the work of Morita, other obstructions for surjectivity of the Johnson homomorphisms have been found. 
For a recent survey, see \cite{Hain_survey}.

\begin{prob}[Johnson image problem]
For each $k\ge 1$,
find an $\mathfrak{sp}_{2g}$-module $M_k$ (with an explicit decomposition into irreducible $\mathfrak{sp}_{2g}$-modules) and a surjective $\mathfrak{sp}_{2g}$-homomorphism ${\rm Der}^k_{\omega}(L) \to M_k$ with kernel ${\rm Im}\, \tau_k$.
\end{prob}

In \cite{ES}, Enomoto and Satoh introduced a refinement of Morita's map. 
They defined a homomorphism
${\sf ES}_k \colon {\rm Der}^k_{\omega}(L) \to |A|_k$ such that ${\sf ES}_k \circ \tau_k=0$ for each $k\ge 2$. 
It factors through the map ${\rm Tr}_k$ via the natural map $|A|_k = H^{\otimes k}/\mathbb{Z}_k \to S^k H$. 
The definition of the {\em Enomoto-Satoh trace} ${\sf ES}_k$ is as follows: 
\[
{\sf ES}_k \colon 
{\rm Der}^k_{\omega}(L)
\hookrightarrow {\rm Hom}(H, L_{k+1})
\hookrightarrow H^* \otimes H^{\otimes k+1}
\xrightarrow{\text{cont}}
H^{\otimes k}
\xrightarrow{\text{projection}}
|A|_k.
\]
Here, the first map sends the derivation $u$ to its restriction to the degree one part $u|_H$, the second map uses the inclusion $L_{k+1} \subset H^{\otimes (k+1)}$, and the third map is the contraction map $f \otimes a_1 \otimes a_2 \otimes \cdots \otimes a_{k+1} \mapsto f(a_1) a_2 \otimes \cdots \otimes a_{k+1}$. 

\begin{rem}
Using Hain's result mentioned above, Morita, Sakasai and Suzuki \cite{MSS} determined the structure of $M_k$ for $k=6$.
For $k\le 5$, see the references in \cite{MSS}.
\end{rem}

\subsection{Filtered version and geometric Johnson homomorphism}

In this section, we consider the filtered version of the Johnson homomorphism, and we recall a more geometric interpretation of its target by using the Goldman bracket. 

For any element $\varphi$ in the Torelli group $\mathcal{I}$, one can consider its logarithm $\log \varphi$ as a derivation of $\widehat{\K \pi}$. This derivation has a  positive degree with respect to the weight filtration.  
The map $\log \varphi$ is a Hopf derivation, and it annihilates $\log \gamma_0$ since $\varphi$ preserves $\gamma_0$.
This construction gives rise to the following filtration-preserving injective group homomorphism
\[
\tau^{\rm filt} \colon
\mathcal{I} \rightarrow 
{\rm Der}^+_{\log \gamma_0} (\widehat{\K \pi}, \Delta),
\quad
\varphi \mapsto \log \varphi,  
\]
where the group structure of the target is given by the BCH series.
Recall that the associated graded of ${\rm Der}_{\log \gamma_0}^+ (\widehat{\K \pi}, \Delta)$ is canonically isomorphic to ${\rm Der}_{\omega}^+(L) = {\rm Der}_{\omega}^+(A, \Delta)$. Hence, one can define the associated graded of the map $\tau^{\rm filt}$ as a map
$$
(\tau^{\rm filt})_{\rm gr}: {\rm Lie}_{\rm gr}(\mathcal{I}) \to {\rm Der}^+_{\omega}(L).
$$
In more detail, for $\varphi \in \mathcal{I}(k)$ denote by $[\varphi] \in {\rm Lie}_{\rm gr}(\mathcal{I})$ the corresponding element in the
graded Lie algebra ${\rm Lie}_{\rm gr}(\mathcal{I})$. Note that for
$\varphi \in \mathcal{I}(k)$ and $\psi \in \mathcal{I}(k+1)$ the filtration degree of the element
\[
\log(\varphi \psi) - \log(\varphi) = \log \psi + \frac{1}{2}[\log \varphi, \log \psi] + \cdots
\]
is at least $k+1$.  Hence, the formula
$$
(\tau^{\rm filt})_{\rm gr}([\varphi])
= {\rm gr}(\log \varphi)
$$
defines a map
$$
(\tau^{\rm filt})_{\rm gr}: \mathcal{I}(k)/\mathcal{I}(k+1)
\to {\textstyle \frac{{\rm Der}_{\log \gamma_0}^+ (\widehat{\K \pi}, \Delta)(k)}{{\rm Der}_{\log \gamma_0}^+ (\widehat{\K \pi}, \Delta)(k+1)}}
= {\rm Der}_{\omega}^k(L).
$$

\begin{prop} \label{prop:grtauun}
The map $(\tau^{\rm filt})_{\rm gr}$ coincides with the Johnson homomorphism $\tau$.
\end{prop}

\begin{proof}
For any $k \ge 1$, $\varphi \in \mathcal{I}(k)$ and $\gamma \in \pi$, we have
\[
(\log \varphi) (\gamma - 1)  \equiv_{k+2} 
(\varphi - {\rm id})(\gamma - 1) 
  = (\varphi(\gamma)\gamma^{-1} - 1) \gamma 
  \equiv_{k+2}  (\varphi(\gamma)\gamma^{-1} - 1).
\]
Here, $\equiv_{k+2}$ denotes equivalence modulo $\widehat{\K \pi}(k+2)$, and we use the fact that $\varphi(\gamma)\gamma^{-1} - 1 \in \K \pi(k+1)$. 
We conclude that 
$$
\left( (\tau^{\rm filt})_{\rm gr} ([\varphi]) \right) ([\gamma]) = [\varphi(\gamma)\gamma^{-1}] = \left( \tau([\varphi]) \right)([\gamma]),
$$
as desired (see also Remark \ref{rem:onLiegrpi}). 
\end{proof}

In \cite{KK15, KK16}, a more geometric description of the map $\tau^{\rm filt}$ was introduced. 
By Proposition~\ref{prop:ker_sigma}, the kernel of the map $\sigma \colon |\widehat{\K \pi}| \to {\rm Der}(\widehat{\K \pi})$ is $\K {\bf 1}$. 
Thus the restriction of $\sigma$ to $|\widehat{\K \pi}(1)|$ is a linear isomorphism. (Note that the space $|\widehat{\K\pi}(1)|$ is not closed under the Goldman bracket: for instance, $[|\log(\alpha_i)|, |\log(\beta_i)|]_{\rm gr} = {\bf 1} \notin |\widehat{\K \pi}(1)|$.)
We consider a pro-nilpotent Lie subalgebra of $|\widehat{\K \pi}|$ defined by 
\[
\mathcal{L}^+ = \mathcal{L}^+(\Sigma) := |\widehat{\K \pi}(1)| \cap \sigma^{-1}({\rm Der}_{\log \gamma_0}^+ (\widehat{\K \pi}, \Delta )) \subset |\widehat{\K \pi}(3)|.
\]
The map $\sigma$ restricts to a Lie algebra isomorphism from $\mathcal{L}^+$ to ${\rm Der}^+_{\log \gamma_0}(\widehat{\K \pi}, \Delta)$.
Hence we obtain a map
\[
\tau^{\rm geom}\colon \mathcal{I} \rightarrow 
\mathcal{L}^+, 
\quad
\varphi \mapsto \sigma^{-1}(\log \varphi)
\]
called the {\em geometric Johnson homomorphism}.
The maps $\tau^{\rm filt}$ and $\tau^{\rm geom}$  fit into the following commutative diagram:
\begin{equation} \label{eq:ungeom}
\xymatrix@R=1em{
 & {\rm Der}^+_{\log \gamma_0}(\widehat{\K \pi}, \Delta) \\
 \mathcal{I} \ar[ur]^{\tau^{\rm filt} \hspace{1em}} \ar[dr]_{\tau^{\rm geom}} &  \\
  & \mathcal{L}^+ \ar[uu]^{\cong}_{\sigma}
}
\end{equation}

\begin{exple}[\cite{KK14, MT13}]
Let $C$ be a null-homologous simple closed curve on $\Sigma$, $\gamma \in \pi$ a based loop homotopic to $C$ and $t_C \in \mathcal{I}(2)$ the right handed Dehn twist along $C$.
Then, $\tau^{\rm geom}(t_C) = |(1/2)(\log \gamma)^2|$.
\end{exple}

\begin{rem}
When $n > 0$, the Dehn-Nielsen action $\mathcal{M} \to {\rm Aut}(\pi)$ is not injective. 
In fact, the Dehn twist along the $j$th boundary component ($1\le j \le n$) is a nontrivial element in the kernel. 
In order to obtain a faithful action of $\mathcal{M}$, one needs to consider the fundamental groupoid of $\Sigma$ where the basepoints are chosen from each boundary component. 
Accordingly, one needs a groupoid version of the operations $\kappa$ and $\sigma$ to introduce the geometric Johnson homomorphism in the case of $n>0$.  
For more details, see \cite{KK15, KK16}. 
\end{rem}

\subsection{Enomoto-Satoh trace and framed Turaev cobracket}

In this section, we fix a framing $f$ on $\Sigma$ and consider the framed Torelli group $\mathcal{I}^f$.
Recall that we have a Lie subalgebra of ${\rm Lie}_{\rm gr}(\mathcal{I})$ defined by 
\[
\Lie_{\rm gr}(\mathcal{I}^f) = \bigoplus_{k=1}^{\infty} \mathcal{I}^f(k)/\mathcal{I}^f(k+1). 
\]
We will study the restriction of the Johnson homomorphism to $\Lie_{\rm gr}(\mathcal{I}^f)$. 
First, we consider the filtered version. 
Recall from Proposition \ref{prop:div^f} the cocycle ${\sf div}^f \colon {\rm Der}(\widehat{\K \pi}, \Delta) \to |\widehat{\K \pi}|$ which depends only on the framing $f$ and is given by the formula
\[
{\sf div}^f(u) = {\sf div}_{x,y}(u) + u({\bf r} - {\bf p}^f). 
\]
Here, ${\rm tDer}(\widehat{\K \pi}, \Delta) = {\rm Der}(\widehat{\K \pi}, \Delta)$ and ${\sf b}^f = 0$ as $n=0$, and we simply denote ${\sf div}_{x,y}$ instead of ${\sf div}_{x,y,z}$.
By Proposition~\ref{prop:div^f} and Theorem~\ref{prop:delta=c_sigma}, there is a commutative diagram which relates the map ${\sf div}^f$ to the framed Turaev cobracket:
\begin{equation} \label{eq:divcob}
\xymatrix{
{\rm Der}^+_{\log \gamma_0}(\widehat{\K \pi}, \Delta) \ar[r]^{\hspace{2em} {\sf div}^f} & |\widehat{\K \pi}| \ar[d]^{\tilde{\Delta}} \\
\mathcal{L}^+ \ar[u]^{\sigma}_{\cong} \ar[r]^{\delta^f} & 
|\widehat{\K \pi}|^{\otimes 2}
}
\end{equation}

\begin{thm} \label{thm:deltataugeom}
For any $\varphi \in \mathcal{I}^f$, we have ${\sf div}^f(\tau^{\rm filt}(\varphi)) = 0$ and  $\delta^f(\tau^{\rm geom}(\varphi)) = 0$.
\end{thm}

\begin{proof}
By \eqref{eq:Torellimap} and Theorem~\ref{thm:KRV_acts}, the framed Torelli group $\mathcal{I}^f$ is a subgroup of ${\rm KV}^f_0$. 
Therefore, for any $\varphi \in \mathcal{I}^f$, we have $\log \varphi \in \mathfrak{kv}_0^f$ which implies
\[
{\sf div}^f (\tau^{\rm filt}(\varphi)) = 
{\sf div}^f (\log \varphi) =0.
\]
By commutativity of the diagrams~\eqref{eq:ungeom} and \eqref{eq:divcob}, we also have $\delta^f(\tau^{\rm geom}(\varphi))=0$.
\end{proof}

\begin{rem}
The fact that $\delta^f$ vanishes on $\tau^{\rm geom}(\mathcal{I}^f)$ can also be proved by the same argument as \cite[Theorem 6.2.1]{KK15}, where a similar statement for the unframed version of the Turaev cobracket was proved.  
\end{rem}

We investigate implications of Theorem \ref{thm:deltataugeom} to the classical Johnson homomorphism.
In general, suppose that $\lambda \colon V \to W$ is a linear map of some filtration degree between filtered $\K$-vector spaces.
Then, we have 
\[
{\rm gr}\, {\rm Ker}(\lambda) \subset {\rm Ker}(\lambda_{\gr})
\]
in ${\rm gr}\, V$.
The two subspaces are not equal in general.
The apparently bigger space ${\rm Ker}(\lambda_{\gr})$ is more ``computable'' than ${\rm gr}\, {\rm Ker}(\lambda)$ but only ``captures'' the lowest order term of $\lambda$.

\begin{lem} \label{lem:formal_generality}
Suppose that $\lambda$ is formal, i.e., there exist filtration-preserving isomorphisms $\theta_V \colon V \to {\rm gr}\, V$ and $\theta_W \colon W \to {\rm gr}\, W$ such that the following diagram commutes.
\[
\xymatrix{
V \ar[r]^{\lambda} \ar[d]^{\cong}_{\theta_V} & W \ar[d]^{\theta_W}_{\cong} \\
{\rm gr}\, V \ar[r]^{\lambda_{\gr}} & {\rm gr}\, W 
}
\]
Then, it holds that ${\rm gr}\, {\rm Ker}(\lambda) = {\rm Ker}(\lambda_{\gr})$ and $\theta_V$ restricts to a filtration-preserving isomorphism ${\rm Ker}(\lambda) \cong {\rm gr}\, {\rm Ker}(\lambda)$.
\end{lem}

\begin{proof}
The map $\theta_V$ restricts to a filtration-preserving isomorphism ${\rm Ker}(\lambda) \cong {\rm Ker}(\lambda_{\gr})$.
Taking the associated graded, we obtain ${\rm gr}\, {\rm Ker}(\lambda) = {\rm gr}\, {\rm Ker}(\lambda_{\rm gr}) = {\rm Ker}(\lambda_{\rm gr})$. 
\end{proof}

We apply the above generalities to our setting.
Since $\tau^{\rm filt}$ and $\tau^{\rm geom}$ are essentially equivalent (see \eqref{eq:ungeom}), we first deal with $\tau^{\rm filt}$ and then translate it to $\tau^{\rm geom}$.

\begin{prop}
Suppose that $g\ge 2$ or $g=1$ and $f$ is adapted.
Then, the map ${\sf div}^f \colon {\rm Der}^+_{\log \gamma_0}(\widehat{\K \pi}, \Delta) \to |\widehat{\K \pi}|$ is formal.
\end{prop}

\begin{proof}
Let $\theta = F \circ \theta_{\exp}$ be a special expansion defined by some $F \in {\rm SolKV}^f_{(g,1)}$.
We have the isomorphism ${\rm Der}^+_{\log \gamma_0}(\widehat{\K \pi}, \Delta) \cong {\rm Der}^+_{\omega}(L)$ induced by $\theta$, and 
by Proposition~\ref{prop:c_transfer}, we have 
\[
{\sf div}^f_{\theta}(\tilde{u})
= {\sf div}^f_{\rm gr}(\tilde{u}) + u(R^f(\tilde{F}))
= {\sf div}^f_{\rm gr}(\tilde{u})
\]
for any $u \in {\rm Der}^+_{\omega}(L)$.
Here, we use the fact that $R^f(\tilde{F}) \in |\K[[\omega]]|$.
Therefore ${\sf div}^f_{\theta} = {\sf div}^f_{\gr}$, as desired.
\end{proof}

Applying Lemma \ref{lem:formal_generality} to ${\sf div}^f$ and taking the associated graded of Theorem \ref{thm:deltataugeom}, we obtain the following:
\begin{align*}
\tau({\rm Lie}_{\rm gr}(\mathcal{I}^f))
& \subset   
{\rm gr}\, {\rm Ker}\big( {\sf div}^f\colon
{\rm Der}^+_{\log \gamma_0}(\widehat{\K \pi}, \Delta) \to |\widehat{\K \pi}| \big) \\
& =  
{\rm Ker} \big( {\sf div}^f_{\rm gr} \colon {\rm Der}^+_{\omega}(L) \to |A| \big).
\end{align*}
Next, we identify the associated graded map of the divergence map ${\sf div}^f$ with the Enomoto-Satoh trace.

\begin{prop} \label{prop:grdiv=ES}
We have ${\sf div}^f_{\gr} = {\sf div}_{x,y}$ on ${\rm gr}\, {\rm Der}(\widehat{\K \pi}, \Delta) \cong {\rm Der}(L)$.
Furthermore, we have ${\sf div}_{x,y} = - {\sf ES}$ on ${\rm Der}_{\omega}^+(L)$. 
\end{prop}

\begin{proof}
The first assertion is clear from the defining formula of ${\sf div}^f$. 
To prove the second assertion, let $\{ x_i^*, y_i^* \}_i$ be the dual basis of the symplectic basis $\{ x_i, y_i \}_i$. 
By using intersection pairing on $H$, we can write $x_i^* = \langle -y_i, \cdot \rangle$ and $y_i^* = \langle x_i, \cdot \rangle$. 
Let $u \in {\rm Der}_{\omega}^k(L)$.
Then, 
\[
u|_H = \sum_{i=1}^g \left( x_i^* \otimes u(x_i) + y_i^* \otimes u(y_i) \right) \in H^* \otimes H^{\otimes k+1}
\]
corresponds to
\[
u^{\sharp} = 
\sum_{i=1}^g \left( -y_i \otimes u(x_i) + x_i \otimes u(y_i) \right)
\in H\otimes H^{\otimes k+1} = H^{\otimes k+2}
\]
through the Poincar\'e duality $H \cong H^*, a \mapsto \langle a, \cdot \rangle$. 
The condition $u(\omega) = 0$ is equivalent to 
\[
\sum_{i=1}^g \left([u(x_i), y_i] + [x_i, u(y_i)] \right) = 0,
\]
which is equivalent to the fact that $u^{\sharp}$ is cyclically invariant.

On the one hand, we have ${\sf div}_{x,y}(u) = \sum_{i=1}^g |d_{x_i}u(x_i) + d_{y_i}u(y_i)| = -|C_{n1}(u^{\sharp})|$, where $C_{n1}$ is the contraction of the last and first components by the intersection form. 
On the other hand, we have ${\sf ES}(u) = |C_{12}(u^{\sharp})|$, where $C_{12}$ is the contraction of the first and second components by the intersection form. 
Since $u^{\sharp}$ is cyclically invariant, we have $C_{n1}(u^{\sharp}) = C_{12}(u^{\sharp})$. 
This proves ${\sf ES}(u) = - {\sf div}_{x,y}(u)$. 
\end{proof}

As a summary of the discussion above, we obtain the following result.

\begin{thm} \label{thm:JohnsonES}
Suppose that $g \ge 2$ or $g = 1$ and $f$ is adapted.
Then we have 
    \[
  \tau({\rm Lie}_{\rm gr}(\mathcal{I}^f))
  \subset {\rm Ker}({\sf ES}).
    \]
Moreover, the space ${\rm Ker}({\sf ES})$, the kernel of the Enomoto-Satoh trace, is isomorphic to the kernel of ${\sf div}^f$ restricted to ${\rm Der}^+_{\log \gamma_0}(\widehat{\K \pi}, \Delta)$.
\end{thm}

In \cite{ES}, Enomoto and Satoh  proved the first statement of Theorem \ref{thm:JohnsonES} by using Hain's Theorem \ref{thm:HainJohnson}.
The proof presented here is more conceptual, and together with the second statement it clarifies the topological meaning of the Enomoto-Satoh trace. 

Finally, let us rephrase our result from a more geometric viewpoint.
Note that the framed Turaev cobracket $\delta^f$ is formal by Theorems \ref{thm:GThomomorphic} and \ref{thm:KVsolve}.
By taking the associated graded of the diagram \eqref{eq:divcob}, we have the following commutative diagram:
\[
\xymatrix{
{\rm Der}^+_{\omega}(L) \ar[r]^{\hspace{0.5em} \sf ES} & |A| \ar[d]^{\tilde{\Delta}} \\
{\rm gr}\, \mathcal{L}^+ \ar[u]^{\sigma_{\gr}}_{\cong} \ar[r]^{\hspace{0.5em} \delta^f_{\rm gr}} & 
|A|^{\otimes 2}
}
\]
The map $\sigma_{\rm gr}$ restricts to a Lie algebra isomorphism ${\rm Ker}({\delta^f_{\gr}}|_{{\rm gr}\, \mathcal{L}^+}) \cong {\rm Ker}({\sf ES})$.

\begin{thm} \label{thm:grKer=Kergr}
Suppose that $g\ge 2$ or $g=1$ and $f$ is adapted.
Then we have 
\[
(\tau^{\rm geom})_{\rm gr}({\rm Lie}_{\rm gr}(\mathcal{I}^f)) \subset
{\rm Ker}({\delta^f_{\gr}}|_{{\rm gr}\, \mathcal{L}^+}),
\]
and the space ${\rm Ker}({\delta^f_{\gr}}|_{{\rm gr}\, \mathcal{L}^+})$ is isomorphic to ${\rm Ker}(\delta^f|_{\mathcal{L}^+})$.
\end{thm}

Theorems~\ref{thm:JohnsonES} and \ref{thm:grKer=Kergr} show that the obstruction for $\tau$ coming from the framed Turaev cobracket coincides exactly  with the Enomoto-Satoh obstruction. 

\begin{rem}
The fact that $\delta^f_{\gr}$ vanishes on the image of the Johnson homomorphisms (in degree at least two) is a refinement of Corollary 6.3.3 in \cite{KK15}, where the unframed version of the Turaev cobracket was used. 
The two versions, framed and unframed, of the Turaev cobracket have explicit differences in constraints for the Johnson image; see \cite{EKS}. 
\end{rem}

\begin{rem}
In genus 0, the following fact is analogous to Theorem~\ref{thm:JohnsonES} above.
Let $\mathfrak{t}_n$ be the Drinfeld-Kohno Lie algebra of infinitesimal pure braids viewed as a Lie sublagebra of the Lie algebra
\[
\mathfrak{sder} = \mathfrak{sder}_{(0,n+1)}
= \{ \tilde{u} \in {\rm tDer}_{(0,n+1)} ; u(z_1 + \cdots + z_n) = 0\} .
\] 
Then, $\mathfrak{t}_n$ is in the kernel of the divergence map.

Moreover, \v{S}evera and Willwacher \cite{SeW} define a decreasing filtration $\mathfrak{sder} \supset {\rm Ker}({\sf div}) \supset \cdots \supset \mathfrak{t}_n$ by using a certain spectral sequence related to the Kontsevich graph complex
(see also \cite{Felder}.)
This construction can be thought of as a solution of the genus 0 analogue of the Johnson image problem.
\end{rem}

\begin{rem}
Tsuji \cite{Tsu1, Tsu2, Tsu3, Tsu4} introduced two skein analogues of the group 
homomorphism $\tau$ using the Kauffman bracket skein algebra and the HOMFLY-PT skein algebra of the product $\Sigma\times [0,1]$, respectively.
Both of them give rise to new constructions of finite type invariants of integral 
homology $3$-spheres. 
\end{rem}

As we saw in Section~\ref{subsec:GRTGT}, for a certain binary rooted trees there are embeddings of the Grothendieck--Teichm\"uller group ${\rm GT}_1$ into the higher genus Kashiwara-Vergne groups ${\rm KV}^f$.
In view of the Johnson image problem, it would be interesting to clarify how the images of ${\rm GT}_1$ under these embeddings and the framed Torelli group $\mathcal{I}^f$ are related in ${\rm KV}^f$.

\appendix

\section{Appendix} \label{sec:Appendix}

In this section, we prove some algebraic results used in the main body of the text.

\subsection{On the weight filtration on $\K \Gamma$} \label{subsec:alt_wt_filt}

In this section, we give a proof of Proposition~\ref{prop:KGamma(m)}.

\begin{proof}[Proof of Proposition \ref{prop:KGamma(m)}]
Let $\langle \mathcal{Z}^{(m)} \rangle$ be the two-sided ideal of $\K\Gamma$ generated by the set $\mathcal{Z}^{(m)}$.
Since $\mathcal{Z}^{(m)}\subset \K\Gamma(m)$, we have $\langle \mathcal{Z}^{(m)} \rangle \subset \K\Gamma(m)$.
Consider the following morphism of short exact sequences
\[
\xymatrix{
0 \ar[r] & \langle \mathcal{Z}^{(m+1)} \rangle \ar[r] \ar[d] & 
\langle \mathcal{Z}^{(m)} \rangle \ar[r] \ar[d] & 
\langle \mathcal{Z}^{(m)} \rangle/ \langle \mathcal{Z}^{(m+1)} \rangle \ar[r] \ar[d] & 0
\\
0 \ar[r] & \K \Gamma(m+1) \ar[r] & 
\K \Gamma(m) \ar[r] & 
\K \Gamma(m)/ \K \Gamma(m+1) \ar[r] & 0. 
}
\]
By the following lemma, the right vertical map is an isomorphism for all $m\ge 0$.
Since $\langle \mathcal{Z}^{(0)} \rangle = \K \Gamma(0) = \K \Gamma$, one conclude that $\langle \mathcal{Z}^{(m)} \rangle = \K\Gamma(m)$ for all $m\ge 0$ by induction on $m$ and application of the five lemma.
Furthermore, the fact that $\K\Gamma(1) = I\Gamma$ follows from the first statement of the proposition and the fact that $I\Gamma$ is generated by elements $Z_j = \gamma_j-1$ as a two-sided ideal.
This completes the proof.
\end{proof}

\begin{lem} \label{lem:BGamma}
The natural map
$
\langle \mathcal{Z}^{(m)} \rangle / \langle \mathcal{Z}^{(m+1)} \rangle \to \K \Gamma(m) / \K \Gamma(m+1)
$
is a $\K$-linear isomorphism for all $m\ge 0$.
\end{lem}

\begin{proof}
As we saw in a paragraph preceding Proposition~\ref{prop:KGamma(m)}, the map $\K \langle J ; {\rm wt}(Z_J) = m \rangle \to \K \Gamma(m) / \K \Gamma(m+1), J \mapsto Z_J \mod \K\Gamma(m+1)$, is a $\K$-linear isomorphism.
It factors through the natural map in the lemma:
\[
\xymatrix{
\K \langle J ; {\rm wt}(Z_J) = m \rangle \ar[r]^{\hspace{-1em} \cong} \ar[d] & \K \Gamma(m) / \K \Gamma(m+1) \\
\langle \mathcal{Z}^{(m)} \rangle / \langle \mathcal{Z}^{(m+1)} \rangle  \ar[ur] &  
}
\]
Therefore, to prove the lemma, it is sufficient to show that the vertical map 
\[
\K \langle J ; {\rm wt}(Z_J) = m \rangle \to \langle \mathcal{Z}^{(m)} \rangle / \langle \mathcal{Z}^{(m+1)} \rangle,
\quad
J \mapsto Z_J \ {\rm mod}\ \langle \mathcal{Z}^{(m+1)} \rangle
\]
is surjective.
It is enough to prove that $xZ_Jy \equiv Z_J \ {\rm mod}\ \langle \mathcal{Z}^{(m+1)} \rangle$ for every $x,y \in \Gamma$ and multi-index $J$ such that ${\rm wt}(Z_J) =m$.
We use induction on $l(x) + l(y)$, where $l(x)$ is the word length of $x$ with respect to the free generators $\{\gamma_j \}_j$.

The case where $l(x) + l(y) =0$ is clear.
Let $l(x) + l(y) >0$.
We have either (i) $x=x'\gamma_j^{\pm 1}$ with $l(x')<l(x)$ or (ii) $y=\gamma_j^{\pm 1}y'$ with $l(y') < l(y)$.
We only consider the case where $x=x'\gamma_j$, the other cases being similar.
If $x=x'\gamma_j$, we compute
\[
xZ_Jy - x'Z_J y = x'(\gamma_j-1)Z_J y \in \langle \mathcal{Z}^{(m+1)} \rangle.
\]
Thus $xZ_Jy \equiv x'Z_J y \equiv Z_J \ {\rm mod}\ \langle \mathcal{Z}^{(m+1)} \rangle$, where in the second equality we have used the induction hypothesis.
This completes the proof of the lemma.
\end{proof}

\subsection{Inner derivation theorem} \label{subse:inn_der_thm}

Let $V$ be a finite-dimensional $\K$-vector space and let $A = \widehat{T}(V)$ be the completed tensor algebra generated by $V$.
We work with the polynomial degree of $A$.

The following statement is a refinement of \cite[Theorem A.1]{genus0}.

\begin{thm} \label{thm:idt} 
Let $u\in {\rm Der}(A)$ and suppose that there is a positive integer $N$ such that $|u(a)| = 0$ for any $a \in A_{\ge N}$. Then, $u$ is an inner derivation.
\end{thm}

The proof we present here basically follows the original argument.
In particular, we will use the following two assertions from \cite[Appendix A]{genus0}.

\begin{lem}[\cite{genus0}, Proposition A.2] \label{lem:A2}
Let $x$ be a non-zero element in $V$ and let $a\in A$ be a homogeneous element of degree $m\ge 1$.
If $|ax^l| = 0$ for some $l\ge m-1$, then $a\in [x,A]$.
\end{lem}

\begin{lem}[\cite{genus0}, Lemma A.5] \label{lem:A5}
Let $x$ and $y \in V$ be linearly independent.
Suppose that $a\in A$ is homogeneous of degree $d\ge 1$.
If $[x^d, a] \in [y, A]$ and $[y^d, a] \in [x, A]$, then $a \in \K\, x^d \oplus \K\, y^d$.
\end{lem}

In fact, the statement of \cite[Lemma A.5]{genus0} is slightly different, but its proof actually shows the statement above.

\begin{proof}[Proof of Theorem \ref{thm:idt}]
Without loss of generality, we may assume that $u$ is homogeneous.
Let $\{ x_i \}_i$ be a basis for $V$ and let $N$ be a positive integer as in the assumption of the theorem.
Pick an $l$ such that $l \ge \max(N-1, 0, \deg u)$.
Then, for any $i$ we compute
\begin{equation} \label{eq:uxl+1}
0 = |u({x_i}^{l+1})|
= (l + 1) | u(x_i) {x_i}^l |.
\end{equation}
If $\deg u \le 0$, then this shows that $u(x_i) = 0$ and hence $u = 0$.

Assume that $d = \deg u > 0$.
Then, by Lemma \ref{lem:A2} and \eqref{eq:uxl+1}, there is an element $a_i \in A$ of degree $d$ such that $u(x_i) = [x_i, a_i]$.

Let $i\neq j$.
If $p + q \ge N$, then
\[
0 = |u({x_i}^p {x_j}^q) |
= | [{x_i}^p, a_i] {x_j}^q + {x_i}^p [{x_j}^q, a_j] |
= 
| [{x_i}^p, a_i -a_j]{x_j}^q |.
\]
Letting $p = d$ and $q \ge \max(2d-1, N-d)$ and applying Lemma \ref{lem:A2}, we conclude that $[{x_i}^d, a_i - a_j] \in [x_j, A]$.
Similarly, we obtain $[{x_j}^d, a_i - a_j] \in [x_i, A]$.
By Lemma \ref{lem:A5}, we obtain $a_i - a_j \in \K\, {x_i}^d \oplus \K\, {x_j}^d$.

We have shown the following: for any $i, j$ with $i \neq j$, there are constants $c_{ij}, c_{ji} \in \K$ such that $a_i - a_j = c_{ij}{x_i}^d - c_{ji}{x_j}^d$.
It is easy to see that $w:= a_i - c_{ij} {x_i}^d$ is independent of the choice of $i$ and $j$.
Then we have $u(x_i) = [x_i, w]$ for any $i$, and conclude that $u$ is an inner derivation by $w$.
\end{proof}

\subsection{Centralizers in tensor algebra}
\label{subsec:center|A|}

This section is a preliminary for the next section.
For the moment, let $V$ be a finite dimensional $\K$-vector space.

\begin{lem}[Euclidean algorithm]\label{lem:Euclid}
Let $l, m$ be non-negative integers with $l+m>0$.
Assume that non-zero vectors $u\in V^{\otimes l}$ and $v\in V^{\otimes m}$ are commutative, i.e.,
\[
uv = vu \in V^{\otimes (l+m)}.
\]
Then there exist some $w \in V^{\otimes \gcd(l,m)}$ and $c, d \in \K$ such that $u = c w^{l'}$ and $v = d w^{m'}$, where
$l' = l/\gcd(l,m)$ and $m' = m/\gcd(l,m)$.
\end{lem}
\begin{proof}
The lemma is trivial when $l+m = 1$ or $lm = 0$.
In order to use induction on $l+m$, 
let $k>1$ and assume that the lemma holds if $l+m<k$.

Suppose $l+m=k$.
We may assume that $0 < l \leq m$. 
Take a $\K$-basis $\{u_i\}_{1 \leq i \leq (\dim V)^l}$
of $V^{\otimes l}$ with $u_1 = u$. Then we have 
$
v = \sum^{(\dim V)^l}_{i=1} u_iv'_i
$
for some $v'_i \in V^{\otimes(m-l)}$, and so 
$$
\sum^{(\dim V)^l}_{i=1} u_1u_iv'_i
= \sum^{(\dim V)^l}_{i=1} u_iv'_iu_1.
$$
Since $\{u_i\}_{1 \leq i \leq (\dim V)^l}$ are linearly independent, we have
$v'_i = 0$ for any $i \neq 1$. This means $v = uv'_1$, 
$uuv'_1 = uv'_1u$ and so $uv'_1 = v'_1u$. Since $l+(m-l) = m < l+m = k$, 
we can apply the inductive assumption to $uv'_1 = v'_1u$, that is,
there exist some $w \in V^{\otimes \gcd(l,m)}$ and $c, d'
\in \K$ such that $u = c w^{l'}$ and $v'_1 = d' w^{m'-l'}$.
Hence we have $v = uv'_1 = c d' w^{m'}$.
This completes the induction. 
\end{proof}

A non-zero homogeneous element 
$u_0 \in V^{\otimes l}$, $l \geq 1$, is {\it reduced}
if the equation $u_0 = \lambda {v_0}^d$ with 
$\lambda \in \K$ and $v_0 \in 
V^{\otimes (l/d)}$ implies $d=1$. 
For example, 
any non-zero element of $V$ or $\wedge^2 V$  
is reduced.
\begin{prop}\label{prop:centralizer}
Let $u_0 \in V^{\otimes l}\setminus\{0\}$ be reduced.
Then we have 
$$
\{v \in \widehat{T}(V) ; v{u_0} = {u_0}v\} = \K[[u_0]].
$$
\end{prop}
\begin{proof}
Let $v\in \widehat{T}(V) \setminus \{ 0\}$ satisfy $vu_0 = u_0v$.
We may assume that $v$ is homogeneous of degree $m \geq 1$.
By Lemme~\ref{lem:Euclid}, 
there exist some $w \in V^{\otimes \gcd(l,m)}$ and $c, d
\in \K$ such that $u_0 = c w^{l'}$ and $v = d w^{m'}$.
Since $u_0$ is reduced, we have $l' = 1$, i.e., $v \in \K[[u_0]]$.
This completes the proof.
\end{proof}

By a similar method we obtain the following strengthened version, which is used in our paper \cite[the proof of Lemma 3.6]{AKKN_new}.

\begin{prop}\label{prop:normalizer}
Let $u_0 \in V^{\otimes a}\setminus\{0\}$ be reduced. Then we have 
$$
\{v \in \widehat{T}(V) ; [{u_0}, v] \in \K[[u_0]]\}
= \K[[u_0]].
$$
\end{prop}
\begin{proof}
Let $v$ be a homogeneous element satisfying $[u_0,v]\in \K[[u_0]]$.
If $a$ does not divide $\deg v$, then $[u_0,v] = 0$.
Hence, by Proposition \ref{prop:centralizer}, we obtain $v \in \K[[u_0]]$. 
So it suffices to consider the case $\deg v = al$, $l \geq 0$. 
We prove $v\in \K[[u_0]]$ by induction on $l \geq 0$.
If $l=0$, it is trivial. Assume $l \geq 1$. 
Then we have $vu_0 = u_0(v - \lambda{u_0}^l)$ for some $\lambda \in \K$. 
By a similar way to the proof of Lemma \ref{lem:Euclid}, we have some $v_0 \in V^{\otimes a(l-1)}$ such that $v = u_0v_0$.
Then $u_0v_0u_0 = u_0(u_0v_0 - \lambda{u_0}^l)$, so 
$v_0u_0 = u_0v_0 - \lambda{u_0}^l$. Applying the inductive assumption to $v_0$, we obtain $v_0 \in \K[[u_0]]$, which implies $v \in \K[[u_0]]$.
This completes the proof.
\end{proof}

Let us go back to the case of $V = {\rm gr}^{\rm wt}\, H$, where $H = H_1(\Sigma, \K)$.

\begin{lem}\label{lem:cent-omega}
We have
$
\{a\in A ; a\omega = \omega a\}
= \K[[\omega]]
$.
\end{lem}
\begin{proof}
Recall from Example~\ref{ex:capping_grded} an injective algebra homomorphism $i_*\colon A = {\rm gr}^{\rm wt}\, \K \pi \to \overline{A} = {\rm gr}^{\rm wt}\, \K \bar{\pi}$ obtained by the capping trick.
It maps $\omega$ to $\overline{\omega} = \sum_{i=1}^{g+n} [\bar{x}_i, \bar{y}_i]$.
The lemma follows from Proposition \ref{prop:centralizer} for $\overline{\omega}$.
\end{proof}

\subsection{The bracket operations associated with $\kappa_{\rm gr}$} \label{subsec:bra_kappa_gr}

In this section, we turn our attention to the bracket operations associated with $\kappa_{\gr}$ and prove Proposition~\ref{prop:im_sigma_gr}, Proposition~\ref{prop:ker_sigma_gr} and Theorem~\ref{thm:center}.

\begin{lem} \label{lem:formulas2}
Let $a \in A$ and $h(z) \in \K[[z]]$. 
\begin{enumerate}
\item[$(i)$]
We have $\kappa_{\rm gr}( \omega, a) = 
\kappa_{\rm gr}(a, \omega) =
a \otimes 1 - 1 \otimes a$.
\item[$(ii)$]
For any $j\in \{ 1,\ldots, n\}$, we have
$\{ |h(z_j)|, a\}_{\rm gr} = 0$.
\item[$(iii)$]
We have $\{ |h(\omega)|, a \}_{\rm gr} = [\dot{h}(\omega), a]$. Here, $\dot{h}(z)$ is the derivative of $h(z)$.
\end{enumerate}
\end{lem}

\begin{proof}
(i)
By inspection, the formula $\kappa_{\rm gr}(\omega, a) = a \otimes 1 - 1\otimes a$ is verified when $a$ is any element of generators $x_i, y_i$ and $z_j$.
If the formula is valid for $a,b \in A$, it is also valid for the product $ab$:
\begin{align*}
\kappa_{\rm gr}( \omega, ab) &= 
(a \otimes 1) \kappa_{\rm gr}( \omega, b) + \kappa_{\rm gr}( \omega, a) (1 \otimes b) \\
& =  (a \otimes 1)(b\otimes 1 - 1 \otimes b) + (a\otimes 1 - 1 \otimes a)(1 \otimes b) \\
& =  ab \otimes 1 - 1 \otimes ab.
\end{align*}
Hence, $\kappa_{\rm gr}(\omega, a) = a\otimes 1 - 1\otimes a$ holds for all $a\in A$.

By taking the associated graded of equation \eqref{eq:yxxy}, we obtain $\kappa_{\rm gr}(v,u) = - \kappa_{\rm gr}(u,v)^{\circ}$ for any $u,v \in A$.
Therefore, $\kappa_{\rm gr}(a,\omega) = -\kappa_{\rm gr}(\omega, a)^{\circ} = a\otimes 1 - 1 \otimes a$.

(ii)
It is sufficient to prove that $\{ |{z_j}^k|, a \}_{\rm gr} = 0$ for any element $a$ of the generators.
This is clear for $a \neq z_j$.
For $a = z_j$, we compute
$$
\kappa_{\rm gr}({z_j}^k, z_j) = \sum_{l=0}^{k-1} (1 \otimes {z_j}^l)(z_j \otimes 1 - 1 \otimes z_j)({z_j}^{k-l-1} \otimes 1) = {z_j}^k \otimes 1 - 1 \otimes {z_j}^k
$$
by using Lemma~\ref{lem:kgruv}.
This shows that $\{ |{z_j}^k|, z_j \}_{\rm gr} = 0$, as required.

(iii) Using (i) we compute
$$
\kappa_{\rm gr}(\omega^k, a) = \sum_{l=0}^{k-1} (1 \otimes \omega^l)(a \otimes 1 - 1 \otimes a)(\omega^{k-l-1} \otimes 1),
$$
and hence $\{ |\omega^k|, a \}_{\rm gr} = \kappa_{\rm gr}(\omega^k, a)' \kappa_{\rm gr}(\omega^k, a)'' = k(a\omega^{k-1} - \omega^{k-1}a) = [k\omega^{k-1},a]$.
This proves the assertion.
\end{proof}

\begin{proof}[Proof of Proposition \ref{prop:im_sigma_gr}]
Recall from Section~\ref{subsec:filtonKpi} the symmetric bilinear operation $\mathfrak{z}$ defined on ${\rm gr}^{\rm wt}\, H$.
We introduce the following notation:
for $1\le j \le n$ and $u\in {\rm gr}^{\rm wt}\, H$, we write $\mathfrak{z}(u,z_j) = \mathfrak{z}_j(u) z_j$.
Namely, $\mathfrak{z}_j(u) \in \K$ reads the coefficient of $z_j$ in $u$.

Let $|a|=|a_1\cdots a_m| \in |A|$, where $a_1,\ldots,a_m \in {\rm gr}^{\rm wt}\, H$. 
By Proposition \ref{prop:kgr} and an argument similar to the proof of Proposition \ref{prop:lifthk}, we see that the derivation $\tilde{\sigma}_{\rm gr}(|a|)$ is given as follows:
\begin{equation} \label{eq:sigmagr|a|}
\sigma_{\rm gr}(|a|)\colon
\begin{cases}
x_i \mapsto
\sum_k \langle a_k, x_i \rangle a_{k+1}\cdots a_m a_1 \cdots a_{k-1} \\
y_i \mapsto 
\sum_k \langle a_k, y_i \rangle a_{k+1}\cdots a_m a_1 \cdots a_{k-1} \\
z_j \mapsto 
[z_j, \sum_k \mathfrak{z}_j(a_k) a_{k+1}\cdots a_m a_1 \cdots a_{k-1}],
\end{cases}
\end{equation}
and the $j$th tangential component of $\tilde{\sigma}_{\rm gr}(|a|)$ is equal to $\sum_k \mathfrak{z}_j(a_k) a_{k+1}\cdots a_m a_1 \cdots a_{k-1}$.
Since $\kappa_{\rm gr}(a, \omega) = a \otimes 1 - 1\otimes a$ by Lemma~\ref{lem:formulas2} (i), we have $\sigma_{\rm gr}(|a|)(\omega) = \{|a|, \omega \}_{\rm gr} = 0$.

Conversely, let $\tilde{u} =(u, u_1, \dots, u_n) \in {\rm tDer}(A)$ be a tangential derivation such that $u(\omega) = 0$. 
We may assume that $u(x_i)$, $u(y_i)$, and $u_j$ are homogeneous of polynomial degree $m$.
Consider the element
\[
a_u:= \sum_i \left( x_i u(y_i) - y_i u(x_i) \right) + \sum_j z_j u_j \in A.
\]
Then, the assumption $u(\omega) = 0$ implies that $a_u$ is invariant under the cyclic permutation of components. Indeed, we have
\begin{align*}
0 = u(\omega) & =  \textstyle\sum_i([x_i, u(y_i)] + [u(x_i), y_i]) + \textstyle\sum_j [z_j, u_j] \\
& = \textstyle\sum_i \left( x_i u(y_i) - y_i u(x_i) \right) + \textstyle\sum_j z_j u_j \\
& \hspace{1em} - \textstyle\sum_i \left( u(y_i) x_i - u(x_i) y_i \right) - \textstyle\sum_j u_j z_j \\
& = a_u - {\rm shift}(a_u)
\end{align*}
where ${\rm shift}(a_u)$ is the cyclic permutation of $a_u$ by one step.

Hence, we can compute the brackets of $|a_u|$ with generators:
\[
\sigma_{\rm gr}\left( \frac{1}{m+1} \, |a_u| \right)\colon
\begin{cases}
x_i \mapsto \langle -y_i, x_i \rangle u(x_i) = u(x_i), \\
y_i \mapsto \langle x_i, y_i \rangle u(y_i) = u(y_i), \\
z_j \mapsto [z_j, u_j] = u(z_j),
\end{cases}
\]
and the $j$th tangential component of $\tilde{\sigma}_{\rm gr} \left( \frac{1}{m+1}|a_u| \right)$ is equal to $(\mathfrak{z}_j \otimes ({\rm id}_{{\rm gr}^{\rm wt}\, H})^{\otimes m}) (a_u) = u_j$.
We conclude that $\tilde{\sigma}_{\rm gr}\left( \frac{1}{m+1} \, |a_u| \right) = \tilde{u}$. 
This completes the proof of the assertion on ${\rm im}(\tilde{\sigma}_{\rm gr})$.

Let $N\colon |A|_{\ge 1} \to A$ be the injective linear map defined by
\[
N(|a|) = \sum_k a_k \cdots a_m a_1\cdots a_{k-1},
\]
where $a = a_1\cdots a_m$ with $a_1,\ldots, a_m \in {\rm gr}^{\rm wt}\, H$.
Equation~\eqref{eq:sigmagr|a|} shows that for any $|a| \in |A|_{\ge 1}$ one has
\begin{equation} \label{eq:N|a|}
N(|a|) = \sum_{i=1}^g \left( x_i (\sigma_{\rm gr}(|a|)(y_i)) - y_i (\sigma_{\rm gr}(|a|)(x_i)) \right) + \sum_{j=1}^n z_j \sigma_{\rm gr}(|a|)_j,
\end{equation}
where $\sigma_{\rm gr}(|a|)_j$ is the $j$th tangential component of $\tilde{\sigma}_{\rm gr}(|a|)$.

To compute $\ker(\tilde{\sigma}_{\rm gr})$, first notice that $\K {\bf 1}$ is clearly in the kernel of $\tilde{\sigma}_{\rm gr}$.
Let $|a| \in |A|_{\ge 1}$ and assume that $|a| \in \ker(\tilde{\sigma}_{\rm gr})$.
By equation~\eqref{eq:N|a|}, we obtain $N(|a|) = 0$ and hence $|a| = 0$.
This proves that $\ker{\tilde{\sigma}_{\rm gr}} = \K {\bf 1}$.
\end{proof}

\begin{proof}[Proof of Proposition \ref{prop:ker_sigma_gr}]
The assertion on ${\rm im}(\sigma_{\rm gr})$ follows from the assertion on ${\rm im}(\tilde{\sigma}_{\rm gr})$ in Proposition~\ref{prop:im_sigma_gr}.

Lemma~\ref{lem:formulas2}~(i) shows that for any $j$ and $m\ge 1$, $|{z_j}^m |$ is in the kernel of $\sigma_{\rm gr}$.
Let $|a| \in |A|_{\ge 1}$ and assume that $|a| \in \ker(\sigma_{\rm gr})$.
We have $\sigma_{\rm gr}(|a|)(x_i) = \sigma_{\rm gr}(|a|)(y_i) = 0$.
Furthermore, since
\[
0 = \sigma_{\rm gr}(|a|)(z_j) = [z_j, \sigma_{\rm gr}(|a|)_j],
\]
we have $\sigma_{\rm gr}(|a|)_j \in \K [[z_j]]$ by Proposition~\ref{prop:centralizer}.
From \eqref{eq:N|a|} we obtain $N(|a|) \in \bigoplus_{j=1}^n \K [[z_j]]$ and hence $|a| \in \bigoplus_{j=1}^n |\K [[z_j]]|$.
This proves the assertion on $\ker(\sigma_{\rm gr})$. 
\end{proof}

\begin{proof}[Proof of Theorem \ref{thm:center}]
Lemma~\ref{lem:formulas2} (ii)(iii) shows that the elements $| \omega^m |$ and $| {z_j}^m|$ are in the center of the Lie bracket $[ \cdot, \cdot ]_{\gr}$.

Suppose that $|a| \in Z(|A|, [\cdot, \cdot ]_{\gr})$. 
Since $|\{|a|, b\}_{\gr}| = 0$ for any $b \in A$, Theorem~\ref{thm:idt} implies that there exists some $u_0 \in A$ such that 
$\{|a|, b\}_{\gr} = [u_0, b]$ for any $b \in A$.
By Proposition~\ref{prop:im_sigma_gr}, $[u_0, \omega] = \{|a|, \omega\}_{\gr}
= 0$. Hence, from Lemma \ref{lem:cent-omega}, we have $u_0= f(\omega)$
for some $f(z) \in \K[[z]]$.
Set $\hat f(z) := \int_0f(z)dz$.
Then, by Lemma~\ref{lem:formulas2} (iii), we have 
\[
\{|a-\hat f(\omega)|, b\}_{\gr} = \{|a|, b\}_{\gr} - [f(\omega), b] = \{|a|, b\}_{\gr} - [u_0, b] = 0
\]
for any $b \in A$.
Applying Proposition~\ref{prop:ker_sigma_gr} to $|a-\hat{f}(\omega)|$, we conclude that $|a|$ has the desired form.
\end{proof}

\subsection{Integration of Lie algebra $1$-cocycles} \label{subsec:App_integration}

We give a more detail on the integration of Lie algebra $1$-cocycles described in Section~\ref{sec:cocycles}.

\subsubsection*{Integration and differentiation}

Let $\mathfrak{g}$ be a Lie algebra, and $M$ a $\mathfrak{g}$-module.
Any $1$-cocycle of $\mathfrak{g}$ with values in $M$ has two aspects: it defines a section of the projection of the semi-direct product $\mathfrak{g} \ltimes M \to \mathfrak{g}$ and an extension cocycle of an extension 
\begin{equation}
\label{eq:exntension_M}
0 \to \mathbb{K} \overset{i}\to M' \to M \to 0
\end{equation}
of $\mathfrak{g}$-modules.
Here $\mathbb{K}$ means the trivial $\mathfrak{g}$-module. 
In what follows, we will give two interpretations of the integration of a Lie algebra $1$-cocycle of $\mathfrak{g}$ with values in $M$ along each of these two aspects.

We begin by considering the semi-direct product aspect.
Let $\mathfrak{g} \ltimes M$ be the semi-direct product associated with 
a $\mathfrak{g}$-module $M$. Its bracket is defined by 
$$
[(u, m), (v, l)] = ([u,v], u(l)- v(m))
$$
for any  $(u, m), (v, l) \in \mathfrak{g} \ltimes M$.
It admits a natural Lie algebra surjective homomorphism
$\mathfrak{g} \ltimes M \to \mathfrak{g}$, $(u, m) \mapsto u$.
Any $\K$-linear section of this surjection is of the form $({\rm id}_{\mathfrak{g}}, {\sf b})\colon \mathfrak{g} 
\to \mathfrak{g} \ltimes M$, where ${\sf b}\colon \mathfrak{g} \to M$ is a $\K$-linear map.
The section $({\rm id}_{\mathfrak{g}},{\sf b})$ is a Lie algebra homomorphism 
if and only if ${\sf b}\colon \mathfrak{g} \to M$ is a Lie algebra $1$-cocycle.

Similarly, one can define the semi-direct product $\mathcal{G} \ltimes M$
associated with any $\mathcal{G}$-module $M$ of a group $\mathcal{G}$.
The group law is given by 
$$
(F, m)(G, l) = (FG, m+F(l))
$$
for any $(F, m), (G, l) \in \mathcal{G} \ltimes M$.
A map ${\sf c}\colon \mathcal{G} \to M$ is a group $1$-cocycle 
if and only if the map $({\rm id}_{\mathcal{G}},{\sf c})\colon \mathcal{G} \to 
\mathcal{G} \ltimes M$ is a group homomorphism.

Suppose that $\mathfrak{g}$ is pro-nilpotent and equipped with a complete filtration, 
and that $M$ is equipped a complete filtration compatible with the filtration of $\mathfrak{g}$.
Let $\mathcal{G} = \exp(\mathfrak{g})$ be the exponention of the Lie algebra $\mathfrak{g}$. 
Then the semi-direct product $\mathfrak{g} \ltimes M$ is also 
pro-nilpotent and equipped with a complete filtration. 
Hence we can take the exponentaion of the semi-direct product $\mathfrak{g} \ltimes M$, which we denote by
the same symbol $\mathfrak{g} \ltimes M$.

\begin{prop} \label{prop:exp(gltimesM)}
There is a group isomorphism
\[
\Upsilon \colon \mathfrak{g} \ltimes M \overset{\cong}{\to} \mathcal{G} \ltimes M, \quad
(u,m) \mapsto ( \exp(u),  (\textstyle\frac{e^u - 1}{u})m ).
\]
\end{prop}

\begin{proof}
In the free Lie algebra on two indeterminates $U$ and $V$, the linear term in V of the BCH series $(U+V)\ast(-U)$ is given by $\big(\frac{e^{\ad U} - 1}{\ad U}\big)V$ (see e.g., \cite[ch.2., \S6, Proposition~5]{Bou71}).
In other words, we have $\left. \frac{d}{dt}\, (U+tV)\ast(-U) \right|_{t=0} = \big(\frac{e^{\ad U} - 1}{\ad U}\big)V$. 
This implies that 
$$
(u, m)\ast (-u, 0) = \big( 0, (\textstyle\frac{{e^u-1}}{u})m \big)
$$
in our situation. Hence we have 
$$
(u,m) = \big(0, (\textstyle\frac{e^u-1}{u})m \big) \ast (u,0),
$$
and so
\begin{align*}
(u, m)\ast (v, l) & = (0, \big( \textstyle\frac{e^u-1}{u})m \big) \ast (u, 0) \ast
\big( 0, (\textstyle\frac{e^v-1}{v})l \big) \ast (v, 0) \\
& = \big( 0, (\textstyle\frac{e^u-1}{u})m \big) \ast
\big( 0, e^u(\textstyle\frac{e^v-1}{v})l \big) \ast (u, 0)\ast (v, 0)\\
& = \big( 0, (\textstyle\frac{e^u-1}{u})m + e^u(\textstyle\frac{e^v-1}{v})l \big) \ast (u\ast v, 0).
\end{align*}
This formula implies that the map in the statement of the lemma is a group homomorphism from $\mathfrak{g} \ltimes M$ to $\mathcal{G} \ltimes M$, whose inverse is given by $(\exp(u), m) \mapsto \big( u, (\textstyle\frac{u}{e^u - 1})m \big)$.
\end{proof}

Any continuous Lie algebra $1$-cocycle ${\sf b}\colon \mathfrak{g} \to M$
defines a continuous Lie algebra homomorphism $({\rm id}_{\mathfrak{g}}, {\sf b})\colon \mathfrak{g} \to \mathfrak{g} \ltimes M$, which can be regarded as a continuous group homomorphism with respect to the BCH series.
Through the isomorphism $\Upsilon$ in Proposition~\ref{prop:exp(gltimesM)}, it corresponds to a continuous group $1$-cocycle 
${\sf c}\colon \mathcal{G} \to M, \exp(u) \mapsto (\frac{e^u-1}{u})b(u)$, 
which is exactly the integration of the Lie algebra cocycle ${\sf b}$. 
In other words, the diagram 
$$
\xymatrix{
\mathfrak{g} \ar[r]^{\hspace{1em} ({\rm id}_{\mathfrak{g}}, {\sf b}) \hspace{2em}} \ar[d]_{\exp} & 
\mathfrak{g} \ltimes M \ar[d]^{\Upsilon} \\
\mathcal{G} \ar[r]^{\hspace{1em} ({\rm id}_{\mathcal{G}},{\sf c}) \hspace{2em}} & \mathcal{G} \ltimes M
}
$$
commutes, and characterizes the integration ${\sf c}$ of the Lie algebra $1$-cocycle
${\sf b}$.
Conversely, any continuous group cocycle ${\sf c}\colon \mathcal{G} \to M$ defines a continuous Lie algebra homomorphism $\mathfrak{g} \to \mathfrak{g} \ltimes M, 
u \mapsto \big( u, (\textstyle\frac{u}{e^u-1}){\sf c}(\exp u) \big)$ through the isomorphism $\Upsilon$.
Hence the differentiation of the cocycle ${\sf c}$ is given by 
${\sf b}(u) = (\textstyle\frac{u}{e^u-1}){\sf c}(\exp u)$. 

Now we consider another aspect of Lie algebra $1$-cocycles, that is, an extension cocycle of the extension \eqref{eq:exntension_M}.  
Since $\mathbb{K}$ is a field, the extension~\eqref{eq:exntension_M} splits as $\mathbb{K}$-modules. Hence we identify $M' = \mathbb{K}p\oplus M$, where $p = i(1) \in M'$ is the image of $1 \in \K$ by $i$.
Since the action of $\mathfrak{g}$ on $\mathbb{K}$ is trivial, we have $u(p) \in M$ for any $u \in \mathfrak{g}$. Then the linear map ${\sf b}\colon \mathfrak{g} \to M$, $u\mapsto u(p)$, is a Lie algebra $1$-cocycle, which is called the extension cocycle of the extension~\eqref{eq:exntension_M}. 
Conversely, any Lie algebra $1$-cocycle ${\sf b}\colon \mathfrak{g} \to M$ defines an extension~\eqref{eq:exntension_M} of $\mathfrak{g}$-modules
by defining $u(p) = {\sf b}(u) \in M$ for any $u \in \mathfrak{g}$. \par
Suppose that $\mathfrak{g}$ is pro-nilpotent and equipped with a complete filtration, 
and that $M$ is equipped a complete filtration compatible with the filtration of $\mathfrak{g}$. Then we can take the exponentation $\mathcal{G}$ of the Lie algebra $\mathfrak{g}$, and regard
the extension~\eqref{eq:exntension_M} as an extension of $\mathcal{G}$-modules.
Its extension cocycle ${\sf c}$ given by 
\[
{\sf c}(\exp(u)) = \exp(u)(p) - p
= \left(\frac{e^u-1}{u}\right)(u(p))
= \left(\frac{e^u-1}{u}\right){\sf b}(u)
\]
is nothing but the integration of the Lie algebra $1$-cocycle ${\sf b}$. 
Conversely, if an extension~\eqref{eq:exntension_M} of $\mathcal{G}$-modules is given by a group $1$-cocycle ${\sf c}\colon \mathcal{G} \to M$, then the induced Lie action of $\mathfrak{g}$ make it an extension of $\mathfrak{g}$-modules. 
Then its extension cocycle ${\sf b}\colon \mathfrak{g} \to M$ is given by 
$$
{\sf b}(u) = u(p) =  \left(\frac{u}{e^u-1}\right)(e^u-1)(p) = 
\left(\frac{u}{e^u-1}\right){\sf c}(\exp(u))
$$
for any $u \in \mathfrak{g}$. 

Since the constant term of $\left(\frac{e^u-1}{u}\right)$ equals $1$, the formula ${\sf c}(\exp(u)) = \textstyle\frac{e^u - 1}{u}({\sf b}(u))$ yields 
$$
\left. \frac{d}{dt}\, {\sf c}(\exp(tu)) \right|_{t=0} 
= \left. \frac{d}{dt}\left( \frac{e^{tu}-1}{tu} \right) \big( {\sf b}(tu) \big) \right|_{t=0}
= \left. \frac{e^{tu}-1}{tu} \right|_{t=0} \big( {\sf b}(u) \big)
= {\sf b}(u).
$$

\begin{rem}
The discussion above establishes the following statement: the integration of a Lie algebra $1$-cocycle ${\sf b}\colon \mathfrak{g} \to M$ is a unique group $1$-cocycle ${\sf c} \colon \mathcal{G} \to M$ satisfying $\left. \frac{d}{dt}\, {\sf c}(\exp(tu)) \right|_{t=0} = {\sf b}(u)$.
Indeed, let ${\sf c}'\colon \mathcal{G} \to M$ be a group $1$-cocycle satisfying $\left. \frac{d}{dt}\, {\sf c}'(\exp(tu)) \right|_{t=0} = {\sf b}(u)$.
There is a Lie algebra $1$-cocycle ${\sf b}'\colon \mathfrak{g} \to M$ which integrates to ${\sf c}'$.
Since $\left. \frac{d}{dt}\, {\sf c}'(\exp(tu)) \right|_{t=0} = {\sf b}'(u)$ as well, we conclude that ${\sf b}' = {\sf b}$ and ${\sf c}'$ is the integration of ${\sf b}$.
\end{rem}

\subsubsection*{Proof of Proposition~\ref{prop:bFu}}

The Lie algebra $\mathfrak{g}$ acts on $Z^1(\mathfrak{g},M)$ in a natural way: for $v\in \mathfrak{g}$ and ${\sf b}\in Z^1(\mathfrak{g},M)$, we have $(v\cdot {\sf b})(u) = v({\sf b}(u)) - {\sf b}([v,u])$.
By the $1$-cocycle property of ${\sf b}$, the right hand side is equal to $u({\sf b}(v)) = d({\sf b}(v))(u)$.
Namely, we have $v\cdot {\sf b} = d({\sf b}(v))$.
In particular, we have $v\cdot dm = d(dm(v)) = d(v(m))$ for any $v \in \mathfrak{g}$ and $m \in M$.
Restricted to the Lie subalgebra $\mathfrak{g}_+$, this action integrates to the group action ${\sf b} \mapsto {\sf b}^F$ of the group $\mathcal{G}_+$ on $Z^1(\mathfrak{g},M)$.

Returning to the situation of Proposition~\ref{prop:bFu}, we have $F = \exp(v)$ for some $v \in \mathfrak{g}_+$.
Applying $\big( \frac{e^v -1}{v} \big)$ to the both sides of the equation $v \cdot {\sf b} = d({\sf b}(v))$, we compute
\[
{\sf b}^F - {\sf b} = \frac{e^v -1}{v} v\cdot {\sf b} = \frac{e^v -1}{v} d({\sf b}(v)) = d\left(\frac{e^v -1}{v} ({\sf b}(v)) \right)
= d \big( {\sf c}(\exp(v)) \big) = d({\sf c}(F)).
\]
This completes the proof. \qed

\subsection{Tips for non-commutative calculus} \label{subsec:tnc}

In this section, we collect formulas about calculus with non-commutative algebras used in the main body of the text.

Let $\mathfrak{A}$ be a (topological) augmented associative $\K$-algebra and $\mathfrak{B}$ an $\mathfrak{A} \otimes \mathfrak{A}^{\rm op}$-module.
From time to time, we use the following shorthand notation for the action of $\mathfrak{A} \otimes \mathfrak{A}^{\rm op}$:
\[
(a \otimes a')(b) = aba', \quad
\text{where $a, a' \in \mathfrak{A}$ and $b \in \mathfrak{B}$}.
\] 
A $\K$-linear map $u \colon \mathfrak{A} \to \mathfrak{B}$ is called a {\em derivation} if 
\[
u(aa') = (1 \otimes a')(u(a)) + (a\otimes 1)(u(a'))
= u(a) a' + a u(a')
\]
for any $a, a' \in \mathfrak{A}$.
For instance, if $\mathfrak{B} = \mathfrak{A}$ is regarded as an $\mathfrak{A} \otimes \mathfrak{A}^{\rm op}$-module in the natural way (namely $(a \otimes a')(b) := aba'$, where we use the multiplication of $\mathfrak{A}$ in the right hand side), then this definition coincides with the usual definition of derivations of $\mathfrak{A}$.

\subsubsection*{Differentiation under the $|\cdot |$-sign}

Let $|\mathfrak{A}| = \mathfrak{A}/[\mathfrak{A}, \mathfrak{A}]$ be the trace space of $\mathfrak{A}$.

\begin{lem} \label{lem:ufaga}
Let $u \in {\rm Der}(\mathfrak{A})$, and let $f(s), g(s) \in \K[[s]]$. For any $a \in \mathfrak{A}_{\ge 1}$, one has
\[
|u(f(a)) g(a)| = |u(a) f'(a) g(a) |.
\]
\end{lem}

\begin{proof}
We may assume that $f(s) = s^m$.
Since $u(f(a)) = \sum_{j=1}^m a^{j-1} u(a) a^{m-j}$, 
\[
|u(f(a))g(a)|
=  \sum_{j=1}^m | a^{j-1} u(a) a^{m-j} g(a)|
=  | u(a) m a^{m-1} g(a) |
=  | u(a) f'(a) g(a) |.
\]
Here we have used the fact that $a$ and $g(a)$ commute.
\end{proof}

Any derivation $u\in {\rm Der}(\mathfrak{A})$ induces a $\K$-linear map $u\colon |\mathfrak{A}| \to |\mathfrak{A}|$ defined by $u(|a|) := |u(a)|$, since $u([a,a']) = [u(a),a'] + [a,u(a')] \in [\mathfrak{A}, \mathfrak{A}]$ for any $a,a' \in \mathfrak{A}$.

\begin{exple} \label{ex:ulog}
Put $f(s) = \log (1+s) = \sum_{k=1}^{\infty} \frac{(-1)^{k-1}}{k} s^k$ and $g(s) = 1$.
Then, $f'(s) = \sum_{k=0}^{\infty} (-1)^k s^k =(1+s)^{-1}$.
Lemma~\ref{lem:ufaga} yields the following: for any $u \in {\rm Der}(\mathfrak{A})$ and $a \in \mathfrak{A}_{\ge 1}$, 
\[
u \left( |\log(1+a) | \right) = |u(a) (1+a)^{-1}|.
\]
\end{exple}

\subsubsection*{Partial differential operator}

In what follows, we consider the case $A = \K \langle \langle z_1, \ldots, z_n \rangle \rangle$.
We regard $A \otimes A$ as an $A \otimes A^{\rm op}$-module by  formula:
\[
(a\otimes a').(b \otimes b') := ab \otimes b'a',
\quad
\text{where $a, a', b, b' \in A$}.
\]

For each $j \in \{ 1,\ldots, n\}$, define the operator 
\[
\partial_j = \partial_{z_j} \colon A \to A \otimes A
\]
as the unique derivation defined by $\partial_j z_k = \delta_{jk} (1 \otimes 1)$.
Using the notation of Sweedler type, we denote
\[
\partial_j a = (\partial_j^{'}a) \otimes (\partial''_j a).
\]

\begin{lem}
Let $B$ be an $A\otimes A^{\rm op}$-module and let $u \colon A \to B$ be a derivation.
Then, for any $a\in A$,  we have
\[
u(a) = \sum_j (\partial_j a).u(z_j) = 
\sum_j (\partial_j^{'}a) u(z_j) (\partial''_j a).
\]
\end{lem}

\begin{proof}
First one checks the formula on the generators $\{ z_j\}_j$.
Second one notices that the map sending $a \in A$ to the right hand side of the formula is a derivation.
Hence the formula holds for any $a\in A$.
\end{proof}

\begin{exple} \label{ex:a11a}
For any $a \in A$, 
\[
a \otimes 1 - 1 \otimes a = \sum_j
\left( (\partial'_j a)z_j \otimes (\partial''_j a)
- (\partial'_j a) \otimes z_j (\partial''_j a) \right).
\]
To see this, note that the map $A \to A \otimes A, a \mapsto a \otimes 1 - 1 \otimes a$ is a derivation.
Hence $a \otimes 1 - 1 \otimes a =
\sum_j \partial_j(a).(z_j \otimes 1 - 1 \otimes z_j)$.
\end{exple}

The following is a non-commutative version of the chain rule formula.

\begin{exple} \label{ex:chainrule}
Let $\psi\colon A \to B$ be a $\K$-algebra homomorphism and $\partial\colon B \to B \otimes B$ a derivation.
Then, for any $a \in A$,
\[
\partial (\psi(a)) = \sum_j
\psi(\partial'_j(a)) \partial(\psi(z_j)) \psi(\partial''_j(a)).
\]
To see this, notice that $\partial \circ \psi \colon A \to B \otimes B$ is a derivation, where we regard $B \otimes B$ as an $A\otimes A^{\rm op}$-module through $\psi$: the action is given by $(a \otimes a').(b \otimes b') = \psi(a)b \otimes b'\psi(a')$.
Hence 
\[
\partial (\psi(a)) = 
\sum_j (\partial_j a).(\partial \psi(z_j))
= \sum_j
\psi(\partial'_j(a)) \partial(\psi(z_j)) \psi(\partial''_j(a)).
\]
\end{exple}

As in Section~\ref{subsec:div_der_Lie}, there is also a free Lie algebra version of partial derivatives, for which we have the following chain rule formula:
\begin{lem} \label{ex:chainrule_Lie}
Let $\psi \colon L(z_1,\ldots,z_n) \to L(w_1, \ldots, w_m)$ be a Lie algebra homomorphism between free Lie algebras.
For any $u \in L(z_1, \ldots, z_n)$ and $1 \le j \le m$, we have
\[
\partial_{w_j} \psi(u)
= \sum_{i=1}^n \psi (d_{z_i}u) d_{w_j}(\psi(z_i)).
\]
\end{lem}
\begin{proof}
The map $\psi$ naturally extends to a $\K$-algebra homomorphism $A(z_1,\ldots,z_n) \to A(w_1,\ldots, w_m)$ between free associative algebras.
Since $u = \sum_{i=1}^n (d_{z_i}u)z_i$, we have
\[
\psi(u) = \sum_{i=1}^n \psi(d_{z_i} u) \psi(z_i)
= \sum_{i=1}^n \psi(d_{z_i} u) \sum_{j=1}^m d_{w_j}(\psi(z_i)) w_j.
\]
The result follows from this equation.
\end{proof}

In the proof of Propositions~\ref{prop:ell_Lie} and \ref{prop:delta_2n}, we use the following lemma.
\begin{lem} \label{lem:dsfadxa}
Let $a \in A(z_1,\ldots,z_n)$ and $f(s) \in s\K[[s]]$.
For any $1\le j\le n$, we have
\[
| d_{z_j} (f({\rm ad}_{z_j})(a)) | = |f(z_j) d_{z_j}(a) - \textstyle\frac{f(z_j)}{z_j} a |. 
\]
\end{lem}
\begin{proof}
It is sufficient to prove the case where $f(s) = s^m$, $m\ge 1$.
To simplify notation, we denote $z = z_j$.
Since ${{\rm ad}_{z}}^m (a) = \sum_{i=0}^m (-1)^i \binom{m}{i} z^{m-i} a z^i$, we compute
\begin{align*}
| d_z( {{\rm ad}_z}^m(a) ) |
& = | z^m d_z(a) + \textstyle\sum_{i=1}^m (-1)^i \binom{m}{i} z^{m-i} a z^{i-1} | \\ 
& = | z^m d_z(a) + \textstyle\sum_{i=1}^m (-1)^i \binom{m}{i} z^{m-1} a | \\
& = |z^m d_z(a) - z^{m-1} a |,
\end{align*}
which proves the assertion.    
\end{proof}

\subsection{Details on tangential derivations and automorphisms} \label{subsec:tgtl}

In this subsection, we prove several statements on the structure of tangential derivations and automorphisms.

Let $A = \K \langle \langle z_1, \ldots, z_n \rangle \rangle$ be the free associative algebra endowed with the weight filtration and $\mathcal{T} = \{ t_1, \ldots, t_k \}$ a finite subset of $A$.
As in \S 4.10, we consider the Lie algebra of tangential derivations of $A$:
\[
{\rm tDer}_{\mathcal{T}}(A)
= \{ (u, u_1, \ldots, u_k) \in {\rm Der}(A) \times A^k ; u(t_j) = [t_j, u_j] \}.
\]
We have the injective Lie algebra homomorphism
\[
{\rm tDer}_{\mathcal{T}}(A)
\longrightarrow 
{\rm Der}(A) \ltimes A^k,
\quad
\tilde{u} = (u,u_1,\ldots,u_k) \mapsto (u, (u_1, \ldots, u_k))
\]
to the semi-direct product of ${\rm Der}(A)$ and $A^k$, where $A^k$ is viewed as a Lie algebra under the algebra commutator of components.

Let $L = L(z_1, \ldots, z_n)$ be the (completed) free Lie algebra.
If all elements $t_j$ belong to $L$, one can define the Lie algebra of tangential derivations of $L$:
\[
{\rm tDer}_{\mathcal{T}}(L) = 
\{ (u, u_1, \ldots, u_k) \in {\rm Der}(L) \times L^k ; u(t_j) = [t_j, u_j] \}.
\]
It is a Lie subalgebra of ${\rm tDer}_{\mathcal{T}}(L)$ and injects to the semi-direct product ${\rm Der}(L) \ltimes L^k$.

In what follows, we use the following convention:
\begin{itemize}
    \item we use the shorthand notation $\vec{u} = (u_1, \ldots, u_k)$.
    \item
In ${\rm Der}(A) \ltimes A^k$, we sometime express elements additively: we simply denote $(u, \vec{b}) = u + \vec{b}$.
For instance, one can write elements of ${\rm tDer}_{\mathcal{T}}(A)$ as $\tilde{u} = u + \vec{u}$.
Note that for $u \in {\rm Der}(A)$ and $\vec{b} \in A^k$, their Lie brackets in ${\rm Der}(A) \ltimes A^k$ is given by 
\begin{equation} \label{eq:uvecb}
[u, \vec{b}] = u(\vec{b}) = (u(b_1), \ldots, u(b_k)).
\end{equation}
\end{itemize}

\subsubsection*{The exponentiation of tangential derivations}

We give an explicit description of the exponentiation of positive tangential derivations.
For simplicity, we restrict ourselves to the case of free Lie algebra. 
Suppose that we assign positive integral weights $w_j = {\rm wt}(z_j)$, and assume that all elements $t_j$ belong to $L$.
Let ${\rm Der}^+(L)$ be the space of positive derivations with respect to the weight filtration.
Then the BCH series converges in ${\rm Der}^+(L) \ltimes L^k$, and one can define the exponentiation $\exp({\rm Der}^+(L) \ltimes L^k)$.
Let
\[
{\rm tDer}^+_{\mathcal{T}}(L) : = {\rm tDer}_{\mathcal{T}}(L) \cap ({\rm Der}^+(L) \ltimes L^k)
\]
and
\[
{\rm tAut}_{\mathcal{T}}^+(L) := \exp({\rm tDer}^+_{\mathcal{T}}(L)).
\]
Any element of ${\rm tAut}_{\mathcal{T}}^+(L)$ is denoted by $\exp(\tilde{u})$ for some $\tilde{u} \in {\rm tDer}^+_{\mathcal{T}}(L)$, and the group structure is given by the BCH series: $\exp(\tilde{u}) \exp(\tilde{v}) = \exp(\bch(\tilde{u}, \tilde{v}))$.
There is a natural group homomorphism
\[
\rho \colon {\rm tAut}^+_{\mathcal{T}}(L) \to {\rm Aut}^+(L), 
\quad 
\exp(\tilde{u}) \mapsto \exp(u) = \sum_{m=0}^{\infty} \frac{1}{m!} u^m.
\]

Recall that ${\rm Aut}^+(L)$ is the subgroup of ${\rm Aut}(L)$ consisting of automorphisms $F$ such that $F(z_j) \in z_j + L_{> w_j}$ for all $j$.
We introduce the following notation:
\[
\widetilde{{\rm Aut}}^+(L)
:=  \{ (F, f_1, \ldots, f_k) \in {\rm Aut}^+(L) \times \exp(L)^k ; F(t_j) = f_j^{-1} t_j f_j \}.
\]
The group structure of $\widetilde{{\rm Aut}}^+(L)$ is given as follows: for $\tilde{F} = (F, f_1, \ldots, f_k), \tilde{G} = (G, g_1, \ldots, g_k)$,
\[
\tilde{F} \cdot \tilde{G} = (FG, f_1F(g_1), \ldots, f_kF(g_k)).
\]
There is a natural group homomorphism
\[
\rho' \colon \widetilde{{\rm Aut}}^+(L) \to {\rm Aut}^+(L), 
\quad 
\tilde{F} \to F.
\]

\begin{prop} \label{prop:tAut_2des}
There is a group isomorphism
\[
{\rm tAut}_{\mathcal{T}}^+(L) \cong
\widetilde{{\rm Aut}}^+(L)
\]
compatible with the two maps to ${\rm Aut}^+(L)$, $\rho$ and $\rho'$.
\end{prop}

\begin{rem} \label{rem:tAut(A)}
\begin{enumerate}
\item[(i)] 
Based on this observation, one can define the group of all tangential automorphisms to be 
\[
{\rm tAut}_{\mathcal{T}}(L) :=
\{ (F, f_1, \ldots, f_k) \in {\rm Aut}(L) \times \exp(L)^k ; F(t_j) = f_j^{-1} t_j f_j \}
\]
with the same group structure as that of $\widetilde{{\rm Aut}}^+(L)$.
\item[(ii)]
In the case of the free associative algebra $A$, one can similarly define
\[
{\rm tDer}^+_{\mathcal{T}}(A) : = {\rm tDer}_{\mathcal{T}}(A) \cap ({\rm Der}^+(A) \ltimes A^k)
\]
which exponentiates to 
\[
{\rm tAut}_{\mathcal{T}}^+(A) := \exp({\rm tDer}^+_{\mathcal{T}}(A)).
\]
This group is isomorphic to the direct product of the group 
\[
\{ (F, f_1, \ldots, f_k) \in {\rm Aut}^+(A) \times \exp(A_+)^k ; F(t_j) = f_j^{-1} t_j f_j \}
\]    
with the abelian factor $\K^k$ corresponding to the constant terms in $A^k$.
\end{enumerate}
\end{rem}

\begin{proof}[Proof of Proposition \ref{prop:tAut_2des}]
The proof consists of three steps.

{\em Step 1.}
First we describe the map $\rho$ in terms of tangential components. Let $\tilde{u} \in {\rm tDer}^+_{\mathcal{T}}(L)$. 
We define $\vec{f} = (f_1, \ldots, f_k) \in \exp(L)^k$ by 
\begin{equation} \label{eq:logf}
\log \vec{f} = (\log f_1, \ldots, \log f_k):=
\bch(u + \vec{u}, -u).
\end{equation}
By \eqref{eq:uvecb}, the right hand side converges in $L^k \subset {\rm Der}^+(L) \ltimes L^k$.
More explicitly
\begin{equation} \label{eq:fu}
\log \vec{f} = \vec{u} + \frac{1}{2} u(\vec{u}) 
+ \frac{1}{12} \big( 2 u( u(\vec{u}) ) + [\vec{u}, u(\vec{u})] \big) + \cdots.
\end{equation}

\begin{lem} \label{lem:expuai}
For all $1 \le j \le k$, we have $\rho(\tilde{u})(t_j) = \exp(u)(t_j) = f_j^{-1} t_j f_j$.
\end{lem}

\begin{proof}
Since $u(t_j) = [t_j, u_j]$, we have 
$
[ u + \vec{u}, \vec{t} ] = u(\vec{t}) + [\vec{u}, \vec{t}] = 0
$
in ${\rm Der}^+(L) \ltimes L^k$.
Therefore $\exp( u + \vec{u})$ and $\exp( \vec{t})$ commute in the exponentiation. 
Using this fact, we compute
\begin{align*}
\exp(u)(\vec{t}) = \exp({\rm ad}_u)(\vec{t})
& = u * \vec{t} * (-u) \\
& = \left( u* (-u - \vec{u}) \right) * \vec{t} * \left( (u + \vec{u} ) * (-u) \right) \\ 
& = (- \log \vec{f}) * \vec{t} * (\log \vec{f}) \\
& = (f_1^{-1}t_1f_1, \ldots, f_k^{-1}t_k f_k).
\end{align*}
Here we have used \eqref{eq:uvecb} in the first line.
This completes the proof.
\end{proof}

{\it Step 2.}
We construct a map
$\Psi \colon {\rm tAut}_{\mathcal{T}}^+(L) \to \widetilde{{\rm Aut}}^+(L)$ as follows.
For $\exp(\tilde{u}) \in {\rm tAut}_{\mathcal{T}}^+(L)$, we set 
\[
\Psi(\exp(\tilde{u})) := 
\left( (\exp(u), f_1,\ldots, f_k), \vec{u}_0 \right),
\]
where $\log( \vec{f} )$ is defined by \eqref{eq:logf}.
By Lemma \ref{lem:expuai}, it holds that $\rho = \rho' \circ \Psi$. 
It is straightforward to show that $\Psi$ is a group homomorphism.

{\it Step 3.}
Finally we give the inverse map of $\Psi$.
Let $\tilde{F} = (F, f_1, \ldots, f_k) \in \widetilde{{\rm Aut}}^+(L)$.
Then the derivation $\log F = \sum_{m=1}^{\infty} ((-1)^{m-1}/m )(F-{\rm id})^m$ converges and is an element of ${\rm Der}^+(L)$.
We define $\Phi(\tilde{F}) \in L^k$ by
\[
\Phi(\tilde{F}) := \log (\vec{f}) * \log F - \log F. 
\]
By \eqref{eq:uvecb}, the right hand side is a well-defined element in $L^k \subset {\rm Der}^+(L) \ltimes L^k$.
Then, using the same letter, we have the following map:
\[
\Phi \colon \widetilde{{\rm Aut}}^+(L) \to {\rm tAut}^+_{\mathcal{T}}(L), \quad
\tilde{F} \mapsto \exp( (\log F, \Phi(\tilde{F}))).
\]
We can check that the map $\Phi$ is the inverse of $\Psi$.

This completes the proof of Proposition \ref{prop:tAut_2des}.
\end{proof}

\subsubsection*{The adjoint action}

The group ${\rm tAut}^+_{\mathcal{T}}(L)$ is a subgroup of ${\rm tAut}^+_{\mathcal{T}}(A)$ (see Remark~\ref{rem:tAut(A)}), which acts on ${\rm tDer}_{\mathcal{T}}(A)$ by the adjoint action.
The formula is as follows: 
for $\tilde{F} = \exp(\tilde{v}) \in {\rm tAut}^+_{\mathcal{T}}(A)$ and $\tilde{u} \in {\rm tDer}_{\mathcal{T}}(A)$, 
\[
{\rm Ad}_{\tilde{F}}(\tilde{u}) := 
\exp({\rm ad}_{\tilde{v}})(\tilde{u}) 
= \tilde{v} * \tilde{u} * (- \tilde{v}).
\]
Restricting ourselves to ${\rm tAut}^+_{\mathcal{T}}(L)$, we give an explicit formula for this action.
With a suitable modification the result extends to ${\rm tAut}^+_{\mathcal{T}}(A)$ as well, but we omit it.

\begin{prop} \label{prop:tad}
Let $\tilde{F} \in {\rm tAut}^+(L)$ and write it $\tilde{F} = (F, f_1,\ldots, f_k)$ using the perspective of Proposition~\ref{prop:tAut_2des}, and let $\tilde{u} = (u, u_1,\ldots,u_k) \in {\rm tDer}_{\mathcal{T}}(A)$.
Then ${\rm Ad}_{\tilde{F}}(\tilde{u}) = ({\rm Ad}_{F}(u), u'_1,\ldots,u'_k)$, where 
\[
u'_j = f_j \cdot {\rm Ad}_F(u)(f_j^{-1}) + f_j F(u_j) f_j^{-1}.
\]
\end{prop}

To prove Proposition \ref{prop:tad}, we need the following lemma.

\begin{lem} \label{lem:bu-b}
In the semi-direct product ${\rm Der}(A) \ltimes A_+$, the following formula holds: 
for any $u\in {\rm Der}(A)$ and $b \in A_+$,
\[
b * u * (-b) = \exp({\rm ad}_b)(u) = e^b\cdot u(e^{-b}).
\]
\end{lem}

\begin{proof}
Since ${\rm ad}_b(u) = [b,u] = -u(b)$ in ${\rm Der}(A) \ltimes A_+$, we have
\[
\exp({\rm ad}_b)(u) = - \frac{\exp({\rm ad}_b) - 1}{{\rm ad}_b} u(b).
\]
Moreover, this is equal to $e^b \cdot u(e^{-b})$ (see \cite[Lemma 6.4]{genus0} for a similar formula).
\end{proof}

\begin{proof}[Proof of Proposition \ref{prop:tad}]
Since the natural map ${\rm tDer}_{\mathcal{T}}(A) \to {\rm Der}(A)$ is a Lie algebra homomorphism, the first component of ${\rm Ad}_{\tilde{F}}(\tilde{u})$ is ${\rm Ad}_F(u)$.
To look at the tangential components $u'_j$, write $\tilde{F} = \exp(\tilde{v})$, where $\tilde{v} \in {\rm tDer}^+_{\mathcal{T}}(L)$.
Since ${\rm Ad}_{\tilde{F}}$ is linear, we have ${\rm Ad}_{\tilde{F}}(\tilde{u}) = {\rm Ad}_{\tilde{F}}(u) + {\rm Ad}_{\tilde{F}}(\vec{u})$.
Note that $\tilde{v} = v + \vec{v} = ((v + \vec{v})*(-v)) * v = \log \vec{f} * v$.
For the first term, we compute
\begin{align*}
{\rm Ad}_{\tilde{F}}(u) & = \log \vec{f} * (v * u * (-v) ) * ( -\log \vec{f}) \\
& = \log \vec{f} * {\rm Ad}_F(u) * (-\log \vec{f}) \\
& = (f_1\cdot {\rm Ad}_F(u) (f_1^{-1}), \ldots, f_k \cdot {\rm Ad}_F(u) (f_k^{-1})).
\end{align*}
Here we have used Lemma \ref{lem:bu-b} in the last line.
For the second term, we compute
\begin{align*}
{\rm Ad}_{\tilde{F}}(\vec{u})
& = \log \vec{f} * ( v * \vec{u} * (-v) ) * ( - \log \vec{f}) \\
& = \log \vec{f} * F(\vec{u}) * ( - \log \vec{f}) \\
& = (f_1 F(u_1) f_1^{-1}, \ldots, f_k F(u_k) f_k^{-1}).
\end{align*}
This completes the proof.
\end{proof}


\end{document}